%% file: main.tex
\documentclass{article}
\usepackage[utf8]{inputenc}
\usepackage[dvipsnames]{xcolor}
\usepackage{amsfonts,amsmath,fullpage,amssymb,amsthm,tikz,mathtools,lmodern,paralist,mathrsfs,enumitem,tikz-3dplot,cases,thmtools,bm}
\usepackage{subcaption}
\usepackage{graphicx}
\usepackage{placeins}
\usepackage[color]{xy}
\usepackage{graphicx}
\usepackage[backref=page]{hyperref}
\hypersetup{colorlinks, citecolor=red, filecolor=black, linkcolor=blue, urlcolor=blue}
\usepackage{cleveref}
\usepackage{mathalpha}
\usepackage[hang]{footmisc}
\DeclareMathAlphabet{\mathdutch}{U}{dutchcal}{m}{n}

\setlist[enumerate]{itemsep=0.18em, topsep=0.36em, parsep=0.08em}
\setlist[itemize]{itemsep=0.18em, topsep=0.36em, parsep=0.08em}

\theoremstyle{plain}
\newtheorem{theorem}{Theorem}[section]
\newtheorem{lemma}[theorem]{Lemma}
\newtheorem{proposition}[theorem]{Proposition}
\newtheorem{corollary}[theorem]{Corollary}

\newtheorem{conjecture}[theorem]{Conjecture}

\theoremstyle{definition}
\newtheorem{definition}[theorem]{Definition}

\newtheorem{remark}[theorem]{Remark}
\newtheorem{example}[theorem]{Example}

\theoremstyle{remark}

\newtheorem{notation}{Notation}

\title{Polytopal Bases for Barycentric Subdivisions}
\author{Spencer Backman and Federico Castillo}
\date{}

\begin{document}

\maketitle
\begin{abstract}
Several important fans studied at the interface of combinatorics and algebraic geometry arise from barycentric subdivisions of other fans; the central example is the braid arrangement.  The collection of faces of the standard simplex forms a basis for the deformation cone of the braid arrangement, i.e. the cone of generalized permutahedra.  We reinterpret this simplicial basis as the collection of deep truncations of the standard simplex, and systematically abstract this perspective to produce polytopal bases for the deformation cones of barycentric subdivisions of simplicial projective fans.  We further demonstrate that these bases restrict to bases for deformation cones of fans obtained by a sequence of stellar subdivisions induced by a building set.  When the starting polytope is smooth with all edge lengths equal to 1, we show that we can upgrade our basis to a collection of flat truncations.
We then investigate two extensions of the braid arrangement where we are able to further upgrade our flat truncation basis to an indecomposable polytopal basis:
	\begin{itemize}
		\item the barycentric subdivision of the normal fan of a product of standard simplices and
		\item  the barycentric subdivision of the braid arrangement.
	\end{itemize}
    
We discuss connections to, and implications for, Archimedean solids and regular polytopes, permutahedral plates, root polytopes, permutoassociahedra, simple permutonestohedra, cosmohedra, omnitruncations of Coxeter permutahedra,  bipermutahedra, $\pi$-colored fans, and tropical $\alpha$ and $\beta$ classes.
\end{abstract}

\setcounter{tocdepth}{2}
\tableofcontents

\section{Introduction}

The collection of all polytopes whose normal fans coarsen a fixed projective fan $\Sigma$ forms a polyhedral cone called the \textbf{deformation cone} of $\Sigma$, denoted $\operatorname{Def}(\Sigma)$ \cite{mcmullen1973representations,castillo2022deformation,padrol2023deformation}.\footnote{We may also refer to the deformation cone of a polytope, by which we mean the deformation cone of its normal fan.  By this convention, any two polytopes with the same normal fan have the same deformation cone.}  This cone is of intrinsic interest to polyhedral geometers as the cone of weakly convex functions on the fan $\Sigma$, but it also plays an important role in the setting of toric geometry where it is naturally identified with the cone of nef divisors on the projective toric variety associated to $\Sigma$ \cite{demazure1970sous, oda1988convex}.\footnote{In the toric setting, $\Sigma$ is always taken to be rational.}  If $\mathscr{P} \subset \operatorname{Def}(\Sigma)$ is a collection of polytopes such that every polytope in $\operatorname{Def}(\Sigma)$ can be expressed uniquely as a signed Minkowski sum of polytopes from $\mathscr{P}$, then we say that $\mathscr{P}$ is a \textbf{polytopal basis} for $\operatorname{Def}(\Sigma)$.\footnote{We may simply say that $\mathscr{P}$ is a basis for $\Sigma$ if no confusion will arise.}

  The braid arrangement in $\mathbb{R}^n$ is the collection of all hyperplanes $\{x_i-x_j=0:1\leq i <j \leq n\}$.
As with any central hyperplane arrangement, the braid arrangement can alternately be viewed as a projective fan, which is the normal fan of a zonotope; in the case of the braid arrangement, this zonotope is the standard permutahedron (see Figure \ref{fig:standard-permutahedron-intro}).

\begin{figure}[ht]
    \centering
    \includegraphics[width=.25\textwidth]{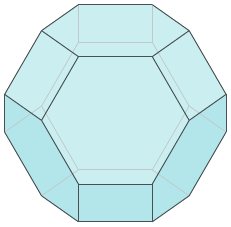}
    \caption{The standard permutahedron is a translation of the Minkowski sum of the line segments which connect type-A positive roots to the origin.}
    \label{fig:standard-permutahedron-intro}
\end{figure}

Generalized permutahedra are the polytopes with edge directions parallel to the vectors $\{e_i-e_j:1\leq i < j \leq n\}$.  These polytopes are ubiquitous in mathematics and make appearances in computer science, physics, and economics \cite{danilov2000cores, morton2009convex, goemans2015smallest, mohammadi2018generalized, early2024generalized}.  They can alternately be characterized as those polytopes whose normal fans coarsen the braid arrangement \cite{postnikov2009permutohedra, postnikov2008faces, ardila2020coxeter, castillo2022deformation}, hence the collection of all generalized permutahedra corresponds to the deformation cone of the braid arrangement.\footnote{Several of the objects mentioned here have alternate names in the literature: the braid arrangement is also the type-A Coxeter arrangement.
When viewing this arrangement as a fan, some authors use the name permutahedral fan, e.g. \cite{adiprasito2018hodge}.    
The generalized permutahedra are translation equivalent to the polymatroid base polytope of Edmonds \cite{edmonds1970submodular}, and some authors naturally use the term deformed permutahedra; see Pilaud \cite{pilaud2017nestohedra}.  
The deformation cone of a polytope is the closure of McMullen's type cone \cite{mcmullen1973representations}.  When $\Sigma$ is rational, McMullen’s type cone, modulo translations, is equivalent to the ample cone of the associated toric variety.  The deformation cone of the braid arrangement is sometimes called the submodular cone as its points can be identified with submodular functions \cite{edmonds1970submodular,loho2025manyrays}.}

It was shown independently by Danilov--Koshevoy \cite{danilov2000cores} and  Ardila--Benedetti--Doker \cite{ardila2010matroid} that every generalized permutahedron can be expressed uniquely as a signed Minkowski sum of standard simplices, i.e. the standard simplices form a polytopal basis for the cone of generalized permutahedra.
The utility of such a description was demonstrated by Postnikov \cite{postnikov2009permutohedra} who applied it to give a volume formula for generalized permutahedra.

Fine Bergman fans of matroids were introduced by Ardila--Klivans \cite{ardila2006bergman} 
as a distinguished collection of  subfans of the braid arrangement.  
Utilizing the Chow ring of a matroid, due to Feichtner--Yuzvinsky \cite{feichtner2004chow}, Adiprasito--Huh--Katz \cite{adiprasito2018hodge} settled Rota's conjecture on the log-concavity of the characteristic polynomial.  
In work of Yuzvinsky \cite{yuzvinsky2002small}  and Backman--Eur--Simpson \cite{backman2023simplicial}, it was shown that the restrictions of the standard simplices to a fine Bergman fan give a nef generating set for the Chow ring of a matroid, and the latter set of authors employed this generating set for giving a second proof of Rota's conjecture.  

The braid arrangement can alternately be described as the barycentric subdivision of the normal fan of the standard simplex.\footnote{From this perspective it may appear as something of a miracle that the braid arrangement is indeed a hyperplane arrangement; however, we find that this is a general phenomenon for regular polytopes; see Remark \ref{rmk:regularpoly}.}
In recent years, authors have become increasingly interested in investigating collections of fans which generalize the braid arrangement, the deformation cones of such fans, and the restrictions of these support functions to certain distinguished subfans and coarsenings.   
While these various fans have different motivations and constructions, they can often be understood as originating with the barycentric subdivision of some projective fan.  
The main examples of this phenomenon, at the time of writing, are the type-B Coxeter arrangement \cite{bastidas2021polytope, eur2024signed}, the $\pi$-colored fans \cite{clader2024multimatroids}, the nested braid fan \cite{castillo2022deformation}, the poset associahedra \cite{galashin2024p,sack2023realization}, the acyclonestohedra \cite{mantovani2025facial}, the permutoassociahedron \cite{kapranov1993permutoassociahedron, reiner1994coxeter, castillo2023permuto}, the permutonestohedra \cite{gaiffi2015permutonestohedra}, the cosmohedra \cite{arkani2025cosmohedra, ardila2026combinatorics} and, to some extent, the bipermutahedron \cite{ardila2023lagrangian, ardila2022bipermutahedron}.

Researchers often view the simplicial basis for generalized permutahedra as the set of faces of the standard simplex (equivalently, as encoding pullbacks of maps to projective space).  However, this interpretation does not naturally generalize to produce bases for deformation cones of other polytopes.
Our work begins with a conceptually different interpretation of the family of standard simplices:
\begin{center}
    
{\bf The standard simplices in $\mathbb{R}^n$ are the deep truncations of a single standard simplex.}   
	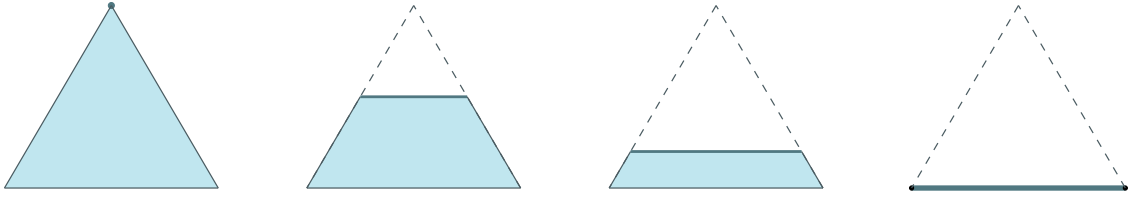
\begin{figure}[ht]
		\centering
		\tdplotsetmaincoords{45}{135}
         \input{tikz/triangle.tex}
		\caption{A standard simplex, two  shallow truncations, and a deep truncation realizing a line segment.}
\label{fig:split}
	\end{figure}
\end{center}

In this article we systematically abstract this perspective for constructing polytopal bases for deformation cones of fans which are barycentric subdivisions of simplicial projective fans.

We first provide an informal treatment of the terminology which is necessary for understanding the statement of our main results; we refer the reader to Subsection \ref{subsec:truncations} for further details.  By \textbf{truncating} a face $F$ of a polytope in direction $u$ by distance $\epsilon$, we mean to take the polytope obtained by moving the $F$-supporting hyperplane in direction $u$ by distance $\epsilon$ towards the polytope.  We typically take our truncations to be \textbf{central} so that $u$ is the sum of fixed normals for the facets containing the supporting face $F$.  We distinguish two notions of truncation of a face $F$: a \textbf{shallow truncation} takes $\epsilon$ small enough so that the truncation shaves off $F$ without touching any faces in the link of $F$, whereas a \textbf{deep truncation} has a defining hyperplane which passes through a nearest vertex neighboring $F$ (this determines $\epsilon$ uniquely).  A deep truncation is a \textbf{flat truncation} if it passes through all of the vertices neighboring $F$.   The \textbf{fragment} of a truncation is the piece of the polytope containing $F$ which is cut off by the truncation.

Shallow truncations are dual to stellar subdivisions, and central stellar subdivisions correspond to blow-ups in the setting of toric geometry.  Recall that the barycentric subdivision of a fan $\Sigma$ can be obtained by applying a sequence of central stellar subdivisions to the cones of $\Sigma$ in reverse order of containment.  An \textbf{omnitruncation} of a polytope is obtained by applying shallow truncations to each proper face of the original polytope in order of containment.  The normal fan of an omnitruncation is the barycentric subdivision of the normal fan of the original polytope.
	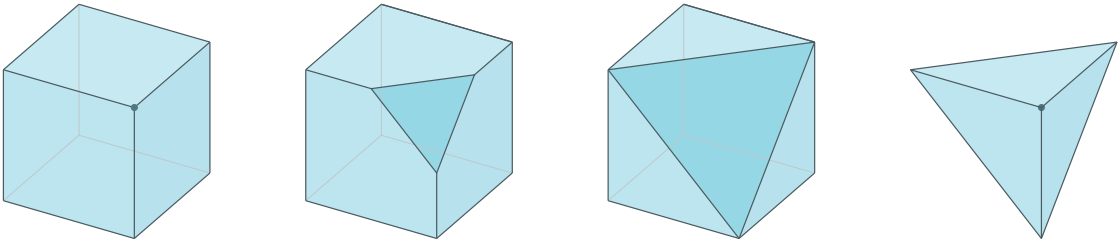
\begin{figure}[ht]
		\centering
		\tdplotsetmaincoords{60}{120}
		\input{tikz/split.tex}
		\caption{A cube, a shallow truncation, a flat vertex truncation, and a fragment.}
\label{fig:split_duplicated}
	\end{figure}
The following is our first main result. 

\begin{theorem}(\ref{thm:basis_seed}, \ref{thm:building}) Let $\mathsf{S}$ be a full-dimensional simple polytope with the origin in its interior, and let $\Psi$ be the barycentric subdivision of its normal fan. For $\epsilon$ sufficiently small, the set of $\epsilon$-truncations of $\mathsf{S}$ gives a polytopal basis for the deformation cone of $ \Psi $.  The Minkowski sum of these truncations is normally equivalent to an omnitruncation of $\mathsf{S}$.
Furthermore, if $ \mathsf{S} $ is smooth with all edge lengths 1, then the deep truncations of $\mathsf{S}$ are all flat truncations, and form a polytopal basis for the deformation cone of $\Psi$.
\end{theorem} 

The standard simplex is a smooth polytope with all edge lengths equal to 1.  As a special case of a codimension-one version of the previous theorem, we recover the simplicial basis for the cone of generalized permutahedra (see \Cref{prop:key_non_full}, \Cref{lem:standard_simplex_deep},
and \Cref{ex:postnikov_basis}).

De Concini and Procesi first introduced building sets and nested set complexes as the combinatorial framework underlying their theory of wonderful compactifications \cite{de1995wonderful}.  Feichtner--Kozlov extended building sets to more general posets and observed that, in the setting of face posets of fans, they are closely related to the theory of stellar subdivisions \cite{feichtner2004incidence}.   Nestohedra are smooth polytopes which arise from the truncations of the standard simplex according to a building set, and their associated toric varieties are the wonderful compactifications of the torus.
Feichtner--Sturmfels and Postnikov described each nestohedron as a Minkowski sum of standard simplices \cite{feichtner2005matroid,postnikov2009permutohedra}.\footnote{Postnikov showed that this presentation gives combinatorial vertex coordinates.  In recent work of Mantovani--Padrol--Pilaud, this has been utilized for giving nice vertex coordinates for building set trimmings of general polytopes via slices with the standard simplex \cite{mantovani2025facial}.}

An omnitruncation of a polytope $ \mathsf{P} $ can be understood as truncating $ \mathsf{P}$ according to the maximum building set.  The following theorem shows that our previous result extends to arbitrary building sets on simple polytopes.  In the special case of the standard simplex this recovers a result of Padrol--Pilaud--Poullot \cite{padrol2022hypergraphic}.

\begin{theorem}(\ref{thm:building}, \ref{cor:basis_building})
Let $S$ be a simple polytope with the origin in its interior, and let $B$ be a building set for $S$.  Let $P$ be a truncation of $S$ along $B$, and $\Psi$ be the normal fan of $P$.  The shallow truncations of $S$ along the individual faces in $B$ form a polytopal basis for the deformation cone of $\Psi$.  The Minkowski sum of these truncations is normally equivalent to $P$.  When $S$ is smooth with edge lengths 1, the corresponding flat truncations give a basis for the deformation cone of $\Psi$.
\end{theorem}

We note that the full versions of Theorems \ref{thm:basis_seed} and \ref{cor:basis_building} contain additional results concerning a canonical generating set of deep truncations for simple polytopes, and an integral generating set for smooth polytopes.  

We clarify the relationship between various types of deformations of an omnitruncation.  Let 
$\mathcal{T}_{\operatorname{Tr}}(\mathsf{S})$ be the cone of positive sums of shallow truncations of $\mathsf{S}$, and let $\mathcal{T}_{\operatorname{Omni}}(\mathsf{S})$ be the collection of all omnitruncations of $\mathsf{S}$, considered up to scaling and translation.  

\begin{theorem}\label{thm:conecontain}
Let $\mathsf S$ be a simple polytope, and let $\Psi$ be the barycentric subdivision of its normal fan.  Then $\mathcal{T}_{\operatorname{Tr}}(\mathsf{S})$ and $\mathcal{T}_{\operatorname{Omni}}(\mathsf{S})$  are full-dimensional open polyhedral cones with nonempty intersection.  In general, we find that $\mathcal{T}_{\operatorname{Tr}}(\mathsf{S}) \neq \mathcal{T}_{\operatorname{Omni}}(\mathsf{S})\neq  \operatorname{Interior}(\operatorname{Def}(\Psi))$, e.g. when $\Psi$ is the braid arrangement for $n\geq 4$.
\end{theorem}

The previous theorem is proven at the level of general building sets (see \Cref{lem:buildingtruncationcone,lem:sumsofdeeptruncations,prop:diffcones}).  We then propose a definition of higher generalized nestohedra (see Definition \ref{def:highernesto}), and explain how the collection of all generalized 2-nestohedra naturally forms a fan, which is no longer a cone (see Proposition \ref{prop:fan_of_2_nestohedra}).

A polytope is indecomposable if it cannot be expressed as a nontrivial Minkowski sum of polytopes.  A basis for a deformation cone is indecomposable if each element is indecomposable; equivalently, the translation classes of its elements span extremal rays of the nef cone.  While the simplicial basis is clearly indecomposable, this is not always the case for our truncation bases.   We investigate two natural extensions of the permutahedron where we are able to upgrade our truncation bases to an indecomposable basis.

The first setting is the deformation cone of a barycentric subdivision of the normal fan of a product of standard simplices.  
This generalizes the type-A and type-B Coxeter arrangements, which correspond to a single simplex and a product of line segments, respectively.  However, this framework leaves the setting of standard Coxeter theory.  We explain how the $\pi$-colored fans of Clader--Damiolini--Eur--Huang--Li arise as subfans of such fans (\Cref{prop:pi_colored_subfan_polysimplex}).
Here we show that the deep truncations form an indecomposable basis.  
As noted above, this basis agrees with the simplicial basis in type A, but in type B this already gives a new basis. 

Recall we define the fragment of a truncation to be the piece of the polytope cut off by that truncation.  In general, the fragments do not live in the deformation cone of an omnitruncation of a polytope, but we show that they do in the special case of products of standard simplices.  Moreover, these fragments form a second, complementary basis, which specializes to the basis of Eur--Fink--Larson--Spink in type B and whose support functions restrict to the Clader--Damiolini--Eur--Huang--Li basis for the nef cones of $\pi$-colored fans (\Cref{prop:pi_colored_truncations_basis,prop:pi_colored_fragments_h_basis}).

\begin{theorem}(\ref{thm:basis_polysimplex_deep}), (\ref{thm:basis_polysimplex_fragments}):
	Let $\Delta(\mathbf S)$ be a product of standard simplices associated to a partition $\mathbf S$ of $[n]$, and let $\Psi$ be the barycentric subdivision of the normal fan of $\Delta(\mathbf S)$.  The indecomposable factors of the deep truncations and the fragments of $\Delta(\mathbf S)$, together with the simplicial bases for each factor of $\Delta(\mathbf S)$, give a pair of indecomposable bases for the deformation cone of $\Psi$.
\end{theorem}

\begin{figure}[ht]
    \centering
    \includegraphics[width=0.8\textwidth]{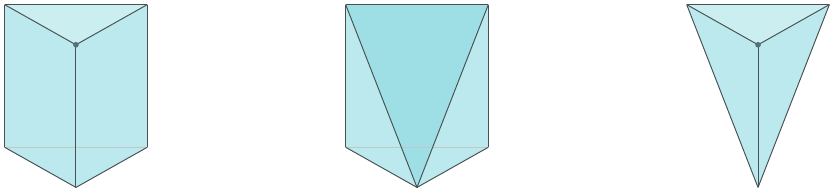}
    \caption{A prism $\Delta_3 \times \Delta_2$, a flat vertex truncation,  and the associated fragment.}
    \label{fig:triangle-segment-vertex-fragment}
\end{figure}

\begin{figure}[ht]
    \centering
    \includegraphics[width=.20\textwidth]{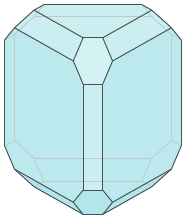}
    \caption{An omnitruncation of $\Delta_3 \times \Delta_2$.}
    \label{fig:triangle-segment-prism-omnitruncation}
\end{figure}
The second setting we investigate is the barycentric subdivision of the braid fan of Gaiffi \cite{gaiffi2015permutonestohedra} and Castillo--Liu \cite{castillo2022deformation}, which we refer to as the \textbf{$2$-braid fan}.\footnote{Castillo--Liu call this fan the nested braid fan, the normal fan of a polytope they call the nested permutahedron.  Gaiffi calls this polytope the permutopermutohedron.} Equivalently, this is the second barycentric subdivision of the normal fan of the standard simplex, and it is the normal fan of an omnitruncation of a permutahedron.  We define a $2$-permutahedron to be an $S_n$-invariant polytope whose normal fan is the $2$-braid fan (see Figure \ref{fig:castillo-liu-nested-permutahedron}).  While the deep truncations of the faces of the standard permutahedron are not indecomposable, they factor naturally and we identify a distinguished set of factors which form an indecomposable basis.

\begin{theorem}\label{thm:main2-permgherkin}(\ref{thm:basis_nested})
The apex truncations of  graphical zonotopes associated to oriented chains of complete bipartite graphs with at least 3 parts, together with the simplicial basis, give an indecomposable basis for the deformation cone of the 2-permutahedron.
	
\end{theorem}

\begin{figure}[ht]
    \centering
    \includegraphics[width=.35\textwidth]{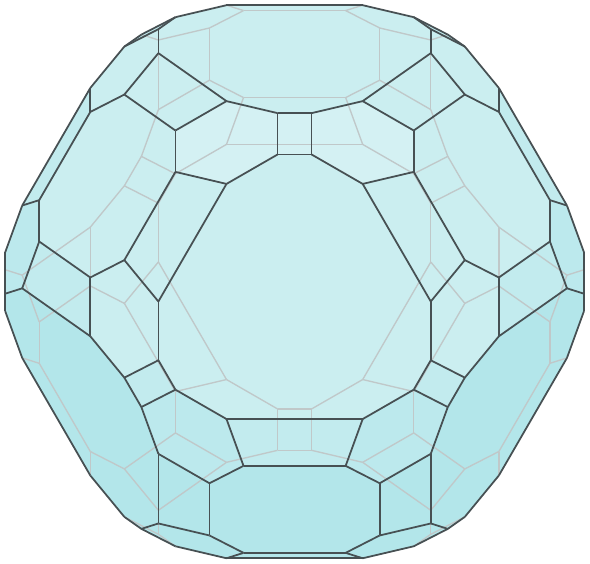}
    \caption{Castillo and Liu's realization of a 2-permutahedron.}
    \label{fig:castillo-liu-nested-permutahedron}
\end{figure}

We were naturally led to generalize Theorem \ref{thm:main2-permgherkin} to the setting of primitive Delzant zonotopes (see \Cref{thm:primitive_delzant_zonotope}), and in the process we prove an indecomposability result for the apex truncations of zonotopes over polytopes (see Theorem \ref{thm:gherkin}).  While this work was in the final stages of preparation, we learned that Padrol--Poullot, motivated by separate considerations of the submodular cone, have independently discovered this indecomposability result \cite[Corollary~5.1.9 and Example~5.1.10]{padrol2026indecomposability}.

We conclude the article in Section \ref{applicationssection} with connections and applications to Archimedean solids and regular polytopes, permutahedral plates, root polytopes, permutoassociahedra, simple permutonestohedra, cosmohedra, omnitruncations of Coxeter permutahedra,  bipermutahedra, $\pi$-colored fans, and tropical $\alpha$ and $\beta$ classes.

\subsection*{Future work}
We will investigate the mixed volumes of elements of our truncation bases in the Chow ring of the corresponding barycentrically subdivided fan to obtain general volume formulas for polytopes in their deformation cones.  Additionally, we will investigate the 2-root polytopes from \Cref{subsec:pemplatesrootpoly}.

\subsection*{Acknowledgements}
We thank Federico Ardila-Mantilla, Sof\'ia Err\'azuriz, Fu Liu, Arnau Padrol, Vincent Pilaud, Felipe Rinc\'on and Raman Sanyal for helpful  comments and conversations.
We thank Arnau Padrol and Germain Poullot for informing us of their related work on indecomposability of certain truncated zonotopes.  Additional thanks to Chris Eur for informing us that his student Matthew Snodgrass has recently taken some first steps in trying to understand the Chow rings of barycentric subdivisions of normal fans of products of simplices.  In particular, he has established a Dragon marriage type result for the fragments.

The first author was supported by NSF Grant (DMS-2246967) and Simons Foundation Gift \# 854037. 
The second author was partially supported by ANID Fondecyt-Regular Project N°1260970.
This research was partially conducted during the Intensive Research Program on Combinatorial Geometries and Geometric Combinatorics, held at the Centre de Recerca Matemàtica in October-November 2025, and funded by the Severo Ochoa and María de Maeztu Program for Centers and Units of Excellence in R\&D (CEX2020-001084-M), the Institut de Matemàtica de la Universitat de Barcelona, and the Spanish projects PID2022-137283NB-C21 and RED2024-153572-T of MICIU/AEI /10.13039/501100011033.

\subsection*{AI disclosure}  This research was conducted and written up over the period 2021-2026.  In 2026, while completing our article, we utilized OpenAI's Codex and GPT 5.5 and 5.6 Pro for the following tasks:
\begin{itemize}

\item producing and editing figures, which were generated with Sage and TikZ, and were intentionally modeled on earlier TikZ figures generated by the authors; 
\item utilizing Sage for gathering evidence for \Cref{mockconj}; 
\item assistance with proving \Cref{thm:2root} as guided by the authors;
 \item employing Sage for performing explicit Minkowski sum decomposition calculations in Section \ref{applicationssection};  
\item assistance with writing up the definitions and results in Subsection \ref{subsec:picolored} for consistency with the framework of Clader--Damiolini--Eur--Huang--Li \cite{clader2024multimatroids};
\item providing the change of basis in Proposition \ref{prop:explicit_truncation_fragment_transform}.  Here, the authors had suggested to GPT that a polysimplex version of the corresponding relation from Clader--Damiolini--Eur--Huang--Li \cite[proof of Lemma~2.11, equations~(7) and~(9)]{clader2024multimatroids} should exist, and Proposition \ref{prop:explicit_truncation_fragment_transform} is what GPT proposed;
\item performing literature searches and assisting with citation verification;
\item checking for correctness and offering suggestions for repairs when minor errors were found;
\item editing, mostly copy editing, but also proposing rewording of some sentences and providing text describing explicit calculations which were reviewed and revised by the authors.
\end{itemize}

 Apart from the above items, all other contributions of this article were human generated.

\section{Preliminaries}
We review the necessary terminology from discrete geometry to be used in the present paper.
For further background, we refer the reader to~\cite{ziegler2012lectures, grunbaum1967convex, villavicencio2024polytopes}.

Let $ \mathbb{R}^{n} $ be a finite-dimensional real vector space with the usual dot product written as $ x^{\intercal} y $.
This space has the canonical basis $\left\{ e_{1}, \dots, e_{n}\right\} $: the vector $ e_i $ has every entry equal to zero except for an entry equal to one in the coordinate $ i $.
More generally, for every subset $ I \subseteq [n] $ we define the vector
\begin{equation}\label{eq:eI}
	e_{I} \coloneqq \sum_{ i \in I } e_{i}.
\end{equation}

A lattice $N$ is a free abelian group of finite rank. We write
$N_{\mathbb Q}=N\otimes_{\mathbb Z}\mathbb Q$ and
$N_{\mathbb R}=N\otimes_{\mathbb Z}\mathbb R$. We also write
$M=\operatorname{Hom}(N,\mathbb Z)$ for the dual lattice; elements of $M$
are the integral linear functions on $N$. When $N=\mathbb Z^n$, we identify
$N_{\mathbb R}$ with $\mathbb R^n$ in the usual way, and the dot product
identifies $M$ with $\mathbb Z^n$.
A linear subspace $\mathsf L\subseteq N_{\mathbb R}$ is rational if it is
spanned over $\mathbb R$ by vectors in $N_{\mathbb Q}$. If $\mathsf L$ is
rational, then $N/(N \cap {\mathsf L})$ is a lattice,
which we call the quotient lattice. Unless otherwise stated, the lattice in
$\mathbb R^n$ is $\mathbb Z^n$.  A nonzero vector $u\in N$ is called \textbf{primitive} if whenever
$u=kw$ for some $w\in N$ and $k\in\mathbb Z_{>0}$, then $k=1$.  In this article, a recurring example will be the case $N=\mathbb Z^n$ and
$\mathsf L=\operatorname{Span}_{\mathbb R}\{e_{S_i}\mid i\in[k]\}$,
where $\{S_i\}_{i=1}^{k}$ is a (potentially trivial) partition of $[n]$.

A \textbf{rational polyhedron} $ \mathsf{P} \subseteq \mathbb{R}^{n}$ is the solution set to a system of finitely many inequalities with integer coefficients, i.e.,
a set of the form $ \left\{ x \in \mathbb{R}^{n} \,\middle|\, \mathbf{A}x \geq h \right\} $, with $ \mathbf{A} $ an integer matrix and $ h $ an integer vector.
Polyhedra in this paper will typically be taken to be rational, so we will omit the adjective.  A \textbf{polytope} $ \mathsf{P} \subseteq  \mathbb{R}^n $ is a bounded polyhedron.
Alternatively, it is the convex hull of a finite set of points in $ \mathbb{R}^n $.
If each of its vertices is integral, we say that $ \mathsf{P} $ is an integral polytope.

For a polyhedron $ \mathsf{P} \subseteq \mathbb{R}^{n} $ we define its \textbf{lineality space} as the largest linear subspace $\mathsf L$ satisfying $\mathsf P+\mathsf L=\mathsf P$.
We say $ \mathsf{P} $ is a $ d $-polyhedron if $ \dim(\mathsf{P}) = d$.
For any two polyhedra $ \mathsf{S}, \mathsf{T} \subseteq \mathbb{R}^{n} $ there are two operations that result in new polyhedra:
the \textbf{Cartesian product} $ \mathsf{S} \times \mathsf{T} $ and the \textbf{Minkowski sum}  
$ \mathsf{S}+\mathsf{T} =\left\{ s+t \,\middle|\, (s,t) \in \mathsf{S} \times  \mathsf{T}\right\} $. 

Associated to every polyhedron $ \mathsf{P} $ we have the \textbf{support function}:
\begin{equation}\label{eq:def_supp}
	h(\mathsf{P}, \cdot ) : \mathbb{R}^{n} \to \mathbb{R} \cup \left\{ - \infty \right\},\quad  h(\mathsf{P}, u) \coloneqq \min_{y \in \mathsf{P} }u^{\intercal}y . 
\end{equation}
The polyhedron and its support function uniquely determine each other.
We have
\begin{equation}\label{eq:support_sum}
	h(\mathsf{P} + \mathsf{Q}, \cdot ) =	h(\mathsf{P}, \cdot ) + h(\mathsf{Q}, \cdot ).
\end{equation}

\begin{remark}\label{rem:single_point}
When $ \mathsf{P} = \{ p \} $ is a single point, then we have that $ h( \mathsf{P}, u) = p^{\intercal}u $ is the linear functional induced by the inner product with $ p $.
\end{remark}

Given two polytopes $ \mathsf{P} $ and $ \mathsf{Q} $, we define the Minkowski difference $ \mathsf{P} - \mathsf{Q} $ to be the polytope $ \mathsf{R} $ such that $ \mathsf{P} = \mathsf{Q} + \mathsf{R} $.
If such an $ \mathsf{R} $ does not exist, we leave the subtraction undefined.
When defined, the Minkowski subtraction has the property $ (\mathsf{P}-\mathsf{Q}) + \mathsf{Q} = \mathsf{P}$.
This implies that $h(\mathsf P-\mathsf Q,\mathord\cdot)=h(\mathsf P,\mathord\cdot)-h(\mathsf Q,\mathord\cdot) $, whenever $ \mathsf{P} - \mathsf{Q} $ exists.

\begin{remark}\label{rem:subtraction}
We warn the reader that the difference of support functions is not always a support function.
Given polytopes $ \mathsf{P}, \mathsf{Q} $, we can always define
\begin{equation}\label{eq:support_difference}
	\mathsf{R}  =\left\{x:u^tx\geq h(\mathsf{P},u)-h(\mathsf{Q},u)\;\forall u\in\mathbb{R}^d\right\}.
\end{equation}
However, Equation \eqref{eq:support_difference} does not guarantee that $ \mathsf{P} = \mathsf{Q} + \mathsf{R} $.
This is because the inequality description is not guaranteed to be tight; that is, in Equation \eqref{eq:support_difference} some of the inequalities may never hold with equality for points in $ \mathsf{R} $.
\end{remark}

\begin{example}\label{ex:badminkowsi}
Let $\mathsf{P} = [0,2]^2 \subseteq \mathbb{R}^{2}$ and let $\mathsf{Q}$ be the triangle with vertices $ \left\{  (1,1),(1,0),(0,1) \right\}$. 
The polytope $\mathsf{R}$ defined by \eqref{eq:support_difference} is equal to $[0,1]^2$. However, $\mathsf{R}+\mathsf{Q} \neq \mathsf{P}$.

\begin{figure}[ht]
    \centering
    \input{tikz/badminkowski_example.tex}
    \caption{From left to right: $P$, $Q$, $R$, and $Q+R$.}
    \label{fig:badminkowski-example}
\end{figure}
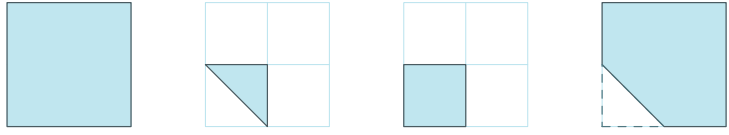
\end{example}

In general, any piecewise linear function on a polytopal fan can be obtained as the difference of two support functions of polytopes \cite[Proposition 5.13]{manecke2024inscribable}, so in most cases, taking the difference of support functions will not provide a support function.
\begin{definition}\label{def:weak_summand}
We say a polytope $ \mathsf{Q} $ is a Minkowski summand of a polytope $ \mathsf{P} $ if $ \mathsf{P} - \mathsf{Q} $ exists.
More generally, we say that a polytope $\mathsf{Q}$ is a \textbf{weak Minkowski summand} or \textbf{deformation} of $\mathsf{P}$ if there exists a $\lambda>0$ such that $\mathsf{Q}$ is a Minkowski summand of $\lambda\mathsf{P}$.
\end{definition}

For a vector $ u \in \mathbb{R}^n	$ we define $ \mathsf{P}^{u} $	to be the face of $ \mathsf{P} $ in direction $ u $, $  \mathsf{P}^{u} \coloneqq\left\{ x \in \mathsf{P} \,\middle|\, u^{\intercal}x = h(\mathsf{P}, u)\right\} $.
Note that $ \mathsf{P} $ is a face of itself (in direction $ u = \mathbf{0} $); additionally by convention we let $ \emptyset $ be a face of $ \mathsf{P} $.
Faces of codimension one are called facets.
The face poset $ \mathscr{P} $ of a polyhedron is the set of nonempty faces partially ordered by inclusion.  The proper part of the face poset is the face poset with the maximum (the polyhedron itself) removed.

A polyhedron $ \mathsf{K} \subseteq \mathbb{R}^{n} $ is a \textbf{cone} if $ x\in \mathsf{K} $ implies that $ \lambda x \in \mathsf{K} $ for any nonnegative scalar $ \lambda $.
Cones used in this paper are always polyhedral and rational so we omit both adjectives.
A cone is pointed if it has a trivial lineality space (the largest linear subspace the cone contains).
In general, a cone $ \mathsf{K} $ can be decomposed as a Minkowski sum $ \mathsf{C} + \mathsf{L} $, where $ \mathsf{C} $ is a pointed cone, and $ \mathsf{L} $ is its lineality space.
The cone $ \mathsf{C}$ in this decomposition is not unique.

\begin{definition}\label{def:primitive}
A one dimensional pointed cone $ \tau \subseteq N_{\mathbb R} $ is called a ray. 
A \textbf{ray generator} of $ \tau $ is a vector $ r  $ such that $ \operatorname{Cone}(\left\{ r\right\}) = \tau $.
Thus any nonzero vector in $\tau$ is a ray generator.  However, since we are often assuming that $ \tau $ is rational, we will typically take the canonical choice of ray generator, namely the unique primitive vector in $\tau$ which we denote $u_{\tau}$.

For any pointed cone $ \sigma $ with finite set of rays $ \tau_{1}, \dots, \tau_{k} $ we define the following vector in its relative interior 
\begin{equation}\label{eq:canonical_vector}
u_{\sigma} = \sum_{i=1}^{k} u_{\tau_{i}}.
\end{equation}

\end{definition}

Faces of cones are themselves cones.
A cone $ \sigma \subseteq \mathbb{R}^{n} $ is \textbf{simplicial} if it is pointed and generated by a linearly independent set of vectors of $ \mathbb{R}^{n} $.
If $ \sigma $ is simplicial, then every face of it is simplicial too.
A \textbf{fan} $ \Sigma $ is a finite set of cones satisfying two properties:
\begin{enumerate}
  \item if $\sigma  \in \Sigma$ and $\tau$ is a face of $\sigma$, then $\tau \in \Sigma$, and
\item  if $\sigma, \tau \in \Sigma$, then $\sigma \cap \tau$ is a face of both $\sigma$ and $\tau$.
\end{enumerate}

A fan $ \Sigma $ is said to be \textbf{complete} if the union of all of its cones is equal to its ambient vector space $\mathbb{R}^n$.
The face poset $ \mathscr{F} $ of a fan $ \Sigma $ is the partial order on its elements given by inclusion.
It has a unique minimum element: the intersection of all members $ \bigcap_{\sigma \in \Sigma} \sigma $, but it may not have a maximum element.  Note that each cone in a fan must have the same lineality space $\mathsf{L}(\Sigma)$, which we call the lineality space of $\Sigma$; this is the minimum element of the face poset $\mathscr{F}$.  When a fan has lineality, we say it is simplicial if its quotient by its lineality space is simplicial.

\begin{example}[Standard simplices]\label{ex:std_simplices}

	Let $ S $ be a nonempty finite set and $ \mathbb{R}^{S} $ the vector space of all real valued functions on $ S $.
	We define the \textbf{standard simplex}
	\begin{align*}
		\Delta(S)  &= \operatorname{ConvexHull}\left\{ e_{s} \,\middle|\, s \in S\right\} \subseteq \mathbb{R}^{S},\\
		           &= \left\{ x \in \mathbb{R}^{S} \,\middle|\,  x(s) \geq 0,\,s\in S\, \quad \text{and} \quad  \sum_{s\in S} x(s) = 1 \right\},
	\end{align*}
		where the vector $ e_{s} $ is the function from $ S $ to $ \mathbb{R} $ that is 1 on $ s $ and 0 elsewhere.
	This is a $ (|S|-1) $-dimensional polytope.
	Its face poset, after adjoining a minimum element, is the Boolean poset $ \mathscr{B}(S) $ consisting of all subsets of $ S $ ordered by inclusion:
	For every subset $ I \subseteq S  $ we have the face $ \Delta(I) \coloneqq \operatorname{ConvexHull} \left\{ e_{i} \,\middle|\, i \in I \right\}$.	
	Often we identify $ S $ with $ [n] $ and simply write $ \Delta_{n} $ for $\Delta_{[n]} $.  Its inequality description is
	\[
		\Delta_{n} =  \left\{ x \in \mathbb{R}^{n} \,\middle|\, x_{i} \geq 0\quad \forall i \in [n], \quad \sum_{i=1}^{n} x_{i} = 1 \right\}.
	\]

\end{example}

Polytopes induce complete fans via the following construction.

\begin{definition}[Normal fan]\label{def:normal_fan}
For a given polytope $\mathsf{P}$, several directions $u$ can define the same nonempty face $ \mathsf{F} $.
We define $\operatorname{ncone}(\mathsf{F}, \mathsf{P}) \coloneqq\left\{ u \in \mathbb{R}^n \,\middle|\, \mathsf{F} \subseteq \mathsf{P}^{u}\right\}$, the normal cone of $ \mathsf{F} $ with respect to $ \mathsf{P} $.
Its relative interior $ \operatorname{ncone}(\mathsf{F}, \mathsf{P})^{\circ} $ consists of the directions $ u $ such that $ \mathsf{P}^{u} = \mathsf{F} $.
Note that $\mathsf{P}^\mathbf{0} = \mathsf{P}$.
The \textbf{normal fan} $ \Sigma( \mathsf{P} ) $ of $ \mathsf{P} $ is the collection of all normal cones of nonempty faces with respect to $ \mathsf{P} $.

\end{definition}

When $ \mathsf{P} \subseteq \mathbb{R}^{n} $ is full-dimensional, each normal cone is pointed and $ \dim\left( \operatorname{ncone}(\mathsf{F}, \mathsf{P}) \right) = \dim (\mathsf{P}) - \dim (\mathsf{F})$.
Let $ \mathsf{F} \subseteq \mathsf{P}	$ be a nonempty face of a full-dimensional polyhedron.
We define the \textbf{canonical inner face normal} of a nonempty proper face $\mathsf{F}$, denoted $ u_{\mathsf{F}} $, to be $ u_{\sigma} $ defined in Equation~\eqref{eq:canonical_vector},
where $ \sigma = \operatorname{ncone}(\mathsf{F}, \mathsf{P}) $.\footnote{Because we have settled on the convention that support functions are of the form $h(\mathsf P,u)=\min_{x\in \mathsf P} u^{\intercal}x$, the vector $ u_{\mathsf{F}} $ points into the polyhedron.
If we had instead used the convention that $h(\mathsf P,u)=\max_{x\in \mathsf P} u^{\intercal}x$, they would point out of the polyhedron.}
As a useful convention we set $u_{\mathsf P} = 0$.
Note that $ u_{\mathsf{F}} \in \operatorname{ncone}(\mathsf{F}, \mathsf{P})^{\circ}$.
The inner face normals are primitive for facets (by definition) but they may not be primitive for lower dimensional faces.
Thus when $ \mathsf{G} \subseteq \mathsf{P} $ is a facet, its normal cone is a ray $ \tau $, and its inner facet normal is the canonical generator of this ray.

When $ \mathsf{P} \subseteq \mathbb{R}^{n} $ is not full-dimensional then none of the normal cones are pointed.
In fact, every normal cone has the same lineality space which can be described as
\begin{equation}\label{eq:lineality_space}
	\mathsf{L}(\mathsf{P}) = \left\{ w\in \mathbb{R}^{n} \,\middle|\, \exists C \in \mathbb{R} : w^{\intercal} x = C, \forall x\in \mathsf{P} \right\}.
\end{equation}
The lineality space is the orthogonal complement to the affine span of $ \mathsf{P}$ translated to the origin.  If a fan $\Sigma$ is the normal fan of a polytope we say the fan is \textbf{polytopal}.\footnote{Some authors say \emph{regular} or \emph{projective} (in the latter case, one should assume that $\Sigma$ is rational).}

Suppose the fan $\Sigma \subset N_{\mathbb R}$ has lineality space $\mathsf L$. Projecting each cone in $\Sigma$ along $\mathsf L$ gives a pointed fan in $ N_{\mathbb R}/\mathsf L$. We define a \textbf{pointed representative} for $\Sigma$ to be a pointed fan ${}^{\wedge}\Sigma$ in the original ambient vector space such that the map ${}^{\wedge}\sigma \mapsto {}^{\wedge}\sigma+\mathsf L$ induces a bijection from the cones of ${}^{\wedge}\Sigma$ to the cones of $\Sigma$. Such a pointed representative is not unique, and if $\Sigma$ is complete, then ${}^{\wedge}\Sigma$ need not be complete in the original ambient vector space.  We say that ${}^{\wedge}\Sigma$ is a \textbf{primitive pointed representative} if,
for every ray ${}^{\wedge}\tau$ of ${}^{\wedge}\Sigma$, the image of its
primitive ray generator in the quotient lattice of
$N_{\mathbb R}/\mathsf L$ is the primitive ray generator of the
corresponding ray of the quotient fan. 
When $\Sigma$ has lineality and a primitive pointed representative ${}^{\wedge}\Sigma$ has been fixed, we write $u_{\sigma}\coloneqq u_{{}^{\wedge}\sigma}$, where ${}^{\wedge}\sigma$ is the cone corresponding to $\sigma$. In particular, for a non-full-dimensional polytope $\mathsf{P}$, the vectors $u_{\mathsf{F}}$ are defined using a fixed primitive pointed representative of $\Sigma(\mathsf{P})$.

\begin{remark}\label{rem:fan_decomposition}
We observe that one canonical choice is
\[
{}^{\wedge}\Sigma=\{\sigma\cap \mathsf L^\perp \mid \sigma\in\Sigma\}.
\]
This is complete as a fan in $\mathsf L^\perp$, but not as a fan in $\mathbb R^n$. However, as the following example indicates, this is not the pointed representative we will want to use.
\end{remark}

\begin{example}\label{ex:braid_zero}
	The normal fan $ \Sigma(\Delta_{n}) $ of the standard simplex from \Cref{ex:std_simplices} is equal to the following set of cones:
\begin{equation}\label{eq:braid_zero}
	\mathcal{B}^{0}_{n} \coloneqq \left\{ \operatorname{Cone} \left\{ e_{i} \,\middle|\, i \in I\right\} + \operatorname{Span}\left\{ e_{[n]}\right\} \,\middle|\, I \subsetneq [n] \right \}.
\end{equation}
Notice that the lineality space of the fan is the one dimensional line spanned by $ e_{[n]} $.  Here we naturally take the pointed representative 
\[{}^{\wedge}\mathcal{B}^{0}_{n}
=
\left\{
\operatorname{Cone}\{e_i \mid i\in I\}
\,\middle|\,
I\subsetneq[n]
\right\}.
\]
The reason for the zero superscript in the notation is justified by \Cref{ex:k_braid}.
The face poset $ \mathscr{B}^{0}_{n} $ is isomorphic to the Boolean lattice of subsets of $ [n] $ with the maximal element removed.
\end{example}

\begin{definition}\label{def:coarsening}
Given fans $ \Sigma_{1} $ and $ \Sigma_{2} $ we say that $ \Sigma_{1} $ is a \textbf{refinement} of $ \Sigma_{2} $ if every cone of $ \Sigma_{2}  $ is a union of cones in $ \Sigma_{1} $.
In this context we say $ \Sigma_{2} $ is a \textbf{coarsening} of $ \Sigma_{1} $.
\end{definition}

Given two fans $ \Sigma_{1} $ and $ \Sigma_{2} $ we define $ \Sigma_{1} \land \Sigma_{2} = \left\{  \sigma_{1} \cap \sigma_{2} \,\middle|\, \sigma_{1} \in \Sigma_{1}, \sigma_{2}\in \Sigma_{2} \right\} $ to be their common refinement.
The relation between Minkowski sums and normal fans is described by the following equation:
\begin{equation}\label{eq:sum_refinement}
 \Sigma(\mathsf{P} + \mathsf{Q}) = \Sigma(\mathsf{P}) \land \Sigma(\mathsf{Q}).
\end{equation}
See \cite[Proposition 7.2]{ziegler2012lectures}.
It follows that for any weak Minkowski summand $ \mathsf{Q} $ of $ \mathsf{P} $ we must have that $ \Sigma(\mathsf{Q}) $ is a coarsening of $ \Sigma(\mathsf{P}) $.
The converse is true, thus we have the following statement. 
\begin{proposition}[{see \cite[Chapter~15]{grunbaum1967convex}}]\label{prop:coarsening_summand}
A polytope $ \mathsf{Q} $ is a deformation (or, equivalently, a weak Minkowski summand) of $ \mathsf{P} $ if and only if $ \Sigma(\mathsf{Q}) $ coarsens the fan $ \Sigma(\mathsf{P}) $.
\end{proposition}

We call a polytope of the form $\lambda \mathsf{P}$, with $\lambda > 0$, a \textbf{homothety} of $\mathsf{P}$. 
Homotheties of $\mathsf{P}$ are regarded as trivial deformations of $\mathsf{P}$. 
We say that a deformation $\mathsf{Q}$ of $\mathsf{P}$ is nontrivial if it is neither a point nor a translation of a homothety of $\mathsf{P}$.

\begin{definition}\label{def:indecomposable}
A polytope $\mathsf{P}$ is called \textbf{indecomposable} if it admits no nontrivial deformations.
\end{definition}

The study of indecomposable polytopes was initiated by Meyer \cite{meyer1974indecomposable}.
Some criteria were developed by Shephard \cite{shephard1963decomposable} and McMullen \cite{mcmullen1987indecomposable}.
In Section \ref{sec:zonotopes} we present a criterion for indecomposability.

\section{Truncations}\label{subsec:truncations}

\subsection{Shallow and deep truncations}
\begin{definition}\label{def:truncation}
	Let $ \mathsf{F} \subsetneq  \mathsf{P}	$ be a nonempty proper face of a polytope and $ u $ a vector in $ \operatorname{ncone}(\mathsf{F}, \mathsf{P})^{\circ} $.
	We define the \textbf{truncation} of $ \mathsf{P} $ along $ u$ at distance $ \epsilon > 0 $ to be the polytope
	\begin{equation}\label{eq:truncation}
		\operatorname{Tr}^{u}_{\epsilon}\left[\mathsf{P}, \mathsf{F}\right] \coloneqq\left\{ x\in\mathsf{P} \,\middle|\, u^{\intercal}  x \geq h( \mathsf{P}, u) + \epsilon\right\}.
	\end{equation}

	With the particular choice of $ u = u_{\mathsf{F}} $, the canonical inner face normal introduced immediately after $\Cref{def:normal_fan}$, we refer to $ \operatorname{Tr}^{u}_{\epsilon}\left[\mathsf{P}, \mathsf{F}\right] $ as the \textbf{central truncation at distance $ \epsilon $} and denote it by $ \operatorname{Tr}_{\epsilon}\left[\mathsf{P}, \mathsf{F}\right] $ omitting the superscript $u$.
	By convention, we define $\operatorname{Tr}_{\epsilon}\left[\mathsf{P}, \emptyset \right] = \mathsf{P} $ for any $ \epsilon $.

	We say a vertex $ v $ is a \textbf{neighbor of a face} $ \mathsf{F} \subseteq \mathsf{P}$ if $ v \notin \mathsf{F} $ and there is an edge between $ v $ and a vertex $ w \in \mathsf{F} $.
	Let
\[
\delta_{\mathsf{F}} \coloneqq \min\left\{u_{\mathsf{F}}^{\intercal}y - h( \mathsf{P},u_{\mathsf F}) \,\middle|\, y \text{ is a neighbor of } \mathsf{F}\right\}.
\]
\begin{enumerate}
    \item If $ \epsilon < \delta_{\mathsf{F}}$, we call the truncation $ \operatorname{Tr}_{\epsilon}\left[\mathsf{P}, \mathsf{F}\right] $ a \textbf{shallow truncation} of $\mathsf{P}$ along $\mathsf{F}$. When the value of $\epsilon$ is understood from context, we omit it from the notation.
    \item For the specific choice $ \epsilon = \delta_{\mathsf{F}}$, we define the truncation
    \[
    \operatorname{DTr}\left[\mathsf{P}, \mathsf{F}\right] \coloneqq \operatorname{Tr}_{\delta_{\mathsf{F}}}\left[\mathsf{P}, \mathsf{F}\right]
    \]
    and refer to it as the \textbf{deep truncation} of $\mathsf{P}$ along $\mathsf{F}$.
    By convention, we set $\operatorname{DTr}[\mathsf{P},\emptyset] = \mathsf{P}$.
    \item We also introduce the associated polytope
    \[
    \operatorname{Fr}\left[\mathsf{P}, \mathsf{F}\right] \coloneqq \left\{ x \in \mathsf{P} \,\middle|\, u_\mathsf{F}^{\intercal}x \leq h(\mathsf{P}, u_\mathsf{F}) + \delta_{\mathsf{F}} \right\},
    \]
    which we call the \textbf{fragment} of $\mathsf{P}$ along $\mathsf{F}$.
\end{enumerate}
\end{definition}

\begin{example}\label{ex:split}
  Consider the unit cube in $\mathbb{R}^{{3}}$.
	In \Cref{fig:split_duplicated} we have a shallow central truncation at distance $ 1/2 $ along a vertex, a deep truncation at distance $ 1 $ along the same vertex, and the fragment along the same vertex.

\end{example}

The counterpart of a shallow truncation at the level of fans is that of a central stellar subdivision; see \cite{ewald1974stellar} for more information.
Given a cone $ \sigma $ in a simplicial fan $ \Sigma $ we define $ \mathrm{St}_{\Sigma}(\sigma) = \left\{ \kappa \in \Sigma \,\middle|\,  \sigma \subseteq \kappa \right\}, $ which is called the \textbf{open star} of $ \sigma $,
and $\overline{\mathrm{St}}_{\Sigma}(\sigma) = \left\{ \nu \in \Sigma \,\middle|\,  \nu \subseteq \kappa \in \mathrm{St}_{\Sigma}(\sigma) \right\}, $ which is called the \textbf{closed star} of $ \sigma $.

\begin{definition}\label{def:central_stellar}
Let $ \Sigma $ be a rational simplicial fan, let $ \sigma \in \Sigma $ be a cone, and let $ u \in \sigma^{\circ} $.\footnote{This definition extends to non-simplicial fans, but we won't need that level of generality for this paper.}
We define the \textbf{geometric stellar subdivision} of $ \sigma $ in $ \Sigma $ centered at $ u $
to be the fan $ \mathcal{C}_{u}( \Sigma, \sigma) $ defined as follows. 
\begin{equation}\label{eq:central_stellar}
\mathcal{C}_{u}( \Sigma, \sigma) \coloneqq \Sigma\backslash \mathrm{St}_{\Sigma}(\sigma) \cup \bigcup_{\nu \in \overline{\mathrm{St}}_{\Sigma}(\sigma) \setminus \mathrm{St}_{\Sigma}(\sigma)}   \operatorname{Cone} \bigg(\left\{ u\right\} \cup \nu \bigg)  .
\end{equation}

\end{definition}

\begin{lemma}\label{lem:shallow_truncation_stellar}
Let $\mathsf P$ be a simple polytope, let $\mathsf F$ be a nonempty
proper face, and let $\sigma_{\mathsf F}=\operatorname{ncone}(\mathsf F,\mathsf P)$
be the corresponding cone of $\Sigma(\mathsf P)$.  For $0<\epsilon<\delta_{\mathsf F}$, the normal fan of the central shallow truncation
$\operatorname{Tr}_{\epsilon}[\mathsf P,\mathsf F]$ is the central stellar
subdivision of $\Sigma(\mathsf P)$ along $\sigma_{\mathsf F}$:
\[
\Sigma(\operatorname{Tr}_{\epsilon}[\mathsf P,\mathsf F])
=
\mathcal C(\Sigma(\mathsf P),\sigma_{\mathsf F}).
\]
\end{lemma}

If $\Sigma$ has lineality space $\mathsf L(\Sigma)$, we perform stellar subdivisions on a pointed representative, equivalently on the quotient fan $\Sigma/\mathsf L(\Sigma)$, and then add the common lineality space back to every cone.

By definition, a central stellar subdivision of a fan $ \Sigma $ is a refinement (see \Cref{def:coarsening}) of $ \Sigma $.
Informally, a geometric stellar subdivision of a cone $ \sigma $ introduces a new ray generated by a vector $u \in \sigma^\circ$, subdivides the cones containing $\sigma$ by replacing them with cones obtained by adjoining this new ray to their proper faces not containing $\sigma$, and leaves all cones not containing $\sigma$ unchanged.
Notice that every cone $ \kappa \in \Sigma$ not containing $ \sigma $  is also in the fan $ \mathcal{C}( \Sigma, \sigma ) $.
See \Cref{fig:blow_fail} for an example of two central stellar subdivisions performed in succession.

The following lemma allows us to describe truncations of a face $\mathsf{F}$ as convex combinations of other truncations of the same face.

\begin{lemma}\label{lem:convtrunc}
Let $\mathsf{P}$ be a polytope, let $\mathsf{F}$ be a proper nonempty face of $\mathsf{P}$, and $0\leq \epsilon_1, \epsilon_2 \leq \delta_{\mathsf{F}}$.  Then, for $0 \leq \alpha\leq 1$,

\[
    \alpha \operatorname{Tr}_{\epsilon_1}\left[\mathsf{P}, \mathsf{F}\right]+ (1-\alpha) \operatorname{Tr}_{\epsilon_2}\left[\mathsf{P}, \mathsf{F}\right] = \operatorname{Tr}_{\beta}\left[\mathsf{P}, \mathsf{F}\right],
    \]
    where $\beta = \alpha \epsilon_1+ (1-\alpha) \epsilon_2$.  In particular, we have the following specialization:
    
\[
    \alpha  \operatorname{DTr}\left[\mathsf{P}, \mathsf{F}\right] + (1-\alpha) \mathsf{P} = \operatorname{Tr}_{\alpha \delta_\mathsf{F}}\left[\mathsf{P}, \mathsf{F}\right].
    \]
   
\end{lemma}

\begin{proof}
Each shallow truncation has the same normal fan, while $\mathsf P$ and $\operatorname{DTr}[\mathsf P,\mathsf F]$ are both deformations of any shallow truncation. It therefore suffices to compare the support functions on the rays of the normal fan of a shallow truncation. For all of the facet normals of the original polytope, with the possible exception of $\mathsf{F}$, this is clear.  For the direction $u_{\mathsf{F}}$, suppose that $h( \mathsf{P}, u_{\mathsf{F}})=\gamma$.  We can see that  $h( \operatorname{Tr}_{\epsilon}\left[\mathsf{P}, \mathsf{F}\right], u_{\mathsf{F}}) = \gamma+\epsilon$, hence $h( \alpha  \operatorname{Tr}_{\epsilon_1}\left[\mathsf{P}, \mathsf{F}\right]+ (1-\alpha)\cdot \operatorname{Tr}_{\epsilon_2}\left[\mathsf{P}, \mathsf{F}\right], u_\mathsf{F}) = \alpha (\gamma+ \epsilon_1) + (1-\alpha)(\gamma +\epsilon_2) = \gamma +\beta$ as desired.
\end{proof}

\subsection{Flat truncations}

Recall that in \Cref{def:truncation}, we defined the deep truncation $\operatorname{DTr} \left[\mathsf{P},\mathsf{F} \right]$.
\begin{definition}
	
	If we have that
	\begin{equation}\label{eqb:def_P-Q}
		\operatorname{DTr} \left[\mathsf{P},\mathsf{F} \right] = \operatorname{Conv}\left\{ v \in \operatorname{Vert}(\mathsf{P}) \,\middle|\, v \notin \mathsf{F}\right\},
	\end{equation}
	then we say that the deep truncation is a \textbf{flat truncation}.\footnote{This is equivalent to saying that the value $u_{\mathsf{F}}^{\intercal}y - h( \mathsf{P}, u_{\mathsf{F}})$ is the same for each neighbor $y$ of $\mathsf{F}$.}
\end{definition}

For example, the third polytope from the left in \Cref{fig:split_duplicated} is the flat truncation of a vertex in a cube.
In a simple polytope, the deep truncation at every vertex is flat, but this is not true for all faces.
In \Cref{prop:delzant} below, we prove that certain properties of the polytope are sufficient to ensure that all faces can be flatly truncated.

\begin{example}\label{ex:nonexample}
	Consider the non-regular cube in $ \mathbb{R}^{3}$ defined as the convex hull of the columns of the matrix
	\[
		\mathsf{P}=\operatorname{Conv}
		\begin{bmatrix}
			0&1&0&0&1&1&0&1\\
			0&0&1&0&1&0&1&1\\
			0&0&0&1&0&1&1/2&1/2
		\end{bmatrix}.
	\]
	When the highlighted edge in \Cref{fig:nonexample} is removed, the four adjacent vertices are not coplanar.
	They form two faces in the convex hull of the remaining six vertices, so the deep truncation at this edge is not flat.
	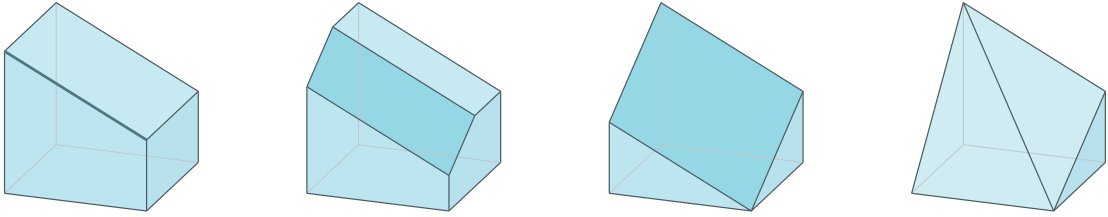
\begin{figure}[ht]
		\centering
		\tdplotsetmaincoords{70}{110}
		\input{tikz/nonexample.tex}
		\caption{From left to right: First we have a non-regular cube with an edge highlighted. 
		Then we have a shallow truncation of said edge in the central direction $u$ (the sum of the inner normals of the two facets containing the edge).
	Next is the deep truncation. Lastly, we have the convex hull of the remaining six points. Since this is not equal to the deep truncation, we conclude that the deep truncation is not flat.}
\label{fig:nonexample}
	\end{figure}
\end{example}

\begin{definition}
	\label{def:Delzant}
Let $N$ be a lattice and let $\mathsf P\subseteq N_{\mathbb R}$ be a full-dimensional integral polytope. We say that $\mathsf P$ is Delzant (or smooth) with respect to $N$ if, at every vertex, the primitive lattice vectors parallel to the incident edges form a basis of $N$.
\end{definition}

The polytope in \Cref{ex:nonexample} is Delzant with respect to the lattice $N = \mathbb{Z}e_1 \oplus \mathbb{Z}e_2 \oplus \frac{1}{2}\mathbb{Z}e_3$.
We describe a condition which ensures that the deep truncation along every face is flat.
An edge in a lattice polytope is called a \textbf{primitive edge} if its only lattice points are its endpoints.
A Delzant polytope whose edges are primitive is called a \textbf{primitive Delzant polytope}.

\begin{proposition}\label{prop:delzant}
	Let $\mathsf{P}$ be a Delzant polytope such that every edge is primitive.  
	Then the deep truncation of any proper nonempty face $ \mathsf{F} $ is flat.
	Furthermore, if a truncation has the same dimension as $\mathsf{P}$ then there is a unique new facet whose vertices are precisely the neighbors of $\mathsf{F}$.
\end{proposition}

\begin{proof}
	Let $\mathsf{F} \subset \mathsf{P}$ be a face defined by the facet equalities $\eta_i^{\intercal} x=b_i$ for $i=1,\dots,k$, where the $\eta_i\in M$ are primitive inner facet normals, and let $W$ be the set of vertices of $\mathsf{P}$ which are neighbors of $\mathsf{F}$.  Let $u_{\mathsf{F}} = \sum_i\eta_i$ be the direction of the truncation of $\mathsf{F}$, and let $b = \sum_ib_i$.
	We have $u_{\mathsf{F}}^{\intercal} x\geq b$ for all $x\in\mathsf{P}$ with equality if and only if $x\in\mathsf{F}$.

	We claim that the set $\operatorname{conv}(W)$ is the face of the convex hull of vertices of  $\mathsf{P}$ not in $\mathsf{F}$ where the functional $u_{\mathsf{F}}^{\intercal} $ is minimized.
	Let $w\notin\mathsf{F}$ be an adjacent vertex to $v\in\mathsf{F}$.
	Since $\left\{w,v\right\}$ is an edge of $\mathsf{P}$ and $ \mathsf{P} $ is simple, there is a unique facet that contains $v$ and not $w$.  Since $w\notin \mathsf{F}$ we can assume that it is the facet given by $\eta_1^{\intercal} x=b_1$, with $ \eta_1 $ primitive.
	We have $\eta_1^{\intercal} w>\eta_1^{\intercal} v$, hence
	\[
	\eta_1^{\intercal}(w-v)>0.
	\]
	Since $\mathsf{P}$ is Delzant and $\left\{v,w\right\}$ is a primitive edge, after choosing a basis of $N$ we may assume that the primitive edge directions at $v$ are $e_1,\dots,e_n$, with $e_1=w-v$.

	The facet defined by $\eta_1^{\intercal}x=b_1$ contains all edges at $v$ except the one in the direction $e_1$, hence $\eta_1^{\intercal} e_j = 0$ for $j \neq 1$. 
	Since $\eta_1$ is primitive, we must have that
	\[
	\eta_1^{\intercal} e_1 =\eta_1^{\intercal}(w-v)= 1.
	\]
	Therefore
	\[
	\eta_1^{\intercal} w = \eta_1^{\intercal} v + 1 = b_1 + 1.
	\]
	Summing over all $i$, we obtain
	\[
	u_{\mathsf{F}}^{\intercal} w = 1 + \sum_i b_i, 
	\]
	as each other facet defining hyperplane which contains $v$ contains $w$.
	 This implies that all of the neighbors of $\mathsf{F}$ lie on a common hyperplane orthogonal to $u_{\mathsf{F}}$ as desired.
	
	Next, suppose that $\dim(\operatorname{DTr} \left[\mathsf{P},\mathsf{F} \right])=\dim(\mathsf{P})$.  The new inequality is not implied by the defining inequalities of $\mathsf{P}$ because ignoring it results in a different polytope.  This demonstrates that this new inequality is facet defining as $\dim(\operatorname{DTr} \left[\mathsf{P},\mathsf{F} \right])=\dim(\mathsf{P})$.
	
Let $\mathsf{F}'$ be the new facet introduced by the truncation. Let $w' \in \mathsf{P}$ be a vertex of $\mathsf{F}'$.  In particular, $u_{\mathsf{F}}^{\intercal} w'=b+1$.  We must show that $w'$ is a neighbor of $\mathsf{F}$.  For $w'$ to be a vertex of $\mathsf{F}'$, it must be that $w'$ lies on an edge $e$ of $\mathsf{P}$ which is not contained in the truncation.  Move along $e$ to the endpoint $v'$ which minimizes $u_{\mathsf{F}}^{\intercal} x$.  
Because $\mathsf{P}$ is integral, $v'$ is integral, and $b\leq u_{\mathsf{F}}^{\intercal}v' < b+1$, hence $u_{\mathsf{F}}^{\intercal}v' = b$ and $v' \in \mathsf{F}$.  So $w'$ lies on an edge connecting $v'$ to a neighbor $w''$ outside of $\mathsf{F}$.  By the first part of this proof, $u_{\mathsf{F}}^{\intercal}w'' = b+1 = u_{\mathsf{F}}^{\intercal}w'$.  Because the truncation cuts this edge $e$, it must be that $w'' = w'$.  In particular, $w'$ is a neighbor of a vertex of $\mathsf{F}$ as desired.
\end{proof}

The following lemma is central to this work.
\begin{lemma}\label{lem:standard_simplex_deep}
Let $I \subseteq [n]$ be a nonempty proper subset.
The deep truncation of the face $\Delta(I) \subseteq \Delta([n])$ is $\Delta([n]\setminus I)$.
\end{lemma}
\begin{proof}
By \Cref{prop:delzant}, every deep truncation of the standard simplex is flat, hence equals the convex hull of the vertices not in the truncated face.
For $\Delta([n])$, this is precisely $\Delta([n]\setminus I)$.
\end{proof}

Note that the flat truncations of the standard simplex are always of smaller dimension than the original simplex.  We finish this section with some remarks on linear orders.

\begin{definition}
	\label{def:orderings}
	Let $ \mathcal{P} $ be a poset with $ n $ elements.
	A \textbf{linear extension} on $ \mathcal{P} $ is an order preserving bijective map $ E: \mathcal{P} \to [n] $.
	This induces a total order on the elements of $ \mathcal{P} $, which is compatible with its defining partial order $ \preceq $.
	
\end{definition}

If a poset $ \mathcal{P} $ is ranked, we can always take a linear extension $E$ of $ \mathcal{P} $ such that  $ E(s) \leq E(t) $ whenever $ s $ has a lower rank than $ t $.  In particular, if $ \mathcal{P} $ is the face poset of a polytope $ \mathsf{P} $, then we can take a total order $E$ in which all vertices of $ \mathsf{P} $ come first, then all edges of $ \mathsf{P} $, and so on.

\section{Deformation cones for barycentric subdivisions}

\subsection{Barycentric subdivisions}

\begin{definition}\label{def:barycentric}
Let $ \Sigma $ be a fan of dimension $d+\dim\mathsf L(\Sigma)$.  We define a sequence of fans
\begin{equation*}
 \Sigma^{0}, \Sigma^{1}, \Sigma^{2}, \cdots, \Sigma^{d-1}, \Sigma^{d}.
\end{equation*}
Take $\Sigma = \Sigma^{0} $ and for $0\leq k \leq d-1 $, construct $ \Sigma^{k+1}$ from $ \Sigma^{k} $ by performing central stellar subdivisions along each of the codimension-$k$ cones of $ \Sigma $ viewed as cones in $ \Sigma^{k} $.\footnote{The subdivisions performed when passing from $ \Sigma^{k} $ to $ \Sigma^{k+1} $ are independent and therefore can be performed simultaneously or in any particular order.}
 We define the \textbf{barycentric subdivision }of $\Sigma$ to be $\operatorname{Bar}(\Sigma) = \Sigma^{d}$.\footnote{Observe that stellar subdivisions of rays in simplicial fans have no effect, so $ \Sigma^{d-1} =  \Sigma^{d} $.  We include the latter for notational convenience.}

\end{definition}

The face poset of $\operatorname{Bar}(\Sigma)$ is the order complex of the face poset of $\Sigma$ with its minimum element removed.
More precisely, for each non-empty chain $\mathsf{L}(\Sigma) \subsetneq \sigma_1 \subsetneq \dots \subsetneq \sigma_k $ of cones in $ \Sigma $, the cone
\begin{equation*}
\operatorname{Cone} \left( \sigma_1 \subsetneq \dots \subsetneq \sigma_k  \right) \coloneqq \operatorname{Cone}(\left\{ u_{\sigma_{i}} \,\middle|\, i \in [k]\right\} ) 
\end{equation*}
is a cone in $\operatorname{Bar}(\Sigma)$, and every cone in  $\operatorname{Bar}(\Sigma) $ arises in this way.  Note that the underlying simplicial complex of $\operatorname{Bar}(\Sigma)$ is a flag complex: a collection of its rays spans a cone if and only if every pair spans a cone.

\begin{example}\label{ex:braid_one}
Recall we defined a fan $ \mathcal{B}^{0}_{n} $ in \Cref{ex:braid_zero} -- it is the normal fan of the standard simplex.  We define the \textbf{braid fan}, also known as the \textbf{braid arrangement}, as
\begin{equation*}
\mathcal{B}^{1}_{n} \coloneqq \operatorname{Bar}(\mathcal{B}^{0}_{n}).
\end{equation*}

Recall that $\sigma_I=  \operatorname{Span}\left\{ e_{[n]}\right\} + \operatorname{Cone} \left\{ e_{i} \,\middle|\, i\in  I \right\}$ has pointed representative ${}^{\wedge}\sigma_I \coloneq \operatorname{Cone} \left\{ e_{i} \,\middle|\, i\in  I \right\}$.  Observe that $e_I$ canonically generates the barycenter of ${}^{\wedge}\sigma_I $. 
The correspondence between chains in the face lattice of a fan and the cones of its barycentric subdivision allows us to describe the cones of $  \mathcal{B}^{1}_{n} $ combinatorially:
for each chain of proper subsets $ \emptyset \subsetneq S_1 \subsetneq S_2 \subsetneq \dots \subsetneq S_k \subsetneq [n]  $ we have the cone 
\begin{equation*}
\sigma(S_1,\dots,S_k) = \operatorname{Cone}\left\{ e_{S_i} \,\middle|\, 1\leq i \leq k\right\} + \operatorname{Span}\left\{ e_{[n]}\right\}.
\end{equation*}
When the chain is empty we obtain the lineality space $\operatorname{Span}\left\{ e_{[n]}\right\}$; when the chain consists of a single subset we get all $ 2^{n}-2 $ possible rays modulo lineality.
Finally, when the chain is saturated, we obtain one of the $n!$ full-dimensional cones, which are naturally in bijection with permutations of $[n]$.
This fan can be alternatively described as the fan induced by the hyperplane arrangement\footnote{Although this is a hyperplane arrangement, general barycentric subdivisions of simplicial fans are not.}
\[
	\mathcal{A}_{n} = \left\{ H_{ij} \,\middle|\, 1\leq i < j \leq n \right\}, \qquad \text{where} \qquad H_{ij} = \left\{ x\in \mathbb{R}^n \,\middle|\, x_i - x_j = 0 \right\}.
\]
The braid fan has a distinguished pointed representative 
${}^{\wedge}\mathcal{B}^{1}_{n}$ consisting of the cones ${}^{\wedge}\sigma(S_1,\dots,S_k)=  \operatorname{Cone}\left\{ e_{S_i} \,\middle|\, 1\leq i \leq k\right\}$ ranging over chains of proper subsets.
\end{example}

\begin{example}\label{ex:k_braid}
More generally, we define the \textbf{$k$-braid fan} $ \mathcal{B}^{k}_{n} = \operatorname{Bar}(\mathcal{B}^{k-1}_{n})$.  For $k=0$, this is the normal fan of the standard simplex.  For $k=1$, this is the braid fan.  For $k=2$ this is the 2-braid fan, i.e. the nested braid fan or the normal fan of the permutopermutohedron.
\end{example}

We now describe the polytopal companion to barycentric subdivisions.

\begin{definition}[\textbf{Omnitruncation}]\label{def:totalnew}
Let $\mathsf{S}$ be a $d$-polytope and let $\boldsymbol{\epsilon}
=
\{\epsilon_{\mathsf{F}}> 0
\mid \emptyset\subsetneq \mathsf{F}\subsetneq \mathsf{S}\}$.  Take some linear extension $\mathsf F_1 < \ldots <  \mathsf F_k$ of the proper part of the face poset of  $\mathsf{S}$, and let $\mathsf{S}_0 \coloneq \mathsf{S}$.  We say that $\boldsymbol{\epsilon}$ is \textbf{shallow} for $\mathsf{S}$ if for each $1\leq i \leq k$, the polytope $\mathsf{S}_i \coloneq \operatorname{Tr}_{\epsilon_{\mathsf F_i}}[\mathsf{S}_{i-1},\mathsf F_i\cap \mathsf{S}_{i-1}]$ is a shallow truncation of $\mathsf{S}_{i-1}$.  If  $\boldsymbol{\epsilon}$ is shallow for $ \mathsf{S}$, the associated omnitruncation of $\mathsf{S}$ is $\operatorname{Omni}_{\boldsymbol\epsilon}(\mathsf{S})\coloneq \mathsf{S}_k$.

\end{definition}

Suppose $\mathsf{S}=\{x\in\operatorname{aff}(\mathsf S) \mid u_{\mathsf{G}}^{\intercal}x\geq b_{\mathsf{G}}
\text{ for every facet } \mathsf{G}\subsetneq \mathsf{S}\}$. For every nonempty proper face $\mathsf{F}\subsetneq \mathsf{S}$, let
$u_{\mathsf{F}}$ be the canonical inner face normal of
$\mathsf{F}$, and set $b_{\mathsf{F}} \coloneqq h(\mathsf{S},u_{\mathsf{F}})$.  By construction, 
\[
\operatorname{Omni}_{\boldsymbol{\epsilon}}(\mathsf{S})
=
\left\{
x\in\operatorname{aff}(\mathsf S) \,\middle|\,
u_{\mathsf{F}}^{\intercal}x
\geq
b_{\mathsf{F}}+\epsilon_{\mathsf{F}}
\text{ for every }
\emptyset\subsetneq \mathsf{F}\subsetneq \mathsf{S}
\right\}.
\]

The normal fan of an omnitruncation of $ \mathsf{S} $ is the  barycentric subdivision of the normal fan of $ \mathsf{S}$:
 
\[
\operatorname{Bar}(\Sigma(\mathsf{S})) = \Sigma \left( \operatorname{Omni}_{\boldsymbol{\epsilon}}(\mathsf{S}) \right).
\]

When the choice of $\boldsymbol{\epsilon}$ is immaterial, we write $\operatorname{Omni}(\mathsf{S})$ for a fixed but arbitrary omnitruncation of $\mathsf{S}$.

\begin{remark}\label{rmk:legitepsilon}
Note that \emph{a priori} the shallowness of ${\boldsymbol{\epsilon}}$ depends on the given linear extension used.  However, it can be shown that the definition is independent of the linear extension chosen.  Additionally, one might naturally allow that $\epsilon_{\mathsf{F}}=0$ when
$\mathsf{F}$ is a facet; our choice was made so that the set of omnitruncations is an open set in the deformation cone.  
\end{remark}

\begin{definition}\label{def:coneofomnitruncations}
Let $\mathcal{T}_{\operatorname{Omni}}(\mathsf{S}) \coloneq \{ \lambda\mathsf{P} + v: \lambda >0, v \in \mathbb{R}^n, \mathsf{P} \text{\,\, is an omnitruncation of\,\,} \mathsf{S}\}$.
\end{definition}

\begin{lemma}\label{lem:omnicone}
Let $\mathsf{S}$ be a simple polytope.  Then $\mathcal{T}_{\operatorname{Omni}}(\mathsf{S})$ is an open cone contained in $\operatorname{Def}(\operatorname{Bar}(\Sigma(\mathsf{S})))$.
\end{lemma}

In \Cref{subsect:Btruncations} we will extend the above constructions and results to the setting of general building sets on projective simplicial fans, and \Cref{lem:omnicone} will follow from \Cref{lem:buildingtruncationcone}.

\begin{example}
	A \textbf{permutahedron} is defined to be the convex hull of the $S_n$-orbit of a point in $\mathbb{R}^n$ with distinct coordinates.  
    Any symmetric omnitruncation of the standard simplex $ \Delta_{n} $ is a permutahedron, but not all permutahedra are symmetric omnitruncations of $\Delta_n$ as the next example shows.	 \end{example}

\begin{example}[The standard permutahedron is not an omnitruncation]
\label{ex:standard_perm_not_omni}
The standard permutahedron is
\[
\Pi_n
\coloneqq
\sum_{1\leq i<j\leq n}\Delta(\{i,j\}).
\]
Its normal fan is the braid fan, which is the barycentric
subdivision of the normal fan of $\Delta_n$.  Thus $\Pi_n$ is
normally equivalent to an omnitruncation of $\Delta_n$.
Nevertheless, when $n\geq5$, the standard permutahedron does not
belong to the closure of
$\mathcal{T}_{\operatorname{Omni}}(\Delta_n)$.
\end{example}

\begin{proof}
For a polytope in
$\mathcal{T}_{\operatorname{Omni}}(\Delta_n)$, translate its support
function so that its values on the vectors $e_i$ are zero, and write
\[
q_A\coloneqq h(\mathsf P,e_A).
\]
For distinct $i,j\in[n]$, the requirement that the vertex
truncations at $i$ and $j$ be performed shallowly implies
\[
q_{[n]}
>
q_{[n]\setminus\{i\}}
+
q_{[n]\setminus\{j\}}.
\]
Indeed, whichever of the two vertices is truncated first shortens the
edge joining them, and shallowness of the second truncation gives the
displayed inequality.  Consequently, every point in the closure of
the omnitruncation cone satisfies the corresponding weak inequality.

For the standard permutahedron, additivity of support functions gives
\[
h(\Pi_n,e_A)=\binom{|A|}{2},
\]
because the segment $\Delta(\{i,j\})$ contributes $1$ precisely when
$\{i,j\}\subseteq A$.  Hence
\[
q_{[n]}=\binom{n}{2}
\qquad\text{and}\qquad
q_{[n]\setminus\{i\}}=\binom{n-1}{2}.
\]
For $n\geq5$, we have
\[
2\binom{n-1}{2}-\binom{n}{2}
=
\frac{(n-1)(n-4)}{2}
>
0.
\]
Thus $\Pi_n$ violates the weak inequality defining the closure of
$\mathcal{T}_{\operatorname{Omni}}(\Delta_n)$.

We remark that $\Pi_4$ is not an omnitruncation of $\Delta_4$, however it does live on the boundary of $\mathcal{T}_{\operatorname{Omni}}(\Delta_4)$.  This is the truncated octahedron realization; the octahedron is a limit of vertex truncations of $\Delta_4$, and the vertices of the octahedron are the limits of the edges of the truncated $\Delta_4$.
\end{proof}

\begin{remark}\label{rmk:omniDelta4}
For the interested reader, we note that for $n\geq 4$, the simplex $\Delta_n$ and the negative simplex $-\Delta_n$ have disjoint omnitruncation cones.  This can be proven as follows: For $n=4$ and any $\boldsymbol{\epsilon}$ that is shallow for $\Delta_4$, the sum of the areas of the facets of $\operatorname{Omni}_{\boldsymbol{\epsilon}}(\Delta_4)$ corresponding to truncations of vertices of $\Delta_4$ is strictly less than the sum of the areas of the facets of $\operatorname{Omni}_{\boldsymbol{\epsilon}}(\Delta_4)$ corresponding to truncations of facets of $\Delta_4$.  The opposite inequality holds for any omnitruncation of $-\Delta_4$.  The statement then holds for general $n$ by considering faces.
\end{remark}

\begin{proposition}\label{prop:omni_product}
Let $\mathsf P$ and $\mathsf Q$ be simple polytopes.
Then $\operatorname{Omni}(\mathsf P)\times\operatorname{Omni}(\mathsf Q)$ is a deformation of $\operatorname{Omni}(\mathsf P\times\mathsf Q)$.
\end{proposition}

\begin{proof}
By \Cref{prop:coarsening_summand} it suffices to show that $\Sigma(\operatorname{Omni}(\mathsf P\times\mathsf Q))$ refines $\Sigma(\operatorname{Omni}(\mathsf P))\times\Sigma(\operatorname{Omni}(\mathsf Q))$.
Write $d=\dim\mathsf P$, $e=\dim \mathsf Q$, and $\Sigma_{\mathsf P}=\Sigma(\mathsf P)$, $\Sigma_{\mathsf Q}=\Sigma(\mathsf Q)$, so that
\[
\Sigma(\operatorname{Omni}(\mathsf P\times\mathsf Q)) = \operatorname{Bar}(\Sigma_{\mathsf P}\times\Sigma_{\mathsf Q})
\qquad\text{and}\qquad
\Sigma(\operatorname{Omni}(\mathsf P))\times\Sigma(\operatorname{Omni}(\mathsf Q)) = \operatorname{Bar}(\Sigma_{\mathsf P})\times\operatorname{Bar}(\Sigma_{\mathsf Q}).
\]
Both are complete fans of dimension $d+e$, so it suffices to show every maximal cone of the first is contained in a maximal cone of the second.

We first treat the case where both polytopes are full dimensional.
Every vertex $v$ of $\operatorname{Omni}(\mathsf P\times\mathsf Q)$ corresponds to a saturated chain of faces from a vertex of $\mathsf P\times\mathsf Q$ up to $\mathsf P\times\mathsf Q$ itself,
\[
\mathsf H_0\subsetneq\mathsf H_1\subsetneq\cdots\subsetneq\mathsf H_{d+e-1}\subsetneq\mathsf P\times\mathsf Q.
\]
Since every face of $\mathsf P\times\mathsf Q$ is a product of a face of $\mathsf P$ with a face of $\mathsf Q$, and each covering relation in a product poset changes exactly one factor, this single chain determines two saturated chains
\[
\mathsf F_0\subsetneq\cdots\subsetneq\mathsf F_{d-1}\subsetneq \mathsf{F}_{d} = \mathsf{P}
\quad,
\qquad
\mathsf G_0\subsetneq\cdots\subsetneq\mathsf G_{e-1}\subsetneq \mathsf{G}_{e} = \mathsf{Q}
\quad,
\]
namely the distinct values taken by the two factors as $\mathsf H_\bullet$ is traversed, and every $\mathsf H_r$ equals $\mathsf F_i\times\mathsf G_j$ for the values of $i,j$ current at step $r$.

The chains $\mathsf F_\bullet$ and $\mathsf G_\bullet$ are themselves the saturated chains corresponding to a vertex $p$ of $\operatorname{Omni}(\mathsf P)$ and a vertex $q$ of $\operatorname{Omni}(\mathsf Q)$, giving a maximal cone
\[
\tau
\coloneqq
\operatorname{ncone}(p,\operatorname{Omni}\mathsf P)\times\operatorname{ncone}(q,\operatorname{Omni}\mathsf Q)
=
\operatorname{Cone}\bigl(u_{\mathsf F_0},\ldots,u_{\mathsf F_{d-1}},u_{\mathsf G_0},\ldots,u_{\mathsf G_{e-1}}\bigr)
\]
of $\Sigma(\operatorname{Omni}(\mathsf P))\times\Sigma(\operatorname{Omni}(\mathsf Q))$ (recall that $ u_{\mathsf{P}} = u_{\mathsf{Q}} = 0 $).
Using $u_{\mathsf F\times\mathsf G}=u_{\mathsf F}+u_{\mathsf G}$, every ray of the normal cone of $v$ satisfies $u_{\mathsf H_r}=u_{\mathsf F_i}+u_{\mathsf G_j}\in\tau$, so
\[
\operatorname{ncone}(v,\operatorname{Omni}(\mathsf P\times\mathsf Q))
=
\operatorname{Cone}\bigl(u_{\mathsf H_0},\ldots,u_{\mathsf H_{d+e-1}}\bigr)
\subseteq
\tau.
\]

Since $v$ was an arbitrary vertex, every maximal cone of $\operatorname{Bar}(\Sigma_{\mathsf P}\times\Sigma_{\mathsf Q})$ lies in a maximal cone of $\operatorname{Bar}(\Sigma_{\mathsf P})\times\operatorname{Bar}(\Sigma_{\mathsf Q})$, proving the refinement.

We now extend to the case where $\mathsf P,\mathsf Q$ need not be full-dimensional.
Let $\mathsf L(\mathsf P),\mathsf L(\mathsf Q)$ be the lineality spaces of $\Sigma_{\mathsf P},\Sigma_{\mathsf Q}$ (\Cref{eq:lineality_space}), and ${}^\wedge\Sigma_{\mathsf P},{}^\wedge\Sigma_{\mathsf Q}$ pointed representatives, so $\Sigma_{\mathsf P}={}^\wedge\Sigma_{\mathsf P}+\mathsf L(\mathsf P)$ and likewise for $\mathsf Q$ (\Cref{prop:eta}); concretely, ${}^\wedge\Sigma_{\mathsf P}$ is the normal fan of $\mathsf P$ realized as a full-dimensional polytope $\mathsf P_V$ in its affine hull direction, and $\mathsf P_V$ is simple exactly when $\mathsf P$ is.

For any fan $\Sigma$ with lineality space $L$, the map $\sigma\mapsto{}^\wedge\sigma$ is an isomorphism of face posets $\Sigma\to{}^\wedge\Sigma$, so chains of proper faces of the two agree and
\[
\operatorname{Bar}(\Sigma) = \operatorname{Bar}({}^\wedge\Sigma) + L.
\]
Since $\sigma_1\times\sigma_2 = ({}^\wedge\sigma_1\times{}^\wedge\sigma_2)+(\mathsf L(\mathsf P)\times\mathsf L(\mathsf Q))$ for $\sigma_1\in\Sigma_{\mathsf P}$, $\sigma_2\in\Sigma_{\mathsf Q}$, the fan $\Sigma_{\mathsf P}\times\Sigma_{\mathsf Q}$ has lineality space $\mathsf L(\mathsf P)\times\mathsf L(\mathsf Q)=\mathsf L(\mathsf P\times\mathsf Q)$ and pointed representative ${}^\wedge\Sigma_{\mathsf P}\times{}^\wedge\Sigma_{\mathsf Q}$.
Applying the displayed identity to $\Sigma_{\mathsf P}\times\Sigma_{\mathsf Q}$, and separately to $\Sigma_{\mathsf P}$ and $\Sigma_{\mathsf Q}$, gives
\[
\operatorname{Bar}(\Sigma_{\mathsf P}\times\Sigma_{\mathsf Q}) = \operatorname{Bar}({}^\wedge\Sigma_{\mathsf P}\times{}^\wedge\Sigma_{\mathsf Q})+\mathsf L(\mathsf P\times\mathsf Q),
\qquad
\operatorname{Bar}(\Sigma_{\mathsf P})\times\operatorname{Bar}(\Sigma_{\mathsf Q}) = \bigl(\operatorname{Bar}({}^\wedge\Sigma_{\mathsf P})\times\operatorname{Bar}({}^\wedge\Sigma_{\mathsf Q})\bigr)+\mathsf L(\mathsf P\times\mathsf Q).
\]
The argument above, applied to the full-dimensional simple polytopes $\mathsf P_V,\mathsf Q_W$, shows $\operatorname{Bar}({}^\wedge\Sigma_{\mathsf P}\times{}^\wedge\Sigma_{\mathsf Q})$ refines $\operatorname{Bar}({}^\wedge\Sigma_{\mathsf P})\times\operatorname{Bar}({}^\wedge\Sigma_{\mathsf Q})$, and adding the common lineality space $\mathsf L(\mathsf P\times\mathsf Q)$ to every cone on both sides preserves this refinement, giving the claim in general.

\end{proof}

Fixing $\mathsf F_\bullet,\mathsf G_\bullet$ and letting $\mathsf H_\bullet$ range over the $\binom{d+e}{d}$ possible interleavings shows how $\tau$ is triangulated.  This triangulation is analogous to the staircase triangulation of a product of two simplices, with the simplices replaced by the corresponding simplicial cones.\footnote{See \cite[Chapter~6]{de2010triangulations} for an introduction to staircase triangulations.}

\begin{corollary}
If $ \mathsf{P}, \mathsf{Q} $ are simple polytopes, then $\operatorname{Omni}(\mathsf P)$ is a deformation of $\operatorname{Omni}(\mathsf P\times\mathsf Q)$.
\end{corollary}

In the statement, we consider $ \operatorname{Omni}(\mathsf P) $ as living in the same ambient space as $\operatorname{Omni}(\mathsf P\times\mathsf Q)$ by embedding it with a zero in the second coordinate.

\begin{proof}
This follows from \Cref{prop:omni_product} and the fact that a Cartesian product is a Minkowski sum.
\end{proof}

\subsection{Parameterizing deformations}\label{sec:basis}

We recall the space of all deformations of a polytope $ \mathsf{P} $.
We will focus for now on the case where $ \mathsf{P} $ is full-dimensional, so that its normal fan is pointed, and at the end we address the general case.

By \Cref{prop:coarsening_summand}, in terms of normal fans, our goal is to describe all polytopes whose normal fan coarsens $\Sigma(\mathsf{P})$. 
We adopt this perspective, which is classical and has also been extensively developed in the context of toric varieties.

Let $ \Sigma $ be a pointed polytopal fan in $ \mathbb{R}^{n} $.
A piecewise linear function supported on $ \Sigma $ is a function $ \phi: \mathbb{R}^{n} \to \mathbb{R} $ such that there exist linear functions $ \left\{ \ell_\sigma \,\middle|\, \sigma \in \Sigma \right\} $,
such that $ \phi\vert_\sigma = \ell_\sigma $ for every cone $ \sigma \in \Sigma$.
We define $ \operatorname{PL}(\Sigma) $ as the vector space of all piecewise linear functions on $ \Sigma $.

We define the set
\begin{equation}\label{eq:ray_generators}
	\mathbf{U}_{\Sigma} = \left\{ u_{\tau} \,\middle|\, \tau \in \Sigma,\, \dim(\tau) = 1 \right\}
\end{equation}
of all primitive vectors that generate rays of the fan.
Note that when $\Sigma$ is the normal fan of a polytope, the set $\mathbf{U}_{\Sigma}$ consists of the inner facet normals of the polytope.
Restricting the domain of each piecewise linear function from $ \mathbb{R}^{n} $ to $ \mathbf{U}_{\Sigma} $ we have the following injective map of vector spaces
\begin{equation}\label{eq:iota}
	\iota: \operatorname{PL}(\Sigma) \hookrightarrow \mathbb{R}^{\mathbf{U}_{\Sigma}}.
\end{equation}

If the fan $ \Sigma $ is simplicial, the map $ \iota $ is an isomorphism;
any function $ f: \mathbf{U}_{\Sigma} \to \mathbb{R} $ can be extended \emph{uniquely} to a piecewise linear function $ \phi $ satisfying $ \phi = f $ on $ \mathbf{U}_{\Sigma} $.
If the fan is not simplicial, such an extension may not exist.
In this case, the image of $  \operatorname{PL}(\Sigma) $ under $ \iota $ is a linear subspace of $ \mathbb{R}^{\mathbf{U}_{\Sigma}} $.

Let $ \mathsf{Q} $ be a polytope such that its normal fan is a coarsening of $ \Sigma $.
This polytope can be uniquely described in the form
\begin{equation*}
\mathsf{Q} = \left\{ x \,\middle|\, u^{\intercal}x \geq h \left( {\mathsf{Q}} , u \right) , \, u \in \mathbf{U}_{\Sigma} \right\},
\end{equation*}
using the support function $ h \left( {\mathsf{Q}} , \cdot \right) $.
When $\Sigma$ is understood, we write $ h ( \mathsf{Q} ) $ for this function as a vector in $ \mathbb{R}^{\mathbf{U}_{\Sigma}} $.

\begin{definition}\label{def:def}
Let $\mathsf{P}$ be a polytope. 
The vector space $\mathrm{DS}(\mathsf{P}) = \operatorname{PL}(\Sigma(\mathsf{P}))$ is called the \textbf{deformation space} of $\mathsf{P}$.
\end{definition}

\begin{definition}\label{def:deformationcone}

The \textbf{deformation cone} of $\mathsf{P}$, denoted $\operatorname{Def}(\mathsf{P})$, is
\[
\operatorname{Def}(\mathsf{P})\coloneqq\left\{h(\mathsf{Q},\mathord\cdot)\,\middle|\,\mathsf{Q}\text{ is a deformation of }\mathsf{P}\right\}\subseteq\mathrm{DS}(\mathsf{P}).
\]
Equivalently, by \Cref{prop:coarsening_summand}, it consists of the support functions of polytopes whose normal fans coarsen $\Sigma(\mathsf{P})$.

When $\Sigma(\mathsf P)$ is a rational fan in $N_{\mathbb R}$, we define
the \textbf{integral deformation space} $\mathrm{DS}_{\mathbb Z}(\mathsf P)$ to be
the subgroup of $\mathrm{DS}(\mathsf P)$ consisting of piecewise linear
functions whose restriction to each cone of $\Sigma(\mathsf P)$ belongs to
$M$. We define the \textbf{integral deformation cone} by
$\operatorname{Def}_{\mathbb Z}(\mathsf P)
=\operatorname{Def}(\mathsf P)\cap\mathrm{DS}_{\mathbb Z}(\mathsf P)$.
\end{definition}

We will freely identify a deformation $\mathsf Q$ of $\mathsf P$ with its support function $h(\mathsf Q)\in\operatorname{Def}(\mathsf P)$. Accordingly, we use the same basis and generating-set terminology for collections of deformations and for the corresponding collections of support functions.

The fact that $\operatorname{Def}(\mathsf P)$ is a polyhedral cone is proven in \cite[Theorem 3]{mcmullen1973representations}. 
When $\Sigma(\mathsf P)$ is a rational simplicial fan, its defining inequalities can be described using Batyrev's criterion~\cite[Theorem 6.4.9]{cox2024toric}. 
Recall that a \textbf{primitive collection} is a set $ \mathbf C=\{u_1,\ldots,u_k\}\subseteq \mathbf U_{\Sigma(\mathsf P)}$
of canonical ray generators which does not span a cone of $\Sigma(\mathsf P)$, but for which every proper subset does. 
Batyrev's criterion states that a function $h\in\operatorname{PL}(\Sigma(\mathsf P))$ belongs to $\operatorname{Def}(\mathsf P)$ if and only if, for every primitive collection $\mathbf C$, it satisfies
\[
h(u_1)+\cdots+h(u_k)
\leq
h(u_1+\cdots+u_k).
\]

To express these inequalities through the isomorphism $\iota$ (Equation \eqref{eq:iota}),
let $\sigma$ be the unique cone of $\Sigma(\mathsf P)$ whose relative interior contains $u_1+\cdots+u_k$. 
Writing
\[
u_1+\cdots+u_k
=
c_1v_1+\cdots+c_\ell v_\ell,
\qquad
c_1,\ldots,c_\ell>0,
\]
where $v_1,\ldots,v_\ell$ are the canonical ray generators of $\sigma$, the inequality associated with $\mathbf C$ becomes
\begin{equation}\label{eq:batyrev}
h(u_1)+\cdots+h(u_k)
\leq
c_1h(v_1)+\cdots+c_\ell h(v_\ell).
\end{equation}

\begin{definition}
Let $\mathscr{A}$ be a collection of deformations of $\mathsf{P}$.  
If the support functions of the polytopes in $\mathscr{A}$ linearly generate the deformation space $\mathrm{DS}(\mathsf{P})$, we call  $\mathscr{A}$ a \textbf{polytopal generating set} (for the deformation space of $\mathsf{P}$). If the support functions of the polytopes in $\mathscr{A}$ are a basis for the deformation space of $\mathsf{P}$, we call $\mathscr{A}$ a \textbf{polytopal basis} (for the deformation space of $\mathsf{P}$).  We further call a polytopal generating set $\mathscr{A}$ indecomposable if each element of $\mathscr{A}$ is indecomposable, and we call $\mathscr{A}$ integral if it generates $\mathrm{DS}_{\mathbb{Z}}(\mathsf{P})$ over $\mathbb{Z}$.
\end{definition}

Having a polytopal basis for $\mathrm{DS}(\mathsf{P})$ allows us to write any deformation of $\mathsf{P}$ as a \textit{signed} Minkowski sum of the polytopes in the basis.

\begin{lemma}\label{lem:signed}
	Let $ \left\{ \mathsf{Q}_1, \dots,\mathsf{Q}_m\right\} \subset \operatorname{Def}(\mathsf{P})$ be a polytopal basis of $ \mathrm{DS}(\mathsf{P})$.
	Then for any $ \mathsf{Q} \in \operatorname{Def}(\mathsf{P})$ there exist unique real scalars $ \left\{ \lambda_i \right\}_{i\in [m]} $ such that
	\begin{equation}\label{eq:signed_supports}
	h(\mathsf{Q}) = \sum_{i=1}^{m} \lambda_{i} h(\mathsf{Q}_{i}),
	\end{equation}
	and this is equivalent to the following equality of Minkowski sums of polytopes:
	\begin{equation}\label{eq:signed_decomp}
	\mathsf{Q} + \sum_{\lambda_{i}<0} (-\lambda_{i})\mathsf{Q}_{i} = \sum_{\lambda_{i}>0} \lambda_{i}\mathsf{Q}_{i}.
	\end{equation}
	
\end{lemma}

\begin{proof}
	 Equation \eqref{eq:signed_supports} follows by construction.
	It follows that 
	\begin{equation}\label{eq:temp}
		h(\mathsf{Q}) + \sum_{\lambda_{i}<0} (-\lambda_{i})h(\mathsf{Q}_{i}) = \sum_{\lambda_{i}>0} \lambda_{i}h(\mathsf{Q}_{i}).
	\end{equation}
	The left-hand side of Equation \eqref{eq:temp} is the support function of the sum $ \mathsf{Q} + \sum_{\lambda_{i}<0} (-\lambda_{i})\mathsf{Q}_{i}  $.
	Analogously, the right-hand side of Equation \eqref{eq:temp} is the support function of the sum $\sum_{\lambda_{i}>0} \lambda_{i}\mathsf{Q}_{i}$.
	Since the support function uniquely determines the polytope, the second conclusion follows.
\end{proof}

\begin{remark}\label{rmk:subtletyminkowski}
For the signed Minkowski sum identity in \Cref{lem:signed}, it is important that  $ \mathsf{Q} \in \operatorname{Def}(\mathsf{P})$; it is not sufficient for $\mathsf{Q}$ to be a polytope cut out by inequalities associated to the rays of $\Sigma(\mathsf{P})$.  To see an interesting instance of this subtlety, see \Cref{subsec:cosmo} where the cosmohedron is not a signed Minkowski sum of the Gherkins although its support values on the rays of the $2$-braid fan are a unique linear combination of the corresponding support values of the Gherkins.
\end{remark}

By \eqref{eq:support_sum} and \Cref{rem:single_point}, for a polytope $ \mathsf{P} $ and a point $ c \in \mathbb{R}^{n} $, we have
\begin{equation*}
	h(\mathsf{P} + \left\{ c \right\}, \cdot ) =	h(\mathsf{P}, \cdot ) + h( \left\{ c \right\}, \cdot ) =	h(\mathsf{P}, \cdot ) + c^{\intercal}(\cdot) .
\end{equation*}
In other words, the translation of a polytope corresponds to adding a global linear function to its support function.  The deformation cone $ \operatorname{Def}(\mathsf{P})$ has a lineality space $\mathsf{L}$ corresponding to the collection of translations in the ambient space.

\begin{definition}\label{def:nef_cone}
The image of the cone $ \operatorname{Def}(\mathsf{P}) $ in the quotient $\mathrm{DS}(\mathsf{P}) / \mathsf{L}$ is a pointed cone called the \textbf{nef cone} and is denoted $ \operatorname{Nef}( \mathsf{P} ) $.
\end{definition}

The extremal rays of the nef cone correspond to the indecomposable deformations of $ \mathsf{P} $, up to translation.
 The following lemma and its corollaries are well-known, but we include their proofs since they illustrate techniques that will be used later.

When $\mathsf{P}$ is a simple polytope, its deformation space $\mathrm{DS}(\mathsf{P})$ has dimension equal to the number of facets of $\mathsf{P}$; since $\mathsf P$ is simple, its normal fan is simplicial, and \eqref{eq:iota} identifies $\mathrm{DS}(\mathsf P)$ with the vector space of functions on the facets of $\mathsf P$.

\begin{lemma}\label{lem:P_affine_basis}
Let $ \mathsf{P} $ be a simple full-dimensional polytope and let $ \mathcal{F} $ be its set of facets.
For $\epsilon$ small enough, the set of vectors $ \left\{ h(\mathsf{P}) - h(\operatorname{Tr}_{\epsilon} [\mathsf{P}, \mathsf{F}] )  \,\middle|\, \mathsf{F} \in \mathcal{F} \right\}$ is a (not necessarily polytopal) basis for the deformation space $ \mathrm{DS}(\mathsf{P}) $.
\end{lemma}

\begin{proof}
The set has the right cardinality, so it suffices to show the elements are linearly independent.
Evaluating the functions on each inner normal $ u_{\mathsf{H}} $, where $\mathsf{H} \subseteq \mathsf{P}$ is a facet, we get
\begin{equation*}
		h(\mathsf{P} , u_{\mathsf{H}}) -  h(\operatorname{Tr}_{\epsilon} (\mathsf{P}, \mathsf{F}) , u_{\mathsf{H}}) =
		\begin{cases}
			-\epsilon,& \text{ if }  \mathsf{H} = \mathsf{F}, \\
			0,& \text{ else. } 
	\end{cases}
	\end{equation*}
It follows that the set is a basis.
\end{proof}

\begin{remark}\label{rmk:origin_interior}
In some results in this article, we will impose a technical condition that the origin is in the interior of a full-dimensional polytope $\mathsf{P}$.  This is equivalent to the condition $h(\mathsf{P} , u_{\mathsf{F}})<0$ for all facets $\mathsf{F}$ of $\mathsf{P}$.  This can be achieved by translation of $\mathsf{P}$ and, as noted previously, such translations are encoded by elements of the lineality space of $\operatorname{Def}(\mathsf{P}) $.  This condition is essentially harmless, as one typically cares about $\operatorname{Def}(\mathsf{P}) $ modulo lineality, i.e., $ \operatorname{Nef}( \mathsf{P} ) $.  We note that if $\mathsf{P}$ is further taken to be integral, the condition that the origin is in the interior is equivalent to the condition $h(\mathsf{P} , u_{\mathsf{F}})\leq -1$.  This condition cannot always be obtained by a lattice translation.
\end{remark}

\begin{corollary}\label{cor:P_in_the_interior}
Let $ \mathsf{P} $ be a simple full-dimensional polytope with origin in the interior and let $ \mathcal{F} $ be its set of facets.
The support vector $ h(\mathsf{P}) $ is in the interior of 
$ \operatorname{Cone}\left\{ h(\operatorname{Tr}_{\epsilon} [\mathsf{P}, \mathsf{F}] )  \,\middle|\, \mathsf{F} \in \mathcal{F} \right\} $ for $\epsilon>0$ sufficiently small.
Furthermore, if $ \mathsf{P} $ is integral, then the support vector $ h(\mathsf{P}) $ is also in the interior of 
$ \operatorname{Cone}\left\{ h(\operatorname{Tr}_1 [\mathsf{P}, \mathsf{F}] )  \,\middle|\, \mathsf{F} \in \mathcal{F} \right\} $.
\end{corollary}

\begin{proof}
	
	Since the origin is in the interior of $ \mathsf{P} $, we see that
	\begin{equation} \label{eq:beta}
		\beta_\mathsf{F} \coloneqq h(\mathsf{P} , u_{\mathsf{F}}) < 0,
	\end{equation}
	for every facet $ \mathsf{F} $.
	Using \Cref{lem:P_affine_basis} we obtain the equality
	\begin{equation} \label{eq:p_linear_combo}
		h(\mathsf{P}) = \sum_{\mathsf{F} \in \mathcal{F}} \lambda_\mathsf{F} \left(h(\mathsf{P}) - h(\operatorname{Tr}_{\epsilon} [\mathsf{P}, \mathsf{F}] ) \right), \qquad \text{where} \qquad \lambda_\mathsf{F} = \frac{\beta_{\mathsf{F}}}{-\epsilon} > 0.
	\end{equation}
	The above equation can be rewritten as
	\begin{equation} \label{eq:p_linear_combo_II}
\sum_{\mathsf{F} \in \mathcal{F}} \lambda_{\mathsf{F}} h(\operatorname{Tr}_{\epsilon} [\mathsf{P}, \mathsf{F}] ) = \left( \left(\sum_{\mathsf{F} \in \mathcal{F}} \lambda_{\mathsf{F}} \right) - 1 \right) h(\mathsf{P})   .
	\end{equation}

	For $ \epsilon $ small, Equation \eqref{eq:p_linear_combo_II} shows that $ h(\mathsf{P}) $ is contained in the interior of $\operatorname{Cone}\left\{ h(\operatorname{Tr}_{\epsilon} [\mathsf{P}, \mathsf{F}] )  : \mathsf{F} \in \mathcal{F} \right\}$.
	If $ \mathsf{P} $ is integral, then  $\delta_{\mathsf F}\geq1$, and we can take $ \epsilon = 1 $.  Moreover, $\beta_{\mathsf F}\leq-1$ which gives that $\lambda_\mathsf{F} \geq 1$ for each facet $\mathsf{F}$.  Because there are at least two facets,  $  \left(\sum_{\mathsf{F} \in \mathcal{F}} \lambda_\mathsf{F} \right) - 1 $ is positive and the conclusion follows.
\end{proof}

\begin{example}
Let $ \mathsf{P} $ be the triangle obtained as the convex hull of the points $ (0,0), (1,0), (0,1) $. 
This polytope is primitive Delzant and does not contain the origin in its interior.   Furthermore, none of its lattice translations contain the origin in their interior, since there are no lattice points in the interior of $ \mathsf{P} $.
In this case, the deep truncations of the three facets are the three vertices.
Clearly, there is no way to obtain the original triangle as a Minkowski sum of points.
So \Cref{cor:P_in_the_interior} does not necessarily hold without the condition that the origin is in the interior.
\end{example}

\begin{corollary}\label{cor:truncated_facets_base}
Let $ \mathsf{P} $ be a simple full-dimensional polytope with origin in the interior and let $ \mathcal{F} $ be its set of facets.
The set $ \left\{ h(\operatorname{Tr}_{\epsilon} [\mathsf{P}, \mathsf{F}] )  \,\middle|\, \mathsf{F} \in \mathcal{F} \right\} $ of support functions is a polytopal basis for $ \mathrm{DS} (\mathsf{P}) $ when $\epsilon >0$ is sufficiently small.
\end{corollary}

\begin{corollary}\label{cor:simplicial_nef}
	Let $ \mathsf{P} $ be a simple full-dimensional polytope, with a simplicial normal fan $ \Sigma $.
	The deformation cone $ \operatorname{Def}(\mathsf{P}) $ is full-dimensional inside the deformation space $\mathrm{DS}(\mathsf{P})$.
	In particular, its dimension is equal to the number of facets of $ \mathsf{P} $.
\end{corollary}

\begin{proof}
	By applying a translation if necessary we can assume that the origin is in the interior.
	The support functions of the facet truncations are all deformations because $ \mathsf{P} $ is simple.
	Hence, we have that the cone $\operatorname{Cone}\left\{ h(\operatorname{Tr}_{\epsilon} [\mathsf{P}, \mathsf{F}] )  \,\middle|\, \mathsf{F} \in \mathcal{F} \right\}$ is in the deformation cone $ \operatorname{Def}(\mathsf{P}) $.
	By \Cref{cor:truncated_facets_base} the deformation cone contains a basis for $\mathrm{DS} (\mathsf{P}) $ so it is full-dimensional.
\end{proof}

\begin{corollary}\label{cor:deepfacetbasis}
Let $ \mathsf{P} $ be a simple full-dimensional polytope with origin in the interior and let $ \mathcal{F} $ be its set of facets.
The set $ \left\{ h(\operatorname{DTr}  \left[\mathsf{P}, \mathsf{F} \right] )  \,\middle|\, \mathsf{F} \in \mathcal{F} \right\} \cup\{ h(\mathsf{P})\} $ is a polytopal generating set for $ \mathrm{DS} (\mathsf{P}) $ of cardinality 1 greater than $\text{dim}(\operatorname{Def}(\mathsf P))$.  
\end{corollary}

\begin{proof}
This follows from \Cref{cor:truncated_facets_base} combined with \Cref{lem:convtrunc}.
\end{proof}

We note that while \Cref{cor:deepfacetbasis} does not give a basis, its advantage over  \Cref{cor:truncated_facets_base} is that the generating set from \Cref{cor:deepfacetbasis}  is canonical.

\begin{remark}
Notice that when $ \mathsf{P} $ is not simple, the normal fan of a facet truncation may not be a coarsening of the fan $ \Sigma(\mathsf{P}) $. 
Hence the restriction of its support function to the rays of $\Sigma(\mathsf{P})$ need not lie in
$\operatorname{Def}(\Sigma(\mathsf{P}))$.
\end{remark}

Lastly, we address the case where $ \mathsf{P} $ is a simple polytope but not full-dimensional in its ambient space $ \mathbb{R}^{n} $.
The prototypical example is the standard simplex $ \Delta_{n} $, which is $ (n-1) $-dimensional but is naturally defined in $ \mathbb{R}^{n} $.

\begin{proposition}\label{prop:eta}
Let $\Sigma$ be a fan with lineality space $\mathsf L(\Sigma)$, and let
${}^{\wedge}\Sigma$ be a pointed representative of $\Sigma$.
Then restriction to the pointed representative and to the lineality space gives an isomorphism
\[
\operatorname{PL}(\Sigma)
\cong
\operatorname{PL}({}^{\wedge}\Sigma)
\oplus
\mathsf L(\Sigma)^\ast .
\]
\end{proposition}

\begin{proof}
For each cone $\sigma\in\Sigma$, let ${}^{\wedge}\sigma$ denote the corresponding cone of
${}^{\wedge}\Sigma$, so that
\[
\sigma = {}^{\wedge}\sigma+\mathsf L(\Sigma).
\]
The defining property of a pointed representative gives that a linear function on a cone $\sigma$ decomposes uniquely as a linear function on ${}^{\wedge}\sigma$ together with a linear function on $\mathsf L(\Sigma)$.  This decomposition is compatible with passing to faces, thus decomposing any piecewise-linear function on $\Sigma$.  This decomposition respects linear combinations, implying the statement of the proposition.  \end{proof}

\begin{example}\label{ex:simplex_non_full}
    The standard simplex $\Delta_n \subseteq \mathbb{R}^n$ is not full-dimensional.
    Its normal fan $\mathcal{B}^{0}_{n}$ has lineality space $\mathsf{L} = \operatorname{Span}\left\{ e_{[n]}\right\}$.
	    Following \Cref{ex:braid_zero} we can write $\mathcal{B}^{0}_{n}$ as ${}^{\wedge}\mathcal{B}^{0}_{n} + \mathsf{L}$, where
\begin{equation}
		{}^{\wedge}\mathcal{B}^{0}_{n} = \left\{ \operatorname{Cone} \left\{ e_{i} \,\middle|\, i \in I\right\} \,\middle|\, I \subsetneq [n] \right \},  
\end{equation}
	in other words, the fan ${}^{\wedge}\mathcal{B}^{0}_{n}$ consists of all the proper faces of the positive orthant in $\mathbb{R}^n$.  This is an $(n-1)$-dimensional fan in $\mathbb{R}^n$ not contained in any linear subspace.
	We have that $\operatorname{PL}(\mathcal{B}^{0}_{n})$ is isomorphic to $\operatorname{PL}({}^{\wedge}\mathcal{B}^{0}_{n}) \oplus \mathsf{L}^\ast $.  The deformation cone of $\Delta_n$ is the sum of a ray and a linear space of dimension $n$.  This can be seen from the fact that, up to homothety, all deformations of $\Delta_n$ are translation equivalent to $\Delta_n$ or a single point.
\end{example}

\begin{example}\label{ex:omni_truncation}
    Similar to \Cref{ex:simplex_non_full}, an omnitruncation $\mathsf{P}_n$ of $\Delta_n$ is not full-dimensional either.
    It has the same lineality space $\mathsf{L} = \operatorname{Span}\left\{ e_{[n]}\right\}$.
	    Following \Cref{ex:braid_one} we can write its normal fan $\mathcal{B}^{1}_{n}$ as ${}^{\wedge}\mathcal{B}^{1}_{n} + \mathsf{L}$, where the ray generators of ${}^{\wedge}\mathcal{B}^{1}_{n}$ are the vectors $e_I$ for $I$ a nonempty proper subset of $[n]$.
	    We have that $\operatorname{PL}(\mathcal{B}^{1}_{n})$ is isomorphic to $\operatorname{PL}({}^{\wedge}\mathcal{B}^{1}_{n}) \oplus \mathsf{L}^\ast $.  See \Cref{cor:supermodular_simplex} below for a classical description of $\operatorname{Def}(\mathcal{B}^{1}_{n})$.
\end{example}

\subsection{Deformation spaces of barycentric subdivisions}\label{subsec:defbary}

In this section we describe a polytopal basis for barycentric subdivisions of normal fans of simple polytopes.  
Our main tool is truncations.

\begin{notation}
Let $ \mathsf{S} $ be a simple polytope (informally referred to as the seed) and $ \mathsf{P} = \operatorname{Omni}( \mathsf{S} ) $ be an omnitruncation of $ \mathsf{S} $.
For now we assume that $ \mathsf{S} $ is full-dimensional and treat the general case at the end.
At the level of normal fans we have that $ \Psi = \Sigma( \mathsf{P})	 $ is the barycentric subdivision of $ \Sigma = \Sigma ( \mathsf{S} ) $.
\end{notation}

Since $\Psi$ is simplicial, we freely identify piecewise linear functions on $\Psi$ with functions on its set of canonical ray generators $\mathbf{U}_{\Psi}$ via the inclusion map~\eqref{eq:iota}.
For every nonempty proper face $\mathsf{F}$ of $\mathsf{S}$, let $v_{\mathsf{F}}\in\mathbf{U}_{\Psi}$ denote the canonical ray generator of the ray $\operatorname{Cone}(u_{\mathsf{F}})$.
\footnote{Note that the vector $u_\mathsf{F}$ is not necessarily primitive, hence we could have $v_\mathsf{F}\neq u_\mathsf{F}$.}

\begin{proposition}\label{prop:barycentric_deformation_inequalities}
Let $\mathsf{S}$ be a full-dimensional simple polytope and let $\mathsf{P}$ be an omnitruncation of $\mathsf{S}$, with normal fan $\Psi=\operatorname{Bar}(\Sigma(\mathsf{S}))$. 
A function $h\in\operatorname{PL}(\Psi)$ belongs to $\operatorname{Def}(\mathsf{P})$ if and only if
\begin{equation}\label{eq:barycentric_deformation_inequalities}
    h(v_{\mathsf{F}})+h(v_{\mathsf{G}})
    \leq
    h(v_{\mathsf{F}}+v_{\mathsf{G}})
\end{equation}
for every pair of incomparable nonempty proper faces
$\mathsf{F},\mathsf{G}\subsetneq\mathsf{S}$.
Equivalently, let $ \mathsf{H}_1\subsetneq\cdots\subsetneq\mathsf{H}_{\ell} $ be the unique chain of faces of $\mathsf{S}$ such that $v_{\mathsf{F}}+v_{\mathsf{G}}$ lies in the relative interior of $ \operatorname{Cone} \left\{ v_{\mathsf{H}_1},\ldots,v_{\mathsf{H}_{\ell}} \right\}, $ and write
\[
v_{\mathsf{F}}+v_{\mathsf{G}}
=
c_1v_{\mathsf{H}_1}+\cdots+c_{\ell}v_{\mathsf{H}_{\ell}},
\qquad
c_1,\ldots,c_{\ell}>0.
\]
Then \eqref{eq:barycentric_deformation_inequalities} is equivalent to the linear inequality
\[
h(v_{\mathsf{F}})+h(v_{\mathsf{G}})
\leq
c_1h(v_{\mathsf{H}_1})+\cdots+c_{\ell}h(v_{\mathsf{H}_{\ell}}).
\]
\end{proposition}

\begin{proof}
A collection of rays of $\Psi$ spans a cone if and only if the corresponding faces of $\mathsf{S}$ form a chain. 
Consequently, the minimal collections of rays which do not span a cone are precisely the pairs indexed by incomparable faces: every collection which is not a chain contains a pair of incomparable elements. 
Thus the primitive collections of $\Psi$ are exactly the sets $\{v_{\mathsf{F}},v_{\mathsf{G}}\}$ with $\mathsf{F}$ and $\mathsf{G}$ incomparable.

The first description now follows from Batyrev's criterion.
The equivalent linear description is Equation \eqref{eq:batyrev} applied to this case.
\end{proof}

\begin{corollary}\cite{edmonds1970submodular, ardila2020coxeter, castillo2022deformation}\label{cor:supermodular_simplex}
Let $\mathsf{P}_n$ be an omnitruncation of the standard simplex $\Delta([n])$, so that $ \Sigma(\mathsf{P}_n)=\mathcal{B}^1_n. $
For a function $z\colon 2^{[n]}\to\mathbb{R}$ with $z(\emptyset)=0$, consider the polyhedron
\begin{equation}\label{eq:generalized_permutahedron_supermodular}
\mathsf{Q}(z)
=
\left\{
x\in\mathbb{R}^n
\,\middle|\,
\begin{array}{ll}
\displaystyle\sum_{i\in I}x_i\geq z(I),
& \emptyset\subsetneq I\subsetneq[n],\\[6pt]
\displaystyle\sum_{i\in[n]}x_i=z([n]).
\end{array}
\right\}.
\end{equation}

Then $\mathsf{Q}(z)$ is a deformation of $\mathsf{P}_n$ if and only if $z$ is \textbf{supermodular}, that is,
\begin{equation}\label{eq:supermodular_simplex}
z(I)+z(J)
\leq
z(I\cap J)+z(I\cup J)
\end{equation}
for all $I,J\subseteq[n]$.
Equivalently, it is enough to impose
\eqref{eq:supermodular_simplex} when $I$ and $J$ are incomparable.
\end{corollary}

\begin{proof}
The braid fan has lineality space
\[
\mathsf{L}
=
\operatorname{Span}\{e_{[n]}\},
\]
and its pointed representative has canonical ray generators $e_I$ indexed by the nonempty proper subsets $I\subsetneq[n]$. 
By \Cref{prop:eta}, a piecewise linear function on $\mathcal{B}^1_n$ is therefore determined by its values
\[
h(e_I)=z(I),
\qquad
\emptyset\subsetneq I\subsetneq[n],
\]
together with its restriction to $\mathsf{L}$, which is determined by
\[
h(e_{[n]})=z([n]).
\]

By \Cref{prop:barycentric_deformation_inequalities}, the primitive collections of $\mathcal{B}^1_n$ are precisely the pairs $\{e_I,e_J\}$ indexed by incomparable nonempty proper subsets $I,J\subsetneq[n]$. 
For every such pair, the identity
\[
e_I+e_J=e_{I\cap J}+e_{I\cup J}
\]
gives the associated primitive relation, where we use $e_{\emptyset}=0$. 
If $I\cup J=[n]$, then $e_{I\cup J}=e_{[n]}$ belongs to the lineality space rather than being a ray of the pointed representative. 
Batyrev's inequality is thus precisely
\[
z(I)+z(J)
\leq
z(I\cap J)+z(I\cup J).
\]
This proves the result for incomparable pairs. 
For comparable pairs, the two sides of \eqref{eq:supermodular_simplex} are equal, so the same condition may equivalently be imposed for all $I,J\subseteq[n]$.
\end{proof}

\begin{remark}\label{rem:order}
For the proofs below, it will be convenient to fix a total order on $\mathbf{U}_{\Psi}$. 
The rays of $\Psi$ are naturally indexed by the nonempty proper faces of $\mathsf{S}$, or equivalently by the nonzero cones of $\Sigma$. 
More precisely, a face $\mathsf{F}$ corresponds to the ray $\operatorname{Cone}(u_{\mathsf{F}})$, whose canonical ray generator we denote by $v_{\mathsf{F}}\in\mathbf{U}_{\Psi}$. 
We fix a linear extension $\leq$ of the poset of nonempty proper faces of $\mathsf{S}$ and interpret it as a total order on $\mathbf{U}_{\Psi}$ through the correspondence $\mathsf{F}\mapsto v_{\mathsf{F}}$ (see \Cref{def:orderings}).
\end{remark}

The following argument follows the treatment in \cite[Section 3]{backman2023simplicial}.
For the rest of the section, $ \epsilon $ is a sufficiently small positive real number, so that all central truncations are shallow.

\begin{lemma}\label{lem:key}
	Let $ \mathsf{S} $ be a full-dimensional simple polytope and $ \mathsf{P} $ an omnitruncation of $ \mathsf{S} $.
The set of vectors
\begin{equation*}
	\left\{ h(\operatorname{Tr}_{\epsilon} [\mathsf{S}, \mathsf{F}] ) - h(\mathsf{S}) \,\middle|\, \emptyset \subsetneq \mathsf{F} \subsetneq \mathsf{S} \right\}
\end{equation*}
is a (not necessarily polytopal) basis of the deformation space of $\mathsf{P}$.
\end{lemma}

\begin{proof}
For every proper nonempty face $ \mathsf{F} \subseteq \mathsf{S} $ we have that $  \operatorname{Tr}_{\epsilon} [\mathsf{S}, \mathsf{F}] \subseteq \mathsf{S}$.
Furthermore, the vertices of $ \mathsf{S} $ not contained in $ \operatorname{Tr}_{\epsilon} [\mathsf{S}, \mathsf{F}] $ are exactly the vertices of the face $ \mathsf{F} $.
Thus, evaluating the support vectors at the element $ u_{H} $, for $ \mathsf{H} \subseteq \mathsf{S} $ a proper nonempty face, we get
\begin{equation*}
h(\operatorname{Tr}_{\epsilon} [\mathsf{S}, \mathsf{F}] , u_{H})  - h(\mathsf{S} , u_{H}) = 
\begin{cases}
>0, \text{ if } \emptyset \subsetneq \mathsf{H} \subseteq \mathsf{F}, \\
=0, \text{ else. } 
\end{cases}
\end{equation*}

Using the total order on the vectors $\left\{ h(\operatorname{Tr}_{\epsilon} [\mathsf{S}, \mathsf{F}] ) - h(\mathsf{S}) \,\middle|\, \emptyset \subsetneq \mathsf{F} \subsetneq \mathsf{S}\right\} $ and $ \mathbf{U}_{\Psi} $ described in \Cref{rem:order},
we obtain a square matrix which is upper triangular with non-zero entries in the diagonal.
We conclude that the columns are a basis.
\end{proof}

\begin{remark}\label{rem:key_delzant}
In the case where $ \mathsf{S} $ is Delzant (see \Cref{def:Delzant}), we can be more precise about the upper triangular matrix we built in the proof of \Cref{lem:key}, and even more, we can give its inverse.
We obtain the following description of the difference of support functions
\begin{equation}\label{eq:delzant_difference}
h(\operatorname{Tr}_{1}  \left[\mathsf{S}, \mathsf{F} \right] , u_{H})  - h(\mathsf{S} , u_{H}) = 
\begin{cases}
1, \text{ if }   \emptyset \subsetneq \mathsf{H} \subseteq \mathsf{F}, \\
0, \text{ else. } 
\end{cases}
\end{equation}
This is an integral upper triangular matrix with 1's on the diagonal, thus it can be inverted over $\mathbb{Z}$.  It is a classical fact that the inverse for this matrix can be found using M\"obius inversion on the face poset of $\mathsf{S}$.  
\end{remark}

We are now ready to present one of the main theorems of this article, which describes polytopal bases for deformation cones of omnitruncations of simple polytopes.

\begin{theorem}\label{thm:basis_seed}
Let $ \mathsf{S} $ be a full-dimensional simple polytope with the origin in its interior, and let $ \mathsf{P} $ be an omnitruncation of $ \mathsf{S} $ with normal fan $ \Psi $.
Then the following statements hold:
\begin{enumerate}
\item\label{it:main1} For $ \epsilon>0 $ small enough, the set
\[
	\mathscr{A}_{\epsilon}^{\circ} = \left\{ \operatorname{Tr}_{\epsilon} [\mathsf{S}, \mathsf{F}]   \,\middle|\, \emptyset \subsetneq \mathsf{F} \subsetneq \mathsf{S} \right\}
\]
of shallow proper $\epsilon$-truncations is a polytopal basis for the deformation space of $\mathsf{P}$.  
\item\label{it:main2} The collection
\[
	\mathscr{A}_{\operatorname{DTr}}= \left\{ \operatorname{DTr} [\mathsf{S}, \mathsf{F}]   \,\middle|\, \mathsf{F} \subsetneq \mathsf{S} \right\}
\]
of deep truncations is a polytopal generating set for the deformation space of $\mathsf{P}$.  

\item\label{it:main3} If $ \mathsf{S} $ is a Delzant polytope, then the collection
\[
	\mathscr{A}_{\operatorname{Tr}_1} = \left\{ \operatorname{Tr}_1[\mathsf{S}, \mathsf{F}]   \,\middle|\,  \mathsf{F} \subsetneq \mathsf{S} \right\}
\]
of unit-depth truncations is an integral polytopal generating set for the deformation space of $\mathsf{P}$. 

\item\label{it:main4} If $ \mathsf{S} $ is a primitive Delzant polytope, then the collection

\[
	\mathscr{A}_{\operatorname{DTr}}^{\circ} = \left\{ \operatorname{DTr}[\mathsf{S}, \mathsf{F}]   \,\middle|\, \emptyset \subsetneq \mathsf{F} \subsetneq \mathsf{S} \right\}
	=\left\{ \operatorname{Tr}_1[\mathsf{S}, \mathsf{F}]   \,\middle|\, \emptyset \subsetneq \mathsf{F} \subsetneq \mathsf{S} \right\}
\] of deep proper truncations is a flat, not necessarily integral, basis for the deformation space of $\mathsf{P}$.
\end{enumerate}
\end{theorem}

\begin{proof}

\begin{enumerate}[leftmargin=*]

\item	By \Cref{lem:key} the set of vectors $ \left\{ h(\operatorname{Tr}_{\epsilon} [\mathsf{S}, \mathsf{F}] ) - h(\mathsf{S}) \,\middle|\, \emptyset \subsetneq \mathsf{F} \subsetneq \mathsf{S} \right\} $ forms a basis.
	By \Cref{cor:P_in_the_interior}, the vector $ h(\mathsf{S}) $ is a linear combination of $ \left\{ h(\operatorname{Tr}_{\epsilon} [\mathsf{S}, \mathsf{F}] )  \,\middle|\, \mathsf{F} \subsetneq \mathsf{S}, \mathsf{F} \text{ facet} \right\} $.  Thus $\mathscr{A}_{\epsilon}^{\circ} $ generates the basis $ \left\{ h(\operatorname{Tr}_{\epsilon} [\mathsf{S}, \mathsf{F}] ) - h(\mathsf{S}) \,\middle|\, \emptyset \subsetneq \mathsf{F} \subsetneq \mathsf{S} \right\} $ and has the same cardinality.  Therefore, $\mathscr{A}_{\epsilon}^{\circ} $ is a basis.
	\item	This follows from \Cref{lem:convtrunc} applied to item 1. above.
	\item This follows from \Cref{rem:key_delzant}.
	\item By the proof of \Cref{prop:delzant}, the deep truncations
are flat and equal the unit-depth truncations. The basis statement
follows from \Cref{rem:key_delzant} as in item~\ref{it:main1}.
	
	\end{enumerate}
	\end{proof}
	
\begin{remark}\label{rmk:nonintegralbasis}
To see that the basis $\mathscr{A}_{\operatorname{DTr}}^{\circ}$ in \Cref{it:main4} may not be an integral basis, take the permutahedron $\mathsf{S} = \operatorname{Conv}\{\sigma(-1,0,1): \sigma \in S_3\}$, which is primitive Delzant with respect to $N=\mathbb{Z}^3\cap\{x_1+x_2+x_3=0\}$, has the origin in its interior, and is full-dimensional in $N_{\mathbb{R}}$.  We find that $\mathsf{S}$ is equal to $1/5$ multiplied by the Minkowski sum of the 6 flat facet truncations of $\mathsf{S}$.  These coefficients are not integral, thus $\mathscr{A}_{\operatorname{DTr}}^{\circ}$ is not an integral basis for the deformation space of $\mathsf{P}$.
\end{remark}

We explain how the truncations corresponding to the facets of $ \mathsf{S} $ may be replaced by any polytopal basis for
$ \operatorname{PL}(\Sigma(\mathsf{S}))$.  We note that this formulation does not
require $\mathsf S$ to be full-dimensional or to contain the origin in
its interior.
	
\begin{corollary}\label{cor:basis_augmented_barycentric}
Let $\mathsf{S}$ be a simple polytope, let $\Sigma$ be its normal fan,
and let $\mathsf{P}$ be an omnitruncation of $\mathsf{S}$ with normal fan
$\Psi=\operatorname{Bar}(\Sigma)$.  If $\mathscr{T}$ is a polytopal
basis for $\operatorname{PL}(\Sigma)$, then the set
\[
    \mathscr{C}_{\epsilon}
    \coloneqq
    \left\{
    \operatorname{Tr}_{\epsilon}[\mathsf{S},\mathsf{F}]
    \,\middle|\,
    \emptyset \subsetneq \mathsf{F} \subsetneq \mathsf{S},
    \ \operatorname{codim}(\mathsf{F})>1
    \right\}
    \cup \mathscr{T}
\]
is a polytopal basis for $\operatorname{PL}(\Psi)$.  Moreover, if $\mathsf{S}$ is smooth and $\mathscr{T}$ is an integral basis, then we can take $\epsilon=1$ and obtain an integral basis for $\operatorname{PL}(\Psi)$.
\end{corollary}

\begin{proof}

By \Cref{lem:key,lem:P_affine_basis,prop:eta}, replacing the
facet-indexed part of the difference basis by $\mathscr T$ gives the
basis
\[
\mathscr T
\cup
\left\{
h(\operatorname{Tr}_{\epsilon}[\mathsf S,\mathsf F])-h(\mathsf S)
\,\middle|\,
\operatorname{codim}(\mathsf F)>1
\right\}
\]
of $\operatorname{PL}(\Psi)$.  Since
$h(\mathsf S)\in\operatorname{Span}(\mathscr T)$, adding
$h(\mathsf S)$ to each of the remaining vectors preserves the basis
and gives $\mathscr C_{\epsilon}$.  If $\mathsf S$ is smooth, \Cref{rem:key_delzant} shows that, for
$\epsilon=1$, both changes of basis are unimodular.  Hence
$\mathscr C_1$ is integral.

\end{proof}
\subsection{Codimension-one seeds}\label{subsec:codim1}

We present an alternative way of obtaining a basis for a special choice $ \mathsf{S} $; namely, when it has codimension one with respect to the ambient space. 
Notice that our central example of the standard simplex $ \Delta_{n} \subseteq \mathbb{R}^{n} $ satisfies this property.
In this case we can add the original polytope $ \mathsf{S} $ to the polytopal basis.

Recall that $ \operatorname{Tr}_{\epsilon} (\mathsf{S}, \emptyset) = \mathsf{S}$.

\begin{proposition}\label{prop:key_non_full}
Let $ \mathsf{S} $ be a $ (n-1) $-dimensional simple polytope in $ \mathbb{R}^{n} $ such that its affine span is not linear,
and let $\mathsf{P}$ be an omnitruncation of $\mathsf{S}$.  Then the set of deep truncations $\mathscr{A}_{\operatorname{DTr}}= \left\{ \operatorname{DTr} [\mathsf{S}, \mathsf{F}]   \,\middle|\, \mathsf{F} \subsetneq \mathsf{S} \right\}$ is a polytopal basis for the deformation space of $\mathsf{P}$.
\end{proposition}

\begin{proof}
By assumption, the polytope $\mathsf{S}  $ satisfies an equation of the form $  a^{\intercal} x = b $, with $ b \neq 0 $, since the affine span does not contain the origin.
The vector $ h(\operatorname{Tr}_{\epsilon} [\mathsf{S}, \emptyset] ) = h(\mathsf{S}) $ does not belong to the linear span of the set 
$\left\{ h(\operatorname{Tr}_{\epsilon} [\mathsf{S}, \mathsf{F}] ) - h(\mathsf{S}) \,\middle|\, \emptyset \subsetneq \mathsf{F} \subsetneq \mathsf{S}\right\} $,
because the latter has value zero on the vector $ a $.
This means that
$\left\{ \operatorname{Tr}_{\epsilon} [\mathsf{S}, \mathsf{F}]   \,\middle|\, \emptyset \subseteq \mathsf{F} \subsetneq \mathsf{S}\right\}$
is a polytopal basis of the deformation space of $\mathsf{P}$.  That the deep truncations $\mathscr{A}_{\operatorname{DTr}}$ form a basis now follows from \Cref{lem:convtrunc}.
\end{proof}

Notice that we are using the basis given by \Cref{lem:key} and directly adding the seed polytope $ \mathsf{S} $, without invoking \Cref{cor:P_in_the_interior}.
This means that \Cref{prop:key_non_full} is valid for primitive Delzant polytopes (of codimension one), regardless of whether they contain a lattice point in their relative interior or not, as opposed to \Cref{thm:basis_seed}.

\begin{example}\label{ex:postnikov_basis}
	Consider the case where $ \mathsf{S} = \Delta_{n} $ which has codimension one in $ \mathbb{R}^{n} $.
	This is primitive Delzant with no relative interior lattice points.
	Furthermore, it is codimension one and its affine span is a hyperplane that does not contain the origin.
By \Cref{prop:key_non_full} a polytopal basis for the deformation space of the braid fan is given by the support vectors of the polytopes
\begin{equation*}
	\operatorname{DTr} [\Delta([n]), \Delta(I)]  = \Delta([n]\setminus I), \qquad \text{for all} \qquad I \subsetneq [n].
\end{equation*}
This is called the \textbf{simplicial basis}.
\end{example}

\begin{remark}\label{rem:typeA_shards}
There are alternate indecomposable polytopal bases for $ \operatorname{PL}( \mathcal{B}^{1}_{n})$.   Padrol--Pilaud--Ritter \cite[Proposition 8]{padrol2023shard} describe the basis of \textbf{shard polytopes}, and Hafner--M\'esz\'aros--Setiabrata--St.~Dizier describe the basis of \textbf{Schubert polytopes} \cite[Theorem~1.4]{hafner2024mconvexity}.  The latter basis was recently rediscovered by Backman--Unter; an alternate derivation of that basis will appear in forthcoming work.
\end{remark}

\begin{proposition}\label{prop:key_non_full_general}
Let $ \mathsf{S} $ be the product $ \prod_{i \in I} \mathsf{S}_{i} $ of simple polytopes, each of which has codimension one and whose affine span is not linear.  Let $\Psi$ be the barycentric subdivision of the normal fan of $ \mathsf{S} $.
The set  $\left\{ \operatorname{DTr} [\mathsf{S}, \mathsf{F}]   \,\middle|\, \emptyset \subsetneq \mathsf{F} \subsetneq \mathsf{S}\right\}   \cup \left\{\mathsf{S}_{i} \,\middle|\, i \in I\right\} $ is a polytopal basis of $  \operatorname{PL}(\Psi) $.
\end{proposition}

\begin{proof}
	This follows by a direct calculation generalizing the one given in the proof of \Cref{prop:key_non_full}.
\end{proof}

\section{Building sets}\label{sec:building}

Building sets and nested set complexes were introduced by De Concini and Procesi as the combinatorial framework for their study of wonderful compactifications of complements of complex hyperplane arrangements \cite{de1995wonderful}.  The building set encodes the flats that are blown up, and the nested set complex is the intersection complex for the boundary divisors of the resulting wonderful compactification.  Feichtner and Kozlov generalized building sets from the lattice of flats of a hyperplane arrangement to more general posets and showed that, for a building set $\mathrm{B}$ on the face poset of a fan $\Sigma$, performing stellar subdivisions along the cones of $\mathrm{B}$ in reverse order of containment produces a simplicial fan which is the fan over the associated nested set complex
$\mathcal N(\mathrm{B})$ \cite[Theorem~4.10]{feichtner2004incidence}.\footnote{In a related direction, Devadoss--Forcey--Reisdorf--Showers constructed convex polytopes from nested posets, giving an independent combinatorial route to polytopes generalizing nestohedra and graph associahedra \cite{devadoss2015convex}.}

In the work of Feichtner--Sturmfels and Postnikov, particular attention was paid to the case of the Boolean lattice \cite{feichtner2005matroid,postnikov2009permutohedra}.  This is the setting of wonderful compactifications of the complement of the coordinate hyperplane arrangement (the algebraic torus $(\mathbb{C}^*)^n$); these wonderful compactifications are smooth projective toric varieties associated to \textbf{nestohedra}, the Minkowski sums of standard
simplices corresponding to the elements of a building set; see
Feichtner--Sturmfels \cite[Theorem~3.14]{feichtner2005matroid},
Postnikov \cite[Theorem~7.4]{postnikov2009permutohedra}, or Zelevinsky
\cite[Theorem~5.1, Corollary~5.2, Theorem~6.1]{zelevinsky2006nested}.  The maximum building set corresponds to the proper part of the Boolean lattice, and the nested set complex is the order complex -- this is geometrically realized by the braid arrangement which, as we have seen, is the barycentric subdivision of the normal fan of the standard simplex, and is normal to the permutahedron.\footnote{The standard nestohedral realization for the maximum building set is the Minkowski sum of all standard simplices, which is also known as an \emph{exponential permutahedron} and is normally equivalent to the standard permutahedron.}

In this section, we investigate the deformation cone of the truncation of a simple polytope $\mathsf{S}$ along a building set $\mathrm{B}$, which we call a $\mathrm B$-truncation.\footnote{Almeter calls such an object a $\mathsf{P}$-nestohedron \cite{almeter2020generalizing}.}   We first prove that the collection of all $\mathrm B$-truncations, up to
scaling and translation, forms an open polyhedral cone inside the
deformation cone of a $\mathrm B$-truncation of $\mathsf S$
(\Cref{lem:buildingtruncationcone}). We prove that this cone has a
full-dimensional intersection with the cone of positive Minkowski sums
of individual face truncations of $\mathsf S$ according to $\mathrm B$
(\Cref{lem:convtrunc,lem:sumsofdeeptruncations}), although these cones
are typically not equal (\Cref{prop:diffcones}).   
We then describe how our bases for deformation cones of barycentric subdivisions of normal fans of polytopal simplicial fans extend to general building sets (\Cref{cor:basis_building}) via restriction.  We finish by proposing a notion of higher deformed nestohedra (\Cref{def:highernesto}) and explaining how the collection of all deformed 2-nestohedra forms a fan (\Cref{prop:fan_of_2_nestohedra}).

\subsection{Building sets and stellar subdivisions}
As a convention, we denote by $ \hat{0} $ the minimum element (if it exists) of a poset $ \mathcal{P} $.
We say that an element $ x \in \mathcal{P} $ is \textbf{irreducible} if the interval $ [ \hat{0}, x] $ is not the direct product of two other posets (each with at least two elements).

The following definition, described by Backman--Danner \cite[Proposition~2.11]{backman2024convex}, generalizes one of several equivalent definitions for a building set on the lattice of flats of a hyperplane arrangement given by De Concini--Procesi.  In the case of the Boolean lattice, the face lattice of a simplex, it is the one used by Feichtner--Sturmfels and Postnikov for investigating nestohedra.

\begin{definition}\label{def:building_set_general}
    Let $\mathcal{L}$ be a meet-semilattice with minimum element $\hat{0}$, let $I(\mathcal{L})$ be the irreducible elements of $\mathcal{L}$, and take $ \mathrm{B} \subseteq \mathcal{L}\setminus \{\hat{0}\}$.
		We say that $\mathrm{B}$ is a \textbf{building set} for $\mathcal{L}$ if the following two conditions are satisfied:
\begin{enumerate}
    \item\label{bscon1} $I(\mathcal{L})\subseteq \mathrm{B}$,
    \item\label{bscon2} if $x,y \in \mathrm{B}$ satisfy $x \wedge y \neq \hat{0}$ and $x \vee y \in \mathcal{L}$, then $x \vee y \in \mathrm{B}$.
\end{enumerate} 
\end{definition}

For any polytopal fan $ \Sigma $, its face poset is a meet-semilattice.
Furthermore, if $ \Sigma $ is simplicial, then every interval $ [\hat{0}, x] $ is isomorphic to a Boolean poset, so it is always a Cartesian product.
It follows that the only irreducible elements are the rank one elements, which correspond to the rays of the fan.

\begin{definition}\label{def:building_set_simplicial_fan}
	Let $ \Sigma $ be a fan and $ \mathcal{L} $ be the meet-semilattice formed by its faces.
	A set of cones $ \mathrm{B} \subseteq \Sigma$ is a building set for $\Sigma$ if it satisfies \Cref{def:building_set_general} for the face poset of $ \Sigma $.
	If $ \Sigma $ is a \emph{simplicial} fan, then the definition can be simplified as follows:
\begin{enumerate}
    \item\label{bscon1-fans} $ \mathrm{B} $ contains all the rays.
		\item\label{bscon2-fans} If $ \sigma_{1}, \sigma_{2} \in \mathrm{B} $ intersect non-trivially and  $\sigma_3=\operatorname{Conv}\{\sigma_{1},\sigma_{2} \}\in \Sigma $, then $ \sigma_{3} \in \mathrm{B} $.
\end{enumerate} 
\end{definition}

We note that if one is only interested in the combinatorics of the stellar subdivisions, they may omit the rays, i.e.  condition \ref{bscon1-fans}  above can be ignored.  However, for considerations of truncations and polytopal bases, facets will play a nontrivial role.

Given a building set $ \mathrm{B} $ of a fan $ \Sigma $, we order the elements of $\mathrm{B} $ with a \emph{reverse} linear extension of the face poset (see \Cref{def:orderings}).
We denote by $ \Gamma_{\mathrm B}(\Sigma) $ the iterated central stellar subdivision of $ \Sigma $ along the elements of $ \mathrm{B} $ in the order described above.
The face poset of $ \Gamma_{\mathrm B}(\Sigma) $ was described by Feichtner--Kozlov in \cite{feichtner2004incidence} as the nested set complex associated to $\mathrm{B} $.  

\begin{definition}
A \textbf{nested set}  for $\mathrm{B} $ is a collection $N \subseteq \mathrm{B} $ such that for any set $\{x_1, \ldots, x_k\} \subseteq N$ of incomparable elements with $k\geq 2$, we have $\vee_i x_i \notin \mathrm{B}$ if  $\vee_i x_i$ exists.  The \textbf{nested set complex} $\mathcal{N}(\mathrm{B})$ is the collection of all nested sets.
\end{definition}

\begin{remark}\label{rmk:maxbuildingset}
For a simple seed polytope $\mathsf S$, the non-minimal cones of $\Sigma(\mathsf S)$ form the maximum building set $\mathrm B^{\mathrm{max}}$. Its nested set complex is the order complex of the proper face poset of $\mathsf S$, equivalently the underlying simplicial complex of $\operatorname{Bar}(\Sigma(\mathsf S))$.  In this sense, we may consider results of this section about building sets for simplicial polytopal fans as generalizing the case of the barycentric subdivision of the normal fan of $\mathsf{S}$, although several results are new in that case already, e.g. those of \Cref{subsect:Btruncations}.

\end{remark}

\begin{proposition}\label{prop:local_building}
Let $ \sigma $ be a simplicial cone, let $ \Sigma $ be the non-complete fan consisting of all faces of $ \sigma $, and let $ \mathrm{B} $ be a building set. Then
\begin{equation*}
\Gamma_{\mathrm B}(\Sigma) = \bigwedge_{\sigma \in \mathrm{B}}  \mathcal{C}(\Sigma, \sigma).
\end{equation*}
In other words, $ \Gamma_{\mathrm B}(\Sigma) $ is the common refinement of each individual central stellar subdivision.
\end{proposition}

\begin{proof}
Let $ d  $ be the dimension of $ \sigma $ so that it has $ d $ rays.
We apply a linear transformation $T$ which takes the ray generators for $ \sigma $ to $\left\{ e_{1}, \dots, e_{d} \right\}$, thus mapping $ \sigma $ to $ \operatorname{Cone} \left\{ e_{1}, \dots, e_{d} \right\}$.  We view $\operatorname{Cone} \left\{ e_{1}, \dots, e_{d} \right\}$ as a pointed representative for a single maximal cone in the normal fan of the standard simplex $\Delta_{d+1}$ as described in Example \ref{ex:simplex_non_full}.  We will now reduce the statement of the proposition to a statement about the normal fan of a standard simplex, which then follows from a classical result about normal fans of nestohedra.\footnote{In this proof we are making implicit use of the fact that the stellar subdivisions of the pointed representative of a fan can be identified with the stellar subdivisions of the image in the quotient vector space, so long as the ray generators are chosen in a compatible way.}

We observe that the building set $ \mathrm{B} $, now defined on $\operatorname{Cone} \left\{ e_{1}, \dots, e_{d} \right\}$, can be extended to a building set $ \mathrm{B}^*= \mathrm{B} \cup\{\operatorname{Cone}(e_{d+1})\} $ on the fan $\mathcal{B}^{0}_{d+1}$ by adding the single element corresponding to the remaining ray generator $e_{d+1}$ (which has no material effect, as noted above).

By $\eqref{eq:sum_refinement}$, the normal fan of a Minkowski sum of a collection of polytopes  is equal to the common refinement of the normal fans of these polytopes.  Therefore, the normal fan of a nestohedron is the common refinement of the normal fans associated to the standard simplices in the building set.  

Taking the common refinement of $\mathcal{B}^{0}_{d+1}$ with the normal fan of a single simplex, $\Sigma(\Delta_S)$ for $S \subset [d]$, corresponds to a stellar subdivision of the corresponding cone $\sigma_S$ in $\mathcal{B}^{0}_{d+1}$.  This follows from the facts:

\begin{enumerate}
\item \Cref{lem:standard_simplex_deep}: $\Delta_S$ is a deep truncation of $\Delta_{d+1}$,
\item \Cref{lem:convtrunc}: the Minkowski sum of a polytope $\mathrm{P}$ and a deep truncation of $\mathrm{P}$ is homothetic to a shallow truncation of $\mathrm{P}$, and
\item \Cref{lem:shallow_truncation_stellar}: shallow truncations are dual to central stellar subdivisions.
\end{enumerate}

We can now apply the nestohedral results above for the building set $ \mathrm{B}^* $ on $\mathcal{B}^{0}_{d+1}$.  We then restrict back to $\operatorname{Cone} \left\{ e_{1}, \dots, e_{d} \right\}$ to obtain the desired result for this cone.  Note that this argument makes use of the fact that this restriction respects stellar subdivisions: a stellar subdivision of $\operatorname{Cone} \left\{ e_{1}, \dots, e_{d} \right\}$ is the restriction of a stellar subdivision of the larger fan $\mathcal{B}^{0}_{d+1}$.  Finally, we may conclude the desired statement for the original cone as the coordinates of central stellar subdivisions respect the linear transformation $T$.

\end{proof}

Now we extend \Cref{prop:local_building} to simplicial fans.

\begin{corollary}\label{cor:partial_building}
Let $ \Sigma $ be a simplicial fan and $ \mathrm{B} $ a building set on $ \Sigma $. Then

\begin{equation}\label{eq:partial_building}
\Gamma_{\mathrm B}(\Sigma) = \bigwedge_{\sigma \in \mathrm{B}}  \mathcal{C}(\Sigma, \sigma) .
\end{equation}

In particular, $ \Gamma_{\mathrm B}(\Sigma) $ is a coarsening of the barycentric subdivision of $ \Sigma $.
\end{corollary}

\begin{proof}
We know that both sides of Equation \eqref{eq:partial_building} are refinements of the original fan $ \Sigma $, so we argue locally.
It is enough to show equality restricted to each maximal cone $ \sigma $ of $ \Sigma $.  We observe that for a simplicial fan, Definition \ref{def:building_set_simplicial_fan} implies that $\mathrm{B}$ is a building set for $\Sigma$ if and only if its restriction to each maximal cone $\sigma$ is a building set.  Furthermore, central stellar subdivisions restrict to maximal cones.  Thus we may apply \Cref{prop:local_building}.
\end{proof}

\begin{figure}[ht]
\centering
\includegraphics{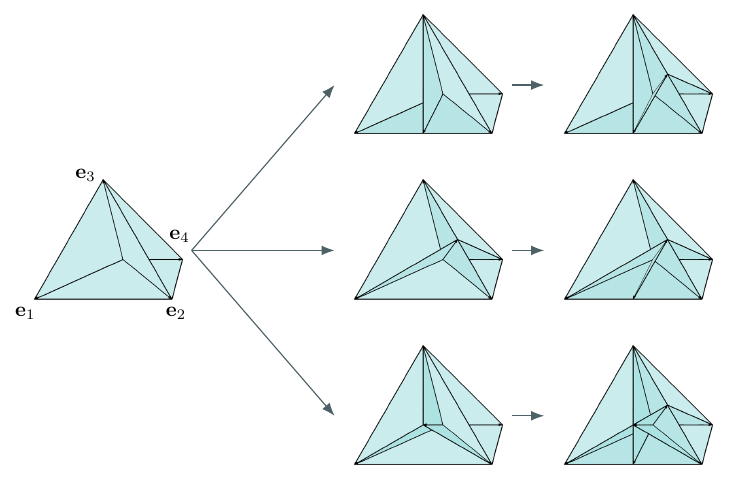}
\caption{The normal fan of a standard simplex (left).  The first two rows show two central stellar subdivisions being performed in two different orders.  Because the two faces being subdivided do not form a building set, the order of the subdivisions matters, and the resulting fans are not coarsenings of the braid arrangement.  The third row shows the result of adding the join of those two cones to the collection.  This creates a building set and the resulting fan (a coarsening of the braid arrangement) is equal to the common refinement of the three stellar subdivisions.}\label{fig:blow_fail}
\end{figure}

\begin{remark}
We are using the simpliciality of our fans in several places in this section, e.g. for non-simplicial fans the stellar subdivisions at individual cones may not be compatible with their barycentric subdivisions; see \Cref{fig:nonsimplicial-square-cone-subdivisions}.
\end{remark}

\begin{figure}[ht]
    \centering
    \includegraphics[width=.35\textwidth]{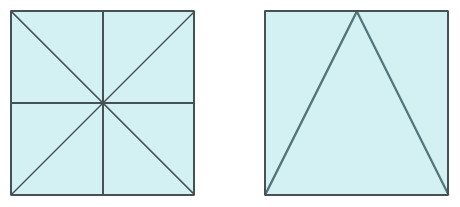}
    \caption{The barycentric subdivision of a square cone viewed from above (left), and a central stellar subdivision of a facet of that cone (right).  Observe that the fan on the right is not a coarsening of the fan on the left.}
    \label{fig:nonsimplicial-square-cone-subdivisions}
\end{figure}

\subsection{The cone of $\mathrm{B}$-truncations}\label{subsect:Btruncations}
We now investigate the polytopal counterpart to building sets on simplicial fans.

\begin{definition}
Let $ \mathsf{S} $ be a polytope.
A set $ \mathrm{B} $ of faces of $ \mathsf{S} $ is a building set for $ \mathsf{S}$  if the set $ \left\{ \operatorname{ncone}(\mathsf{F},\mathsf{S})\,\middle|\, \mathsf{F} \in \mathrm{B}  \right\} $ is a building set for the fan $ \Sigma( \mathsf{S} ) $ according to \Cref{def:building_set_simplicial_fan}.
\end{definition}

Whenever $ \mathrm{B} $ is a building set for a fixed simple polytope $ \mathsf{S} $, we abbreviate $ \Gamma_{\mathrm B}(\Sigma(\mathsf S)) $ to $ \Gamma_{\mathrm B} $.

\begin{proposition}\label{thm:building}
Let $ \mathsf{S} $ be a simple polytope, let $ \mathrm{B} $ be a building set for $ \mathsf{S} $, and let $\mathsf{P}$ be a positive Minkowski sum of shallow truncations $ \operatorname{Tr} _{\epsilon}\left[\mathsf{S},\mathsf{F} \right]$ ranging over each $ \mathsf{F} \in \mathrm{B} $.  Then $\Sigma(\mathsf{P})=\Gamma_{\mathrm B}$.
\end{proposition}

\begin{proof}
	This follows from \Cref{cor:partial_building}, since the normal fan of a Minkowski sum is the common refinement of the normal fans (see \Cref{eq:sum_refinement}).
\end{proof}

The following generalizes omnitruncations to the setting of general building sets for simple polytopes.
\begin{definition}[\textbf{Building-set truncation}]
\label{def:building_set_truncation}
Let $\mathsf{S}$ be a $d$-polytope, let $\mathrm{B}$ be a building
set for $\mathsf{S}$, and let
$\boldsymbol{\epsilon}
=
\{\epsilon_{\mathsf{F}}>0\mid\mathsf{F}\in\mathrm{B}\}$.
Take some linear extension
$\mathsf{F}_1<\ldots<\mathsf{F}_k$ of $\mathrm{B}$, and let
$\mathsf{S}_0\coloneqq\mathsf{S}$.
We say that $\boldsymbol{\epsilon}$ is \textbf{$\mathrm{B}$-shallow} for
$\mathsf{S}$ if, for each $1\leq i\leq k$, we have
$\mathsf{S}_i\coloneqq
\operatorname{Tr}_{\epsilon_{\mathsf{F}_i}}
[\mathsf{S}_{i-1},\mathsf{F}_i\cap\mathsf{S}_{i-1}]$
is a shallow truncation of $\mathsf{S}_{i-1}$.
If $\boldsymbol{\epsilon}$ is $\mathrm{B}$-shallow for
$\mathsf{S}$, the associated
$(\mathrm{B},\boldsymbol{\epsilon})$-truncation of $\mathsf{S}$ is
$\operatorname{Tr}_{\mathrm{B},\boldsymbol{\epsilon}}(\mathsf{S})
\coloneqq\mathsf{S}_k$.

Suppose
$\mathsf{S}=\{x \in\operatorname{aff}(\mathsf S)\mid u_{\mathsf{G}}^{\intercal}x\geq b_{\mathsf{G}}
\text{ for every facet }\mathsf{G}\subsetneq\mathsf{S}\}$.
For every $\mathsf{F}\in\mathrm{B}$, let $u_{\mathsf{F}}$ be the
canonical inner face normal of $\mathsf{F}$, and set
$b_{\mathsf{F}}\coloneqq h(\mathsf{S},u_{\mathsf{F}})$.
Then the associated
$(\mathrm{B},\boldsymbol{\epsilon})$-truncation is
\[
\operatorname{Tr}_{\mathrm{B},\boldsymbol{\epsilon}}(\mathsf{S})
=
\left\{
x \in\operatorname{aff}(\mathsf S) \,\middle|\,
u_{\mathsf{F}}^{\intercal}x
\geq
b_{\mathsf{F}}+\epsilon_{\mathsf{F}}
\text{ for every }\mathsf{F}\in\mathrm{B}
\right\}.
\]

If a polytope $\mathsf{P} = \operatorname{Tr}_{\mathrm{B},\boldsymbol{\epsilon}}(\mathsf{S})$ for a suitable choice of shallow $\boldsymbol{\epsilon}$, we will simply refer to $\mathsf{P}$ as a \textbf{$\mathrm{B}$-truncation}.
\end{definition}

\begin{definition}
\label{def:coneofbuildingtruncations}
Let $\mathcal{T}_{\mathrm{B}}(\mathsf{S})
\coloneqq
\{\lambda\mathsf{P}+v:
\lambda>0,\,
v\in\mathbb{R}^n,\,
\mathsf{P}\text{ is a $\mathrm{B}$-truncation of }\mathsf{S}\}$ be the cone of $\mathrm{B}$-truncations.  
\end{definition}

This terminology is justified by the following proposition.

\begin{proposition}
\label{lem:buildingtruncationcone}
Let $\mathsf{S}$ be a simple polytope, and let $\mathrm{B}$ be a building
set for $\mathsf{S}$.
Then $\mathcal{T}_{\mathrm{B}}(\mathsf{S})$ is a full-dimensional open polyhedral cone contained
in $\operatorname{Def}(\Gamma_{\mathrm B})$.
\end{proposition}

\begin{proof}
It is clear that $\mathcal{T}_{\mathrm{B}}(\mathsf{S})$ is an open subset of $\operatorname{Def}(\Gamma_{\mathrm B})$ as we can perturb the defining parameter vector $\boldsymbol{\epsilon}$ at a point and stay in $\mathcal{T}_{\mathrm{B}}(\mathsf{S})$.  

We next prove that $\mathcal{T}_{\mathrm{B}}(\mathsf{S})$ is a cone.  It is also clear that the statement respects translations, thus we will restrict our attention to the untranslated version.  Now $\mathcal{T}_{\mathrm{B}}(\mathsf{S})$ is closed under scaling, so it suffices to prove that the collection of $\mathrm{B}$-truncations is convex.  Let $0<\lambda <1$ and let $\boldsymbol{\epsilon}^{1}, \boldsymbol{\epsilon}^{2}$ be $\mathrm{B}$-shallow for $\mathsf{S}$.  Let $\mathsf{S}^1_i$ and $\mathsf{S}^2_i$ be the truncations of $\mathsf{S}$ along $\mathsf{F}_1<\ldots<\mathsf{F}_{i}$ using $\boldsymbol{\epsilon}^{1}$ and $ \boldsymbol{\epsilon}^{2}$, respectively.  We know that $\Sigma(\mathsf{S}^1_i) = \Sigma(\mathsf{S}^2_i)$, and $\Sigma(\mathsf{S}^1_{i+1}) = \Sigma(\mathsf{S}^2_{i+1})$.  Thus 
\[
\Sigma(\lambda\mathsf{S}^1_i+ (1-\lambda)\mathsf{S}^2_i)=\Sigma(\mathsf{S}^1_i) = \Sigma(\mathsf{S}^2_i)
\]
and 
\[
\Sigma(\lambda \mathsf{S}^1_{i+1}+ (1-\lambda)\mathsf{S}^2_{i+1})=\Sigma(\mathsf{S}^1_{i+1}) = \Sigma(\mathsf{S}^2_{i+1}).
\]
Moreover, for each facet $\mathsf F$ of $\mathsf{S}^1_i$ (equivalently $\mathsf{S}^2_i$) such that $u_{\mathsf F}\neq u_{\mathsf F_{i+1}}$,
\[
h(\lambda\mathsf{S}^1_i+ (1-\lambda)\mathsf{S}^2_i, u_{\mathsf F})=\lambda h(\mathsf{S}^1_i,u_{\mathsf F})+(1-\lambda)h(\mathsf{S}^2_i,u_{\mathsf F})
\]
\[
=\lambda h(\mathsf{S}^1_{i+1}, u_{\mathsf F})+(1-\lambda)h(\mathsf{S}^2_{i+1},u_{\mathsf F})=h(\lambda\mathsf{S}^1_{i+1}+ (1-\lambda)\mathsf{S}^2_{i+1}, u_{\mathsf F}).
\]
In the direction $u_{\mathsf F_{i+1}}$,
\[
h(\lambda\mathsf{S}^1_i+ (1-\lambda)\mathsf{S}^2_i, u_{\mathsf F_{i+1}})=\lambda h(\mathsf{S}^1_i,u_{\mathsf F_{i+1}})+(1-\lambda)h(\mathsf{S}^2_i,u_{\mathsf F_{i+1}})
\]
\[
<\lambda h(\mathsf{S}^1_{i+1}, u_{\mathsf F_{i+1}})+(1-\lambda)h(\mathsf{S}^2_{i+1},u_{\mathsf F_{i+1}})=h(\lambda\mathsf{S}^1_{i+1}+ (1-\lambda)\mathsf{S}^2_{i+1}, u_{\mathsf F_{i+1}}).
\]

 This shows that  $\lambda \mathsf{S}^1_{i+1}+ (1-\lambda)\mathsf{S}^2_{i+1}$ is obtained from $\lambda\mathsf{S}^1_i+ (1-\lambda)\mathsf{S}^2_i$ by a shallow truncation along $\mathsf{F}_{i+1}$.  By induction on $i$, $\lambda \mathsf{S}^1_i+ (1-\lambda)\mathsf{S}^2_i$ is obtained by a sequence of shallow truncations of the faces $\mathsf{F}_1<\ldots<\mathsf{F}_i$.  Therefore $\lambda \mathsf{S}^1_{i+1}+ (1-\lambda)\mathsf{S}^2_{i+1}$ is obtained by a sequence of shallow truncations of the faces $\mathsf{F}_1<\ldots<\mathsf{F}_{i+1}$, and it follows that the set in question forms a cone.

 It remains to show that this open cone is determined by a finite number of distinct inequalities.  We again proceed by the number of steps in the truncation.  At each step, the truncation parameter for $\mathsf{F}_k$ satisfies $0<\epsilon_{\mathsf F_k}<u_{\mathsf F_k}^{\intercal}v-h(\mathsf S_{k-1},u_{\mathsf F_k})$ for every $v$ which is a vertex neighboring $\mathsf{F}_k$.  Each such vertex $v$ is contained in $d$ facets, which give a finite system of linear equations describing the coordinates of $v$ in terms of the earlier entries of $\boldsymbol{\epsilon}$.  This in turn gives a linear description of $\epsilon_{\mathsf F_k}$ in terms of the earlier entries of $\boldsymbol{\epsilon}$.
\end{proof}

Let $X$ be a finite set of points from a cone.  We let $\operatorname{Cone}^+(X)$ denote the set of all strictly positive combinations of elements of $X$.  Let  $\mathcal{T}_{\mathscr{A}_{\mathrm{B}}}(\mathsf{S}) \coloneqq \operatorname{Cone}^+( \left\{ \operatorname{DTr}[\mathsf{S},\mathsf{F}] \,\middle|\,\mathsf{F}\in \mathrm{B} \right\}  \cup \{\mathsf{S}\})+\mathsf L$, where $\mathsf L$ is the lineality space of $\operatorname{Def}(\Gamma_{\mathrm B})$, corresponding to translations.  Also, observe that $\mathcal{T}_{\mathscr{A}_{\mathrm{B}}}(\mathsf{S})$ is the interior of the cone generated by the deep truncations together with the seed $\mathsf{S}$, or equivalently, by \Cref{lem:convtrunc}, the interior of the cone generated by the shallow truncations of $\mathsf{S}$.

\begin{proposition}\label{lem:sumsofdeeptruncations}
  The open polyhedral cones $\mathcal{T}_{\mathscr{A}_{\mathrm{B}}}(\mathsf{S})$ and $ \mathcal{T}_{\mathrm{B}}(\mathsf{S})$ have full-dimensional intersection.  
  \end{proposition}

\begin{proof}

We first prove that there exists a point in their intersection.  Let $\boldsymbol{\epsilon}$ be such that $\epsilon_{\mathsf F} \gg \epsilon_{\mathsf G}$ whenever $\mathsf{F} \subsetneq \mathsf{G}$.   We will refer to such an $\boldsymbol{\epsilon}$ as \emph{rapidly decreasing}.  It is straightforward to verify that any sufficiently small, rapidly decreasing $\boldsymbol{\epsilon}$ is \textbf{$\mathrm{B}$-shallow} for $\mathsf{S}$.  

Let $\boldsymbol{\lambda}$ be rapidly decreasing, and take the following Minkowski sum
\[
\mathsf{P} = \lambda_{\emptyset}\mathsf{S}+\sum_{\mathsf{F} \in \mathrm{B}} \lambda_{\mathsf{F}}\operatorname{DTr}[\mathsf{S},\mathsf{F}],
\]

with $\lambda_{\emptyset}$ some large real number.  We claim that $\mathsf{P}$ is homothetic to a $\mathrm{B}$-truncation.  First observe that by \Cref{cor:partial_building} and \Cref{lem:convtrunc}, the normal fan of $\mathsf P$ is
\[
\Sigma(\mathsf S)
\wedge
\bigwedge_{\mathsf G\in\mathrm B}
\Sigma(\operatorname{DTr}[\mathsf S,\mathsf G])
=
\Gamma_{\mathrm B},
\]
hence $\mathsf{P}$ has the correct normal fan and it suffices to check that its support function is that of a scalar multiple of a $\mathrm{B}$-truncation. 

Let $b_{\mathsf F}\coloneqq h(\mathsf S,u_{\mathsf F}).$
Define $ a_{\mathsf F,\mathsf G}
\coloneqq
h(\operatorname{DTr}[\mathsf S,\mathsf G],u_{\mathsf F})-b_{\mathsf F}$
so that $a_{\mathsf F,\mathsf G}\geq 0$ and is strictly positive
if and only if 
$\mathsf F\subseteq\mathsf G$.  Moreover, $a_{\mathsf F,\mathsf F}=\delta_{\mathsf F}$.
Set
\[
c
=
\lambda_{\emptyset}
+
\sum_{\mathsf G\in\mathrm B}\lambda_{\mathsf G}.
\]Using additivity and homogeneity of support functions, we obtain, for every $\mathsf F\in\mathrm B$,
\[
h(\mathsf P,u_{\mathsf F})
=
\lambda_{\emptyset}b_{\mathsf F}
+
\sum_{\mathsf G\in\mathrm B}
\lambda_{\mathsf G}
\bigl(b_{\mathsf F}+a_{\mathsf F,\mathsf G}\bigr).
\]Therefore
\[
h(\mathsf P,u_{\mathsf F})
=
c\,b_{\mathsf F}
+
\sum_{\mathsf G\supseteq\mathsf F}
\lambda_{\mathsf G}a_{\mathsf F,\mathsf G}.
\]Define
\[
\epsilon_{\mathsf F}
\coloneqq
\frac{1}{c}
\sum_{\mathsf G\supseteq\mathsf F}
\lambda_{\mathsf G}a_{\mathsf F,\mathsf G}
=\frac{1}{c}
\left(
\lambda_{\mathsf F}\delta_{\mathsf F}
+
\sum_{\mathsf G\supsetneq\mathsf F}
\lambda_{\mathsf G}a_{\mathsf F,\mathsf G}
\right).
\]
Then
\[
h(\mathsf P,u_{\mathsf F})
=
c\bigl(b_{\mathsf F}+\epsilon_{\mathsf F}\bigr).
\]
 Because $\boldsymbol{\lambda}$ is rapidly decreasing, the term $\lambda_{\mathsf F}\delta_{\mathsf F}$ dominates in $\epsilon_{\mathsf{F}}$, and large $\lambda_{\emptyset}$ makes the entries of $\boldsymbol{\epsilon}$ small.  Hence we see that $\boldsymbol{\epsilon}$ is small and rapidly decreasing.  We can conclude that $\mathsf{P}$ is indeed homothetic to a $\mathrm{B}$-truncation: 
\[
\mathsf P
=
c\,
\operatorname{Tr}_{\mathrm B,\boldsymbol{\epsilon}}
(\mathsf S).
\]
 Thus $\mathcal{T}_{\mathscr{A}_{\mathrm{B}}}(\mathsf{S})$ and $ \mathcal{T}_{\mathrm{B}}(\mathsf{S})$ have nonempty intersection.  Because both cones are open, this intersection is full-dimensional.

\end{proof}

\begin{proposition}\label{prop:diffcones}
In general, $\mathcal{T}_{\mathscr{A}_{\mathrm{B}}}(\mathsf{S})\neq \mathcal{T}_{\mathrm{B}}(\mathsf{S}) \neq \operatorname{Interior}(\operatorname{Def}(\Gamma_{\mathrm B}))$.  
\end{proposition}

\begin{proof}
We demonstrate an example which illustrates this phenomenon.  We emphasize that the lack of equality of the cones in question should be viewed as typical behavior.  We will focus on $\Delta_n$ with $n \geq 5$ with maximum building set $\mathrm{B} = \mathrm{B}^{\text{max}}$.  The standard permutahedron does not belong to $\mathcal{T}_{\mathscr{A}_{\mathrm{B}}}(\Delta_n)$.  However, it does live on the boundary of $\mathcal{T}_{\mathscr{A}_{\mathrm{B}}}(\Delta_n)$; it is a Minkowski sum of line segments, and this is arbitrarily close to the corresponding sum plus a very small copy of every other simplex (which is required by our definition of $\mathcal{T}_{\mathscr{A}_{\mathrm{B}}}(\Delta_n)$).   However, for $n\geq 5$, the standard permutahedron does not live in the closure of $\mathcal{T}_{\mathrm{B}}(\Delta_n)$; see \Cref{ex:standard_perm_not_omni}.  On the other hand, the standard permutahedron certainly lives in $\operatorname{Interior}(\operatorname{Def}(\Gamma_{\mathrm B}))$.
\end{proof}

\begin{remark}\label{rmk:sumofridgetruncs}
In this setting, there is a nice generalization of the fact that the Minkowski sum of line segments gives the standard permutahedron.  If one takes a seed $\mathsf{S}$ and Minkowski sums all of the ridge (codimension-2) truncations of $\mathsf{S}$, this gives a polytope which is normally equivalent to $\operatorname{Omni}(\mathsf{S})$.  This follows by inspection of the normal fans restricted to the chambers of $\Sigma(\mathsf{S})$, where the classical result for the braid arrangement applies.
\end{remark}

\begin{example}
Take $\mathrm{B}$ to be the maximum building set on the face poset of $\Delta_4$.  After quotienting by translation, we find all three cones above are 11-dimensional and have the following numbers of facets and rays:
\[
\begin{array}{c|c|c}
\text{cone} & \text{facets} & \text{extreme rays}\\
\hline
\overline{\mathcal T_{\operatorname{Tr}}(\Delta_4)}
& 11 & 11\\
\overline{\mathcal T_{\operatorname{Omni}}(\Delta_4)}
& 24 & 120\\
\operatorname{Def}(\mathcal B_4^1)
& 24 & 37
\end{array}
\]
\end{example}

\subsection{Bases for deformation cones of truncated polytopes}

We now present a building set extension of \Cref{thm:basis_seed}.  

\begin{theorem}\label{cor:basis_building}
Let $\mathsf{S}$ be a full-dimensional simple polytope with the
origin in the interior, and let $\mathrm{B}$ be a building set of faces
of $\mathsf{S}$.  Let $\mathsf{P}_{\mathrm{B}}$ be a $\mathrm{B}$-truncation of $\mathsf{S}$.  Then the
following statements hold:
\begin{enumerate}
\item\label{it:building-main1}
For $\epsilon>0$ small enough, the set
\[
    \mathscr{A}_{\epsilon,\mathrm{B}}^{\circ}
    =
    \left\{
    \operatorname{Tr}_{\epsilon}[\mathsf{S},\mathsf{F}]
    \,\middle|\,
    \mathsf{F}\in \mathrm{B}
    \right\}
\]
of shallow $\epsilon$-truncations along the faces in $\mathrm{B}$ is a
polytopal basis for the deformation space of $\mathsf{P}_{\mathrm{B}}$.

\item\label{it:building-main2}
The collection
\[
    \mathscr{A}_{\operatorname{DTr},\mathrm{B}}
    =
    \left\{
    \operatorname{DTr}[\mathsf{S},\mathsf{F}]
    \,\middle|\,
    \mathsf{F}\in \mathrm{B}
    \right\}
    \cup \{\mathsf{S}\}
\]
of deep truncations together with the seed polytope is a polytopal
generating set for the deformation space of $\mathsf{P}_{\mathrm{B}}$.

\item\label{it:building-main3}
If $\mathsf{S}$ is a Delzant polytope, then
\[
    \mathscr{A}_{\operatorname{Tr}_1,\mathrm{B}}
    =
    \left\{
    \operatorname{Tr}_1[\mathsf{S},\mathsf{F}]
    \,\middle|\,
    \mathsf{F}\in \mathrm{B}
    \right\}
    \cup \{\mathsf{S}\}
\]
is an integral polytopal generating set for the deformation space of
$\mathsf{P}_{\mathrm{B}}$.

\item\label{it:building-main4}
If $\mathsf{S}$ is a primitive Delzant polytope, then
\[
    \mathscr{A}_{\operatorname{DTr},\mathrm{B}}^{\circ}
    =
    \left\{
    \operatorname{DTr}[\mathsf{S},\mathsf{F}]
    \,\middle|\,
    \mathsf{F}\in \mathrm{B}
    \right\}
    =
    \left\{
    \operatorname{Tr}_1[\mathsf{S},\mathsf{F}]
    \,\middle|\,
    \mathsf{F}\in \mathrm{B}
    \right\}
\]
is a flat, not necessarily integral, basis.
\end{enumerate}
\end{theorem}

\begin{proof}
\leavevmode
\begin{enumerate}
\item Every stellar subdivision of a simplicial fan is simplicial.
Since $ \Gamma_{\mathrm B} $ is simplicial, the vector space $  \operatorname{PL}(\Gamma_{\mathrm B})  $ has dimension equal to the number of its rays, which is equal to the cardinality of $ \mathrm{B} $.
Let $ \Psi = \operatorname{Bar}(\Sigma(\mathsf{S})) $.  By \Cref{cor:partial_building}, we have that $ \Psi $ refines $ \Gamma_{\mathrm B} $, so we have an inclusion of vector spaces $ \operatorname{PL}(\Gamma_{\mathrm B}) \subseteq \operatorname{PL}(\Psi)  $.
By \Cref{thm:basis_seed}, we have that the vectors in $ \left\{ h( \mathsf{T} ) \,\middle|\, \mathsf{T} \in \mathscr{A}_{\epsilon,\mathrm{B}}^{\circ} \right\} $ are linearly independent in $ \operatorname{PL}(\Psi)  $, so they are also independent in $  \operatorname{PL}(\Gamma_{\mathrm B})$.
Finally, the set $ \mathscr{A}_{\epsilon,\mathrm{B}}^{\circ}$ has the same cardinality as $ \mathrm{B} $, and since the support vectors are linearly independent by \Cref{thm:basis_seed}, they form a basis.
\item	This follows from \Cref{lem:convtrunc} applied to item 1. above.  
	\item This follows from an argument extending \Cref{rem:key_delzant}; one can restrict the matrix from that remark to the columns and rows indexed by $\mathrm{B}$ and obtain another upper triangular integral matrix with 1's on the diagonal.  This is again invertible over the integers. 
	\item  By the proof of \Cref{prop:delzant}, the deep truncations are flat and equal the unit-depth truncations. The basis statement follows from \Cref{rem:key_delzant}, after restricting the matrix there to the rows and columns indexed by $\mathrm B$, as in item~\ref{it:building-main1}.
		\end{enumerate}
\end{proof}

Similar to \Cref{cor:basis_augmented_barycentric}, we can replace the facet truncations in the generating sets of \Cref{cor:basis_building} items 1, 2, and 4, with any polytopal basis $\mathscr{T}$ for $\operatorname{PL}(\Sigma)$.

In the special case of $\mathsf{S}=\Delta_{n}$, the standard simplex, the flat-truncation statement in \Cref{cor:basis_building} specializes to an observation of Padrol--Pilaud--Poullot \cite[Corollary~4.14]{padrol2022hypergraphic}.

\subsection{Higher nestohedra and the 2-nestohedral fan}

In this section we suggest a theory of higher nestohedra, which is motivated by various works, some of which are discussed in Section \ref{applicationssection}.  We observe that all generalized 2-nestohedra form a fan, which strongly suggests that they should be considered in aggregate.  

Nestohedra are, up to normal equivalence, the set of polytopes which can be obtained from the standard simplex by truncating faces along a building set.  In light of this fact, we propose the following definition.

\begin{definition}\label{def:highernesto}
The $0$-nestohedron is the standard simplex.  A polytope $\mathsf{P}$ is a \textbf{$k$-nestohedron} with $k\geq 1$ if there exists a $k-1$-nestohedron $\mathsf{Q}$ such that $\mathsf{P}$ is a $\mathrm{B}$-truncation of $\mathsf{Q}$ for some building set $\mathrm{B}$ on $\mathsf{Q}$.  We define a \textbf{generalized $k$-nestohedron} to be a polytope obtained as a deformation of a $k$-nestohedron.
\end{definition}

Taking the maximum building set at each stage yields the $k$-permutahedra.

\begin{figure}[ht]
    \centering
    \includegraphics[width=.38\textwidth]{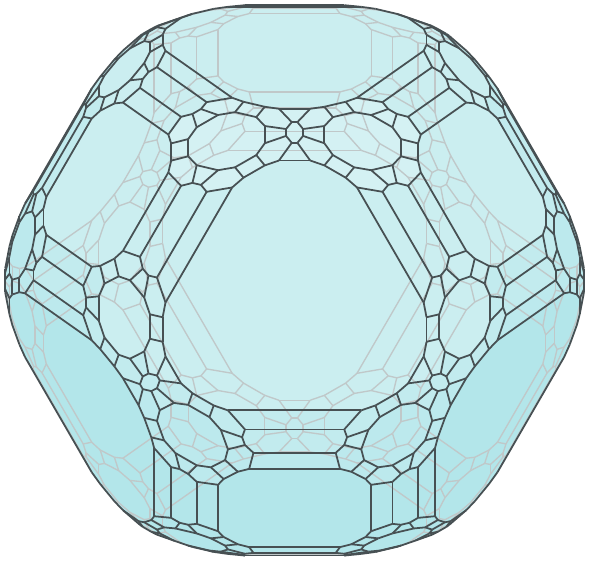}
    \caption{A $3$-permutahedron.}
    \label{fig:three-permutahedron}
\end{figure}

The notion of taking iterated building sets is also discussed in the work of Petr\'ic \cite{petric2014stretching}.  We have nothing deep to say at this juncture about this general definition of a $k$-nestohedron.  We instead focus on the case of generalized 2-nestohedra.

The 2-braid fan $\mathcal{B}^{2}_{n}$ has a distinguished pointed representative
${}^{\wedge}\mathcal{B}^{2}_{n}
\coloneqq
\operatorname{Bar}\left({}^{\wedge}\mathcal{B}^{1}_{n}\right)$.
We can take primitive ray generators for the rays of ${}^{\wedge}\mathcal{B}^{2}_{n}$ as follows: let 
\[
\mathcal{F}=\emptyset \subsetneq S_1\subsetneq \ldots \subsetneq S_k \subsetneq [n]
\]
be a flag of subsets of $[n]$. The corresponding primitive ray generator is
$u_{\mathcal{F}}=\sum_{j=1}^{k} \,\mathbf e_{S_j},$
and the corresponding ray in $\mathcal{B}^{2}_{n}$ is 
\[
\rho_{\mathcal{F}}
=
\operatorname{Cone}\{u_{\mathcal{F}}\}
+
\operatorname{Span}\{\mathbf e_{[n]}\}
\subseteq \mathbb R^n.
\]

We can equivalently characterize the vectors $u_{\mathcal{F}}$ as those nonzero integer vectors $v$ such that if we weakly order the entries of $v$, then $v_1 = 0$, and $v_{i+1}-v_{i} \in \{0,1\}$ for all $1\leq i \leq n-1$, i.e. it has \emph{no jumps}.  This is the characterization we will use below.

We also want to recall the definition of a tree preposet cone.  Let $\mathcal P=(\mathfrak T,\preceq)$ be a poset whose set of elements $\mathfrak T=\{\mathscr T_0,\mathscr T_1,\ldots,\mathscr T_m\}$ forms a partition of $[n]$.  We say that $\mathcal P$ is a \textbf{tree preposet} if the Hasse diagram of $\mathcal P$ is a tree.\footnote{The use of the term \emph{preposet} here, rather than poset, is due to the fact that the elements of our poset are elements of a partition of $[n]$ rather than elements of $[n]$ itself.}   The associated \textbf{tree-preposet cone} is 
\[
C_{\mathcal P} = \{x \in \mathbb{R}^n: x_i \leq x_j \text{\,\,whenever there exists\,\,} \mathscr T_a, \mathscr T_b \in \mathfrak{T} \text{\,\,such that\,\,} i \in \mathscr T_a, j \in \mathscr T_b \text{,\,\,and\,\,} \mathscr T_a \preceq \mathscr T_b\}.
\]

By the cone--preposet dictionary of Postnikov--Reiner--Williams
\cite[Proposition~3.5 and Corollary~3.6]{postnikov2008faces}, the tree-preposet cones are precisely those simplicial
cones which arise as unions of cones in the braid arrangement.

\begin{lemma}\label{forestposetbary}
Let $\mathfrak T=\{\mathscr T_0,\mathscr T_1,\ldots,\mathscr T_m\}$ be a nontrivial partition of $[n]$ and let $\mathcal{P} = (\mathfrak T, \preceq)$ be a tree preposet.  Let $C_{\mathcal{P}}$ be the corresponding tree-preposet cone.  The barycenter of $C_{\mathcal{P}}$ is a ray generator of the 2-braid fan.
\end{lemma}

\begin{proof}
Let $\mathsf T$ be the directed tree which is the Hasse diagram for $\mathcal{P}$. We first recall that $C_{\mathcal{P}}$ has a standard pointed representative ${}^{\wedge}C_{\mathcal{P}}$ with the following ray generators associated to each edge $e = \mathscr T_a \lessdot \mathscr T_b$ of the Hasse diagram:

\[
u_e =  \sum_{\mathscr T_c} \mathbf e_{\mathscr T_c}
\]
where the sum ranges over all  $\mathscr T_c$ which lie in the same connected component of $\mathsf T \setminus e$ as $ \mathscr T_b$.  It suffices to prove that the primitive ray generator $u_{\mathcal{P}}$ for the barycenter of ${}^{\wedge}C_{\mathcal{P}}$ is a primitive ray generator for ${}^{\wedge}\mathcal B^2_n$.  We let

\[
\widetilde{u}_{\mathcal{P}} = \sum_{e}u_e.
\]

Take $i \in \mathscr T_c$ for some $\mathscr T_c \in \mathfrak T$.  The $i$th coordinate of $\widetilde{u}_{\mathcal{P}}$ is equal to the number of edges $\mathscr T_a \lessdot \mathscr T_b$ for which $\mathscr T_c$ lies in the component of $\mathsf T\setminus e$ containing $\mathscr T_b$.  Let $u_{\mathcal{P}}= \widetilde{u}_{\mathcal{P}} - t\cdot \mathbf e_{[n]}$ 
where $t$ is the minimum value of $\widetilde{u}_{\mathcal{P}}$.  We now show that $u_{\mathcal{P}}$ satisfies the characterization of primitive ray generators for ${}^{\wedge}\mathcal B^2_n$ given above.   By construction, $u_{\mathcal{P}}$ is nonnegative and has some entry equal to zero.   We will now prove that ${u_{\mathcal{P}}}$ has no jumps.  Let $i, j \in [n]$ and suppose, without loss of generality, that ${u_{\mathcal{P}}}_i < {u_{\mathcal{P}}}_j$.  Suppose $i \in \mathscr T_c$ and $j \in \mathscr T_d$.  Walk along the path from $\mathscr T_c$ to $\mathscr T_d$.  At each step, the value $u_{\mathcal{P}}$ associated to a vertex either increases or decreases by one.  Therefore, it takes on all intermediate values and has no jumps.

\end{proof}

\begin{lemma}\label{facetof2nesto}
Let $\mathsf{P} \in \mathbb{R}^n$ be a $2$-nestohedron.  Each ray of $\Sigma(\mathsf{P})$ is a ray of the 2-braid fan $\mathcal{B}^{2}_{n}$.  
\end{lemma}

\begin{proof}
Let $\mathsf{Q}$ be a nestohedron such that $\Sigma(\mathsf{P})$ is obtained from $\Sigma(\mathsf{Q})$ by a sequence of central stellar subdivisions along a building set $\mathrm{B}$.  Then each ray $\rho$ of $\Sigma(\mathsf{P})$ is generated by the barycenter of a face $\sigma \in \Sigma(\mathsf{Q})$.  Because $\mathsf{Q}$ is a smooth generalized permutahedron, each cone in $\Sigma(\mathsf{Q})$ is a tree-preposet cone and the desired statement holds by \Cref{forestposetbary}.\end{proof}

Despite the fact that the rays of $\Sigma(\mathsf{P})$ are rays of $\mathcal{B}^{2}_{n}$, $\mathsf{P}$ may not live in the deformation cone of the 2-permutahedron because $\Sigma(\mathsf{P})$ may not be a coarsening of $\mathcal{B}^{2}_{n}$.  On the other hand, McMullen showed that the collection of all support functions of polytopes on a fixed set of rays defines a cone, which he calls the closed inner region, and this cone is subdivided by deformation cones of polytopes whose facet normals are a subset of the given ones \cite{mcmullen1973representations}; see also \cite[Lemma~2.16 and Theorem~2.18]{mcmullen2026monotypic}.  In light of Lemma \ref{facetof2nesto}, these results imply the following fact.

\begin{proposition}\label{prop:fan_of_2_nestohedra}
The collection of all deformation cones of 2-nestohedra, together with their faces, forms a fan.  We call this fan the \textbf{2-nestohedral fan}.
\end{proposition}

This fan is quite attractive as it places so many interesting polytopes all in one place: all generalized permutahedra, all deformations of 2-permutahedra (e.g. the 2-permutahedra, permutonestohedra, and the simple permutonestohedra) as well as the cosmohedra, and presumably many more fascinating polytopes yet to be discovered!  See \Cref{applicationssection} for more on these examples.

\begin{remark}
    The above construction can be extended and generalized in the following ways.  
    \begin{enumerate}
    \item The only fact which we used about $\mathsf{Q}$ in \Cref{facetof2nesto} was that it was a smooth generalized permutahedron.  Thus one could enlarge the 2-nestohedral fan to be the fan of all deformations of polytopes obtained by truncations of smooth generalized permutahedra along building sets. 
    \item Lemma \ref{facetof2nesto} can be upgraded to all Delzant polytopes: any facet normal of a 2-pass truncation of a Delzant polytope $\mathsf{S}$ is a facet normal of a 2-omnitruncation of $\mathsf{S}$.\footnote{We take a  2-omnitruncation of $\mathsf{S}$ to mean an omnitruncation of an omnitruncation of $\mathsf{S}$.}  This is because each cone of the normal fan of $\mathsf{S}$ is unimodular and, by a unimodular transformation, can be treated as a cone in the normal fan of the standard simplex as in the proof of \Cref{prop:local_building}.  Thus the collection of all deformations of  2-pass truncations of $\mathsf{S}$ forms a fan.   This can be further extended to simple nonsmooth polytopes if one is willing to work with noncentral truncation directions.
    \end{enumerate} 
    
    One can combine the two observations above to produce even larger fans.
    \end{remark}

\section{Products of simplices}\label{sec:polysimplex}
In this section, we study the deformation space of the barycentric subdivision of the normal fan of a product of standard simplices.\footnote{Castillo--Doolittle--Goeckner--Ross--Ying proved that the nef cone of
every polytope combinatorially isomorphic to a product of simplices is
simplicial \cite{castillo2022minkowski}.  Here we instead study the
deformation cone obtained after barycentrically subdividing the normal
fan of a product of standard simplices.} 

\subsection{Polysimplices}

\begin{definition}\label{def:setup}
    Let $k \geq 1$, and let $\mathbf{S} = (S_1, \dots, S_k)$ be a tuple of nonempty finite sets.
    Let $\mathscr{S}$ denote the disjoint union of the sets in $\mathbf{S}$.
    We define the  \textbf{polysimplex}
    \begin{equation*}
        \Delta(\mathbf{S}) \coloneqq \prod_{i \in [k]} \Delta(S_i)
        \subseteq \bigoplus_{i \in [k]} \mathbb{R}^{S_i} \cong \mathbb{R}^{\mathscr{S}}.
    \end{equation*}
\end{definition}

\begin{remark}\label{rem:singleton}
    If some $S_i$ is a singleton, then $\Delta(S_i)$ is a point and the factor contributes
    nothing to the combinatorial type of $\Delta(\mathbf{S})$; such factors can be ignored
    without loss of generality.
\end{remark}

\begin{remark}\label{rem:embed}
    We write $\mathbb{R}^{\mathscr{S}} = \bigoplus_{S \in \mathbf{S}} \mathbb{R}^S$ as a direct
    sum to retain the canonical inclusions $\mathbb{R}^S \hookrightarrow \mathbb{R}^{\mathscr{S}}$.
    These allow us to view any polytope defined in $\mathbb{R}^S$, or in a sum over a
    subcollection of $\mathbf{S}$, as naturally embedded in the ambient space $\mathbb{R}^{\mathscr{S}}$.
\end{remark}

Recall that $\{e_s \mid s \in \mathscr{S}\}$ is the canonical basis for $\mathbb{R}^{\mathscr{S}}$.

\begin{example}\label{ex:running_example_polysimplex}
    Consider the product of two triangles $\Delta_3 \times \Delta_3$.
    Set $\mathbf{S} = (S_1, S_2)$ with $S_1 = \{A,B,C\}$ and $S_2 = \{X,Y,Z\}$, so
    $\Delta_3 \times \Delta_3 \cong \Delta(\mathbf{S})$ lives in
    $\mathbb{R}^{S_1} \oplus \mathbb{R}^{S_2} \cong \mathbb{R}^{\{A,B,C,X,Y,Z\}}$
    with canonical basis $\{e_A, e_B, e_C, e_X, e_Y, e_Z\}$. 
    \end{example}

\begin{figure}[htbp]
    \centering
    \includegraphics[width=.6\textwidth]{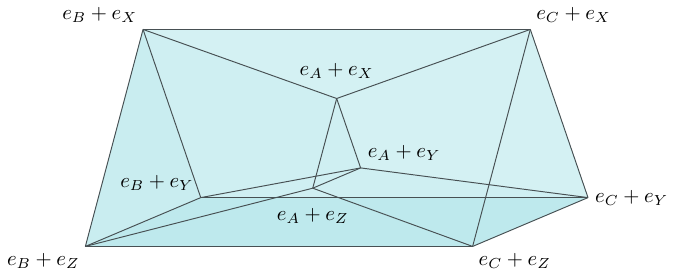}
    \caption{A Schlegel diagram of $\Delta_3 \times \Delta_3$.}
    \label{fig:triangle-product-schlegel}
\end{figure}

We identify $\mathbb{R}^{\mathscr{S}}$ with its own dual via the standard
inner product
\[
x \cdot y = \sum_{s \in \mathscr{S}} x(s)y(s),
\]
so that $e_s$ is identified with the coordinate functional
$x \mapsto x(s)$.

The facet-inequality description of the polysimplex $ \Delta ( \mathbf{S} ) $ is
	\begin{equation}\label{eq:polysimplex_Hrep}
		\Delta( \mathbf{S} ) =
		\left\{ x \in \mathbb{R}^{\mathscr{S}} \,\middle|\, 
			\begin{array}{rl}
				x(s) \geq 0,& \text{ for each }s\in \mathscr{S},\\
				\sum_{s\in S_i} x(s) = 1,& \text{ for each factor }S_i.
\end{array}
	\right\}
	\end{equation}

\begin{definition}\label{def:valid}
A tuple $\mathbf{A} = (A_1, \dots, A_k)$ is \textbf{valid} (with respect to $\mathbf{S}$) if
$\emptyset \subsetneq A_i \subseteq S_i$ for each $i$.
Valid tuples are in bijection with nonempty faces of $\Delta(\mathbf{S})$, with containment
$\Delta(\mathbf{A}) \subseteq \Delta(\mathbf{B})$ if and only if $A_i \subseteq B_i$ for each $i$.
We write $\overline{\mathbf{A}} \coloneqq \mathscr{S} \setminus \bigcup_i A_i$ for the complement of the
support of $\mathbf{A}$.
The face $\Delta(\mathbf{A})$, itself a polysimplex, is
\begin{equation}\label{eq:polysimplex_face}
    \Delta(\mathbf{A}) = \left\{ p \in \Delta(\mathbf{S}) \subseteq \mathbb{R}^{\mathscr{S}}
    \,\middle|\, p(s) = 0,\; s \in \overline{\mathbf{A}} \right\}.
\end{equation}
The inner normal vector for $\Delta(\mathbf A)\subseteq\Delta(\mathbf S)$ is
\begin{equation}\label{eq:normal_vector_polysimplex_face}
    u_{\mathbf{A}} \coloneqq \sum_{s \in \overline{\mathbf{A}}} e_s.
\end{equation}
\end{definition}

We define the \textbf{slack} of a valid tuple by
\begin{equation}\label{eq:gap_defi}
    \mathrm{Slk}(\mathbf{A}) \coloneqq \{ i \in [k] : A_i \neq S_i \},
    \qquad
    \mathrm{slk}(\mathbf{A}) \coloneqq |\mathrm{Slk}(\mathbf{A})|.
\end{equation}

\begin{example}\label{ex:running_example_polysimplex_faces}
In \Cref{ex:running_example_polysimplex} we have a two-dimensional face defined by the equalities $ p(X)=p(Y)=0 $.
This corresponds to the valid tuple $ \mathbf{A} = (\left\{ A,B,C \right\}, \left\{ Z \right\}) $.
The face is the polysimplex $ \Delta (\mathbf{A}) $ which is a triangle.
In this case the tuple has slack equal to one.
\end{example}

\begin{proposition}\label{prop:factor_out}
    Let $\mathbf{A} \subsetneq \mathbf{S}$ be a valid tuple, and partition $\mathbf{S}$ as
    \begin{equation}\label{eq:transversal_A}
        \mathbf{T}_\mathbf{A} = \left( S_i \in \mathbf{S} \,\middle|\, i \in \mathrm{Slk}(\mathbf{A}) \right),
        \qquad
        \overline{\mathbf T}_\mathbf{A} = \left( S_i \in \mathbf{S} \,\middle|\, i \notin \mathrm{Slk}(\mathbf{A}) \right).
    \end{equation}
    Then the deep truncation and fragment decompose as
    \begin{align}
        \operatorname{DTr}\left[\Delta(\mathbf{S}),\, \Delta(\mathbf{A})\right]
            &= \operatorname{DTr}\left[\Delta(\mathbf{T}_\mathbf{A}),\, \Delta(\mathbf{A}|_{\mathbf{T}_\mathbf{A}})\right]
               \times \Delta(\overline{\mathbf T}_\mathbf{A}), \label{eq:factor_out} \\
        \operatorname{Fr}\left[\Delta(\mathbf{S}),\, \Delta(\mathbf{A})\right]
            &= \operatorname{Fr}\left[\Delta(\mathbf{T}_\mathbf{A}),\, \Delta(\mathbf{A}|_{\mathbf{T}_\mathbf{A}})\right]
               \times \Delta(\overline{\mathbf T}_\mathbf{A}). \label{eq:factor_out_II}
    \end{align}
\end{proposition}

\begin{proof}
    By \eqref{eq:normal_vector_polysimplex_face}, the deep truncation and fragment are cut, respectively, from
    the system \eqref{eq:polysimplex_Hrep} by inequalities
    \begin{equation}\label{eq:factor_out_proof}
        \sum_{s \in \overline{\mathbf{A}}} p(s) \geq C \text{ and } \sum_{s \in \overline{\mathbf{A}}} p(s) \leq C
    \end{equation}
    for some constant $C$.
    Under the decomposition $\mathbb{R}^{\mathscr{S}} = \prod_{S_i \in \mathbf{T}_\mathbf{A}} \mathbb{R}^{S_i}
    \times \prod_{S_i \in \overline{\mathbf T}_\mathbf{A}} \mathbb{R}^{S_i}$,
    each of the inequalities \eqref{eq:factor_out_proof} involves only variables from the first factor,
    so the two systems separate and the product decomposition follows.
\end{proof}

\begin{remark}
    Each polytope $\operatorname{DTr}\!\left[\Delta(\mathbf{T}_\mathbf{A}),\,
    \Delta(\mathbf{A}|_{\mathbf{T}_\mathbf{A}})\right]$ is regarded inside $\mathbb{R}^{\mathscr{S}}$
    via the natural embedding of \Cref{rem:embed}.
\end{remark}

\begin{example}\label{ex:deep_truncation_edge_3d}
    See \Cref{fig:factor}. Consider the polysimplex $\Delta(\mathbf{S})$ for $\mathbf{S} = (\{1,2\},\{3,4\},\{5,6\})$,
    the product of three segments, i.e. a cube.
    The valid tuple $\mathbf{A} = (\{1\},\{3\},\{5,6\})$ has slack two, since $A_i \neq S_i$
    only for $i = 1, 2$.
    The corresponding face is the segment with endpoints $e_1 + e_3 + e_5$ and $e_1 + e_3 + e_6$.
    The flat truncation along this face removes these two vertices from the cube.
    By \Cref{prop:factor_out}, the result is the product of the deep truncation of the vertex
    $e_1 + e_3$ inside $\Delta(\{1,2\}) \times \Delta(\{3,4\})$ (a square) with the segment
    $\Delta(\{5,6\})$.
\end{example}
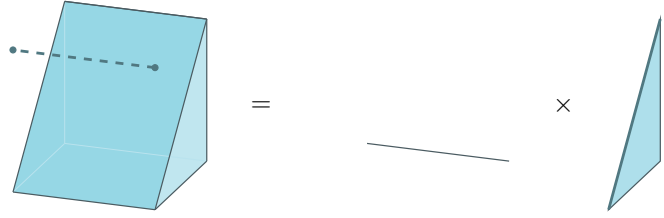
\begin{figure}[!htb]
    \centering
    \tdplotsetmaincoords{70}{110}
    \input{tikz/factor.tex}
    \caption{The flat truncation along an edge of the $3$-cube, exhibiting its factorization as the product of a segment and a triangle.}
    \label{fig:factor}
\end{figure}

The polytopes $\operatorname{DTr}\!\left[\Delta(\mathbf{T}_\mathbf{A}),\,
\Delta(\mathbf{A}|_{\mathbf{T}_\mathbf{A}})\right]$ and $\Delta(\overline{\mathbf T}_\mathbf{A})$
in \eqref{eq:factor_out}--\eqref{eq:factor_out_II} are canonically embedded in
$\mathbb{R}^{\mathscr{S}}$ via \Cref{rem:embed}.
We say that a valid tuple $\mathbf{A} \subseteq \mathbf{S}$ is \textbf{full} if
$\mathrm{slk}(\mathbf{A}) = k$.
Note that $\mathbf{A}$ is always full when restricted to $\mathbf{T}_\mathbf{A}$.

\begin{corollary}\label{cor:multisimplex_drop_dimension}
    Let $\mathbf{A} \subseteq \mathbf{S}$ be a valid tuple.
    \begin{enumerate}
        \item If $\mathrm{slk}(\mathbf{A}) > 1$, then both the deep truncation and the fragment
              have the same dimension as $\Delta(\mathbf{S})$.
        \item If $\mathrm{slk}(\mathbf{A}) = 1$, then the fragment is $\Delta(\mathbf{S})$ itself,
              and the deep truncation is a face of $\Delta(\mathbf{S})$.
        \item If $\mathrm{slk}(\mathbf{A}) = 0$, then $\Delta(\mathbf{A}) = \Delta(\mathbf{S})$.
            \end{enumerate}
\end{corollary}

In the case $\mathrm{slk}(\mathbf{A}) = 1$, say $A_i \subsetneq S_i$ for a unique index $i$,
the deep truncation is
\begin{equation}\label{eq:deep_truncation_slk_1}
    \operatorname{DTr}\left[\Delta(\mathbf{S}),\, \Delta(\mathbf{A})\right] = \Delta(\mathbf{A}'),
\end{equation}
where $\mathbf{A}'$ is the tuple with $A'_i = S_i \setminus A_i$ and $A'_j = S_j$ for $j \neq i$.

\begin{lemma}\label{lem:face_deformation_omni}
    Each face of $\Delta(\mathbf{S})$ is a deformation of $\operatorname{Omni}(\Delta(\mathbf{S}))$.
\end{lemma}
\begin{proof}
    A face $\Delta(A_1) \times \dots \times \Delta(A_k)$ coincides with
    the Minkowski sum $\Delta(A_1) + \dots + \Delta(A_k)$ under the embedding of \Cref{rem:embed}.
If $A_i \subsetneq S_i$, then $\Delta(A_i)$ is a factor of the deep truncation described in \Cref{eq:deep_truncation_slk_1} which is, by construction, a deformation of $\operatorname{Omni}(\Delta(\mathbf{S}))$.  If $A_i = S_i$, then $\Delta(A_i)$ is a factor of $\Delta(\mathbf{S})$, which is a deformation of $ \operatorname{Omni}(\Delta(\mathbf{S}))$.
    \end{proof}

\subsection{The normal fan of a polysimplex and its barycentric subdivision}

Let $\Sigma(\mathbf{S})$ be the normal fan of $\Delta(\mathbf{S})$ and $\Psi(\mathbf{S})$ its
barycentric subdivision, which is the normal fan of $\mathsf{P}(\mathbf{S}) = \operatorname{Omni}(\Delta(\mathbf{S}))$.

\begin{proposition}\label{prop:psi_refines_braid_product}
Let $\mathbf S=(S_1,\dots,S_k)$ be a tuple of finite sets.
The fan $\Psi(\mathbf S)$ refines $\prod_{i\in[k]}\mathcal B^1_{S_i}$, the product of the braid fans of the individual factors.
\end{proposition}

\begin{proof}
We induct on $k$.
For $k=1$, $\Delta(\mathbf S)=\Delta(S_1)$, so $\Psi(\mathbf S)=\mathcal B^1_{S_1}$ and the two fans coincide.
For $k\geq 2$, write $\mathbf S'=(S_2,\dots,S_k)$, so that $\Delta(\mathbf S)=\Delta(S_1)\times\Delta(\mathbf S')$.
Since $\Delta(S_1)$ and $\Delta(\mathbf S')$ are simple, being a simplex resp.\ a product of simplices, \Cref{prop:omni_product} together with \Cref{prop:coarsening_summand} shows that $\Sigma(\operatorname{Omni}(\Delta(\mathbf S)))$ refines $\Sigma(\operatorname{Omni}(\Delta(S_1)))\times\Sigma(\operatorname{Omni}(\Delta(\mathbf S')))$, that is,
\[
\Psi(\mathbf S) \text{ refines } \mathcal B^1_{S_1}\times\Psi(\mathbf S').
\]
By induction, $\Psi(\mathbf S')$ refines $\prod_{i=2}^{k}\mathcal B^1_{S_i}$, and taking the product with the fixed fan $\mathcal B^1_{S_1}$ preserves refinement, so $\mathcal B^1_{S_1}\times\Psi(\mathbf S')$ refines $\prod_{i=1}^{k}\mathcal B^1_{S_i}$.
Composing the two refinements gives the claim.
\end{proof}

The fans $\Sigma(\mathbf{S})$ and $\Psi(\mathbf{S})$ live in $\mathbb{R}^{\mathscr{S}}$ and have nontrivial lineality space.

\begin{lemma}\label{lem:fan_polysimplex}
    The fan $\Sigma(\mathbf{S})$ has lineality space
    \begin{equation}\label{eq:lineality_polysimplex}
        \mathsf{L}(\mathbf{S}) \coloneqq \operatorname{Span}\left\{ e_{S_i} \,\middle|\, i \in [k] \right\},
    \end{equation}
    where $e_{S_i} = \sum_{s \in S_i} e_s$. This lineality space has dimension $k$.

  Let $\mathbf{A} \subseteq \mathbf{S}$ be a valid tuple.  The pointed normal cone of $\Delta(\mathbf{A})$ is
\begin{equation}\label{eq:normal_cone_polysimplex_face}
    {}^{\wedge}\sigma(\mathbf{A}) \coloneqq \operatorname{Cone}\left\{ e_s \,\middle|\, s \in \overline{\mathbf{A}} \right\}.
\end{equation}

The normal cone of the face $\Delta(\mathbf{A})$ is
    \begin{equation}\label{eq:normal_cone_with_lineality}
        \sigma(\mathbf{A}) =   {}^{\wedge}\sigma(\mathbf{A}) + \mathsf{L}(\mathbf{S}),
    \end{equation}
   and the normal fan of $\Delta(\mathbf{S})$ is $\Sigma(\mathbf{S}) = \left\{ \sigma(\mathbf{A}) \,\middle|\, \mathbf{A} \text{ valid} \right\}$, which has the pointed representative 
   $ {}^{\wedge}\Sigma(\mathbf{S}) = \left\{ {}^{\wedge}\sigma(\mathbf{A}) \,\middle|\, \mathbf{A} \text{ valid} \right\}$.
   
    \end{lemma}

\begin{proof}
    The lineality space arises from the $k$ equality constraints $\sum_{s \in S_i} x(s) = 1$ in \eqref{eq:polysimplex_Hrep}: their normal directions $e_{S_i}$ lie in every normal cone.
    The description of $\Sigma(\mathbf{S})$ then follows from the bijection between valid tuples and nonempty faces of $\Delta(\mathbf{S})$, together with $\eqref{eq:polysimplex_Hrep}$.
\end{proof}

An ordered set partition of $\mathscr S$ is a sequence
$\mathfrak T=(\mathscr T_0,\mathscr T_1,\ldots,\mathscr T_m)$
of pairwise disjoint nonempty subsets whose union is $\mathscr S$.
It is \textbf{rooted} (with respect to $\mathbf S=(S_1,\ldots,S_k)$) if
$\mathscr T_0 \cap S_i \neq \emptyset$ for $1\leq i \leq k$.
It is \textbf{saturated} if $|\mathscr T_0\cap S_i|=1$ for $1\leq i \leq k$, and each of $\mathscr T_1,\ldots,\mathscr T_m$ is a singleton.    
	    
\begin{lemma}\label{lem:fan_polysimplex_bary}
    The fan $\Psi(\mathbf{S})$ has the same lineality space $\mathsf{L}(\mathbf{S})$ of dimension $k$, the length of $\mathbf{S}$.
	    We can write $\Psi(\mathbf{S}) = {}^{\wedge}\Psi(\mathbf{S}) + \mathsf{L}(\mathbf{S})$, where
	    \begin{equation}\label{eq:polysimplex_bary_cones}
	        {}^{\wedge}\Psi(\mathbf{S}) = \left\{ {}^{\wedge}\sigma(\mathfrak{T}) \,\middle|\, \mathfrak{T} \text{ is a rooted ordered set partition of } \mathscr{S} \right\},
	    \end{equation}
    where, for $\mathfrak{T} = (\mathscr{T}_{0}, \mathscr{T}_{1}, \dots, \mathscr{T}_{m})$,
    the cone ${}^{\wedge}\sigma(\mathfrak{T})$ consists of all $g \in \mathbb{R}^{\mathscr{S}}$ that are
    constant on each block $\mathscr{T}_{i}$ with value $g_i$, and satisfy
    \[
        0 = g_0 \leq g_1 \leq \dots \leq g_m.
    \]
    Furthermore, ${}^{\wedge}\sigma(\mathfrak{T})$ is maximal if and only if $\mathfrak{T}$ is saturated.
\end{lemma}
\begin{proof}
By definition, the cones of the barycentric subdivision are of the form
\begin{equation}\label{eq:preorder}
    \operatorname{Cone}\left\{ u_{\mathbf{A}_i} \,\middle|\, i \in [p] \right\},
\end{equation}
where $ \mathbf{A}_{1} \subsetneq \dots \subsetneq \mathbf{A}_{p} $ is a chain of valid tuples.
Recall from \Cref{def:valid} that for each tuple $ \mathbf{A}_i $ we write $ \overline{\mathbf{A}}_i $ for the elements of $ \mathscr{S} $ not in any subset of the tuple $ \mathbf{A}_{i} $.
The cone in Equation \eqref{eq:preorder} equals ${}^{\wedge}\sigma(\mathfrak T)$ for the rooted ordered set partition $\mathfrak{T} = (\mathscr{T}_{0}, \mathscr{T}_{1}, \dots, \mathscr{T}_{p})$ defined by
\[
\mathscr T_0
=
\mathscr S\setminus\overline{\mathbf A}_1,
\qquad
\mathscr T_i
=
\overline{\mathbf A}_i
\setminus
\overline{\mathbf A}_{i+1}
\quad(1\leq i<p),
\qquad
\mathscr T_p
=
\overline{\mathbf A}_p.
\]
The validity of $\mathbf{A}_{1}$ ensures that $\mathscr{T}_{0}$ contains at least one element from each factor, so $\mathfrak{T}$ is rooted.
Conversely, every rooted ordered set partition arises from a unique chain of valid tuples in this way, so the description in Equation \eqref{eq:polysimplex_bary_cones} is exhaustive.

For the last claim, maximal cones correspond to maximal chains. Under
the correspondence above, these are exactly the chains for which
$\mathscr T_0$ contains one element from each factor and
$\mathscr T_1,\ldots,\mathscr T_p$ are singletons. Thus
$\mathfrak T$ is saturated.
\end{proof}

\begin{example}\label{ex:running_example_polysimplex_cone_barycentric}
Continuing with \Cref{ex:running_example_polysimplex}, consider the chain $ (\{A\}, \{X,Y\}) \subseteq (\{A,B,C\}, \{X,Y\}) $.
	The corresponding chain of faces is $ \Delta(\{A\}) \times \Delta(\{X,Y\}) \subseteq \Delta(\{A,B,C\}) \times \Delta(\{X,Y\}) $.
	The respective inner normals are $ e_B + e_C + e_Z $ and $ e_Z $, so that the cone in the barycentric subdivision corresponding to this chain is the cone generated by these two elements.
	More precisely, this cone consists of elements of the form $ \lambda e_B + \lambda e_C + (\lambda + \mu)e_Z $ with $ \lambda, \mu \geq 0 $.
	In terms of inequalities this cone can be described as the set of functions $ f \in \mathbb{R}^{\{A,B,C,X,Y,Z\}} $ such that $ f(A) = f(X) = f(Y) = 0 $ and $ 0 \leq f(B) = f(C) \leq f(Z) $,
	which corresponds to the cone $ {}^{\wedge}\sigma(\mathfrak{T}) $ for the ordered set partition $\mathfrak{T} = (\{A,X,Y\}, \{B,C\}, \{Z\}) $.
\end{example}

Now we can describe $\mathrm{DS}(\mathsf{P}(\mathbf{S})) \cong \operatorname{PL}(\Psi(\mathbf{S}))$
combinatorially. By \Cref{prop:eta},
\begin{equation*}
	    \operatorname{PL}(\Psi(\mathbf{S})) \cong \operatorname{PL}({}^{\wedge}\Psi(\mathbf{S}))
    \oplus \mathsf{L}(\mathbf{S})^{\ast}.
\end{equation*}

\begin{definition}\label{def:functions_tuples}
    Following \Cref{lem:fan_polysimplex_bary}, we define:
    \begin{itemize}
        \item $V(\mathbf{S})$: the space of real-valued functions on valid proper tuples of $\mathbf{S}$,
	              isomorphic to $\operatorname{PL}({}^{\wedge}\Psi(\mathbf{S}))$.
        \item $U(\mathbf{S})$: the space of real-valued functions on $[k]$
              (in natural bijection with the factors of $\mathbf{S}$),
              isomorphic to $\mathsf{L}(\mathbf{S})^{\ast}$.
    \end{itemize}
\end{definition}

We identify $\operatorname{PL}(\Psi(\mathbf{S}))$ with pairs $(f, g)$ where $f \in V(\mathbf{S})$
and $g \in U(\mathbf{S})$.  We extend each $f\in V(\mathbf{S})$ to all valid tuples by setting
$f(\mathbf{S})=0$.
Recall that for each valid tuple $\mathbf{A}$ we define $\overline{\mathbf{A}}$ as the complement of the union of sets in the tuple.
If $\mathsf Q$ is a polytope whose normal fan coarsens $\Psi(\mathbf S)$,
and $(f,g)$ are the coordinates of its support function under this
identification, then $\mathsf Q$ is recovered as
\begin{equation}\label{eq:polytope_fg_description}
    \left\{ x \in \mathbb{R}^{\mathscr{S}} \,\middle|\,
        \begin{array}{rl}
            \displaystyle\sum_{s \in \overline{\mathbf{A}}} x(s) \geq f(\mathbf{A}),
                & \text{for each valid } \mathbf{A}, \\[6pt]
            \displaystyle\sum_{s \in S_i} x(s) = g(i),
                & \text{for each factor } S_i.
        \end{array}
    \right\}
\end{equation}

\begin{example}\label{ex:polysimplex_full_support}
    The support function of $\Delta(\mathbf{S})$ is
    \begin{align}
        f(\mathbf{A}) &\coloneqq h(\Delta(\mathbf{S}),\, u_{\mathbf{A}}) = 0,
            &\text{for each valid tuple } \mathbf{A}, \\
        g(i) &\coloneqq h\!\left(\Delta(\mathbf{S}),\, e_{S_i}\right) = 1,
            &\text{for each factor } S_i.
    \end{align}
\end{example}

\begin{lemma}\label{lem:polysimplex_face_support}
    Let $\Delta(\mathbf{A}) \subseteq \Delta(\mathbf{S})$ be a face.
    Its support function is:
    \begin{align}
        f(\mathbf{B}) &\coloneqq h(\Delta(\mathbf{A}),\, u_{\mathbf{B}})
            = |\{i \in [k] \mid A_i \cap B_i = \emptyset\}|,
            &\text{for each valid tuple } \mathbf{B}, \\
        g(i) &\coloneqq h\!\left(\Delta(\mathbf{A}),\, e_{S_i}\right) = 1,
            &\text{for each factor } S_i.
    \end{align}
\end{lemma}
\begin{proof}
    We start with $g$.
    Since $e_{S_i} \cdot x = \sum_{s \in S_i} x(s) = 1$ for all $x \in \Delta(\mathbf{S})$
    by \eqref{eq:polysimplex_Hrep}, it is constant on every face, giving $g(i) = 1$.

    Now we argue for $f$.
    By definition
$u_{\mathbf B}=\sum_{s\in\overline{\mathbf B}} e_s$, so
\[
h(\Delta(\mathbf A),u_{\mathbf B})
=
\min_{x\in\Delta(\mathbf A)}
\sum_{s\in\overline{\mathbf B}}x(s).
\]
    Since $\Delta(\mathbf A)=\prod_i\Delta(A_i)$ in disjoint coordinate spaces,
this minimum is the sum over $i$ of the minima of
$\sum_{s\in\overline{\mathbf B}\cap S_i}x_i(s)$ on $\Delta(A_i)$.
    Since $\overline{\mathbf B}\cap S_i=S_i\setminus B_i$
and $x_i$ is supported on $A_i$, the $i$-th term reduces to
\[
\min_{x_i\in\Delta(A_i)}
\sum_{s\in A_i\setminus B_i}x_i(s).
\]
    If $A_i \cap B_i \neq \emptyset$, we can put a 1 on any element of $A_i \cap B_i$ and zero on the remaining coordinates, giving minimum $0$.
    If $A_i \cap B_i = \emptyset$, then $x_i$ is supported entirely on $A_i \setminus B_i$, forcing the sum to equal $1$.
    Summing over $i$ gives $|\{i \in [k] \mid A_i \cap B_i = \emptyset\}|$ as desired.
\end{proof}

\begin{example}
Consider the edge $\Delta(\mathbf{A})$ with $\mathbf{A} = (\{A,B\}, \{X\})$, whose vertices are $e_A + e_X$ and $e_B + e_X$, in the polysimplex of \Cref{ex:running_example_polysimplex}.
Take the valid tuple
$\mathbf B=(\{A,C\},\{Y\})$; then
$\overline{\mathbf B}=\{B,X,Z\}$
and $u_{\mathbf B}=e_B+e_X+e_Z$.
The support function evaluates to
\[
    h(\Delta(\mathbf{A}),\, u_{\mathbf{B}})
    = \min\!\left\{ (e_B + e_X + e_Z)\cdot(e_A + e_X),\;
                   (e_B + e_X + e_Z)\cdot(e_B + e_X) \right\}
    = \min\{1, 2\} = 1,
\]
consistent with the formula in \Cref{lem:polysimplex_face_support}: only the second entries are disjoint ($\{X\} \cap \{Y\} = \emptyset$), while the first entries share $A$, contributing exactly one to the count.
\end{example}

We record the support functions of the deep truncations in the $(f,g) \in V(\mathbf{S}) \times U(\mathbf{S})$ format of \Cref{def:functions_tuples}.

\begin{lemma}\label{lem:polysimplex_truncation_support}
    Let $\Delta(\mathbf{A}) \subsetneq \Delta(\mathbf{S})$ be a proper face.
    The support function of $\operatorname{DTr}\!\left[\Delta(\mathbf{S}),\, \Delta(\mathbf{A})\right]$ is
    \begin{align}
        f(\mathbf{B}) &=
        \begin{cases}
            1 & \text{if } B_i \subseteq A_i \text{ for all } i, \\
            0 & \text{otherwise,}
        \end{cases}\\
    \text{and}\,\,\,\,\,\,\,     g(i) &= 1 \qquad \text{for every factor } S_i.
    \end{align}
\end{lemma}

\begin{proof}
    This follows from Equation \eqref{eq:delzant_difference} and \Cref{ex:polysimplex_full_support}.
\end{proof}

The facets of the fragment are of two types.  First, there are the facets of $\Delta(\mathbf{S})$ that meet the face $\Delta(\mathbf{A})$.
Second, there is a unique new facet created by the truncation, whose inner normal is $-u_{\mathbf{A}}$, the negation of the inner normal of the face being removed.

In the following lemma we record the support function of the fragment, even though at this point we do not yet know whether fragments of polysimplices are deformations of their omnitruncation.
Nonetheless, their inner facet normals are a subset of the ray generators of $\Psi(\mathbf{S})$, so we can at least record the minimum of the corresponding linear functional in each of these directions.

\begin{lemma}\label{lem:polysimplex_fragment_support}
    Let $\Delta(\mathbf{A}) \subseteq \Delta(\mathbf{S})$ be a face.
		The support function of $\operatorname{Fr}\left[\Delta(\mathbf{S}),\Delta(\mathbf{A})\right]$ is 
\begin{align}\label{eq:combinatorial_matrix}
    f(\mathbf{B}) &= \max\bigl(0,\ |\{i \in [k] | A_i \cap B_i = \emptyset\}| - 1\bigr), &\text{for each valid tuple } \mathbf{B},\\
    g(i) &= 1, &\text{for each factor } S_i.
\end{align}
\end{lemma}

\begin{proof}
By \Cref{lem:polysimplex_face_support}, the functional $u_{\mathbf{B}}$ attains minimum value $|\{i | A_i \cap B_i = \emptyset\}|$ on the face $\Delta(\mathbf{A})$, achieved at any vertex $v = e_{s_1} + \cdots + e_{s_k}$ whose $i$-th coordinate satisfies $s_i \in A_i \cap B_i$ whenever this intersection is nonempty.
On the full polytope $\Delta(\mathbf{S})$, the same functional attains minimum value $0$ by \Cref{ex:polysimplex_full_support}.
The vertices of the fragment are the vertices of $\Delta(\mathbf{A})$ together with their neighbors in $\Delta(\mathbf{S})$, where a neighbor is obtained by replacing exactly one coordinate $s_i \in A_i$ by an element of $S_i \setminus A_i$.
If the replaced coordinate is chosen at an index where $A_i \cap B_i = \emptyset$, and we replace $s_i$ by some element of $B_i$, the inner product with $u_{\mathbf{B}}$ drops by exactly $1$; any other replacement either leaves the value unchanged or increases it.
Hence the minimum over the fragment is one less than the minimum over $\Delta(\mathbf{A})$, provided there exists an index with $A_i \cap B_i = \emptyset$; otherwise the fragment inherits the minimum value $0$ from $\Delta(\mathbf{A})$.
For $g(i)$: the functional $e_{S_i}$ evaluates to $1$ on every vertex of $\Delta(\mathbf{S})$, hence on every vertex of the fragment.
\end{proof}

\begin{example}\label{ex:fragment_edge_triangle_triangle_segment}
Let $\mathbf{S} = (\{A,B,C\}, \{X,Y,Z\}, \{P,Q\})$, so $\Delta(\mathbf{S}) = \Delta_3 \times \Delta_3 \times \Delta_2$ is the product of two triangles and a segment, with $18$ vertices.
Consider the edge $\Delta(\mathbf{A})$ for $\mathbf{A} = (\{A,B\}, \{X\}, \{P\})$, with vertices $e_A + e_X + e_P$ and $e_B + e_X + e_P$.
This tuple is full since $A_i \subsetneq S_i$ for all $i$.
The fragment $\operatorname{Fr}\left[\Delta(\mathbf{S}), \Delta(\mathbf{A})\right]$ consists of the two edge vertices together with their neighbors in $\Delta(\mathbf{S})$, giving $9$ vertices in total:
\[
    \{e_A,e_B\} \times \{e_X\} \times \{e_P\}
    \;\cup\;
    \{e_C\} \times \{e_X\} \times \{e_P\}
    \;\cup\;
    \{e_A,e_B\} \times \{e_Y,e_Z\} \times \{e_P\}
    \;\cup\;
    \{e_A,e_B\} \times \{e_X\} \times \{e_Q\}.
\]
Note that $e_C + e_X + e_P$ is a neighbor of both edge vertices.
To verify \Cref{lem:polysimplex_fragment_support}, take $\mathbf{B} = (\{C\},\{Y\},\{Q\})$, so $\overline{\mathbf{B}} = \{A,B,X,Z,P\}$ and $u_{\mathbf{B}} = e_A + e_B + e_X + e_Z + e_P$.
Since $A_i \cap B_i = \emptyset$ for all three indices, the formula gives $f(\mathbf{B}) = \max(0,3-1) = 2$.
Indeed, $u_{\mathbf{B}}$ attains minimum value $2$ on the fragment, for instance at $e_C + e_X + e_P$ (contributing $1$ from $e_X$ and $1$ from $e_P$) and at $e_A + e_Y + e_P$ (contributing $1$ from $e_A$ and $1$ from $e_P$), while the edge vertices $e_A + e_X + e_P$ and $e_B + e_X + e_P$ both evaluate to $3$.
\end{example}

\subsection{The deformation cone of an omnitruncated polysimplex}

We describe the full deformation cone of $\mathsf{P}(\mathbf{S})$, the polysimplex analogue of the supermodular cone. 
Throughout, we use the identification
\[
    \operatorname{PL}(\Psi(\mathbf{S}))
    \cong
    V(\mathbf{S})\oplus U(\mathbf{S})
\]
of \Cref{def:functions_tuples}, and write a piecewise-linear function as
a pair $(f,g)$.

Recall that $ u_{\mathbf{A}} = \sum_{s\in\overline{\mathbf{A}}}e_s $ for every valid tuple $\mathbf{A}$. 
If $\mathbf{A}\neq\mathbf{S}$, then $u_{\mathbf{A}}$ is the canonical ray generator of the corresponding ray of ${}^{\wedge}\Psi(\mathbf{S})$.

For two valid tuples $\mathbf{A}$ and $\mathbf{B}$, let
\begin{equation}\label{eq:disjoint_factors}
    Z(\mathbf{A},\mathbf{B})
    \coloneqq
    \left\{
        i\in[k]
        \,\middle|\,
        A_i\cap B_i=\emptyset
    \right\}.
\end{equation}

For each $i\in[k]$, there are two cases:
\begin{enumerate}
\item If $i\notin Z(\mathbf A,\mathbf B)$, then
$A_i\cap B_i\neq\varnothing$, and we use
$A_i\cap B_i\subseteq A_i\cup B_i$.
\item If $i\in Z(\mathbf A,\mathbf B)$, then
$A_i\cap B_i=\varnothing$, and we instead use
$A_i\cup B_i\subseteq S_i$.
\end{enumerate}

\begin{definition}\label{def:polysimplex_uncrossing}
Let $\mathbf{A}$ and $\mathbf{B}$ be valid tuples.  Their \textbf{lower uncrossing} and \textbf{upper uncrossing}, denoted respectively by
\[
    \mathbf{A}\wedge\mathbf{B}
    \qquad\text{and}\qquad
    \mathbf{A}\vee\mathbf{B},
\]
are the valid tuples defined componentwise by
\begin{align}
    \left(\mathbf{A}\wedge\mathbf{B}\right)_i
    &\coloneqq
    \begin{cases}
        A_i\cap B_i,
            & i\notin Z(\mathbf{A},\mathbf{B}),\\
        A_i\cup B_i,
            & i\in Z(\mathbf{A},\mathbf{B}),
    \end{cases}
    \label{eq:lower_uncrossing}\\[4pt]
    \left(\mathbf{A}\vee\mathbf{B}\right)_i
    &\coloneqq
    \begin{cases}
        A_i\cup B_i,
            & i\notin Z(\mathbf{A},\mathbf{B}),\\
        S_i,
            & i\in Z(\mathbf{A},\mathbf{B}).
    \end{cases}
    \label{eq:upper_uncrossing}
\end{align}
\end{definition}

These operations agree with the componentwise intersection and union when $Z(\mathbf{A},\mathbf{B})=\emptyset$. 
They are not, in general, the meet and join in the poset of valid tuples. 
Their essential property is that they always form a comparable pair:
\begin{equation}\label{eq:uncrossing_comparable}
    \mathbf{A}\wedge\mathbf{B}
    \subseteq
    \mathbf{A}\vee\mathbf{B}.
\end{equation}

\begin{lemma}\label{lem:polysimplex_primitive_relation}
For every pair of valid tuples $\mathbf{A}$ and $\mathbf{B}$, we have
\begin{equation}\label{eq:polysimplex_primitive_relation}
    u_{\mathbf{A}}+u_{\mathbf{B}}
    =
    u_{\mathbf{A}\wedge\mathbf{B}}
    +
    u_{\mathbf{A}\vee\mathbf{B}}
    +
    \sum_{i\in Z(\mathbf{A},\mathbf{B})}e_{S_i}.
\end{equation}

In particular, modulo the lineality space $\mathsf{L}(\mathbf{S})$, the vector $u_{\mathbf{A}}+u_{\mathbf{B}}$ belongs to the cone of ${}^{\wedge}\Psi(\mathbf{S})$ indexed by the chain $ \mathbf{A}\wedge\mathbf{B} \subseteq \mathbf{A}\vee\mathbf{B}. $
\end{lemma}

\begin{proof}
We verify \eqref{eq:polysimplex_primitive_relation} separately in each factor $S_i$. 
Recall that the restriction of $u_{\mathbf{A}}$ to $\mathbb{R}^{S_i}$ is the indicator vector of $S_i\setminus A_i$, and we take the convention that $e_{\emptyset} = 0$.

\begin{itemize}
    \item  Suppose first that $i\notin Z(\mathbf{A},\mathbf{B})$.
The usual indicator-vector identity gives
\[
    e_{S_i\setminus A_i}
    +
    e_{S_i\setminus B_i}
    =
    e_{S_i\setminus(A_i\cap B_i)}
    +
    e_{S_i\setminus(A_i\cup B_i)}.
\]
These are precisely the contributions of $u_{\mathbf{A}\wedge\mathbf{B}}$ and $u_{\mathbf{A}\vee\mathbf{B}}$ in the $i$-th factor.

\item Suppose now that $i\in Z(\mathbf{A},\mathbf{B})$. 
Since $A_i\cap B_i=\emptyset$, we have
\[
    e_{S_i\setminus A_i}
    +
    e_{S_i\setminus B_i}
    =
    e_{S_i\setminus(A_i\cup B_i)}
    +
    e_{S_i}.
\]

The first term on the right-hand side is the contribution of $u_{\mathbf{A}\wedge\mathbf{B}}$. 
Since $ \left(\mathbf{A}\vee\mathbf{B}\right)_i=S_i, $
the contribution of $u_{\mathbf{A}\vee\mathbf{B}}$ in this factor is zero, while $e_{S_i}$ belongs to the lineality space.
\end{itemize}

Summing these identities over all factors proves \eqref{eq:polysimplex_primitive_relation}. 
The final statement follows from \eqref{eq:uncrossing_comparable}, since the cones of ${}^{\wedge}\Psi(\mathbf{S})$ are indexed by chains of valid tuples. 
\end{proof}

\begin{theorem}\label{thm:deformation_cone}
The \textbf{supermodular cone of the polysimplex} $\Delta(\mathbf{S})$ is the set of pairs $ (f,g)\in V(\mathbf{S})\oplus U(\mathbf{S}) $ with $f(\mathbf{S})=0$ satisfying
\begin{equation}\label{eq:supermodular_inequality}
    f(\mathbf{A})+f(\mathbf{B})
    \leq
    f\left(\mathbf{A}\wedge\mathbf{B}\right)
    +
    f\left(\mathbf{A}\vee\mathbf{B}\right)
    +
    \sum_{i\in Z(\mathbf{A},\mathbf{B})}g(i)
\end{equation}
for every pair of valid tuples $\mathbf{A}$ and $\mathbf{B}$.
The deformation cone of $\mathsf{P}(\mathbf{S})$ is the supermodular cone of $\Delta(\mathbf{S})$.\footnote{Note that if we switch from the min to max convention for support functions, this replaces the inner normal fan with the outer normal fan.   Appropriately modifying the definitions of $\wedge$ and $\vee$, and the resulting inequalities above, would produce a \emph{submodular cone of $\Delta(\mathbf{S})$}.}
\end{theorem}

\begin{proof}
A collection of rays of the barycentric subdivision $\Psi(\mathbf{S})$ spans a cone if and only if the corresponding valid tuples form a chain. 
Consequently, its primitive collections are precisely the pairs
\[
    \{u_{\mathbf{A}},u_{\mathbf{B}}\}
\]
indexed by incomparable proper valid tuples $\mathbf{A}$ and
$\mathbf{B}$.

Let $h\in\operatorname{PL}(\Psi(\mathbf{S}))$ correspond to the pair $(f,g)$. 
By \Cref{lem:polysimplex_primitive_relation}, we have
\[
    u_{\mathbf{A}}+u_{\mathbf{B}}
    =
    u_{\mathbf{A}\wedge\mathbf{B}}
    +
    u_{\mathbf{A}\vee\mathbf{B}}
    +
    \sum_{i\in Z(\mathbf{A},\mathbf{B})}e_{S_i}.
\]

The first two terms on the right-hand side belong to a common cone of ${}^{\wedge}\Psi(\mathbf{S})$, while the remaining terms belong to $\mathsf{L}(\mathbf{S})$. 
Therefore, under the decomposition of \Cref{prop:eta},
\[
\begin{split}
    h(u_{\mathbf{A}}+u_{\mathbf{B}})
    =\;&
    f\left(\mathbf{A}\wedge\mathbf{B}\right)
    +
    f\left(\mathbf{A}\vee\mathbf{B}\right)\\
    &+
    \sum_{i\in Z(\mathbf{A},\mathbf{B})}g(i).
\end{split}
\]

Here we use $f(\mathbf{S})=0$ whenever one of the two uncrossings is $\mathbf{S}$.

By Batyrev's criterion, in the form of \Cref{prop:barycentric_deformation_inequalities}, the function $h$ belongs to $\operatorname{Def}(\mathsf{P}(\mathbf{S}))$ if and only if
\[
    h(u_{\mathbf{A}})+h(u_{\mathbf{B}})
    \leq
    h(u_{\mathbf{A}}+u_{\mathbf{B}})
\]
for every incomparable pair $\mathbf{A},\mathbf{B}$. 
Under the $(f,g)$-coordinates, these are precisely the inequalities \eqref{eq:supermodular_inequality}.

It remains only to observe that the same inequalities may equivalently be imposed for every pair of valid tuples. 
Indeed, if $\mathbf{A}\subseteq\mathbf{B}$, then $ Z(\mathbf{A},\mathbf{B})=\emptyset, \quad \mathbf{A}\wedge\mathbf{B}=\mathbf{A}, \quad \mathbf{A}\vee\mathbf{B}=\mathbf{B}, $ so \eqref{eq:supermodular_inequality} holds with equality.
\end{proof}

\begin{remark}\label{rem:deformation_cone_simplex}
When $k=1$, \Cref{thm:deformation_cone} recovers the usual (type $A$) supermodular inequalities.  When $|S_i|=2$ for all $1\leq i\leq k$, i.e. $\Delta(\mathbf S)$ is a cube, we recover the type $B/C$ Coxeter supermodular functions; see \cite[Theorem~5.2]{ardila2020coxeter}.\footnote{Due to min vs. max conventions, the version in \cite[Theorem~5.2]{ardila2020coxeter} is the submodular version of our statement.}
\end{remark}

\subsection{Indecomposable truncation and fragment bases}

Deep truncations of polysimplices need not be indecomposable; see \Cref{ex:deep_truncation_edge_3d}. The following theorem gives a sufficient condition for deep truncations and fragments to be indecomposable.

\begin{theorem}\label{thm:small_truncated_indecomposable}
	For every full tuple $ \mathbf{A} \subseteq \mathbf{S} $, the deep truncation $ \operatorname{DTr} \left[ \Delta ( \mathbf{S}), \Delta ( \mathbf{A}) \right] $ and the fragment $ \operatorname{Fr} \left[ \Delta ( \mathbf{S}), \Delta ( \mathbf{A}) \right] $ are both indecomposable.
\end{theorem}

See \Cref{fig:triangle-product-edge-truncation-fragment} for illustrations of indecomposable truncations and fragments.

\begin{proof}
If $k=1$, both polytopes are simplices by \Cref{cor:multisimplex_drop_dimension}, and hence are indecomposable.  We may therefore assume that $k\geq 2$.  We invoke a recent result of Higashitani--Padrol--Sanyal~\cite{higashitani2026indecomposability}, which states that every $0/1$-polytope is either indecomposable or a nontrivial Cartesian product of $0/1$-polytopes.
Both the fragment and the deep truncation are $0/1$-polytopes.  

Both the deep truncation and the fragment have a distinguished facet $\mathsf{F}$ whose vertices are precisely the neighbors of $\Delta(\mathbf{A})$.  Let $\mathsf{P}$ denote either polytope.  Suppose for contradiction that $\mathsf{P}$ is decomposable.  The cited result gives a decomposition
$\mathsf{P}=\mathsf{P}'\times\mathsf{P}''$ along a nontrivial
partition of the coordinate set $\mathscr{S}$.  This partition cannot
split any $S_i$.  Indeed, since $\mathbf{A}$ is full and $k\geq 2$,
every coordinate takes the value $1$ at some vertex of $\mathsf{P}$;
if $S_i$ met both coordinate blocks, the product structure would
therefore produce a point $x\in\mathsf{P}$ with
$\sum_{s\in S_i}x(s)=2$.  Thus the coordinate partition induces a
nontrivial partition
$\mathbf{S}'\sqcup\mathbf{S}''=\mathbf{S}$, with
$\mathsf{P}'\subseteq\Delta(\mathbf{S}')$ and
$\mathsf{P}''\subseteq\Delta(\mathbf{S}'')$.  In particular $\mathsf{F} = \mathsf{F}' \times \mathsf{F}'' \subset \mathsf{P}' \times \mathsf{P}''$.

Since the tuple $\mathbf{A}$ is full, 
\[
\mathsf{F} = \operatorname{Conv}(\bigcup_{i \in [k]} \big( \,\Delta(S_i \setminus {A}_i) \times \prod_{j \neq i} \Delta(A_j) \big)).
\]

Let $S_i \in \mathbf{S}'$ and $S_j \in \mathbf{S}''$.  There exists some vertex $p \in \mathsf{F}$ whose projection $p'$ to $\Delta(\mathbf{S}')$ satisfies $p'_i \notin \Delta(A_i)$.  Similarly, there exists a vertex $q \in \mathsf{F}$ whose projection $q''$ to $\Delta(\mathbf{S}'')$ satisfies $q''_j \notin \Delta(A_j)$.  Thus the vertex $p'\times q'' \in \mathsf{F}$ has projections to $\Delta(S_i)$ and $\Delta(S_j)$ not in $\Delta(A_i)$ and not in $\Delta(A_j)$, respectively, a contradiction.

\end{proof}

\begin{figure}[!htb]
    \centering
    \includegraphics[width=.85\textwidth]{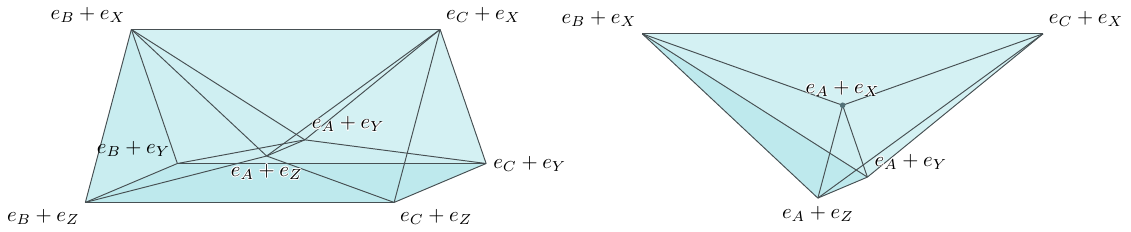}

    \vspace{6mm}
    \includegraphics[width=.85\textwidth]{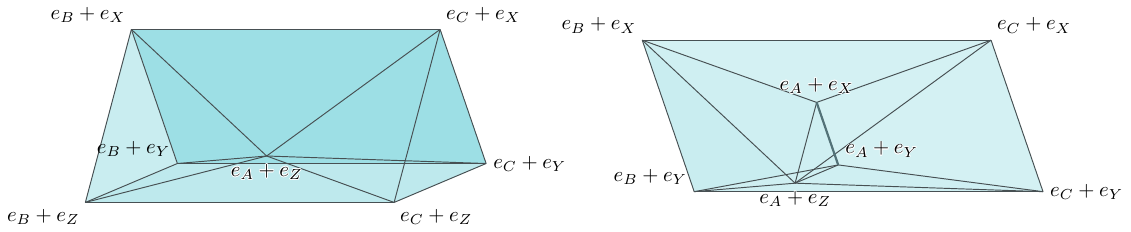}
    \caption{Examples of indecomposable flat truncations and fragments of $\Delta_3 \times \Delta_3$. Top: the flat truncation along the vertex $e_A+e_X$ and its associated fragment. Bottom: the flat truncation along an edge $[e_A+e_X, e_A+e_Y]$ and its associated fragment.}
    \label{fig:triangle-product-edge-truncation-fragment}
\end{figure}

We now apply the results of \Cref{sec:basis} to construct an indecomposable basis for
the deformation space of an omnitruncation of $\Delta(\mathbf{S})$.
We start with a polytopal basis that is not yet indecomposable.

\begin{proposition}\label{prop:basis_polysimplex}
    A polytopal basis for the deformation space of $\mathsf{P}(\mathbf{S})$ is given by
    \begin{equation}\label{eq:old_basis}
        \left\{ \operatorname{DTr}\left[\Delta(\mathbf{S}),\, \Delta(\mathbf{A})\right]
            \,\middle|\, \mathbf{S}\neq\mathbf{A} \text{ valid} \right\}
        \cup
        \left\{ \Delta(S) \,\middle|\, S \in \mathbf{S} \right\}.
    \end{equation}
\end{proposition}

\begin{proof}
    Direct consequence of \Cref{prop:key_non_full_general}.
\end{proof}

We now upgrade the basis of \Cref{prop:basis_polysimplex} to an indecomposable one.

\begin{theorem}\label{thm:basis_polysimplex_deep}
    An indecomposable polytopal basis for the deformation space of $\mathsf{P}(\mathbf{S})$ is given by
    \begin{equation}\label{eq:basis_polysimplex_deep}
        \left\{ \operatorname{DTr}\!\left[\Delta(\mathbf{T}_\mathbf{A}),\,
            \Delta(\mathbf{A}|_{\mathbf{T}_\mathbf{A}})\right]
            \,\middle|\, \mathbf{S}\neq\mathbf{A} \text{ valid} \right\}
        \cup
        \left\{ \Delta(S) \,\middle|\, S \in \mathbf{S} \right\}.
    \end{equation}
\end{theorem}

\begin{proof}
  Starting from the basis of \Cref{prop:basis_polysimplex}, when $\mathbf{A}$ is not full, \eqref{eq:factor_out} shows that $\operatorname{DTr}\!\left[\Delta(\mathbf{S}),\, \Delta(\mathbf{A})\right]$ decomposes as a Cartesian product, from which we extract the factor $\operatorname{DTr}\!\left[\Delta(\mathbf{T}_\mathbf{A}),\,
    \Delta(\mathbf{A}|_{\mathbf{T}_\mathbf{A}})\right]$.
    Each such factor is indecomposable by \Cref{thm:small_truncated_indecomposable}, and the remaining factor $\Delta(\overline{\mathbf T}_\mathbf{A})$ is generated by $\{\Delta(S) \mid S \in \mathbf{S}\}$.
    The set \eqref{eq:basis_polysimplex_deep} therefore spans the basis from \Cref{eq:old_basis}, and has the same cardinality as this basis, hence it is itself a basis.
\end{proof}

For the case of cubes, the specialization of the basis of \Cref{thm:basis_polysimplex_deep} differs from the polytopal bases of \cite{bastidas2021polytope} and \cite{eur2024signed}, which consist entirely of simplices, and from the shard polytopes of \cite{padrol2023shard}.  In \Cref{prop:basis_polysimplex_fragments} we will describe a basis constructed from fragments which specializes, up to translation, to the basis of \cite{eur2024signed}, and is closely related to the one studied in \cite{bastidas2021polytope}.

In general, a fragment of a polytope $\mathsf{S}$ truncated along a face does not live in the deformation cone of $\operatorname{Omni}(\mathsf{S})$.  However, we will now show that the fragments $\operatorname{Fr}\left[\Delta(\mathbf{S}),\Delta(\mathbf{A})\right]$ are deformations of $\mathsf{P}(\mathbf{S})$.

\begin{lemma}\label{lem:fragment_new_facet}
Let $\mathbf{A}$ be a valid tuple with $\mathrm{slk}(\mathbf{A}) \geq 1$, and define the tuple $\mathbf{A}^{\ast}$ by
\[
    A^{\ast}_{i} =
    \begin{cases}
        S_i \setminus A_i, & i \in \mathrm{Slk}(\mathbf{A}),\\
        S_i, & i \notin \mathrm{Slk}(\mathbf{A}).
    \end{cases}
\]
Then $\mathbf{A}^{\ast}$ is valid and
\begin{equation}\label{eq:fragment_new_facet}
    u_{\mathbf{A}} + u_{\mathbf{A}^{\ast}} = \sum_{i \in \mathrm{Slk}(\mathbf{A})} e_{S_i}.
\end{equation}
In particular,  $-u_{\mathbf{A}}$ agrees with $u_{\mathbf{A}^{\ast}}$ modulo the lineality space $\mathsf{L}(\mathbf{S})$, and hence is a ray generator of $\Psi(\mathbf{S})$.
\end{lemma}

\begin{proof}
For $i \in \mathrm{Slk}(\mathbf{A})$ the set $S_i \setminus A_i$ is nonempty by Equation \eqref{eq:gap_defi}, so $\mathbf{A}^{\ast}$ is valid.
The complement of the support of $\mathbf{A}$ meets $S_i$ in $S_i \setminus A_i$, while the complement of the support of $\mathbf{A}^{\ast}$ meets $S_i$ in $A_i$ for $i \in \mathrm{Slk}(\mathbf{A})$ and is empty for $i \notin \mathrm{Slk}(\mathbf{A})$.
By Equation \eqref{eq:normal_vector_polysimplex_face} we therefore have $u_{\mathbf{A}} = \sum_{i \in \mathrm{Slk}(\mathbf{A})} e_{S_i \setminus A_i}$ and $u_{\mathbf{A}^{\ast}} = \sum_{i \in \mathrm{Slk}(\mathbf{A})} e_{A_i}$, and adding these gives Equation \eqref{eq:fragment_new_facet}.
The last claim follows from Equation \eqref{eq:lineality_polysimplex}.
\end{proof}

\begin{remark}
When $\mathrm{slk}(\mathbf{A}) = 1$ the tuple $\mathbf{A}^{\ast}$ is the tuple $\mathbf{A}'$ of Equation \eqref{eq:deep_truncation_slk_1}.
\end{remark}

\begin{proposition}\label{prop:fragment}
    Let $\Delta(\mathbf{A})$ be a proper face of $\Delta(\mathbf{S})$.
    The fragment $\operatorname{Fr}\left[\Delta(\mathbf{S}),\Delta(\mathbf{A})\right]$ is a deformation of $\mathsf{P}(\mathbf{S})$.
\end{proposition}

\begin{proof}
Write $\mathsf{F} = \operatorname{Fr}\left[\Delta(\mathbf{S}),\Delta(\mathbf{A})\right]$ and let $(f,g)$ be the pair recorded in \Cref{lem:polysimplex_fragment_support}.
If $\mathrm{slk}(\mathbf{A}) \leq 1$, then $\mathsf{F} = \Delta(\mathbf{S})$ by \Cref{cor:multisimplex_drop_dimension} and the fact that any polytope is a deformation of its omnitruncation.
By \Cref{rem:singleton} we may also assume that every factor has at least two elements.
We first identify $\mathsf{F}$ with the polytope attached to $(f,g)$ by Equation \eqref{eq:polytope_fg_description}, and then verify that $(f,g)$ lies in the deformation cone.

Let $\mathsf{Q}$ denote the polytope in Equation \eqref{eq:polytope_fg_description} determined by $(f,g)$.
Since $f(\mathbf{B}) = h(\mathsf{F}, u_{\mathbf{B}})$ for every valid tuple $\mathbf{B}$ and $g(i) = 1$ for every factor, all of the defining conditions of $\mathsf{Q}$ hold on $\mathsf{F}$, so $\mathsf{F} \subseteq \mathsf{Q}$. 
For the reverse inclusion, note that $\mathsf{F}$ is cut out by the equalities $\sum_{s \in S_i} x(s) = 1$, the inequalities $x(s) \geq 0$ for $s \in \mathscr{S}$, and the single inequality $u_{\mathbf{A}} \cdot x \leq h(\Delta(\mathbf{S}), u_{\mathbf{A}}) + \delta_{\Delta(\mathbf{A})}$.
For $s \in S_i$, let $\mathbf{B}(s)$ be the valid tuple with $\overline{\mathbf{B}(s)} = \{s\}$, so that $u_{\mathbf{B}(s)} = e_s$.
Since $f(\mathbf{B}(s)) \geq 0$, the corresponding inequality of $\mathsf{Q}$ implies $x(s) \geq 0$.
By \Cref{lem:fragment_new_facet} and the equalities above, the remaining inequality is equivalent to $u_{\mathbf{A}^{\ast}} \cdot x \geq c$ for a constant $c$, and $c \leq h(\mathsf{F}, u_{\mathbf{A}^{\ast}}) = f(\mathbf{A}^{\ast})$ because $\mathsf{F}$ satisfies it.
Hence the corresponding inequality of $\mathsf{Q}$ implies it as well, and $\mathsf{Q} \subseteq \mathsf{F}$.
Therefore $\mathsf{F} = \mathsf{Q}$.

It remains to check that $(f,g)$ satisfies the inequalities from Equation \eqref{eq:supermodular_inequality}.
For a valid tuple $\mathbf{X}$ set $z(\mathbf{X}) \coloneqq |Z(\mathbf{A},\mathbf{X})|$ as in Equation \eqref{eq:disjoint_factors}, so that $f(\mathbf{X}) = \max(0, z(\mathbf{X}) - 1)$ and $g \equiv 1$; since every $A_i$ is nonempty we have $z(\mathbf{S}) = 0$, in accordance with the convention $f(\mathbf{S}) = 0$.
Write $\zeta_i(\mathbf{X}) = 1$ if $A_i \cap X_i = \emptyset$ and $\zeta_i(\mathbf{X}) = 0$ otherwise, so that $z = \sum_{i \in [k]} \zeta_i$.
Fix valid tuples $\mathbf{B}$ and $\mathbf{C}$ and set $t = |Z(\mathbf{B},\mathbf{C})|$.

For $i \notin Z(\mathbf{B},\mathbf{C})$ we have $(\mathbf{B} \wedge \mathbf{C})_i = B_i \cap C_i$ and $(\mathbf{B} \vee \mathbf{C})_i = B_i \cup C_i$ by \Cref{def:polysimplex_uncrossing}, so that
\[
    \zeta_i(\mathbf{B} \wedge \mathbf{C}) \geq \max\left(\zeta_i(\mathbf{B}), \zeta_i(\mathbf{C})\right),
    \qquad
    \zeta_i(\mathbf{B} \vee \mathbf{C}) = \min\left(\zeta_i(\mathbf{B}), \zeta_i(\mathbf{C})\right).
\]
For $i \in Z(\mathbf{B},\mathbf{C})$ we have $(\mathbf{B} \wedge \mathbf{C})_i = B_i \cup C_i$ and $(\mathbf{B} \vee \mathbf{C})_i = S_i$, so that $\zeta_i(\mathbf{B} \wedge \mathbf{C}) = \min(\zeta_i(\mathbf{B}), \zeta_i(\mathbf{C}))$ and $\zeta_i(\mathbf{B} \vee \mathbf{C}) = 0$.
Summing over $i \in [k]$ and using $\max + \min = \zeta_i(\mathbf{B}) + \zeta_i(\mathbf{C})$ together with $\min \geq \zeta_i(\mathbf{B}) + \zeta_i(\mathbf{C}) - 1$ yields
\begin{equation}\label{eq:fragment_supermodular_a}
    z(\mathbf{B} \wedge \mathbf{C}) + z(\mathbf{B} \vee \mathbf{C}) \geq z(\mathbf{B}) + z(\mathbf{C}) - t,
\end{equation}
while discarding the indices in $Z(\mathbf{B},\mathbf{C})$ in the first sum yields
\begin{equation}\label{eq:fragment_supermodular_b}
    z(\mathbf{B} \wedge \mathbf{C}) \geq \max\left(z(\mathbf{B}), z(\mathbf{C})\right) - t.
\end{equation}

Since $g $ is identically 1, the inequality \eqref{eq:supermodular_inequality} for the pair $\mathbf{B}, \mathbf{C}$ reads
\[
    \max(0, z(\mathbf{B}) - 1) + \max(0, z(\mathbf{C}) - 1)
    \leq
    \max(0, z(\mathbf{B} \wedge \mathbf{C}) - 1) + \max(0, z(\mathbf{B} \vee \mathbf{C}) - 1) + t .
\]
Suppose first that $z(\mathbf{B}) \geq 1$ and $z(\mathbf{C}) \geq 1$.
Using $\max(0, x - 1) \geq x - 1$ on the right-hand side and Equation \eqref{eq:fragment_supermodular_a}, the right-hand side is at least $z(\mathbf{B} \wedge \mathbf{C}) + z(\mathbf{B} \vee \mathbf{C}) + t - 2 \geq z(\mathbf{B}) + z(\mathbf{C}) - 2$, which is the left-hand side.
Otherwise one of the two values vanishes, say $z(\mathbf{B}) = 0$ by symmetry, and the left-hand side equals $\max(0, z(\mathbf{C}) - 1)$.
If $z(\mathbf{C}) = 0$ the inequality is immediate, and if $z(\mathbf{C}) \geq 1$ then Equation \eqref{eq:fragment_supermodular_b} gives
\[
    \max(0, z(\mathbf{B} \wedge \mathbf{C}) - 1) + t
    \geq
    z(\mathbf{B} \wedge \mathbf{C}) - 1 + t
    \geq
    z(\mathbf{C}) - 1 .
\]
In either case Equation \eqref{eq:supermodular_inequality} holds.

By \Cref{thm:deformation_cone} the pair $(f,g)$ lies in $\operatorname{Def}(\mathsf{P}(\mathbf{S}))$, and the associated polytope is $\mathsf{Q} = \mathsf{F}$.
Hence $\mathsf{F}$ is a deformation of $\mathsf{P}(\mathbf{S})$.
\end{proof}

We want to construct another polytopal basis for $\operatorname{PL}(\Psi(\mathbf{S}))$ using fragments.
However, there are several faces whose fragments coincide with the whole polytope; see \Cref{cor:multisimplex_drop_dimension}.
This case therefore requires more care.

By \Cref{lem:polysimplex_fragment_support}, the fragment of any tuple $\mathbf{A}$ with $\mathrm{slk}(\mathbf{A}) = 1$ coincides with the entire polytope $\Delta(\mathbf{S})$.
For the remaining tuples, we have the following result.

\begin{proposition}\label{prop:explicit_truncation_fragment_transform}
Let $\mathbf{C}\subseteq\mathbf{S}$ be a valid tuple with $L=\mathrm{Slk}(\mathbf{C})$ and $m=\mathrm{slk}(\mathbf{C})=|L|>1$.
For every subset $U\subseteq L$ with $|U|\geq 2$, let $\mathbf{A}^{U}(\mathbf{C})$ be the valid tuple with
\[
    A_i^{U}(\mathbf{C})=
    \begin{cases}
        S_i\setminus C_i, & i\in U,\\
        S_i, & i\notin U,
    \end{cases}.
\]
Then, as support functions on $\Psi(\mathbf{S})$,
\begin{equation}\label{eq:explicit_truncation_fragment_transform}
    h\!\left(\operatorname{DTr}\left[\Delta(\mathbf{S}),\Delta(\mathbf{C})\right]\right)
    =
    \sum_{\substack{U\subseteq L\\ |U|\geq 2}}
    (-1)^{|U|}\,
    h\!\left(\operatorname{Fr}\left[\Delta(\mathbf{S}),\Delta(\mathbf{A}^{U}(\mathbf{C}))\right]\right)
    +(2-m)\,h\!\left(\Delta(\mathbf{S})\right).
\end{equation}
In particular, every deep truncation with slack greater than one lies in the span of the fragments with slack greater than one together with the factor simplices $\{\Delta(S_i)\mid i\in[k]\}$.
\end{proposition}

\begin{proof}
We compare both sides in the $(f,g)$ coordinates of \Cref{def:functions_tuples}.
Fix a valid tuple $\mathbf{B}$ with associated ray $u_{\mathbf{B}}$.
By \Cref{lem:polysimplex_truncation_support},
\[
    h\!\left(\operatorname{DTr}\left[\Delta(\mathbf{S}),\Delta(\mathbf{C})\right],u_{\mathbf{B}}\right)
    =
    \begin{cases}
        1, & B_i\subseteq C_i \text{ for every } i,\\
        0, & \text{otherwise.}
    \end{cases}
\]
Since $C_i=S_i$ for $i\notin L$, this condition depends only on the indices in $L$, hence is equivalent to $Y(\mathbf{B})=L$, where
\[
    Y(\mathbf{B})\coloneqq\{i\in L\mid B_i\subseteq C_i\}.
\]
For $U\subseteq L$ with $|U|\geq 2$, \Cref{lem:polysimplex_fragment_support} gives
\[
    h\!\left(\operatorname{Fr}\left[\Delta(\mathbf{S}),\Delta(\mathbf{A}^{U}(\mathbf{C}))\right],u_{\mathbf{B}}\right)
    =
    \max\!\left(0,\,\left|\{i\in[k]\mid A_i^{U}(\mathbf{C})\cap B_i=\emptyset\}\right|-1\right)
    =
    \max\!\left(0,\,|U\cap Y(\mathbf{B})|-1\right),
\]
because for $i\notin U$ the set $A_i^{U}(\mathbf{C})=S_i$ meets the nonempty set $B_i$, while for $i\in U$ the condition $(S_i\setminus C_i)\cap B_i=\emptyset$ is equivalent to $B_i\subseteq C_i$.
Writing $Y=Y(\mathbf{B})$, the $f$-coordinates agree once we prove
\begin{equation}\label{eq:boolean_truncation_fragment_identity}
    \sum_{\substack{U\subseteq L\\ |U|\geq 2}}
    (-1)^{|U|}\max(0,|U\cap Y|-1)
    =
    \begin{cases}
        1, & Y=L,\\
        0, & Y\neq L,
    \end{cases}
\end{equation}
for every $Y\subseteq L$.
The terms with $|U|<2$ vanish, so we may sum over all $U\subseteq L$.
The left-hand side is equal to
\begin{equation}\label{eq:two_summands}
    \sum_{U\subseteq L}(-1)^{|U|}|U\cap Y|
    -
    \sum_{\substack{U\subseteq L\\ U\cap Y\neq\emptyset}}(-1)^{|U|}.
\end{equation}
The first summand of \eqref{eq:two_summands} vanishes, since
\[
    \sum_{U\subseteq L}(-1)^{|U|}|U\cap Y|
    =
    \sum_{i\in Y}\sum_{\substack{U\subseteq L\\ i\in U}}(-1)^{|U|}
    =
    -\sum_{i\in Y}\sum_{W\subseteq L\setminus\{i\}}(-1)^{|W|}
    =
    0,
\]
using $m=|L|>1$.
For the second summand of \eqref{eq:two_summands}, if $Y=\emptyset$ it is empty, hence zero.
If $Y\neq\emptyset$, then
\[
    \sum_{\substack{U\subseteq L\\ U\cap Y\neq\emptyset}}(-1)^{|U|}
    =
    \sum_{U\subseteq L}(-1)^{|U|}
    -
    \sum_{U\subseteq L\setminus Y}(-1)^{|U|}
    =
    -\sum_{U\subseteq L\setminus Y}(-1)^{|U|},
\]
which vanishes unless $Y=L$, in which case it equals $-1$.
This proves \eqref{eq:boolean_truncation_fragment_identity}, so the $f$-coordinates of \eqref{eq:explicit_truncation_fragment_transform} agree.

For the $g$-coordinates, \Cref{lem:polysimplex_truncation_support}, \Cref{lem:polysimplex_fragment_support}, and \Cref{ex:polysimplex_full_support} give $g(i)=1$ on every term.
The fragment coefficients sum to $\sum_{r=2}^{m}(-1)^{r}\binom{m}{r}=m-1$, so the right-hand side contributes $(m-1)+(2-m)=1$ in each $g(i)$, matching the left-hand side.
Hence the support functions agree.

Finally, under the embedding of \Cref{rem:embed} the product $\Delta(\mathbf{S})=\prod_i\Delta(S_i)$ is the Minkowski sum $\sum_i\Delta(S_i)$, so $h(\Delta(\mathbf{S}))=\sum_i h(\Delta(S_i))$ lies in the span of the factor simplices.
\end{proof}

\begin{proposition}\label{prop:basis_polysimplex_fragments}
A polytopal basis for the deformation space of $\mathsf{P}(\mathbf{S})$ is given by
\begin{equation*}
\mathscr{F} \cup \mathscr{C},
\qquad
\mathscr{F} = \left\{ \operatorname{Fr}\left[\Delta(\mathbf{S}), \Delta(\mathbf{A})\right] \,\middle|\, \mathrm{slk}(\mathbf{A}) > 1 \right\},
\qquad
\mathscr{C} = \left\{ \Delta(I) \,\middle|\, \emptyset \subsetneq I \subseteq S \in \mathbf{S} \right\}.
\end{equation*}
\end{proposition}

\begin{proof}
We start from the basis $\mathscr{B}$ of \Cref{prop:basis_polysimplex} and split its deep truncations by slack,
\[
	\mathscr{B} = \mathscr{D} \cup \mathscr{D}_1 \cup \mathscr{L},
    \qquad
		\mathscr{D} = \{\operatorname{DTr}[\Delta(\mathbf{S}),\Delta(\mathbf{A})]\mid\mathrm{slk}(\mathbf{A})>1\},
		\qquad
\mathscr{L} =  \{\Delta(S)\mid S\in\mathbf{S}\},
\]
where $\mathscr{D}_1$ collects the deep truncations of slack one.  By \Cref{lem:face_deformation_omni}, the elements of $\mathscr{C}$ belong to the deformation cone of $\mathsf{P}(\mathbf{S})$.  By \eqref{eq:deep_truncation_slk_1} and \Cref{rem:embed}, the sets $\mathscr{D}_1\cup \mathscr{L}$ and $\mathscr{C}$ span the same subspace and have the same cardinality.
By \Cref{prop:explicit_truncation_fragment_transform}, the set $\mathscr{F} \cup \mathscr{C}$ generates $\mathscr{D}$.
Since $\mathscr{F}\cup\mathscr{C}$ therefore generates the basis $\mathscr{D}\cup\mathscr{C}$ and has the same cardinality as $\mathscr{D}\cup\mathscr{C}$, it is itself a basis.
\end{proof}

\Cref{cor:multisimplex_drop_dimension} shows that fragments of non-full tuples are decomposable and explains how to factor them.
Parallel to \Cref{thm:basis_polysimplex_deep}, we obtain the following.

\begin{theorem}\label{thm:basis_polysimplex_fragments}
An indecomposable basis for the deformation space of $\mathsf{P}(\mathbf{S})$ is given by
\begin{equation}\label{eq:basis_polysimplex_fragment}
\left\{ \operatorname{Fr}\!\left[\Delta(\mathbf{T}_\mathbf{A}),\, \Delta(\mathbf{A}|_{\mathbf{T}_\mathbf{A}})\right] \,\middle|\, \mathrm{slk}(\mathbf{A})>1 \right\}
\cup
\left\{ \Delta(I) \,\middle|\,  \emptyset \subsetneq I \subseteq S \in \mathbf{S}  \right\}.
\end{equation}
\end{theorem}

\begin{proof}
Fragments of full tuples are indecomposable by \Cref{thm:small_truncated_indecomposable}.
The rest of the proof is analogous to \Cref{thm:basis_polysimplex_deep}.
\end{proof}

\begin{remark}\label{rmk:integralpolysimplexbasis}
One can show that the bases presented in \Cref{thm:basis_polysimplex_deep} and \Cref{thm:basis_polysimplex_fragments} are both integral bases.  For the flat-truncation basis, this follows by verifying that the results of \Cref{subsec:codim1} admit integral versions using \Cref{thm:basis_seed} together with the integrality of the factorization in \Cref{eq:factor_out}.  For the fragment basis, this follows by confirming that the inverse of the transformation in \Cref{prop:explicit_truncation_fragment_transform} is integral.
\end{remark}

\section{Zonotopes}
\label{sec:zonotopes}

In this section, we investigate our truncation bases for deformation cones of barycentric subdivisions of hyperplane arrangements (equivalently, omnitruncations of zonotopes).  Our main example is the 2-permutahedron.
\subsection{Truncations of zonotopes}

Let $\mathsf E$ be a collection of line segments in $\mathbb R^n$.
The \textbf{zonotope} $\mathsf{Z}(\mathsf E)$ is the Minkowski sum
\[
    \mathsf{Z}(\mathsf E)=\sum_{[u,v]\in\mathsf E}[u,v].
\]

\begin{example}\label{ex:hexagon}
In \Cref{fig:hexagon} we have a regular hexagon written as the Minkowski sum of three segments.
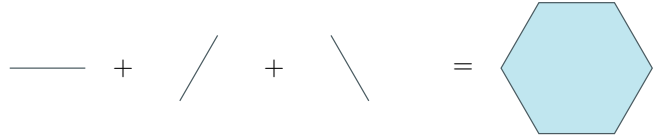
\begin{figure}[ht]
\centering
\begin{tikzpicture}
    \tikzset{
        paperedge/.style={color=Turquoise!25!black, line width=0.46pt},
        paperface/.style={fill=Turquoise!27, draw=none}
    }
    \begin{scope}
        \draw[paperedge] (180:0.5) -- (0:0.5);
    \end{scope}
    \begin{scope}[xshift=2cm]
        \draw[paperedge] (60:0.5) -- (240:0.5);
    \end{scope}
    \begin{scope}[xshift=4cm]
        \draw[paperedge] (120:0.5) -- (300:0.5);
    \end{scope}

    \node at (1,0) {$+$};
    \node at (3,0) {$+$};
    \node at (5.5,0) {$=$};

    \begin{scope}[xshift=7cm]
        \coordinate (hex0) at (0:1);
        \coordinate (hex1) at (60:1);
        \coordinate (hex2) at (120:1);
        \coordinate (hex3) at (180:1);
        \coordinate (hex4) at (240:1);
        \coordinate (hex5) at (300:1);
        \fill[paperface] (hex0) -- (hex1) -- (hex2) -- (hex3) --
            (hex4) -- (hex5) -- cycle;
        \draw[paperedge] (hex0) -- (hex1) -- (hex2) -- (hex3) --
            (hex4) -- (hex5) -- cycle;
    \end{scope}
\end{tikzpicture}
\caption{A zonotope from three segments.}\label{fig:hexagon}
\end{figure}
\end{example}

Recall that the normal fan of $\mathsf{Z}(\mathsf E)$ is induced by the hyperplane arrangement consisting of one linear hyperplane orthogonal to each generating segment.
Faces of zonotopes are zonotopes themselves.
We will need the following lemma about truncating a zonotope.

\begin{lemma}[Contraction]\label{lem:contraction}
    Let $\mathsf E$ be a set of distinct segments and let $\mathsf{F} = \mathsf{Z}(\mathsf W) + p \subseteq \mathsf{Z}(\mathsf E)$ be a face, where $\mathsf W \subseteq \mathsf E$.
    For $0<\epsilon\leq\delta_{\mathsf F}$,
    \begin{equation}\label{eq:lem_contraction}
        \operatorname{Tr}_\epsilon\left[\mathsf{Z}(\mathsf E),\, \mathsf{F}\right]
        = \operatorname{Tr}_\epsilon^{u_{\mathsf{F}}}\left[\mathsf{Z}(\mathsf E \setminus \mathsf W),\, p\right] + \mathsf{Z}(\mathsf W),
    \end{equation}
    where $u_\mathsf{F}$ is the inner face normal of $\mathsf{F}$ in $\mathsf{Z}(\mathsf E)$.
\end{lemma}

\begin{proof}
    We use the Minkowski decomposition $\mathsf{Z}(\mathsf E) = \mathsf{Z}(\mathsf E \setminus \mathsf W) + \mathsf{Z}(\mathsf W)$.
    By definition,
    \[
        \operatorname{Tr}_\epsilon\left[\mathsf{Z}(\mathsf E),\, \mathsf{F}\right]
        = \left\{ x \in \mathsf{Z}(\mathsf E) \,\middle|\, u_\mathsf{F} \cdot x \geq C + \epsilon \right\},
        \qquad C = \min_{x \in \mathsf{Z}(\mathsf E)} u_\mathsf{F} \cdot x.
    \]
    Since $\mathsf{F} = \mathsf{Z}(\mathsf W) + p$ is a face, $u_\mathsf{F}$ is constant on $\mathsf{Z}(\mathsf W)$; call this constant value $\lambda = u_\mathsf{F} \cdot z$ for any $z \in \mathsf{Z}(\mathsf W)$.
    Writing $x = y + z$ with $y \in \mathsf{Z}(\mathsf E \setminus \mathsf W)$ and $z \in \mathsf{Z}(\mathsf W)$,
    \begin{align*}
        \operatorname{Tr}_\epsilon\left[\mathsf{Z}(\mathsf E),\, \mathsf{F}\right]
        &= \left\{ y + z \in \mathsf{Z}(\mathsf E \setminus \mathsf W) + \mathsf{Z}(\mathsf W) \,\middle|\, u_\mathsf{F} \cdot y + \lambda \geq C + \epsilon \right\} \\
        &= \left\{ y \in \mathsf{Z}(\mathsf E \setminus \mathsf W) \,\middle|\, u_\mathsf{F} \cdot y \geq C' + \epsilon \right\} + \mathsf{Z}(\mathsf W) \\
        &= \operatorname{Tr}_\epsilon^{u_\mathsf{F}}\left[\mathsf{Z}(\mathsf E \setminus \mathsf W),\, p\right] + \mathsf{Z}(\mathsf W),
    \end{align*}
    where $C' = C - \lambda$.
\end{proof}

\begin{remark}
    Intuitively, \Cref{lem:contraction} asserts that one can factor out a face by \emph{contracting} it to a vertex: truncating the original face is equivalent to truncating the resulting vertex.
    See \Cref{fig:contraction} for an example.
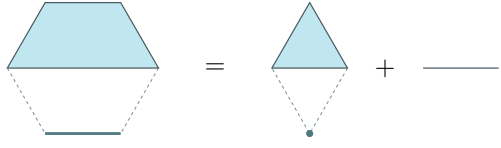
\begin{figure}[ht]
\centering
\begin{tikzpicture}
    \tikzset{
        paperedge/.style={color=Turquoise!25!black, line width=0.46pt},
        paperremovededge/.style={color=Turquoise!40!black!75,
            dash pattern=on 1.2pt off 1.2pt, line width=0.35pt},
        paperface/.style={fill=Turquoise!27, draw=none},
        paperdistinguishededge/.style={color=Turquoise!45!black,
            line width=1.05pt},
        paperdistinguishedvertex/.style={circle,
            draw=Turquoise!45!black, fill=Turquoise!45!black,
            inner sep=0.8pt}
    }
    \coordinate (conA) at (0:1);
    \coordinate (conB) at (60:1);
    \coordinate (conC) at (120:1);
    \coordinate (conD) at (180:1);
    \coordinate (conE) at (240:1);
    \coordinate (conF) at (300:1);

    \fill[paperface] (conA) -- (conB) -- (conC) -- (conD) -- cycle;
    \draw[paperedge] (conA) -- (conB) -- (conC) -- (conD) -- cycle;
    \draw[paperdistinguishededge] (conE) -- (conF);
    \draw[paperremovededge] (conD) -- (conE) (conA) -- (conF);

    \node at (1.75,0) {$=$};
    \node at (4,0) {$+$};

    \begin{scope}[xshift=2.5cm]
        \coordinate (localA) at (0:1);
        \coordinate (localB) at (60:1);
        \coordinate (localF) at (300:1);
        \coordinate (localO) at (0,0);
        \fill[paperface] (localA) -- (localB) -- (localO) -- cycle;
        \draw[paperedge] (localA) -- (localB) -- (localO) -- cycle;
        \draw[paperremovededge] (localA) -- (localF) -- (localO);
        \node[paperdistinguishedvertex] at (localF) {};
    \end{scope}
    \begin{scope}[xshift=5cm]
        \draw[paperedge] (180:0.5) -- (0:0.5);
    \end{scope}
\end{tikzpicture}
\caption{A deep truncation of a face in a zonotope illustrating \Cref{lem:contraction}.}
\label{fig:contraction}
\end{figure}
\end{remark}

In Equation \eqref{eq:lem_contraction} we specify the vector $ u_\mathsf{F} $ since the inner normal vector for the vertex $ p $ in the polytope $ \mathsf{Z} ( \mathsf E\setminus \mathsf W ) $ may be different from the original $ u_{\mathsf{F}} $ with respect to $ \mathsf{Z}(\mathsf E) $.
However, we do have that $ u_{\mathsf{F}} $ is in the interior of the normal cone of $ p $ in $ \mathsf{Z} ( \mathsf E\setminus \mathsf W ) $, since $ p$ uniquely minimizes the dot product 
with $ u_{\mathsf{F}} $ in $ \mathsf{Z} ( \mathsf E\setminus \mathsf W ) $.

\begin{example}\label{ex:contraction_normal}
Let $v_1=(1,0)$, $v_2=(-1,1)$, $v_3=(1,2)$, and let
$\mathsf E=\{[0,v_1],[0,v_2],[0,v_3]\}$.
The zonotope $\mathsf{Z}(\mathsf E)$ has six vertices $(0,0)$, $(1,0)$, $(2,2)$, $(1,3)$, $(0,3)$, and $(-1,1)$.
Consider the edge $\mathsf{F} = \mathsf{Z}(\{[0,v_1]\}) + (0,0)$, connecting $(0,0)$ to $(1,0)$.
Its inner normal in $\mathsf{Z}(\mathsf E)$ is $u_\mathsf{F} = (0,1)$.

After contracting $\mathsf W = \{[0,v_1]\}$, the contracted zonotope $\mathsf{Z}(\{[0,v_2],[0,v_3]\})$ is the parallelogram with vertices $(0,0)$, $(1,2)$, $(0,3)$, $(-1,1)$.
The normal cone of $p = (0,0)$ in $\mathsf{Z}(\{[0,v_2],[0,v_3]\})$ is the cone spanned by the rays $(1,1)$ and $(-2,1)$.
Its canonical inner normal is $(-1,2)$, which differs from $u_\mathsf{F} = (0,1)$.
See \Cref{fig:contraction_normal}.
\end{example}

\begin{figure}[ht]
\centering
\begin{tikzpicture}[scale=0.75]
    \tikzset{
        paperedge/.style={color=Turquoise!25!black, line width=0.46pt},
        paperface/.style={fill=Turquoise!27, draw=none},
        paperdistinguishededge/.style={color=Turquoise!45!black,
            line width=1.05pt},
        paperdistinguishedvertex/.style={circle,
            draw=Turquoise!45!black, fill=Turquoise!45!black,
            inner sep=0.8pt},
        paperarrow/.style={->, color=Turquoise!55!black,
            line width=0.8pt}
    }
    \begin{scope}
        \fill[paperface]
            (0,0)--(1,0)--(2,2)--(1,3)--(0,3)--(-1,1)--cycle;
        \draw[paperedge]
            (0,0)--(1,0)--(2,2)--(1,3)--(0,3)--(-1,1)--cycle;
        \draw[paperdistinguishededge] (0,0)--(1,0)
            node[midway, below=3pt, text=Turquoise!45!black]
            {\small$\mathsf{F}$};
        \draw[paperarrow] (0.5,0)--(0.5,1.2)
            node[above, text=Turquoise!55!black]
            {\small$u_\mathsf{F}$};
        \node[paperdistinguishedvertex] at (0,0) {};
        \node[above, font=\small] at (0.5,3.4) {$\mathsf{Z}(\mathsf E)$};
    \end{scope}

    \begin{scope}[xshift=7.5cm]
        \fill[paperface] (0,0)--(1,2)--(0,3)--(-1,1)--cycle;
        \draw[paperedge] (0,0)--(1,2)--(0,3)--(-1,1)--cycle;
        \draw[paperarrow] (0,0)--(-0.6,1.2)
            node[above right=2pt, yshift=-3pt,
                text=Turquoise!55!black] {\small$u_p$};
        \node[paperdistinguishedvertex] at (0,0) {};
        \node[above, font=\small] at (0,3.4)
            {$\mathsf{Z}(\mathsf E\setminus\{[0,v_1]\})$};
    \end{scope}
\end{tikzpicture}
\caption{The edge $\mathsf{F}$ has inner normal $u_\mathsf{F}=(0,1)$ in $\mathsf{Z}(\mathsf E)$.
After contracting $[0,v_1]$, the vertex $p$ has a two-dimensional normal cone spanned by $(1,1)$
and $(-2,1)$; both $u_\mathsf{F}=(0,1)$ and the canonical inner normal $(-1,2)$ lie inside it,
but only $u_\mathsf{F}$ gives the correct truncation in \Cref{lem:contraction}.}
\label{fig:contraction_normal}
\end{figure}
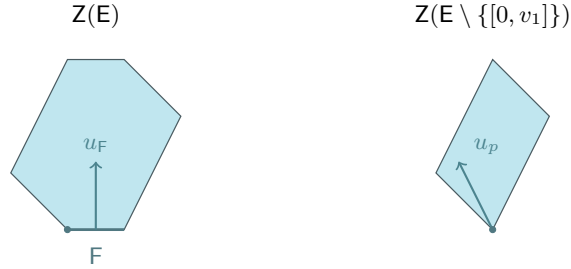

\begin{lemma}\label{lem:delzant_uF}
Let $\mathsf{Z}(\mathsf E) \subseteq \mathbb{R}^{n}$ be a primitive Delzant zonotope and let $\mathsf{F} = \mathsf{Z}(\mathsf W) + p \subseteq \mathsf{Z}(\mathsf E)$ be a face.
Then the deep truncation of $p$ in $\mathsf{Z}(\mathsf E \setminus \mathsf W)$ in direction $u_{\mathsf{F}}$ is flat.
\end{lemma}

\begin{proof}

Since $\mathsf{Z}(\mathsf E)$ is primitive Delzant, \Cref{prop:delzant} implies that the central deep truncation of $\mathsf{F}$ is flat.  By \Cref{lem:contraction}, the deep truncation of $\mathsf{F}$ in $\mathsf{Z}(\mathsf E)$ decomposes as the deep truncation of $p$ in $\mathsf{Z}(\mathsf E \setminus \mathsf W)$ in direction $u_\mathsf{F}$, plus the summand $\mathsf{Z}(\mathsf W)$.  It is straightforward to verify that for any edge $[p,q]$ in $\mathsf{Z}(\mathsf E \setminus \mathsf W)$, there exists some $v\in\operatorname{Vert}(\mathsf Z(\mathsf W))$ such that $[p+v,q+v]$ is an edge of $\mathsf{Z}(\mathsf E)$.  Since $u_{\mathsf{F}}$ is constant on $\mathsf{Z}(\mathsf W)$, this implies that the deep truncation of $p$ in $\mathsf{Z}(\mathsf E \setminus \mathsf W)$ in direction $u_\mathsf{F}$ is also flat.
\end{proof}

\begin{lemma}[Locality]\label{lem:locality}
Let $\mathsf E$ be a set of distinct segments and let $p$ be a vertex of $\mathsf Z(\mathsf E)$.
For each $s\in\mathsf E$, let $p(s)$ denote the endpoint of the segment $s$ selected by $p$, so that
\[
p=\sum_{s\in\mathsf E}p(s).
\]
Let $\mathsf U\subseteq\mathsf E$ be the set of segments whose directions are parallel to edges of $\mathsf Z(\mathsf E)$ incident to $p$.
Set
\[
p_{\mathsf U}\coloneqq\sum_{s\in\mathsf U}p(s).
\]
For $0<\epsilon\leq \delta_{p}$,
\begin{equation}\label{eq:locality}
\operatorname{Tr}_{\epsilon}\left[\mathsf Z(\mathsf E),p\right]
=
\operatorname{Tr}_{\epsilon}\left[\mathsf Z(\mathsf U),p_{\mathsf U}\right]
+
\mathsf Z(\mathsf E\setminus\mathsf U).
\end{equation}
\end{lemma}

\begin{proof}
For every vertex $v$ of $\mathsf Z(\mathsf E)$, there is a unique choice of an endpoint $v(s)$ of each segment $s\in\mathsf E$ such that $v=\sum_{s\in\mathsf E}v(s)$.

Write
\[
\mathsf P\coloneqq\mathsf Z(\mathsf E),
\qquad
\mathsf A\coloneqq\mathsf Z(\mathsf U),
\qquad
\mathsf B\coloneqq\mathsf Z(\mathsf E\setminus\mathsf U),
\]
and set
\[
p_{\mathsf B}\coloneqq\sum_{s\in\mathsf E\setminus\mathsf U}p(s),
\qquad
p=p_{\mathsf U}+p_{\mathsf B}.
\]

Since the edge directions incident to $ p $ and $ p_\mathsf{U} $ are the same, it follows that their corresponding normal cones are equal and so the canonical inner normals coincide.
Call it $ u $.

We first record the key consequence of the choice of $\mathsf U$.
Suppose that $q$ is a vertex of $\mathsf P$ satisfying $q(s)=p(s)$ for every $s\in\mathsf U$ but $ q \neq p $.
Let $v_1,\dots,v_m$ be the vertices of $\mathsf P$ adjacent to $p$, and set $w_i\coloneqq v_i-p \text{for }i\in[m]$.
There exist $\alpha_1,\dots,\alpha_m\geq 0$ such that $q-p=\sum_{i=1}^m\alpha_iw_i$.
Note that for each $ i \in [m] $ the point $ q+w_i = q+v_i-p $ belongs to $ \mathsf{P} $.
Since $q\neq p$, $A\coloneqq\sum_{i=1}^m\alpha_i>0,$ and
\begin{equation}\label{eq:q_combination}
q = \frac{1}{1+A}p + \sum_{i=1}^m\frac{\alpha_i}{1+A}(q+w_i).
\end{equation}
This expresses $q$ as a nontrivial convex combination of points of $\mathsf P$ distinct from $q$, contradicting that $q$ is a vertex.
Therefore $q=p$.
In other words, the restriction of the selector of $p$ to $\mathsf U$ uniquely determines the vertex $p$.

We now prove the inclusion from right to left in \eqref{eq:locality}.
Since $u$ lies in the relative interior of the normal cone of $p$ and $ p_U $, the point $p_{\mathsf B}$, satisfying $ p_{\mathsf{U}} + p_{\mathsf{B}} = p $, minimizes $u$ on $\mathsf B$.
Thus, for $a\in\operatorname{Tr}_{\epsilon}[\mathsf A,p_{\mathsf U}]$ and $ b\in\mathsf B,$ we have
\[
u^{\intercal}(a+b) \geq u^{\intercal}p_{\mathsf U}+\epsilon+u^{\intercal}p_{\mathsf B} = u^{\intercal}p+\epsilon.
\]
It follows that $\operatorname{Tr}_{\epsilon}[\mathsf A,p_{\mathsf U}]+\mathsf B \subseteq \operatorname{Tr}_{\epsilon}[\mathsf P,p]$, since all the other defining inequalities agree.

The edges incident to $p$ and $p_{\mathsf U}$ differ by translation, so the two truncations have the same deep depth. Hence, for $0<\epsilon<\delta_p$, both truncations are shallow.

\begin{description}
	\item[Old vertices] Let $q\neq p$ be a vertex of $\mathsf P$.
Set
\[
q_{\mathsf U}\coloneqq\sum_{s\in\mathsf U}q(s), \qquad q_{\mathsf B}\coloneqq\sum_{s\in\mathsf E\setminus\mathsf U}q(s).
\]
The point $q_{\mathsf U}$ is a vertex of $\mathsf A$.
Since the restriction of the selector determines the vertex, we have that $q_{\mathsf U}\neq p_{\mathsf U}$.
Therefore $q_{\mathsf U}$ is still a vertex of a shallow truncation of $\mathsf A$, and $q=q_{\mathsf U}+q_{\mathsf B} \in \operatorname{Tr}_{\epsilon}[\mathsf A,p_{\mathsf U}]+\mathsf B.$

\item[New vertices]
Let $\mathsf e$ be an edge of $\mathsf P$ incident to $p$, and let $\mathsf e_{\mathsf U}$ be the parallel edge of $\mathsf A$ incident to $p_{\mathsf U}$.
We have $\mathsf e=\mathsf e_{\mathsf U}+p_{\mathsf B}$.  In other words, they differ by translation by $ p_{\mathsf{B}} $.
If $r_{\mathsf e}$ and $r_{\mathsf e_{\mathsf U}}$ are the new vertices created on these edges by the respective truncations, then $r_{\mathsf e}=r_{\mathsf e_{\mathsf U}}+p_{\mathsf B},$ since $p=p_{\mathsf U}+p_{\mathsf B}$ and the two truncating hyperplanes differ by the same translation.
Therefore $r_{\mathsf e}\in \operatorname{Tr}_{\epsilon}[\mathsf A,p_{\mathsf U}]+\mathsf B.$
\end{description}

Thus every vertex of $\operatorname{Tr}_{\epsilon}[\mathsf P,p]$ belongs to the right-hand side of \eqref{eq:locality}.
The inclusion follows.

The equality at $\epsilon=\delta_p$ follows by continuity.
\end{proof}

\begin{figure}[ht]
\centering
\begin{tikzpicture}
    \tikzset{
        paperedge/.style={color=Turquoise!25!black, line width=0.46pt},
        paperremovededge/.style={color=Turquoise!40!black!75,
            dash pattern=on 1.2pt off 1.2pt, line width=0.35pt},
        paperface/.style={fill=Turquoise!27, draw=none},
        paperdistinguishededge/.style={color=Turquoise!45!black,
            line width=1.05pt}
    }
    \coordinate (locA) at (0:1);
    \coordinate (locB) at (60:1);
    \coordinate (locC) at (120:1);
    \coordinate (locD) at (180:1);
    \coordinate (locE) at (240:1);
    \coordinate (locF) at (300:1);
    \coordinate (locG) at (30:{sqrt(3)/2});
    \coordinate (locH) at (330:{sqrt(3)/2});

    \fill[paperface]
        (locF)--(locH)--(locG)--(locB)--(locC)--(locD)--(locE)--cycle;
    \draw[paperedge]
        (locF)--(locH)--(locG)--(locB)--(locC)--(locD)--(locE)--cycle;
    \draw[paperdistinguishededge] (locH)--(locG);
    \draw[paperremovededge] (locG)--(locA)--(locH);

    \node at (1.75,0) {$=$};
    \node at (4,0) {$+$};

    \begin{scope}[xshift=2.5cm]
        \coordinate (fragA) at (0:1);
        \coordinate (fragB) at (60:1);
        \coordinate (fragF) at (300:1);
        \coordinate (fragO) at (0,0);
        \coordinate (fragG) at (30:{sqrt(3)/2});
        \coordinate (fragH) at (330:{sqrt(3)/2});
        \fill[paperface]
            (fragG)--(fragB)--(fragO)--(fragF)--(fragH)--cycle;
        \draw[paperedge]
            (fragG)--(fragB)--(fragO)--(fragF)--(fragH)--cycle;
        \draw[paperdistinguishededge] (fragH)--(fragG);
        \draw[paperremovededge] (fragG)--(fragA)--(fragH);
    \end{scope}
    \begin{scope}[xshift=5cm]
        \draw[paperedge] (180:0.5) -- (0:0.5);
    \end{scope}
\end{tikzpicture}
\caption{A shallow truncation of a vertex in a zonotope illustrating \Cref{lem:locality}.}
\label{fig:locality}
\end{figure}
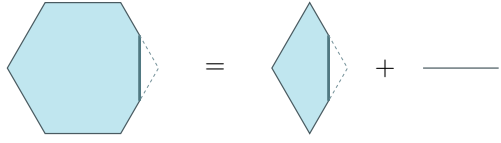

The proof of \Cref{lem:locality} applies verbatim when both truncations are taken in any fixed direction $u$ in the relative interior of their common normal cone. Taking $u=u_{\mathsf F}$ and combining this with \Cref{lem:contraction}, we obtain the following.

\begin{lemma}\label{lem:locality_contraction}
Let $\mathsf{F} = \mathsf{Z}(\mathsf W) + p$ be a face of $\mathsf{Z}(\mathsf E)$, where $\mathsf W \subseteq \mathsf E$ and $p$ is the corresponding vertex of $\mathsf{Z}(\mathsf E \setminus \mathsf W)$.
Let $\mathsf U \subseteq \mathsf E \setminus \mathsf W$ be the set of segments whose directions are parallel to edges of $\mathsf{Z}(\mathsf E \setminus \mathsf W)$ incident to $p$, and set $p_{\mathsf U} \coloneqq \sum_{s \in \mathsf U} p(s)$.
For $0 < \epsilon \leq \delta_\mathsf{F}$,
\begin{align*}
    \operatorname{Tr}_\epsilon\left[\mathsf{Z}(\mathsf E),\, \mathsf{Z}(\mathsf W) + p\right]
    &= \operatorname{Tr}_\epsilon^{u_\mathsf{F}}\left[\mathsf{Z}(\mathsf E \setminus \mathsf W),\, p\right] + \mathsf{Z}(\mathsf W) \\
    &= \operatorname{Tr}_\epsilon^{u_\mathsf{F}}\left[\mathsf{Z}(\mathsf U),\, p_{\mathsf U}\right] + \mathsf{Z}(\mathsf E \setminus \mathsf U),
\end{align*}
where $u_\mathsf{F}$ is the inner face normal of $\mathsf{F}$ in $\mathsf{Z}(\mathsf E)$.
In particular, taking $\epsilon = \delta_\mathsf{F}$,
\[
    \operatorname{DTr}\left[\mathsf{Z}(\mathsf E),\, \mathsf{F}\right]
    = \operatorname{Tr}^{u_\mathsf{F}}_{\delta_{\mathsf{F}}} \left[\mathsf{Z}(\mathsf U),\, p_{\mathsf U}\right] + \mathsf{Z}(\mathsf E \setminus \mathsf U).
\]
\end{lemma}

In what follows we identify $p$ and $p_\mathsf{U}$ when the context makes $\mathsf{U}$ clear.

\subsection{A graphical criterion for indecomposability}\label{sec:criteria}

We start this section with an example of an indecomposable polytope that is not detected by McMullen's criteria \cite{mcmullen1987indecomposable}.

\begin{example}\label{ex:rhombic_dodecahedron}
Consider the Minkowski sum of segments
\[
    \mathsf{Z} = [\mathbf{0}, (1,1,1)] + [\mathbf{0}, (1,-1,1)] + [\mathbf{0}, (-1,1,1)] + [\mathbf{0}, (-1,-1,1)].
\]
This polytope is the rhombic dodecahedron, which has $12$ faces, all of which are parallelograms.
The deep truncation $\operatorname{DTr}[\mathsf{Z}, \mathbf{0}]$ is indecomposable even though it has only four indecomposable faces, and they are not strongly connected to one another.
See \Cref{fig:schlegel} for reference.
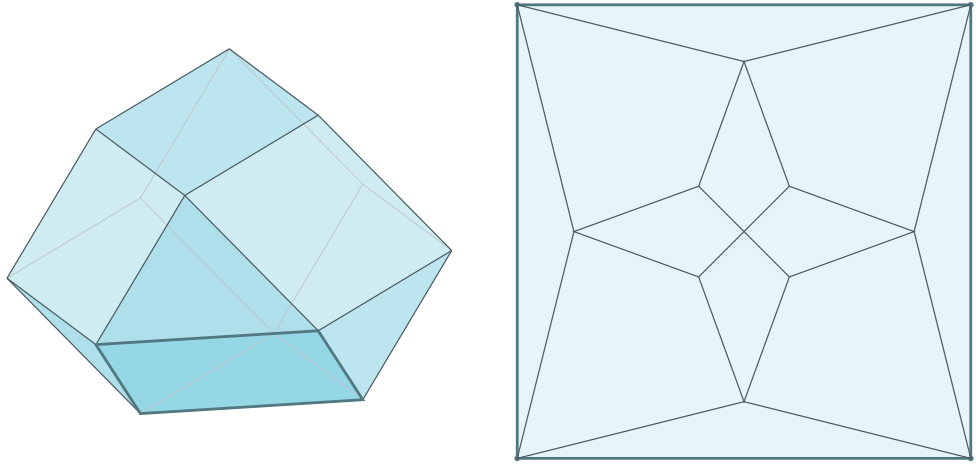
\begin{figure}
    \centering
    \begin{minipage}[c]{0.4\linewidth}
        {\hfuzz=20pt\input{tikz/Img_poly.tex}}
    \end{minipage}
    \begin{minipage}[c]{0.4\linewidth}
        {\hfuzz=20pt\input{tikz/schlegel.tex}}
    \end{minipage}
    \caption{A vertex truncation of the rhombic dodecahedron and its planar graph.}
    \label{fig:schlegel}
\end{figure}
\end{example}

We now present our tool to prove indecomposability, using a different parametrization of the deformation cone.

\begin{definition}\label{def:1_mink_weights}
The set of all deformations of $\mathsf{P}$, up to translation, is parametrized by non-negative \textbf{1-Minkowski weights} \cite[Lemma 8.1]{mcmullen1996weights}.
This is the polyhedral cone
\begin{equation}\label{eq:1_mink_weights}
    \Omega_1(\mathsf{P}) =
    \left\{
        \omega \in \mathbb{R}_{\geq 0}^{\mathcal{F}_1(\mathsf{P})} \,\middle|\,
\sum_{\mathsf E\subseteq\mathsf F}\omega(\mathsf E)\,\vec v_{\mathsf E,\mathsf F}=\mathbf{0},
        \text{ for every 2-dimensional face } \mathsf{F}
    \right\},
\end{equation}
where $\mathcal F_1(\mathsf P)$ is the set of edges of $\mathsf P$, and, for each two-face $\mathsf F$, the vector $\vec v_{\mathsf E,\mathsf F}$ is the edge vector of $\mathsf E$, oriented according to a cyclic orientation of the boundary of $\mathsf F$.  Each equality in \eqref{eq:1_mink_weights} is called a \textbf{balancing condition}.
\end{definition}

We are now ready to present another indecomposability criterion.
This was independently discovered by Padrol--Poullot \cite{padrol2026indecomposability} and is a special case of their Theorem A.

\begin{proposition}[Edge criterion]\label{prop:edge_criterion}
Let $\mathsf{P}$ be a polytope and $\mathcal{F}_1(\mathsf{P})$ its set of edges.
Form a graph $\mathdutch{G}$ with vertex set $\mathcal{F}_1(\mathsf{P})$ by connecting $\mathsf{E}_1, \mathsf{E}_2 \in \mathcal{F}_1(\mathsf{P})$ if either
\begin{enumerate}
    \item they belong to a common triangular 2-face, or
    \item they are opposite edges of a parallelogram 2-face.
\end{enumerate}
If $\mathdutch{G}$ is connected, then $\dim \Omega_1(\mathsf{P}) = 1$, so $\mathsf{P}$ is indecomposable.
\end{proposition}

\begin{proof}
In a triangular face, the balancing condition uniquely determines the weights of any two edges from the third.
In a parallelogram face, opposite edges must carry the same weight.
If $\mathdutch{G}$ is connected, a single edge weight determines all others, so $\Omega_1(\mathsf{P})$ is one-dimensional and $\mathsf{P}$ is indecomposable.
\end{proof}

\begin{example}\label{ex:rhombic_dodecahedron_edge_criterion}
The flat truncation of the rhombic dodecahedron from \Cref{ex:rhombic_dodecahedron} is indecomposable by the above criterion.
The edges not on the square $\square = \operatorname{Conv}\{(1,1,1),(1,-1,1),(-1,-1,1),(-1,1,1)\}$ fall into four parallel classes, each connected internally via parallelograms.
The four triangles adjacent to $\square$ then connect all classes to one another (see \Cref{fig:graph}).
Note that the quadrilaterals in the figure are parallelograms in the rhombic dodecahedron, even if they do not appear so in the two-dimensional picture.
\begin{figure}
    \centering
    \input{tikz/graph.tex}
    \caption{The graph $\mathdutch{G}$ induced by parallelogram and triangle adjacency.}
    \label{fig:graph}
\end{figure}
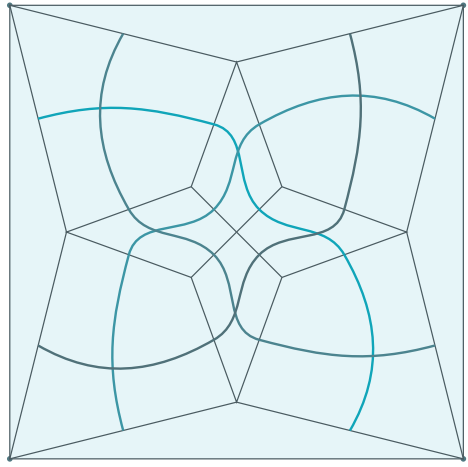
\end{example}

\subsection{Faberg\'e and Gherkin polytopes}

We now describe a general construction that yields indecomposable polytopes.

\begin{definition}\label{def:gherkin}
	Given a $ n $-polytope $ \mathsf{P} \subseteq \mathbb{R}^{n} $ we define the \textbf{Faberg\'e egg} polytope as 
	\begin{equation}\label{eq:faberge}
		\mathrm{F}(\mathsf{P}) \coloneqq \sum_{v\in\operatorname{Vert}(\mathsf P)}	[\mathbf{0}, (v, 1)] \subseteq \mathbb{R}^{n+1}.
	\end{equation}

	We define the \textbf{Gherkin} polytope as
	\begin{equation}\label{eq:gherkin}
		\mathsf{G}(\mathsf{P}) \coloneqq \mathrm{F}(\mathsf{P}) \cap \mathsf{H}, \text{ where } \mathsf{H} \coloneqq  \left\{ x \in \mathbb{R}^{n+1} \,\middle|\, x_{n+1} \geq 1 \right\}  .
	\end{equation}
\end{definition}

The deep truncation of the Faberg\'e polytope $\mathrm{F}(\mathsf{P})$ at the origin in direction $e_{n+1}$ is a flat truncation which is equal to the Gherkin polytope $\mathsf{G}(\mathsf{P})$.
The polytope $\mathsf{P}$ appears as the facet $\mathsf{G}(\mathsf{P}) \cap \left\{ x \in \mathbb{R}^{n+1} \,\middle|\, x_{n+1} = 1 \right\} $  of $ \mathsf{G}(\mathsf{P}) $.  We refer to the origin $\mathbf{0}$ as the \textbf{apex} of $\mathrm{F}(\mathsf{P})$, and we refer to this truncation as the \textbf{apex truncation}.

\begin{theorem}\label{thm:gherkin}
For every $ n $-polytope $ \mathsf{P} $ the Gherkin polytope $ \mathsf{G}(\mathsf{P}) $ is indecomposable.
\end{theorem}

\begin{proof}
	Let $ \mathsf{Z} = \mathrm{F}(\mathsf{P}) $.
	Since no three vertices of $ \mathsf{P}  $ are collinear, no three of the segments in Equation \eqref{eq:faberge} are coplanar.
	This implies that every 2-face of $ \mathsf{Z} $ is a sum of exactly two segments and thus a parallelogram.
	After intersecting with the closed halfspace $ \mathsf{H} $, each 2-face $\mathsf{F}$ of $\mathsf{Z}$ contributes to $\mathsf{G}(\mathsf{P})$ according to the following:
\begin{itemize}
		\item it is a parallelogram if $ \mathsf{F} $ does not contain the origin,
		\item it is a triangle if $ \mathsf{F} $ contains the origin.
	\end{itemize}  
	
We split the set of edges of the Gherkin $  \mathsf{G}(\mathsf{P})  $ into two:
\begin{enumerate}
  \item Type I: Edges parallel to some $  [0, (v,1) ]  $.
  \item Type II: Edges lying in the hyperplane $ \partial \mathsf{H} $.
\end{enumerate}

	Every Type II edge is one side of a triangular 2-face of $\mathsf{G}(\mathsf{P})$ whose other two sides are of Type I.  We now apply \Cref{prop:edge_criterion} by showing that the associated graph $\mathdutch{G}$ on the edge set of $\mathsf{G}(\mathsf{P})$ is connected.  The Type I edges partition into parallel classes $\mathscr C_v$, one for each vertex $v$ of $\mathsf{P}$, where $\mathscr{C}_v$ consists of the edges parallel to $(v,1)$. 
We show in turn that each $\mathscr{C}_v$ is connected in $\mathdutch{G}$, that any two classes $\mathscr{C}_v, \mathscr{C}_w$ are connected in $\mathdutch{G}$, and that every Type II edge is connected to some $\mathscr{C}_v$.

	The 2-faces of $\mathsf{Z}$ having $[\mathbf{0},(v,1)]$ as a Minkowski factor form the \textbf{zone} of $v$, and they are all parallelograms. 
	Together with the edges parallel to $(v,1)$, they form a subcomplex of $\partial \mathsf{Z}$ on which any two edges of $\mathscr{C}_v$ are connected by a chain of adjacent parallelograms. 
	A parallelogram in this zone is cut by $\mathsf{H}$ only when it has $\mathbf{0}$ as a vertex, so after truncation each surviving parallelogram still links two consecutive edges of $\mathscr{C}_v$ via rule~(2) of \Cref{prop:edge_criterion}, and each cut parallelogram produces a triangle that links the two surviving neighbors of the deleted edge via rule~(1). 
	Hence $\mathscr{C}_v$ is connected in $\mathdutch{G}$.

	Now we show that different classes are connected.
	If $vw$ is an edge of $\mathsf{P}$, then there is a linear functional on $\mathbb{R}^{n+1}$ that attains its minimum on $\mathsf{Z}$ precisely along the parallelogram $[\mathbf{0},(v,1)] + [\mathbf{0},(w,1)]$, which is therefore a 2-face of $\mathsf{Z}$ with the origin as a vertex. 
	Its image in $\mathsf{G}(\mathsf{P})$ is a triangle with one edge in $\mathscr{C}_v$, one in $\mathscr{C}_w$, and one of Type II. 
	Rule~(1) of \Cref{prop:edge_criterion} then links $\mathscr{C}_v$ to $\mathscr{C}_w$.
	Since the edge graph of $\mathsf{P}$ is connected, the classes $\{\mathscr C_v\}$ for $v$ ranging over the vertices of $\mathsf P$ are pairwise connected in $\mathdutch{G}$.

	Finally, every Type II edge sits in a triangle with two Type I edges, so it is adjacent in $\mathdutch{G}$ to some $\mathscr{C}_v$. 
	The graph $\mathdutch{G}$ is therefore connected, and by \Cref{prop:edge_criterion} the Gherkin polytope $\mathsf{G}(\mathsf{P})$ is indecomposable.
\end{proof}

A slight generalization of Gherkin polytopes is the following:
Let $ k \in \mathbb{Z}_{>0} $ and define
\begin{equation}\label{eq:generalize_gherkin}
	\mathsf{G}_k(\mathsf{P}) \coloneqq \mathrm{F}(\mathsf{P}) \cap \mathsf{H}, \text{ where } \mathsf{H} \coloneqq  \left\{ x \in \mathbb{R}^{n+1} \,\middle|\, x_{n+1} \geq k \right\}  .
\end{equation}

\Cref{thm:gherkin} only holds for $ \mathsf{G}_1 $.
Indeed, one can check that with $k=2$ in \Cref{ex:rhombic_dodecahedron} the resulting graph has two connected components.  Moreover, the corresponding polytope is decomposable.

\begin{figure}[p]
    \centering
    \includegraphics[width=.86\textwidth]{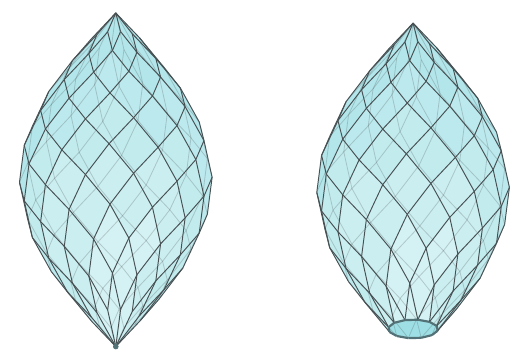}
    \caption{The Faberg\'e egg associated to a regular dodecagon (left) and its corresponding Gherkin (right).}
    \vspace{8mm}
    \includegraphics[width=.86\textwidth]{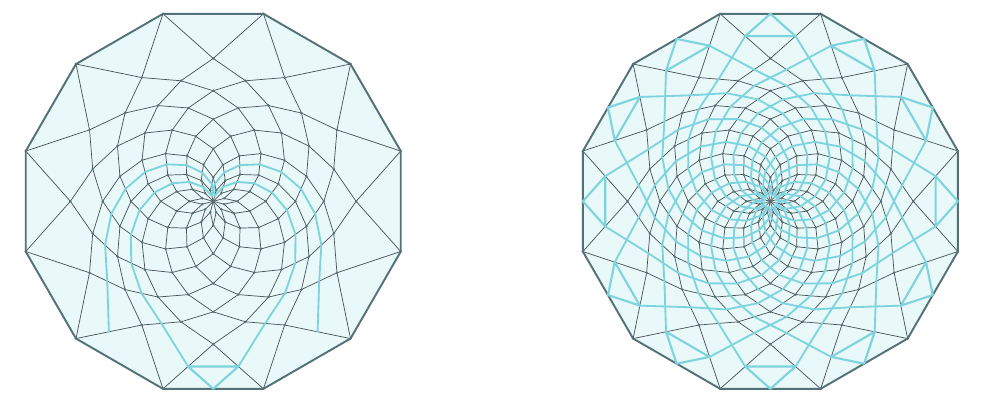}
    \caption{Planar drawings of the dodecagonal Gherkin.  In the left panel, the two paths beginning at the Type~I sides of one boundary triangle are traced, and the small triangle records the pairwise adjacencies of its three sides in $\mathdutch{G}$.  The right panel traces all twelve paths and records the same adjacencies in every boundary triangle.}
    \label{fig:dodecagon-gherkin-planar-paths}
\end{figure}
\FloatBarrier

\subsection{Simple zonotopes}

In this section we study the piecewise linear functions supported on the barycentric subdivision of fans induced by hyperplane arrangements.
Applying \Cref{thm:basis_seed} with seed $\mathsf{S} = \mathsf{Z}(\mathsf E)$, a simple zonotope (where we assume $\mathsf E$ spans the ambient space so that $\mathsf{Z}(\mathsf E)$ is full-dimensional, and $\mathsf{Z}(\mathsf E)$ contains the origin in its interior), we obtain that
\begin{equation}\label{eq:truncation_basis}
    \left\{ \operatorname{Tr}_\epsilon[\mathsf{Z}, \mathsf{F}] \,\middle|\, \emptyset \subsetneq \mathsf{F} \subsetneq \mathsf{Z} \right\}
\end{equation}
is a polytopal basis for $\operatorname{PL}(\Psi)$, where $\Psi = \operatorname{Bar}(\Sigma(\mathsf{Z}))$.

Using \eqref{eq:truncation_basis} as a starting point, we describe an indecomposable basis.
\Cref{lem:locality_contraction} decomposes each face truncation into several pieces, each of which may be further decomposable.
From now on we assume that our seed zonotope is primitive Delzant.

The reason for restricting to primitive Delzant zonotopes is the following corollary to \Cref{lem:locality_contraction}.

\begin{corollary}\label{cor:indecomposable_factor_zonotopes}
Let $\mathsf{Z}(\mathsf E)$ be a primitive Delzant zonotope and $\mathsf{F} = \mathsf{Z}(\mathsf W) + p$ a face.
With $\mathsf U$ and $\mathsf W$ as in \Cref{lem:locality_contraction}, we have the Minkowski decomposition
\begin{equation}\label{eq:zonotope_strong_truncation}
    \operatorname{DTr}\left[\mathsf{Z}(\mathsf E),\, \mathsf{F}\right] = \operatorname{Tr}^{u_{\mathsf F}}_{\delta_{\mathsf F}}\left[\mathsf{Z}(\mathsf U),\, p\right] + \mathsf{Z}(\mathsf E \setminus \mathsf U),
\end{equation}
where $\operatorname{Tr}^{u_{\mathsf F}}_{\delta_{\mathsf F}}\left[\mathsf{Z}(\mathsf U),\, p\right]$ is an indecomposable Gherkin polytope.
\end{corollary}

\begin{proof}
By \Cref{lem:locality_contraction},
\[
    \operatorname{DTr}\left[\mathsf{Z}(\mathsf E),\, \mathsf{Z}(\mathsf W) + p\right] = \operatorname{Tr}^{u_{\mathsf F}}_{\delta_{\mathsf F}}\left[\mathsf{Z}(\mathsf U),\, p\right] + \mathsf{Z}(\mathsf E \setminus \mathsf U).
\]
By \Cref{lem:delzant_uF,lem:locality}, the deep truncation $\operatorname{Tr}^{u_{\mathsf F}}_{\delta_{\mathsf F}}\left[\mathsf{Z}(\mathsf U),\, p\right]$ is flat.  After translating $p$ to the origin and rescaling $u_{\mathsf F}$, the other endpoints of the generating segments lie in the hyperplane $u_{\mathsf F}=1$. Thus $\operatorname{Tr}^{u_{\mathsf F}}_{\delta_{\mathsf F}}\left[\mathsf{Z}(\mathsf U),\, p\right]$ is affinely isomorphic to the corresponding Gherkin, and is indecomposable by \Cref{thm:gherkin}.
\end{proof}

In light of the previous result, we will adopt the notation
\[
 \mathsf{G}(\mathsf{Z}, \mathsf{F}) \coloneqq \operatorname{Tr}^{u_{\mathsf F}}_{\delta_{\mathsf F}}\left[\mathsf{Z}(\mathsf U),\, p\right].
 \]

\begin{theorem}\label{thm:primitive_delzant_zonotope}
Let $\mathsf{Z}(\mathsf E)$ be a primitive Delzant zonotope with $\mathsf E$ integral.  Let $\mathsf{P}$ be an omnitruncation of $\mathsf{Z}$, and $\mathscr{T}$ be an indecomposable polytopal basis for the deformation space of $\mathsf{Z}$.
Then
\[
    \mathscr{G} \coloneqq \left\{ \mathsf{G}(\mathsf{Z}, \mathsf{F}) \,\middle|\, \emptyset \subsetneq \mathsf{F} \subsetneq \mathsf{Z},\, \operatorname{codim}(\mathsf{F}) > 1 \right\} \cup \mathscr{T}
\]
is an indecomposable basis for the deformation space of $\mathsf{P}$.  Moreover, if $\mathscr{T}$ is an integral basis for the deformation space of $\mathsf{Z}$, then $\mathscr{G}$ is an integral basis for the deformation space of $\mathsf{P}$.
\end{theorem}

\begin{proof}
By \Cref{prop:delzant}, the deep truncations below are the
unit-depth truncations.  By \Cref{rem:key_delzant} and the proof of
\Cref{cor:basis_augmented_barycentric}, the set
\[
    \mathscr{C} = \left\{ \operatorname{DTr}[\mathsf{Z}, \mathsf{F}] \,\middle|\, \emptyset \subsetneq \mathsf{F} \subsetneq \mathsf{Z},\, \operatorname{codim}(\mathsf{F}) > 1 \right\} \cup \mathscr{T}
\]
is a polytopal basis.
Write $\mathsf{Z}=\mathsf{Z}(\mathsf E)$. By \eqref{eq:zonotope_strong_truncation}, the zonotope factor $\mathsf{Z}(\mathsf E \setminus \mathsf U)$ in the decomposition of $\operatorname{DTr}[\mathsf{Z}, \mathsf{F}]$ is a deformation of $\mathsf{Z}$, so its support function is generated by elements of $\mathscr{T}$.
Hence the elements of $\mathscr{G}$ generate those of $\mathscr{C}$, and since $\mathscr{G}$ and $\mathscr{C}$ have the same cardinality, $\mathscr{G}$ is also a basis.  For the integrality statement, first observe that $\mathsf{Z}(\mathsf E \setminus \mathsf U)$ is integral.  If $\mathscr{T}$ is an integral basis, it integrally generates $\mathsf{Z}(\mathsf E \setminus \mathsf U)$, hence passing to the set of Gherkins is a unimodular change of basis.
\end{proof}

The following figure illustrates why we cannot omit $\mathscr{T}$ in \Cref{thm:primitive_delzant_zonotope}, and instead include the facet Gherkins.

\begin{figure}[ht]
    \centering
    \includegraphics[width=.55\textwidth]{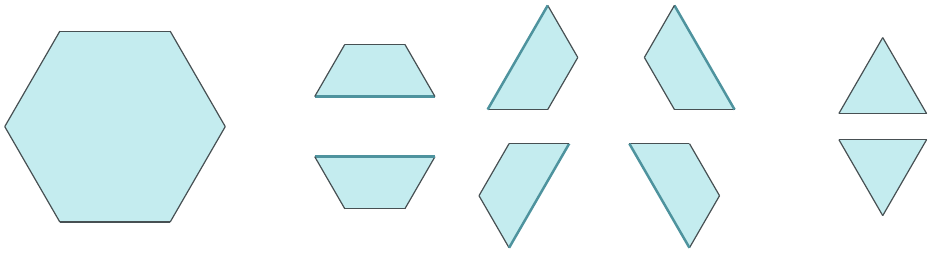}
    \caption{The hexagonal permutahedron (left), its six deep facet truncations (center), and the two translation classes of their Gherkin factors (right).  Each triangular Gherkin occurs three times up to translation.  Thus, although the permutahedron together with its deep facet truncations forms a basis, replacing those truncations by their Gherkin factors leaves only the triangle and its negative, which do not span the deformation space of the hexagon.}
    \label{fig:hexagonal-permutahedron-facet-gherkins}
\end{figure}
\begin{remark}\label{rmk:cubeisboth}
The standard hypercube is both a primitive Delzant zonotope and a product of standard simplices.  The two bases we produce in \Cref{thm:basis_polysimplex_deep,thm:primitive_delzant_zonotope} agree in this case.
\end{remark}

\subsection{The 2-permutahedron}

Previous works of Gaiffi \cite{gaiffi2015permutonestohedra} and Castillo--Liu \cite{castillo2022deformation} investigated polytopes whose normal fan is the barycentric subdivision of the braid arrangement, which they called the permutopermutohedron and nested permutahedron, respectively.  These are normally equivalent to an omnitruncation of the standard permutahedron.  We propose defining a $2$-permutahedron to be an $S_n$-invariant polytope whose normal fan is the $2$-braid fan.  In this section we apply our techniques to describe an indecomposable polytopal basis for the deformation cones of such polytopes.  This was one of the primary motivations for this work.

Recall that the standard permutahedron is a primitive Delzant zonotope that admits the following combinatorial description:
\begin{equation}\label{eq:permutohedron}
    \Pi_n \coloneqq \sum_{1 \leq i < j \leq n} [e_i, e_j]
    = \operatorname{Conv}\left\{ (\pi(1), \dots, \pi(n)) \,\middle|\, \pi \colon [n] \rightarrow \left\{ 0, 1, \dots, n-1 \right\} \text{ bijective}\right\} \subseteq \mathbb{R}^n.
\end{equation}
Nonempty faces are in bijection with ordered set partitions of $[n]$, i.e., tuples $\mathcal{S} = (S_1, \dots, S_k)$ of pairwise disjoint nonempty subsets of $[n]$ whose union is $[n]$.
For each ordered set partition $\mathcal{S}$ there is a corresponding face $\mathsf{F}(\mathcal{S}) \subseteq \Pi_n$.

We define, for every ordered set partition $\mathcal{S}$ with at least two blocks, the indecomposable polytope
\begin{equation}\label{eq:nested_pieces}
    \mathsf{R}(\mathcal{S}) \coloneqq \mathsf{G}(\Pi_n, \mathsf{F}(\mathcal{S})).
\end{equation}

We now describe these polytopes more explicitly.
Let $\mathcal{S} = (S_1, \dots, S_k)$ with $k\geq2$ be an ordered set partition of $[n]$, and let $\mathscr{G}$ be the directed graph with vertex set $[n]$ and a directed edge from $a$ to $b$ whenever $a \in S_i$ and $b \in S_{i+1}$ for some $i$.
The vectors $\{e_b - e_a \mid (a,b) \text{ edge of } \mathscr{G}\}$ are the edge directions connecting a vertex of $\mathsf{F}(\mathcal{S})$ to a vertex outside the face.
The Faberg\'e zonotope is
\begin{equation}
    \mathsf{Z}(\mathcal{S}) = \sum_{(a,b) \text{ edge of } \mathscr{G}} [\mathbf{0},\, e_b - e_a].
\end{equation}
In \Cref{subsec:pemplatesrootpoly} we relate this construction to other objects in the literature. 
Observe that
\[
    \mathsf{R}(\mathcal{S}) =\operatorname{Tr}^{u_{\mathsf F(\mathcal S)}}_{1}\!\left[\mathsf Z(\mathcal S),\mathbf 0\right].
\]

The normal fan of $\Pi_n$ is the braid fan $\mathcal{B}^1_n$ (see \Cref{ex:braid_one}).
We let $\Pi^2_n$ be a 2-permutahedron, and denote its normal fan by $\mathcal{B}^2_n$; the latter is the barycentric subdivision of $\mathcal{B}^1_n$ (see \Cref{ex:k_braid}).  It follows from \Cref{cor:indecomposable_factor_zonotopes} that the polytopes $\mathsf{R}(\mathcal{S})$, for $\mathcal{S}$ an ordered set partition of $[n]$, are deformations of $\Pi^2_n$.

\begin{figure}[p]
    \centering
    \captionsetup{font=small}
    \includegraphics[width=.70\textwidth]{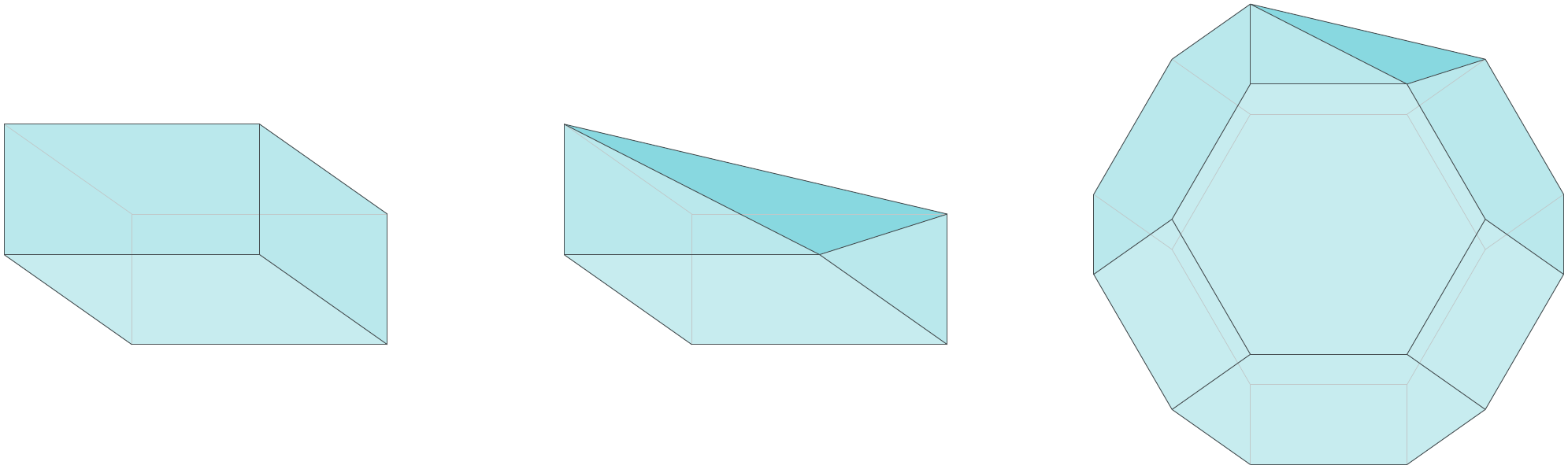}\\[.9em]
    \includegraphics[width=.70\textwidth]{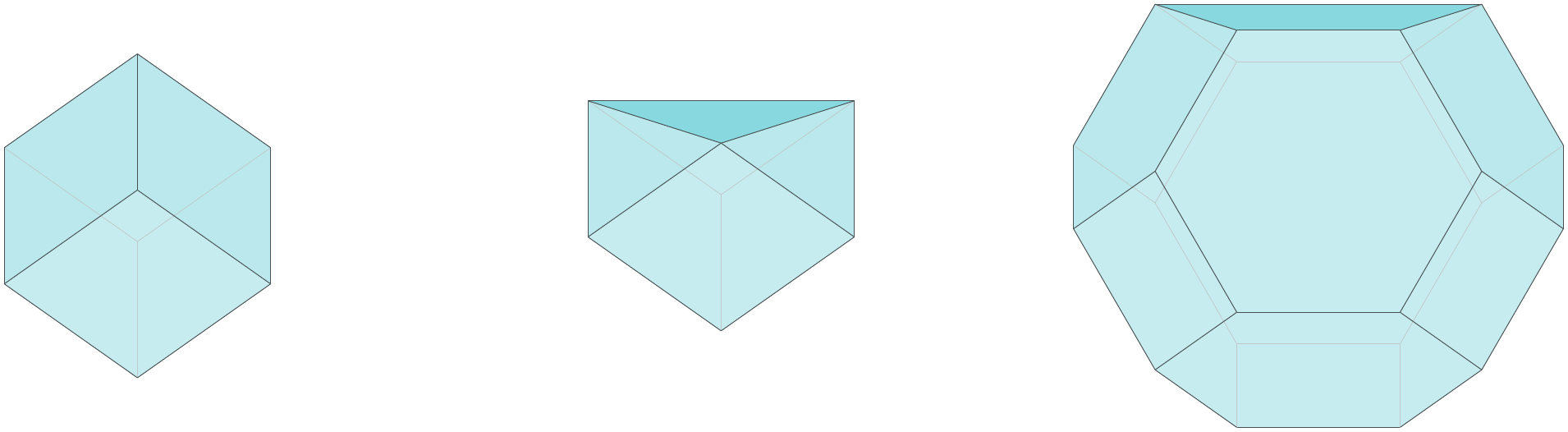}\\[.9em]
    \includegraphics[width=.70\textwidth]{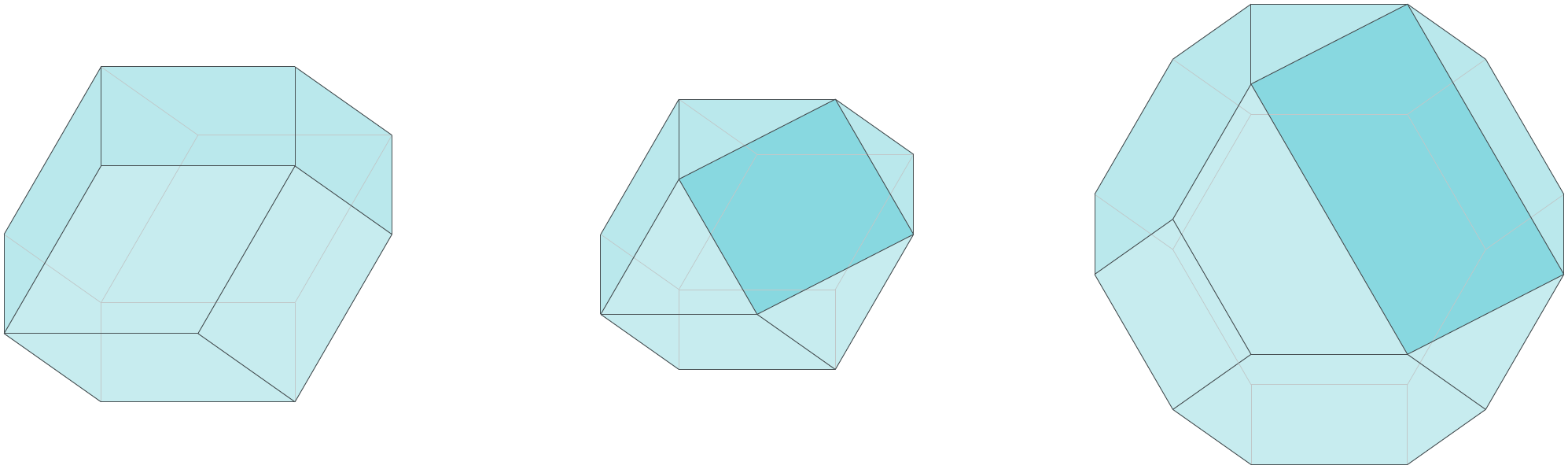}\\[.9em]
    \includegraphics[width=.70\textwidth]{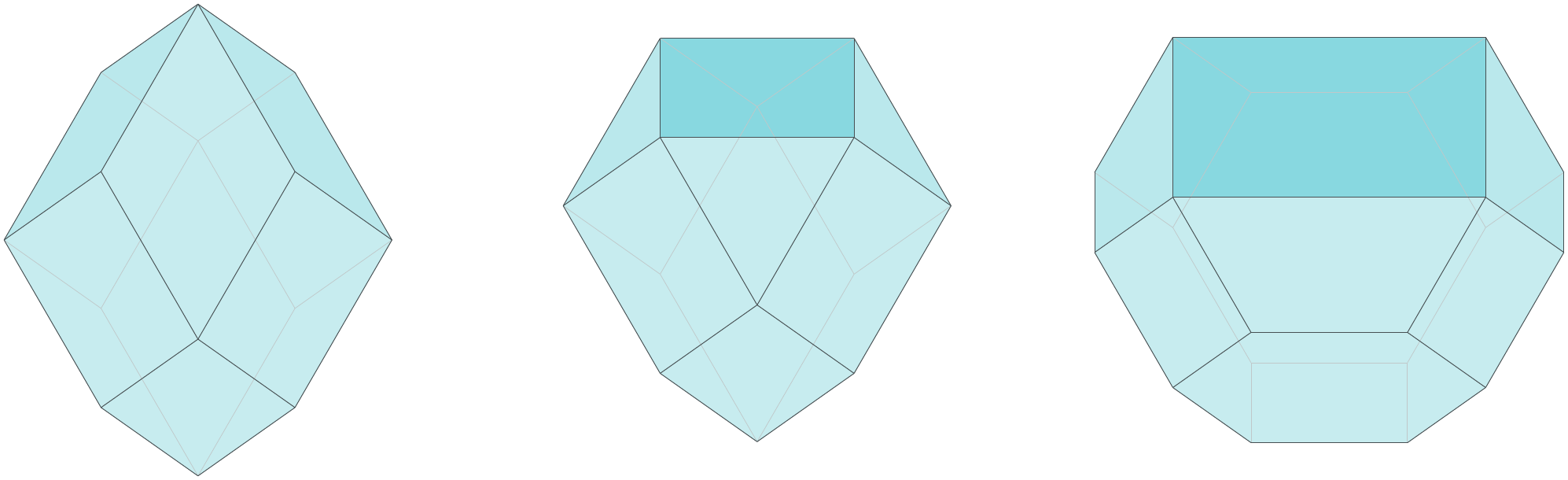}\\[.05em]
    \includegraphics[width=.70\textwidth]{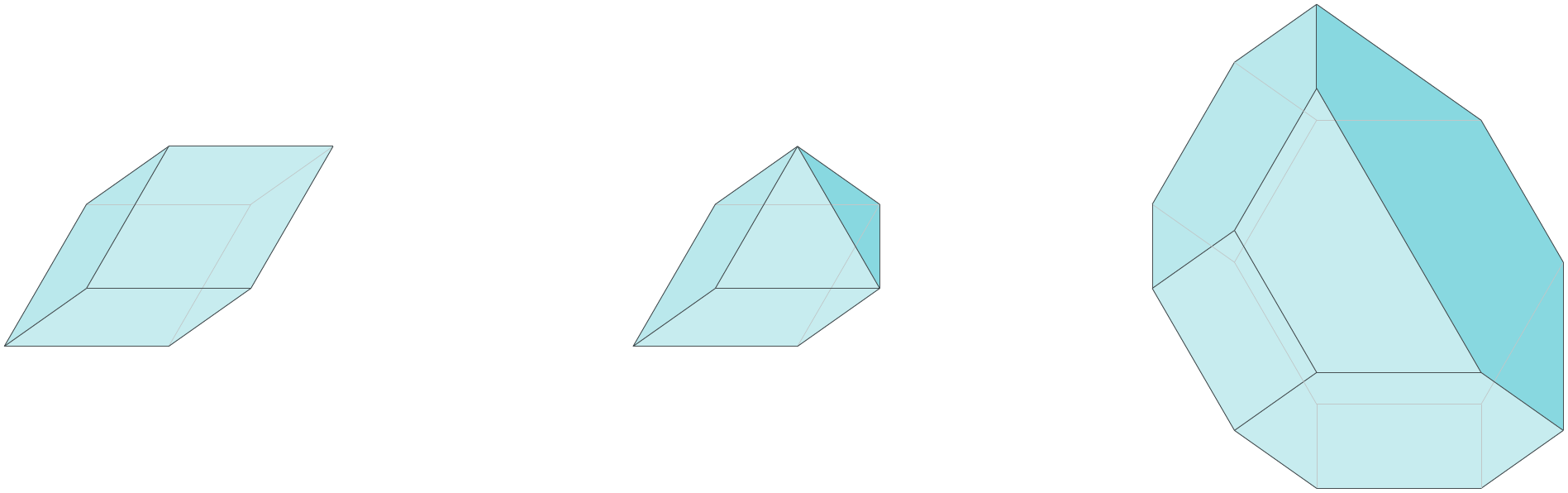}
    \caption{For each row, the first column shows the Faberg\'e egg, the second column shows the associated Gherkin, and the third column shows the corresponding truncated permutahedron.  The rows correspond respectively to truncating a vertex, an edge, the other type of edge, a square facet, and a hexagonal facet of the permutahedron.}
    \label{fig:permutahedron-gherkin-face-types}
\end{figure}

\begin{figure}[htbp]
    \centering
    \begin{minipage}{.25\textwidth}
        \centering
        \includegraphics[width=\linewidth]{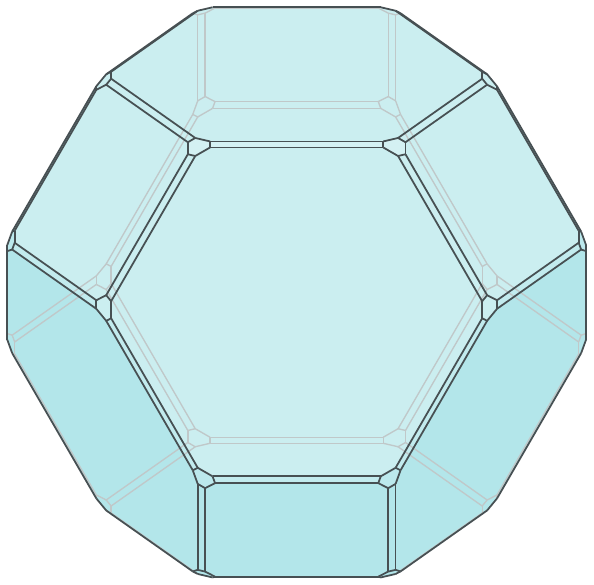}
    \end{minipage}
    \hfill
    \begin{minipage}{.25\textwidth}
        \centering
        \includegraphics[width=\linewidth]{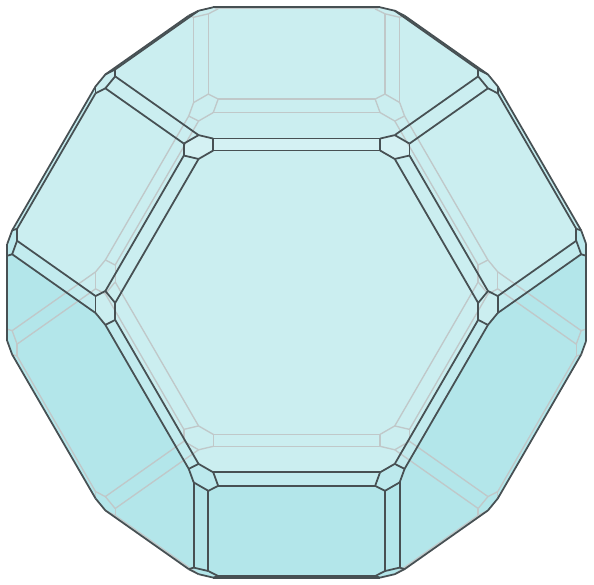}
    \end{minipage}
    \hfill
    \begin{minipage}{.25\textwidth}
        \centering
        \includegraphics[width=\linewidth]{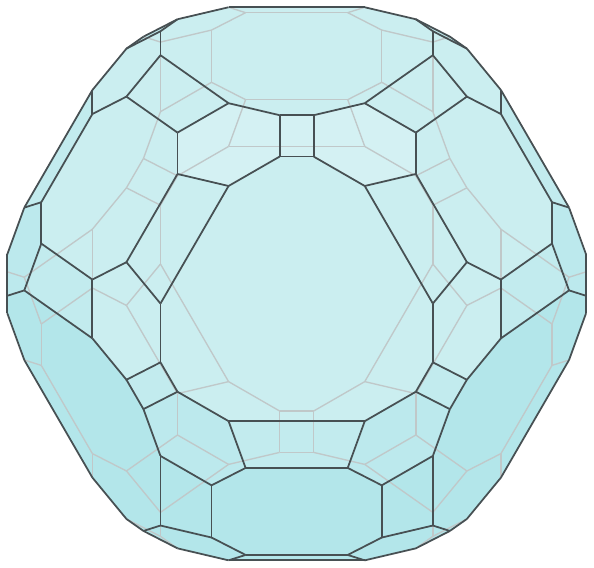}
    \end{minipage}
    \caption{The Minkowski sum of all Gherkins, the Minkowski sum of only the edge Gherkins (see \Cref{rmk:sumofridgetruncs}), and the Castillo--Liu realization of a 2-permutahedron.  All three polytopes are normally equivalent.}
    \label{fig:permutahedron-gherkin-sums}
    \label{fig:castillo-liu-nested-permutahedron-permutohedron-section}
\end{figure}

Let $P_{\mathrm{CL}}$ denote the Castillo--Liu realization of the $2$-permutahedron shown in \Cref{fig:castillo-liu-nested-permutahedron-permutohedron-section}.
For $I\subseteq [4]$, write $\Delta_I=\operatorname{Conv}\{e_i\mid i\in I\}$.
The corresponding signed Minkowski decomposition in the Gherkin basis is the following support-function identity:
\begin{equation}\label{eq:castillo-liu-nested-permutahedron-gherkin-decomposition}
\begin{aligned}
h_{P_{\mathrm{CL}}}
={}&
\sum_{\substack{A\subset[4],\, |A|=2\\ [4]\setminus A=\{i,j\}}}
\Big(
h_{\mathsf{R}(A|i|j)}
+h_{\mathsf{R}(A|j|i)}
+h_{\mathsf{R}(i|A|j)}
+h_{\mathsf{R}(j|A|i)}
+h_{\mathsf{R}(i|j|A)}
+h_{\mathsf{R}(j|i|A)}
\Big) \\
&+
31\sum_{i\in[4]}h_{\Delta_i}
-13\sum_{\substack{I\subset[4]\\ |I|=2}}h_{\Delta_I}
-3\sum_{\substack{I\subset[4]\\ |I|=3}}h_{\Delta_I}
+6h_{\Delta_{[4]}}.
\end{aligned}
\end{equation}
Here the first summation ranges over unordered two-element subsets $A\subset[4]$; if $[4]\setminus A=\{i,j\}$, then both orders of the complementary singleton blocks are included explicitly.
Equivalently, this realization is obtained from the sum of the $36$ edge-type Gherkins together with the simplicial correction appearing in \eqref{eq:castillo-liu-nested-permutahedron-gherkin-decomposition}.

To apply \Cref{thm:primitive_delzant_zonotope} to the permutohedron we use the simplicial basis (see \Cref{ex:postnikov_basis}) for its deformation space.
This gives the following result. 

\begin{theorem}\label{thm:basis_nested}
The set
\[
	\left\{ \mathsf{R}(\mathcal{S}) \,\middle|\, \mathcal{S} \text{ an ordered set partition of $[n]$ of size at least three } \right\}
    \cup
    \left\{ \Delta_I \,\middle|\, \emptyset \subsetneq I \subseteq [n] \right\}
\]
is an indecomposable polytopal basis for $\operatorname{PL}(\mathcal{B}^2_n)$.
\end{theorem}

The ordered set partitions of size two correspond precisely to facets, which are not included in the basis of \Cref{thm:basis_nested}.

\begin{example}\label{ex:braid_fragment}
In contrast to the case of the product of simplices, the fragments of the permutohedron are \emph{not} always deformations of $\Pi^2_n$.   Consider the vertex $(0,1,2,3,4) \in \Pi_5$.
Its fragment has vertex set given by the columns of
\begin{equation}\label{eq:fragment_counterexample}
    \operatorname{ConvexHull}
    \begin{bmatrix}
        0 & 1 & 0 & 0 & 0 \\
        1 & 0 & 2 & 1 & 1 \\
        2 & 2 & 1 & 3 & 2 \\
        3 & 3 & 3 & 2 & 4 \\
        4 & 4 & 4 & 4 & 3
    \end{bmatrix}.
\end{equation}
Full-dimensional cones of the normal fan of $\Pi^2_n$ are described in \cite[Section 4]{castillo2022deformation}.
In particular, the following is a full-dimensional normal cone:
\begin{equation}
    \sigma = \left\{ x \in \mathbb{R}^5 \,\middle|\,
        \begin{array}{c}
            x_2 \geq x_4 \geq x_1 \geq x_5 \geq x_3 \\
            x_2 - x_4 \geq x_4 - x_1 \geq x_1 - x_5 \geq x_5 - x_3
        \end{array}
    \right\}.
\end{equation}
The vector $(3,10,0,6,1) \in \sigma$ attains its minimum over the vertices of \eqref{eq:fragment_counterexample} at the second column, whereas 
$(7,18,0,12,3) \in \sigma$ attains its minimum at the fourth column.
Hence $\sigma$ is not contained in any single normal cone of the fragment, so the fragment is not a deformation of $\Pi^2_5$.
\end{example}

\begin{remark}\label{rmk:graphicalzonotopes}
Graphical zonotopes associated to simple graphs have primitive edges, and if they are simple, they are smooth. Padrol--Pilaud--Poullot \cite{padrol2025deformed} proved that certain collections of standard simplices form a basis for the deformation cones of graphical zonotopes.  Thus, one might hope to use their work to construct bases for deformation cones of omnitruncations of primitive Delzant graphical zonotopes.  However, it was shown by Postnikov--Reiner--Williams \cite[Proposition~5.2]{postnikov2008faces} that graphical zonotopes are simple if and only if the blocks in their block decompositions are complete graphs.  This implies that any simple graphical zonotope associated to a simple graph is unimodularly equivalent to a product of permutahedra, where one has a simplicial basis in each factor.   Here our construction applies and gives a basis for the deformation cone of the omnitruncation of a product of permutahedra.

\end{remark}

\section{Connections and applications}\label{applicationssection}

The main motivation of the present paper was to provide a systematic framework for Minkowski decompositions of polytopes whose normal fans arise from barycentric subdivisions of other fans.
In this section we illustrate the broad scope of this approach via connections and applications to other works in the literature.

\subsection{Archimedean solids}\label{subsec:archimedean}
The Archimedean solids are a collection of 13 different 3-dimensional polytopes  which are vertex-transitive and whose facets are regular polygons.  They naturally extend the Platonic solids by allowing distinct facets.  It can be shown that 11 of the 13 Archimedean solids (all but the snub dodecahedron and the snub cube) live in the deformation cones of the $A_3$, $B_3$ and $H_3$ arrangements.  These arrangements are the normal fans of omnitruncations of the standard simplex, the cube, and the dodecahedron respectively.  Therefore, these Platonic and Archimedean solids can be expressed in our truncation bases.

In this section we focus on the case of $H_3$, which is the most novel aspect of our contribution to this topic.  We observe that in this case, the deep truncations are all flat, despite the fact that the dodecahedron is not smooth (although it is simple).  We calculate that, in addition to the flat truncations, the corresponding fragments live in the deformation cone of the $H_3$-Coxeter arrangement, and also form a basis.  Furthermore, by subtracting the distinguished faces of these fragments, we identify an indecomposable basis of pyramids.  This is like a fragment version of the factoring out of faces in our Gherkin construction.  We then provide some explicit decompositions of the $H_3$-Archimedean solids in our bases.

\begin{figure}[ht]
    \centering
    \includegraphics[width=.95\textwidth]{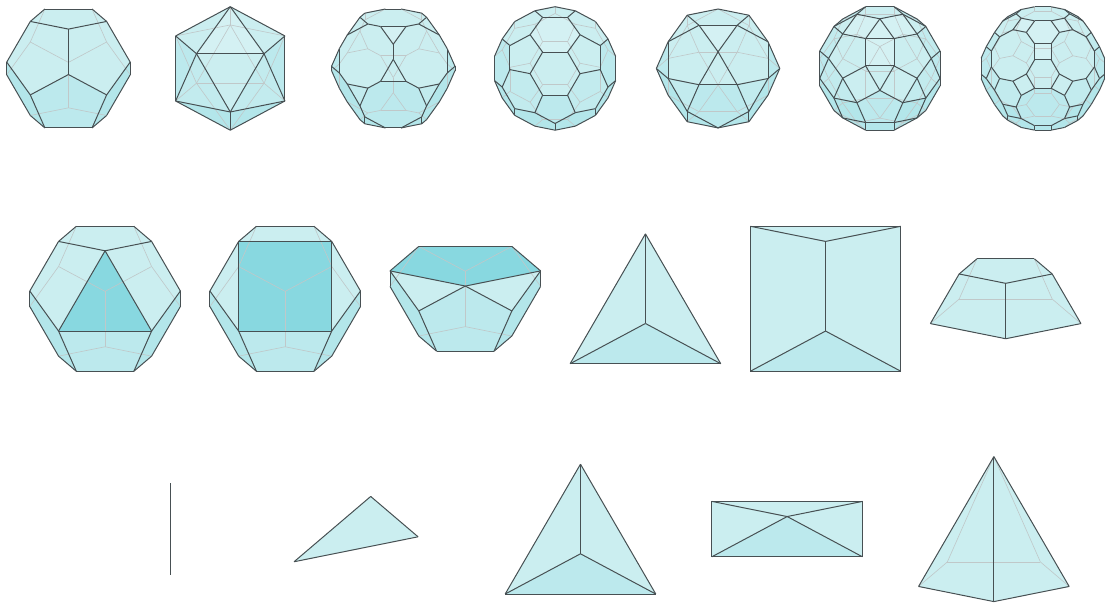}
    \caption{First row from left to right: dodecahedron, icosahedron, truncated dodecahedron, truncated icosahedron, icosidodecahedron, rhombicosidodecahedron, truncated icosidodecahedron.  Second row: representative vertex, edge, and facet truncations of the dodecahedron, and the corresponding fragments.  Third row: representative polytopes from the extremal rays of the nef cones of the fragments.  The $H_3$-orbits of the three on the right form an indecomposable polytopal basis for the deformation cone of the $H_3$ Coxeter arrangement. }
    \label{fig:dodecahedron-truncated-icosidodecahedron}
\end{figure}

The $H_3$ Coxeter zonotope is the truncated icosidodecahedron which, despite its name, is not obtained as a truncation of the icosidodecahedron, which is itself obtained by truncating the dodecahedron (in a non-shallow way).  The truncated icosidodecahedron is, however, an omnitruncation of the dodecahedron.  Interestingly, this omnitruncation cannot be realized as a positive Minkowski sum of shallow truncations.  We give here its decomposition as a signed sum of elements of our deep truncation basis.  Let $\mathsf P$ be the dodecahedron with presentation

\[
\mathsf P =
\left\{
(x,y,z)\in\mathbb R^3 \,\middle|\,
\begin{array}{rcl}
\pm x \pm \varphi z &\geq& -5\varphi,\\
\pm \varphi x \pm y &\geq& -5\varphi,\\
\pm \varphi y \pm z &\geq& -5\varphi.
\end{array}
\right\},
\]

where $\varphi$ is the golden ratio.  
Let $\mathsf Q$ be the Archimedean truncated icosidodecahedron whose
decagonal facet hyperplanes are the facet hyperplanes of $\mathsf P$.  Normalize the facet normals of
$\mathsf{P}$ so that $h_{\mathsf{P}}(u_{\mathsf{F}})=-1$ for every
facet $\mathsf{F}$ of $\mathsf{P}$.  With this normalization, the deep
truncation parameter is
\[
\delta = 3-\sqrt{5}
\]
for every proper face of $\mathsf{P}$.

The polytope $\mathsf{Q}$ is obtained from $\mathsf{P}$ by first
truncating all vertices at depth $3\delta/5$, and then truncating all
original edges at depth $\delta/5$.  In the deep truncation basis, we have that
\[
h_{\mathsf{Q}}
=
h_{\mathsf{P}}
+
\frac{1}{5}
\sum_{\mathsf{E}\in \operatorname{Edges}(\mathsf{P})}
\left(
h_{\operatorname{DTr}[\mathsf{P},\mathsf{E}]}
-
h_{\mathsf{P}}
\right).
\]
Since $\mathsf{P}$ has $30$ edges, this expands to the signed
Minkowski identity
\[
h_{\mathsf{Q}}
=
\frac{1}{5}
\sum_{\mathsf{E}\in \operatorname{Edges}(\mathsf{P})}
h_{\operatorname{DTr}[\mathsf{P},\mathsf{E}]}
-
5h_{\mathsf{P}}.
\]

This shows that, although $\mathsf{Q}$ is an omnitruncation of $\mathsf{P}$, its expression in the deep truncation
basis has a negative coefficient for $\mathsf{P}$.

Let $\varphi=(1+\sqrt{5})/2$, and choose the numbering of the
fundamental weights of $H_3$ so that
\[
\mathsf D=P_{H_3}(\omega_1),\qquad
\mathsf M=P_{H_3}(\omega_2),\qquad
\mathsf I=P_{H_3}(\omega_3)
\]
are the dodecahedron, icosidodecahedron, and icosahedron,
respectively.  We use the common Wythoff normalization in which their
edges have length $\sqrt{2}$.  Minkowski-linearity of weight polytopes
gives
\[
\mathsf{tD}=\mathsf D+\mathsf M,\qquad
\mathsf{tI}=\mathsf M+\mathsf I,
\]
and
\[
\mathsf R=\mathsf D+\mathsf I,\qquad
\mathsf Q=\mathsf D+\mathsf M+\mathsf I;
\]
see \cite[Proposition~6.4, Remark~8.6, and
Figure~5]{ardila2020coxeter}.  It therefore suffices to calculate the
decompositions of $\mathsf D$, $\mathsf M$, and $\mathsf I$.

For $d\in\{0,1,2\}$, define the orbit sums
\[
\begin{aligned}
\mathsf T_d
&=
\sum_{\dim\mathsf F=d}
h_{\operatorname{DTr}[\mathsf D,\mathsf F]},\\
\mathsf F_d
&=
\sum_{\dim\mathsf F=d}
h_{\operatorname{Fr}[\mathsf D,\mathsf F]},\\
\mathsf K_d
&=
\sum_{\dim\mathsf F=d}h_{\Pi_{\mathsf F}},
\end{aligned}
\]
where $\Pi_{\mathsf F}$ denotes the indecomposable pyramid determined, up to translation, by
$\operatorname{Fr}[\mathsf D,\mathsf F]=\mathsf F+\Pi_{\mathsf F}.$  For $\mathsf X\in\{\mathsf T,\mathsf F,\mathsf K\}$, write
\[
[a,b,c]_{\mathsf X}
=a\mathsf X_0+b\mathsf X_1+c\mathsf X_2.
\]
The decompositions in the deep-truncation, fragment, and indecomposable
pyramid bases are as follows.
\begin{center}
\small
\renewcommand{\arraystretch}{1.6}
\setlength{\tabcolsep}{4pt}
\begin{tabular}{c|c|c|c}
 & Deep truncations & Fragments & Indecomposable pyramids \\ \hline
$h_{\mathsf D}$
& $\left[0,0,\dfrac{5-\varphi}{38}\right]_{\mathsf T}$
& $\left[0,\dfrac{3\varphi-2}{11},\dfrac{6-9\varphi}{22}\right]_{\mathsf F}$
& $\left[0,\dfrac{\varphi-1}{2},\dfrac{2-3\varphi}{4}\right]_{\mathsf K}$ \\
$h_{\mathsf M}$
& $\left[\varphi-1,0,\dfrac{9(6-5\varphi)}{19}\right]_{\mathsf T}$
& $\left[1-\varphi,\dfrac{10\varphi-14}{11},\dfrac{10-4\varphi}{11}\right]_{\mathsf F}$
& $\left[1-\varphi,0,\varphi\right]_{\mathsf K}$ \\
$h_{\mathsf I}$
& $\left[1-\varphi,\varphi-1,\dfrac{37-34\varphi}{38}\right]_{\mathsf T}$
& $\left[\varphi-1,\dfrac{4-6\varphi}{11},\dfrac{7\varphi-1}{22}\right]_{\mathsf F}$
& $\left[\varphi-1,-\dfrac12,\dfrac{3-\varphi}{4}\right]_{\mathsf K}$
\end{tabular}
\end{center}
The decompositions of the other four Archimedean solids of type $H_3$ follow
immediately from the preceding Minkowski identities.

\begin{remark}\label{rmk:regularpoly}
If $\mathsf{S}$ is a regular polytope with symmetry group $W$, then the rays of $\operatorname{Bar}(\Sigma(\mathsf{S}))$ are the rays fixed by the stabilizers of the faces of $\mathsf{S}$, so $\operatorname{Bar}(\Sigma(\mathsf{S}))$ is the Coxeter fan of $W$ \cite{humphreys1990reflection,abramenko2008buildings}.
Since $\mathsf{S}$ and its dual $\mathsf{S}^{*}$ have the same symmetry group, their omnitruncations have the same normal fan, and both are normally equivalent to $W$-permutohedra \cite{coxeter1973regular,hohlweg2012permutahedra}.
\end{remark}

The only regular simple polytopes are the polygons, the regular simplex and hypercube in every dimension, the dodecahedron in dimension 3, and the 120-cell in dimension 4. It would be interesting to carry out a similar analysis for the 120-cell, which is the only remaining case.

\subsection{Permutahedral plates and root polytopes}\label{subsec:pemplatesrootpoly}
Permutahedral plates are the tangent cones of the permutahedron. 
The terminology of plates is attributed by Early to private communication of Ocneanu \cite{early2018canonical}, and is also used in Ardila-Sanchez \cite{ardila2021valuations}.  We observe, perhaps for the first time, that the primitive ray generators for a fixed permutahedral plate lie on a common hyperplane.  If we cut the permutahedral plate with this hyperplane, we obtain a codimension 1 polytope, and this polytope is precisely the distinguished facet $F$ of the Gherkins introduced by the truncation of the corresponding Faberg\'e egg.

A type-$A$ root polytope is the convex hull of a collection of type-$A$ roots \cite{gelfand1996combinatorics,postnikov2009permutohedra,meszaros2011root}.  One popular class of root polytopes considered consists of those associated to directed bipartite graphs with a positive root taken for each directed edge; see \cite{kalman2017root,kalman2022root}.  We observe that the distinguished facet $F$ of the Gherkins introduced by the truncation is a root polytope associated to a chain of complete directed bipartite graphs (see Figure \ref{fig:plate-root-polytope-cone-slices}).  These root polytopes, and more generally root polytopes associated to Hasse diagrams of ranked posets, have been investigated by K\'alm\'an--T\'othm\'er\'esz, Numata--Takahashi--Tamaki, and
Rietsch--Williams
\cite{kalman2022root,numata2024faces,rietsch2025root}.

The classical root polytope is the convex hull of all type-$A$ roots.  Its facets correspond to complete bipartite graphs, equivalently bipartitions of $[n]$.  It is a deformation of the standard permutahedron having the same set of facet normals.  We were drawn to wonder whether these root polytopes associated to chains of complete bipartite graphs, i.e. ordered set partitions, could form the facets of a deformation of the 2-permutahedron having the same set of facet normals.  Such a polytope might naturally be considered as a 2-version of the root polytope.  We used GPT to investigate this question and found that the answer is ``no, but yes''.  

Unlike in the case of the classical root polytope, there is some ambiguity about how to abstractly glue the root polytopes associated to ordered set partitions to form the boundary of a polytope.  However, we have found that the following choice is quite natural.

\begin{definition}
Let $T=S_1|\cdots|S_k$ be an ordered set partition of $[n]$  with $k\geq 2$, and let
\[
\mathsf{R}_T
=
\operatorname{conv}
\left\{
\mathbf e_j-\mathbf e_i
:
i\in S_a,\ j\in S_{a+1},\ 1\leq a<k
\right\}.
\]
  An \textbf{admissible refinement} of $T$ is an ordered set
partition obtained by replacing exactly one block $S_a$ by two
consecutive nonempty blocks $C|D$, where $S_a=C\sqcup D$, subject to the following initial and terminal restrictions:
\begin{enumerate}
    \item if $a=1$, then $|C|=1$;
    \item if $a=k$, then $|D|=1$.
\end{enumerate}
Glue a pair of root polytopes $\mathsf{R}_T$ and $\mathsf{R}_{T'}$ along their common facet whenever one is an admissible refinement of the other.  We define the corresponding abstract polyhedral complex to be the \textbf{2-root complex} on $[n]$. 
\end{definition}

The following theorem will be established in future work.

\begin{theorem}\label{thm:2root}
For each $n\geq 2$, the $2$-root complex on $[n]$ admits a realization as the boundary of a convex polytope $\mathsf{P}$, which we call the 2-root polytope.
\end{theorem}

For $n\geq 4$, the $2$-root complex on $[n]$ cannot be realized as the boundary of a deformation of the 2-permutahedron.  However, the 2-root polytope we've found is a deformation of a skew 2-permutahedron, meaning an omnitruncation of the standard permutahedron with respect to noncentral truncation directions.  Moreover, these directions can be chosen arbitrarily close to the central ones (those which can yield a 2-permutahedron).  The facets are, individually, affinely equivalent to the corresponding ordered set partition root polytopes.  On the other hand, we have found that for $n\geq 4$ it is not possible to produce a polytopal realization of the 2-root complex where each facet is isometric (not just affinely equivalent) to the corresponding ordered set partition root polytope.

\begin{figure}[ht]
	\centering
	\includegraphics[width=.88\textwidth]{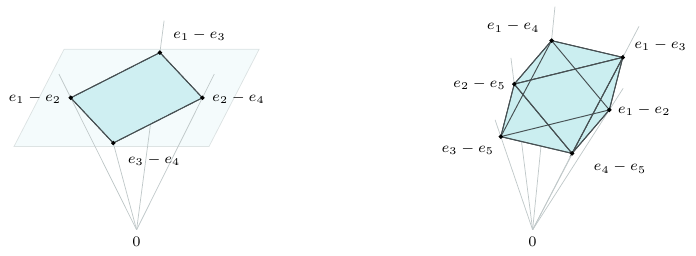}
	\caption{Root polytopes obtained as slices of permutahedral plates.  The primitive ray generators for the plates $1|23|4$ and $1|234|5$ lie in common affine slices, giving respectively a square and an octahedron.}
	\label{fig:plate-root-polytope-cone-slices}
\end{figure}

\begin{figure}[ht]
		\centering
			\includegraphics[width=.65\textwidth]{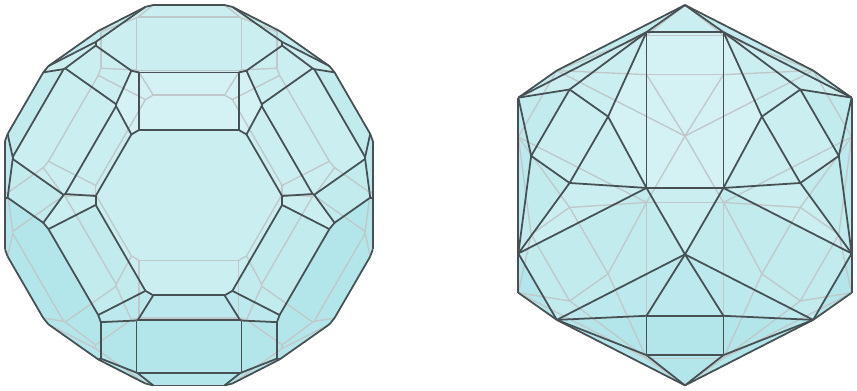}
			\caption{A skew $2$-permutahedron (left) and its corresponding 2-root polytope deformation (right).}
		\label{fig:ordered-set-partition-root-complex}
\end{figure}
\FloatBarrier

\subsection{Deformed 2-permutahedra: the permutonestohedra and simple permutonestohedra}\label{subsec:deformed2perm}

Kapranov introduced the permutoassociahedron as a poset and conjectured that it was realized as the face poset of a polytope \cite{kapranov1993permutoassociahedron}.  This was confirmed by Reiner--Ziegler \cite{reiner1994coxeter}.  Gaiffi introduced permutonestohedra as a generalization of the permutoassociahedra providing a new realization of the latter.  Recently, another realization of the permutoassociahedron was constructed in \cite{castillo2023permuto} which is normally equivalent to Gaiffi's realization.  The permutoassociahedral realizations of Gaiffi and Castillo--Liu are deformations of the 2-permutahedron $\Pi_n^2$ (see Figure \ref{fig:permutoassociahedron-simple-comparison}).  By \Cref{thm:basis_nested}, they can therefore be expressed as a signed Minkowski sum of the elements of our basis of Gherkins.  It is not difficult to see that they cannot be a
positive sum; by inspecting the normal fan, one can observe that the chambers cannot be obtained as common refinements of chambers of the Gherkins.  We provide the Gherkin decomposition here for the $n=4$ realization of Castillo--Liu.

Let $[4]=\{1,2,3,4\}$.  We write
\[
\mathsf R(S_1|\cdots|S_k)
\]
for the Gherkin associated to the ordered partition $S_1|\cdots|S_k$, and
\[
\Delta_I=\operatorname{conv}\{e_i:i\in I\}.
\]
For the three-dimensional permutoassociahedron, the decomposition in the Gherkin basis is
\[
\begin{aligned}
P_{\mathrm{PA}}
={}&
-\sum_{(i,j,k,l)\in \mathfrak S_4}
\mathsf R(i|j|k|l) \\
&+
2\sum_{\substack{|A|=2\\ [4]\setminus A=\{i,j\}}}
\Big(
\mathsf R(A|i|j)+\mathsf R(A|j|i)
+\mathsf R(i|j|A)+\mathsf R(j|i|A)
\Big) \\
&+
\sum_{\substack{|A|=2\\ [4]\setminus A=\{i,j\}}}
\Big(
\mathsf R(i|A|j)+\mathsf R(j|A|i)
\Big) \\
&+
30\sum_{i\in[4]}\Delta_i
-14\sum_{\substack{I\subset[4]\\ |I|=2}}\Delta_I
+
4\Delta_{[4]} .
\end{aligned}
\]
Here the summations over $A$ range over unordered two-element subsets of $[4]$, and if
$[4]\setminus A=\{i,j\}$, then both possible orders of the two complementary singleton blocks are included explicitly.

Barali\'{c}--Ivanovi\'{c}--Petri\'{c} \cite{baralic2019simple} introduced a family of simple permutonestohedra as certain truncations of the permutohedron along a building set.
In \cite[Section 5]{ivanovic2020geometrical} Ivanovi\'{c} produced a Minkowski sum decomposition of the simple permutonestohedra which uses deep truncations of certain nestohedra; this is in the same spirit as our general approach; however, those decompositions are not always indecomposable \cite[Figure 12]{ivanovic2020geometrical}. Using \Cref{thm:basis_nested} one can decompose a normally equivalent polytope as a positive Minkowski sum of Gherkin polytopes, which have the advantage of being indecomposable.

\begin{figure}[htbp]
    \centering
    \includegraphics[width=.95\textwidth]{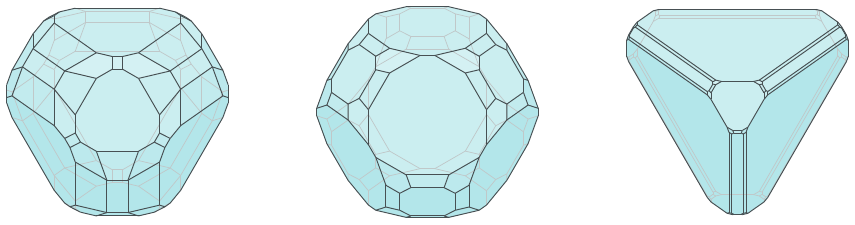}
    \caption{The Castillo--Liu realization of Kapranov's (nonsimple) permutoassociahedron (left), and two realizations of the simple permutoassociahedron of Barali\'{c}--Ivanovi\'{c}--Petri\'{c}: as a building-set truncation of $P_4$ (center) and as the Minkowski sum of the corresponding Gherkins (right). The truncation realization in the center admits a signed Minkowski decomposition in the basis of \Cref{thm:basis_nested}, in which every Gherkin has a nonnegative coefficient; only the standard-simplex coefficients may be negative.}
    \label{fig:permutoassociahedron-simple-comparison}
\end{figure}

\subsection{A deformed 2-nestohedron: the cosmohedron}\label{subsec:cosmo}

Earlier we introduced the notion of a 2-nestohedron as a polytope which is obtained as the 2-pass truncation of a standard simplex.   We anticipate that many interesting polytopes can be obtained as deformations of 2-nestohedra.  In addition to the permutoassociahedron and the simple permutoassociahedron in Subsection \ref{subsec:deformed2perm}, there is also the cosmohedron of Arkani-Hamed--Figueiredo--Vaz\~ao \cite{arkani2025cosmohedra}.  In recent work of Ardila-Mantilla--Arkani-Hamed--Figueiredo--Vaz\~ao \cite{ardila2026combinatorics}, the polytopal realization provided can be obtained as a deformation of an omnitruncation of Loday's associahedron.  Thus its support function restricted to the rays of the 2-braid fan (which include the rays of the cosmohedron's normal fan) can be generated uniquely as a linear combination of the support functions of the Gherkins on the rays of the 2-braid fan.  There is a subtle issue here: the normal fan of the cosmohedron is not a coarsening of the 2-braid fan.  This implies that the corresponding signed Minkowski sum of elements of the Gherkin basis (see \Cref{thm:basis_nested}) does not recover the cosmohedron.  Instead we can apply \Cref{cor:basis_building} to conclude that the cosmohedron is a signed Minkowski sum of shallow truncations of faces of Loday's associahedron.

Let $L$ denote the Loday associahedron cut out by the affine hyperplane
\[
x_1+x_2+x_3+x_4=10
\]
together with the following supporting inequalities:
\[
\begin{array}{lll}
H_0: x_1\ge 1, & H_1: x_2\ge 1, & H_2: x_3\ge 1,\\
H_3: x_4\ge 1, & H_4: x_1+x_2\ge 3, & H_5: x_2+x_3\ge 3,\\
H_6: x_3+x_4\ge 3, & H_7: x_1+x_2+x_3\ge 6, & H_8: x_2+x_3+x_4\ge 6.
\end{array}
\]
Let $F_i$ denote the facet of $L$ obtained by setting $H_i$ to an equality.
We use the Ardila-Mantilla--Arkani-Hamed--Figueiredo--Vaz\~ao realization of the $3$-cosmohedron.
Fix a real number $\delta$ with $0<\delta<1$.
For a nonempty set $A\subseteq\{0,\ldots,8\}$ whose intersection
\[
F_A=\bigcap_{i\in A}F_i
\]
is a face of $L$, let
\[
T_A^\delta=\operatorname{Tr}_{\delta}(L,F_A)
\]
be the shallow truncation of $L$ along $F_A$ at common depth $\delta$.

Then the $3$-cosmohedron $\mathsf C$ satisfies the following signed support-function identity:
\[
\begin{aligned}
h_{\mathsf C}
={}& h_L
-\frac{1}{2\delta}\sum_{i\in\{0,1,3,4,5,6\}} h_{T^\delta_{\{i\}}}
-\frac{3}{4\delta}\sum_{i\in\{2,7,8\}} h_{T^\delta_{\{i\}}} \\
&+
\frac{1}{4\delta}\sum_{\{i,j\}\in E_1} h_{T^\delta_{\{i,j\}}}
+
\frac{1}{2\delta}\sum_{\{i,j\}\in E_2} h_{T^\delta_{\{i,j\}}}
-
\frac{1}{4\delta}\sum_{\{i,j,k\}\in V} h_{T^\delta_{\{i,j,k\}}},
\end{aligned}
\]
where
\[
\begin{aligned}
E_1={}&
\big\{
\{0,7\},\{0,3\},\{0,2\},\{2,7\},\{1,7\},
\{1,3\},\{1,8\},\{3,8\},\{2,8\}
\big\},\\
E_2={}&
\big\{
\{0,4\},\{4,7\},\{3,4\},\{0,6\},\{2,6\},\{3,6\},
\{1,4\},\{5,7\},\{1,5\},\{2,5\},\{5,8\},\{6,8\}
\big\},\\
V={}&
\big\{
\{0,4,7\},\{0,3,4\},\{0,2,6\},\{0,3,6\},
\{1,4,7\},\{1,3,4\},\{1,5,7\},\{2,5,7\},\\
&\qquad
\{1,5,8\},\{2,5,8\},\{2,6,8\},\{3,6,8\}
\big\}.
\end{aligned}
\]
The vertex-face truncations corresponding to the active facet sets
\[
\{0,2,7\}
\qquad\text{and}\qquad
\{1,3,8\}
\]
have coefficient zero and are therefore omitted.

\begin{figure}[htbp]
    \centering
    \includegraphics[width=.95\textwidth]{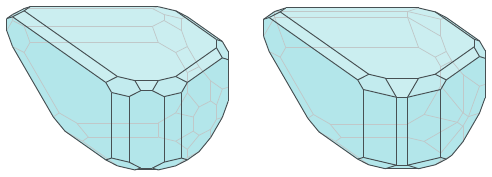}
    \caption{An omnitruncation of Loday's associahedron (left) and the cosmohedron (right).}
    \label{fig:loday-omnitruncation-cosmohedron}
\end{figure}

\subsection{Truncations of crystallographic zonotopes}
For any crystallographic root system $\Phi$, the Weyl permutahedron
$Z_{\Phi}=\operatorname{conv}(W\cdot\rho)$ is, up to translation, a primitive Delzant zonotope with
respect to the root lattice, where $\rho$ is the half-sum of the positive
roots.
Indeed, at the dominant vertex $\rho$, the adjacent edge directions are
$-\alpha_i$, as $\alpha_i$ ranges over the simple roots; these form a
$\mathbb Z$-basis of the root lattice. The same holds at every vertex by
the Weyl group action; see \cite[Chapter~4]{bjorner2005combinatorics}.  We think it would be interesting to study polytopes which are deformations of the omnitruncations of $Z_{\Phi}$.  In Figure \ref{fig:type-b-gherkin-sum-omnitruncation}, we illustrate how the Minkowski sum of the associated Gherkins is normally equivalent to the omnitruncation of the type-B permutahedron.

\begin{figure}[htbp]
    \centering
    \begin{minipage}{.48\textwidth}
        \centering
        \includegraphics[width=\linewidth]{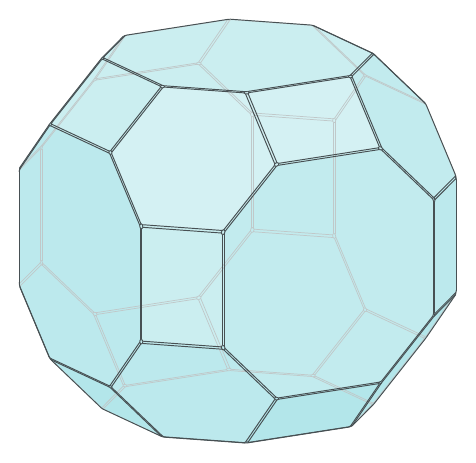}
    \end{minipage}
    \hfill
    \begin{minipage}{.48\textwidth}
        \centering
        \includegraphics[width=\linewidth]{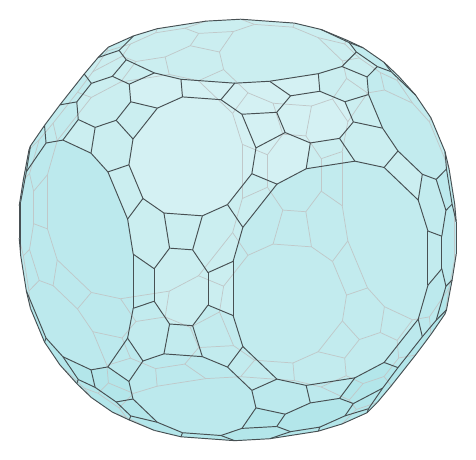}
    \end{minipage}
    \caption{The Minkowski sum of all type $B_3$ edge Gherkins and the result of subtracting $66.5\cdot P_{B_3}$, where $P_{B_3}$ is the type-B permutahedron in dimension 3. The two polytopes are both normally equivalent to an omnitruncation of $P_{B_3}$; note the tiny hexagonal and quadrilateral faces of the polytope on the left.}
    \label{fig:type-b-gherkin-sum-omnitruncation}
\end{figure}

\subsection{Attempting the bipermutahedron}

The bipermutahedron was introduced by Ardila--Denham--Huh in their work on Lagrangian geometry of matroids \cite{ardila2023lagrangian,ardila2023lagrangian2,ardila2022bipermutahedron}.
This is a smooth polytope whose normal fan, the bipermutahedral fan, contains the conormal fan of any matroid as a subfan.  
This story is analogous to the more classical case where the normal fan of the permutahedron, the braid arrangement, contains the fine Bergman fan of any matroid as a subfan.

As was observed in the work of Backman--Danner (see footnote 3 in \cite{backman2024convex}), the face poset of the conormal fan of a matroid behaves like the nested set complex associated to a building set of biflats on the lattice ${\mathcal{L}(M)}^{op}\times \mathcal{L}(M^{\perp})$.
For subtle reasons related to lineality spaces explained in that footnote, the biflats do not quite form a building set.   
Motivated by this opaque connection to building sets, the authors of the present article wondered whether the bipermutahedron could be obtained  from some simpler polytope by a sequence of well-chosen truncations, as is the case for the classical permutahedron.  
While we failed at achieving this goal, we feel that the ways in which we failed are interesting.  
We record here our attempts in hopes that other researchers may glean something useful from our approach.  
We will assume familiarity with the bipermutahedron and the associated combinatorics of bisubsets.

\subsubsection{A mock bipermutahedron}
As in the footnote in Backman--Danner, the collection of bisubsets nearly forms a building set for the product of two Boolean posets $ \mathscr{B}^{\,op} \times  \mathscr{B}$.  
This suggested to us that perhaps the bisubsets could be viewed as a building set for the face poset of a product of two standard simplices of the same size, which would then fit naturally with the framework of \Cref{sec:polysimplex}.
In what follows, we will identify the face poset of the normal fan of this product of simplices with $( \mathscr{B}^{\,op}\setminus \{\emptyset\} )\times ( \mathscr{B}\setminus\{ E\})$.  
If we slightly augment the set of bisubsets by unioning with the pairs of the form $(E\setminus\left\{i\right\},\emptyset)$ for $i$ an element of the ground set, and removing pairs of the form $(F,E)$, we obtain a building set which we call the \textbf{mock bisubsets}.  
Note that these additional pairs of flats added correspond to facets coming from the first simplex and thus play no role in the truncation process, while the pairs removed simply do not make sense in this setting.  
We additionally observe that adding all pairs of the form $(F,\emptyset)$ would not produce a building set.  

Performing truncations of the faces of the product of these simplices associated to the bisubsets, one obtains a polytope with nearly the same set of facet normals as the bipermutahedron; it is missing facets corresponding to $(F,\emptyset)$ for $0<|F|<n$.  We then do a second-pass truncation of these faces, which themselves form a building set on the face poset of the polytope already obtained.  The resulting polytope is not the bipermutahedron -- we call it the  \textbf{mock bipermutahedron}.  Not only is this a smooth polytope with the same facet normals as the bipermutahedron, but we have the following conjecture which has been verified for $n \leq 5$:

\begin{conjecture}\label{mockconj}
The bipermutahedron and the mock bipermutahedron have the same $f$-vector.
\end{conjecture}

\begin{figure}[ht]
    \centering
    \includegraphics[width=.82\textwidth]{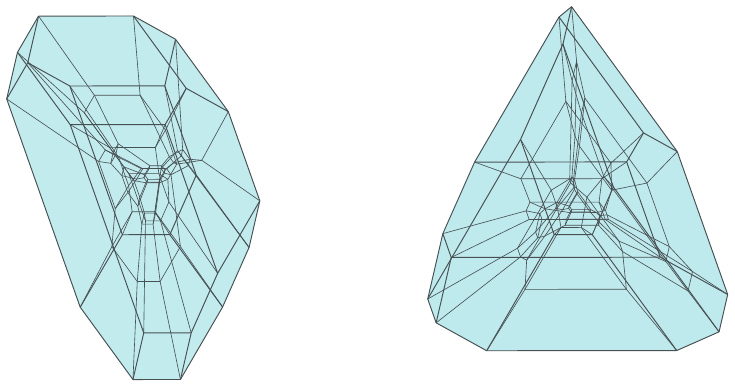}
    \caption{On the left is the bipermutahedron, and on the right is the mock bipermutahedron for $n=3$.  They are both smooth, have the same set of facet normals, and have the same $f$-vector, $(1,90,180,114,24,1)$.}
    \label{fig:bipermutahedron-mock-bipermutahedron}
\end{figure}

We note that a similar phenomenon has been reported recently by Abram--Bastidas--Dequêne--Morales--Park--Thomas \cite{abram2026flows}; they have introduced the mutoperhedron, a polytope with the same $f$-vector as the permutohedron, but a different face poset.

\subsubsection{Including the diagonal simplex} 
The previous approach fails to incorporate the role of the diagonal simplex, i.e. 
the convex hull of all vectors $e_i+f_i$ for $i \in E$.  
The harmonic polytope is the product of two permutahedra plus the diagonal simplex \cite{ardila2023lagrangian,ardila2021harmonic}, and the  harmonic polytope is a deformation of the bipermutahedron.  
So, the authors next considered whether the bipermutahedron could be obtained by starting from a polytope $P$, which is the product of two standard simplices of the same size, Minkowski-summed with the diagonal simplex, and then performing a sequence of truncations of $P$ encoded by the bisubsets.  
This works for $n=3$, producing the 4-dimensional bipermutahedron living in $\mathbb{R}^6$, although we note the sequence of truncations is not performed according to a building set.  However, for $n=4$ this is not possible; the 6-dimensional bipermutahedron living in $\mathbb{R}^8$ cannot be obtained from $P$ by a sequence of truncations.  In fact, one normal cone of $P$ contains four required bipermutahedral rays, and none of the $4!=24$ possible orders of stellar subdivision along these rays produces the corresponding local subdivision of the bipermutahedral fan.

\begin{remark}[Conormal building sets]
It would be interesting if one could build upon the observations of the footnote of Backman--Danner to produce a family of fan structures on the product of a Bergman fan of a matroid and the Bergman fan of its dual which generalizes the conormal fan analogously to the way that building sets on the lattice of flats of a matroid produce different fan structures on the Bergman fan which generalize the fine Bergman fan.  This might go by way of first developing a general theory of wonderful compactifications of conormal bundles for hyperplane arrangement complements.
\end{remark}

\subsection{\texorpdfstring{$\pi$}{}-colored fans}\label{subsec:picolored}

While this work was in progress, the article \cite{clader2024multimatroids} introduced the notion of a $\pi$-colored fan and its associated Chow ring.\footnote{At the time the arXiv version of  \cite{clader2024multimatroids} first appeared, we had produced and verified the flat truncation basis for the product of simplices.  We had also privately conjectured that the fragment basis should be a basis; however, we did not have a proof.  The proof we have presented is inspired by the inclusion-exclusion proof of the restricted statement given in  \cite{clader2024multimatroids}.} 
This Chow ring has been further studied by Nathanson \cite{nathanson2026multipermutohedral}.
We observe that the $\pi$-colored fans are subfans of the barycentric subdivisions of the normal fans of products of standard simplices.  In particular, they are subfans of smooth projective fans (a theme in the area of Hodge theory for matroids; see \cite{mantovani2025facial,backman2024convex,ardila2023lagrangian,braden2022semi,crowley2022bergman,eur2023intersection}).  
We further demonstrate that the restrictions of the support functions of fragments give the $h$-basis of \cite{clader2024multimatroids} and the restrictions of the support functions of deep truncations are closely related to their $f$-basis.
\begin{figure}[ht]
    \centering
    \includegraphics[width=.82\textwidth]{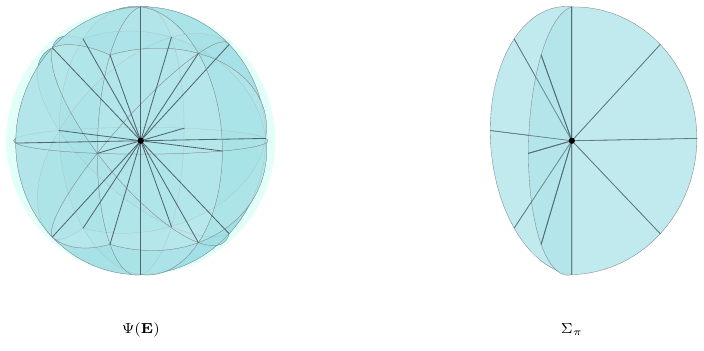}
    \caption{For $\mathbf{E}=(\{a,b,c\},\{x,y\})$, the barycentric subdivision $\Psi(\mathbf{E})$ of the normal fan of $\Delta(\mathbf{E})=\Delta^2\times\Delta^1$ (see Figure \ref{fig:triangle-segment-prism-omnitruncation} for a polytope normal to $\Psi(\mathbf{E})$) and the corresponding $\pi$-colored subfan $\Sigma_{\pi}$.}
    \label{fig:triangle-segment-pi-colored-fan}
\end{figure}

\begin{definition}
Let $\pi:E\to[n]$ be a surjection from a finite set $E$, with
$|\pi^{-1}(i)|\geq2$ for every $i$. Set $E_i=\pi^{-1}(i)$ and
$\mathbf E=(E_1,\ldots,E_n)$, and set
\[
\mathsf L(\mathbf E)
=
\operatorname{Span}\{e_{E_1},\ldots,e_{E_n}\}
\subseteq\mathbb R^E.
\]
A set $S\subseteq E$ is \emph{$\pi$-colored} if
$|S\cap E_i|\leq1$ for every $i$. The pointed $\pi$-colored fan
${}^{\wedge}\Sigma_\pi$ consists of the cones
\[
\operatorname{Cone}(e_{S_1},\ldots,e_{S_k}),
\]
where $S_1\subsetneq\cdots\subsetneq S_k$ is a chain of nonempty
$\pi$-colored sets. Its prequotient and quotient versions are
\[
\widetilde\Sigma_\pi
=
{}^{\wedge}\Sigma_\pi+\mathsf L(\mathbf E)
\subseteq\mathbb R^E,
\qquad
\Sigma_\pi
=
\widetilde\Sigma_\pi/\mathsf L(\mathbf E).
\]
\end{definition}

\begin{proposition}
\label{prop:pi_colored_subfan_polysimplex}
The pointed fan ${}^{\wedge}\Sigma_\pi$ is a subfan of
${}^{\wedge}\Psi(\mathbf E)$. Consequently,
$\widetilde\Sigma_\pi$ is a subfan of $\Psi(\mathbf E)$, and
$\Sigma_\pi$ is a subfan of
$\Psi(\mathbf E)/\mathsf L(\mathbf E)$.
\end{proposition}

\begin{proof}
A chain $S_1\subsetneq\cdots\subsetneq S_k$ of nonempty
$\pi$-colored sets determines the rooted ordered set partition
\[
(E\setminus S_k,
  S_k\setminus S_{k-1},\ldots,
  S_2\setminus S_1,S_1).
\]
Indeed, its first block meets every $E_i$. By
\Cref{lem:fan_polysimplex_bary}, the corresponding cone is
$\operatorname{Cone}(e_{S_1},\ldots,e_{S_k})$. The remaining
statements follow by adding and then quotienting by the lineality
space from \Cref{lem:fan_polysimplex}.
\end{proof}

We offer a conceptual explanation of the above result.  We first observe that the $\pi$-colored fans can alternately be interpreted as the barycentric subdivisions of a product of tropical lines.  Such a tropical line is a subfan of the normal fan of a standard simplex.  Because barycentric subdivisions respect passing to subfans, the result follows.  We note that the same proof applies if $|S\cap E_i|\leq1$ is replaced by
$|S\cap E_i|\leq c_i<|E_i|$; see
\cite[Remark~2.3]{clader2024multimatroids}.  This extension corresponds to replacing tropical lines with tropical linear spaces which are the $c_i$-skeleta of normal fans of standard simplices.

We use the minimum support function $h(\mathsf Q,u)$. Under the
preceding proposition, the rays of $\Sigma_\pi$ are inward normals
for the omnitruncation of a polysimplex and outward normals for its negative.
Accordingly, to compare our minimum support functions with the ray
coefficients in \cite{clader2024multimatroids}, we use the formal
identity
\[
h_{\max}(-\mathsf Q,u)=-h(\mathsf Q,u).
\]
To distinguish the two uses of the letter $h$, write
$\mathrm h_S^\pi$ and $\mathrm f_S^\pi$ for the functions denoted
$h_S$ and $f_S$ in \cite{clader2024multimatroids}; they are determined
on the rays by
\[
\mathrm h_S^\pi(e_R)=\mathbf1_{S\cap R\neq\varnothing},
\qquad
\mathrm f_S^\pi(e_R)=\mathbf1_{S\subseteq R}.
\]

For $s\in E$, let $\ell_s(e_R)=\mathbf1_{s\in R}$. If
$s,t\in E_i$, then $\ell_s-\ell_t$ is the restriction of the global
linear function $u\mapsto u(s)-u(t)$. We denote the common class of
the $\ell_s$, for $s\in E_i$, by $\alpha_i$. The functions
$\mathrm f_S^\pi$ form the zeta basis before quotienting by global
linear functions. Thus, after quotienting, the classes $\alpha_i$
together with $[\mathrm f_S^\pi]$ for $|S|\geq2$ form a basis.

Fix a nonempty $\pi$-colored set $S$ and put
\[
I(S)=\{i:S\cap E_i\neq\varnothing\},
\qquad
\overline E_i(S)=E_i\setminus(S\cap E_i).
\]
Write
\[
\mathbf E_S=(E_i)_{i\in I(S)},
\qquad
\overline{\mathbf E}_S
=(\overline E_i(S))_{i\in I(S)},
\qquad
\overline{\mathbf E}(S)
=(\overline E_1(S),\ldots,\overline E_n(S)).
\]
Define the factored fragment and factored deep truncation
\[
\mathsf F_S
=
\operatorname{Fr}[\Delta(\mathbf E_S),\{e_S\}],
\qquad
\mathsf T_S
=
\operatorname{DTr}
[\Delta(\mathbf E_S),\Delta(\overline{\mathbf E}_S)].
\]
To obtain support functions on the quotient fan, we translate each
factored polytope into $\mathsf L(\mathbf E_S)^\perp$, equivalently
normalizing its restriction to the lineality space to be zero. We use
the translates $\mathsf F_S-e_S$ and $\mathsf T_S-e_S$. Any two such
normalizations differ by a translation in
$\mathsf L(\mathbf E_S)^\perp$, and hence determine the same class
modulo global linear functions on the quotient fan.

\begin{proposition}
\label{prop:pi_colored_fragments_h_basis}
The outward-support coefficients of $e_S+(-\mathsf F_S)$ restrict to
$\mathrm h_S^\pi$. Equivalently,
\[
-h(\mathsf F_S-e_S,e_R)=\mathrm h_S^\pi(e_R)
\]
for every nonempty $\pi$-colored set $R$.
\end{proposition}

\begin{proof}
By \Cref{lem:polysimplex_fragment_support},
\[
h(\mathsf F_S,e_R)
=
\max\{0,|S\cap R|-1\}.
\]
Hence the formal conversion from minimum to outward support gives
\[
-h(\mathsf F_S-e_S,e_R)
=
|S\cap R|-\max\{0,|S\cap R|-1\}
=
\mathbf1_{S\cap R\neq\varnothing}.
\]
\end{proof}

\begin{proposition}
\label{prop:pi_colored_truncations_basis}
The minimum support function of $\mathsf T_S-e_S$ restricts to
\[
h(\mathsf T_S-e_S,e_R)
=
\mathrm f_S^\pi(e_R)-\sum_{s\in S}\ell_s(e_R).
\]
Consequently, in the outward convention, the translated factor
simplices and the translated truncations restrict to the nef basis
\[
\alpha_i,
\qquad
\sum_{i\in I(S)}\alpha_i-[\mathrm f_S^\pi]
\quad (|S|\geq2).
\]
Moreover, if
\[
\widehat{\mathsf T}_S
=
\operatorname{DTr}
[\Delta(\mathbf E),\Delta(\overline{\mathbf E}(S))],
\]
then
\[
\left(
h(\widehat{\mathsf T}_S,\mathord\cdot)
-h(\Delta(\mathbf E),\mathord\cdot)
\right)\big|_{\Sigma_\pi}
=
\mathrm f_S^\pi.
\]
\end{proposition}

\begin{proof}
By \Cref{lem:polysimplex_truncation_support},
\[
h(\mathsf T_S,e_R)=\mathbf1_{S\subseteq R}.
\]
Translation by $-e_S$ therefore gives
\[
h(\mathsf T_S-e_S,e_R)
=
\mathbf1_{S\subseteq R}-|S\cap R|.
\]
Since $R$ is a nonempty $\pi$-colored set, $e_R$ represents a ray of
$\Sigma_\pi$ and $R\cap E_i\subsetneq E_i$; thus, writing
$S\cap E_i=\{s_i\}$ for $i\in I(S)$, we have
\[
h_{\max}\!\left(
\sum_{i\in I(S)}\bigl(e_{s_i}-\Delta(E_i)\bigr),e_R
\right)
=
\sum_{i\in I(S)}\ell_{s_i}(e_R)
=
|S\cap R|,
\]
so the formal sign change gives the class
$\sum_{i\in I(S)}\alpha_i-[\mathrm f_S^\pi]$, and the basis claim
follows from the preceding zeta-basis observation. Finally,
$h(\widehat{\mathsf T}_S,\mathord\cdot)$ and
$h(\Delta(\mathbf E),\mathord\cdot)$ have the same restriction to
$\mathsf L(\mathbf E)$, so their difference descends to the quotient
fan, where its value on $e_R$ is
$\mathbf1_{S\subseteq R}-0=\mathrm f_S^\pi(e_R)$.
\end{proof}

\subsection{The tropical $\alpha$ and $\beta$ classes}

  The $\alpha$ and $\beta$ classes introduced by Huh--Katz \cite{huh2012log} play a fundamental role in the proof of the Heron--Rota--Welsh conjecture \cite{adiprasito2018hodge}.  The $\alpha$ and $\beta$ classes correspond to the restrictions of maximal standard simplices and their negatives to the Bergman fan; see \cite{backman2023simplicial}.  We observe that there appears to be an interesting parallel between the fragment and truncation bases for the barycentric subdivisions of the normal fans of products of simplices and the $\alpha$ and $\beta$ classes, respectively.  
  
The $\alpha$ and $\beta$ classes correspond to restrictions of matroid base polytopes for the dual matroids $U^1_n$ and $U^{n-1}_n$, respectively. In the case of the type $B$ permutahedron, the fragments and vertex truncations in this article also correspond to matroid independence polytopes for the dual matroids $U^1_n$ and $U^{n-1}_n$, respectively (as well as their images under the action of ${\mathbb{Z}/2\mathbb{Z}}^{n}$; see Figure \ref{fig:uniform-matroid-base-independence-polytopes}).   Moreover, the $\alpha$ and $\beta$ classes are related by inclusion--exclusion.  This is also the case for the fragment and truncation bases (see Proposition \ref{prop:explicit_truncation_fragment_transform}) as well as their restricted $\pi$-colored versions \cite[Lemma~2.11]{clader2024multimatroids}.

\begin{figure}[ht]
    \centering
    \includegraphics[width=.50\textwidth]{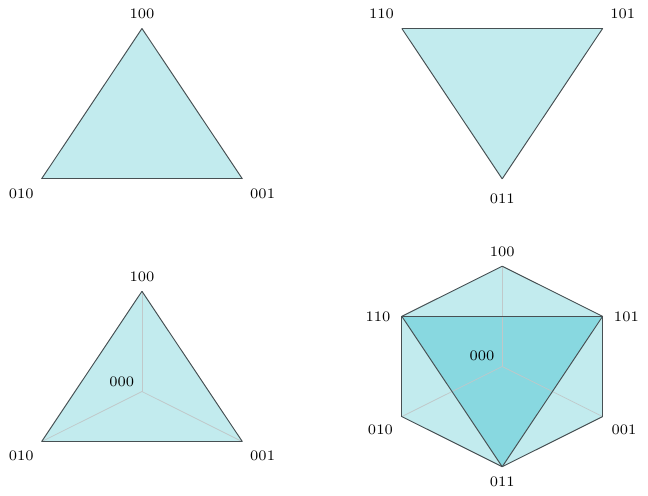}
    \caption{Uniform matroid base and independence polytopes: the top row shows the base polytopes of $U^1_3$ and $U^2_3$, and the bottom row shows the independence polytopes of $U^1_3$ and $U^2_3$.}
    \label{fig:uniform-matroid-base-independence-polytopes}
\end{figure}

The $\alpha$ and $\beta$ classes are tropicalizations of hyperplanes and reciprocal hyperplanes, respectively, which are related via the tropical Cremona transform (multiplication by $-1$).  It would be interesting to have a deeper algebro-geometric understanding of the duality between the truncation and fragment bases for products of simplices.

\FloatBarrier
\addcontentsline{toc}{section}{References}
\begingroup
\linespread{0.97}\selectfont
\bibliography{biblio}
\bibliographystyle{alpha}
\endgroup

\end{document}

%% file: tikz/triangle.tex
\begin{tikzpicture}[
    tdplot_main_coords,
    scale = 2,
    back/.style={loosely dotted, thin},
    edge/.style={color=Turquoise!25!black, line width=0.46pt},
    distinguishededge/.style={color=Turquoise!45!black, line width=1.05pt},
    distinguishedvertex/.style={inner sep=0.8pt,circle,draw=Turquoise!45!black,fill=Turquoise!45!black},
    finaledge/.style={color=Turquoise!45!black, line width=2pt},
    endpoint/.style={inner sep=0.55pt,circle,draw=black,fill=black},
    body/.style={fill=Turquoise!27, draw=none},
    vertex/.style={inner sep=0pt,draw=none,fill=none}]
    \begin{scope}[xshift=-3cm]
    \fill[body] (1,0,0) -- (0,1,0) -- (0,0,1) -- cycle;
    \node[vertex] at (1,0,0) {};
    \node[vertex] at (0,1,0) {};
    \node[distinguishedvertex] at (0,0,1) {};
    \draw[edge] (0,1,0)  -- (1,0,0);
    \draw[edge] (0,0,1)  -- (1,0,0);
    \draw[edge] (0,1,0)  -- (0,0,1);

    \end{scope}
    \begin{scope}[xshift= -1cm]
    \fill[body]
        (1,0,0) -- (0,1,0) -- (0,0.5,0.5) -- (0.5,0,0.5) -- cycle;
    \node[vertex] at (1,0,0) {};
    \node[vertex] at (0,1,0) {};
    \node[vertex] at (0,0,1) {};
    \node[vertex] at (0.5,0,0.5) {};
    \node[vertex] at (0,0.5,0.5) {};
    \draw[edge] (0,1,0)  -- (1,0,0);
    \draw[edge, dashed] (0,0,1)  -- (1,0,0);
    \draw[edge, dashed] (0,1,0)  -- (0,0,1);
    \draw[edge] (0.5,0,0.5)  -- (1,0,0);
    \draw[distinguishededge] (0.5,0,0.5)  -- (0,0.5,0.5);
    \draw[edge] (0,1,0)  -- (0,0.5,0.5);

    \end{scope}
    \begin{scope}[xshift= 1cm]
    \fill[body]
        (1,0,0) -- (0,1,0) -- (0,0.8,0.2) -- (0.8,0,0.2) -- cycle;
    \node[vertex] at (1,0,0) {};
    \node[vertex] at (0,1,0) {};
    \node[vertex] at (0,0,1) {};
    \node[vertex] at (0.8,0,0.2) {};
    \node[vertex] at (0,0.8,0.2) {};
    \draw[edge] (0,1,0)  -- (1,0,0);
    \draw[edge, dashed] (0,0,1)  -- (1,0,0);
    \draw[edge, dashed] (0,1,0)  -- (0,0,1);
    \draw[edge] (0.8,0,0.2)  -- (1,0,0);
    \draw[distinguishededge] (0.8,0,0.2)  -- (0,0.8,0.2);
    \draw[edge] (0,1,0)  -- (0,0.8,0.2);

    \end{scope}

    \begin{scope}[xshift= 3cm]
    \node[vertex] at (1,0,0) {};
    \node[vertex] at (0,1,0) {};
    \node[vertex] at (0,0,1) {};
    \draw[finaledge] (0,1,0)  -- (1,0,0);
    \node[endpoint] at (0,1,0) {};
    \node[endpoint] at (1,0,0) {};
    \draw[edge, dashed] (0,0,1)  -- (1,0,0);
    \draw[edge, dashed] (0,1,0)  -- (0,0,1);

    \end{scope}
\end{tikzpicture}

%% file: tikz/split.tex
\begin{tikzpicture}[
	 tdplot_main_coords,
	 scale = 2,
	 back/.style={solid, line width=0.35pt, color=Turquoise!25!black!30},
	 edge/.style={color=Turquoise!25!black, line width=0.46pt},
	 distinguishedface/.style={fill=Turquoise!45, draw=none},
	 distinguishedvertex/.style={inner sep=0.8pt,circle,draw=Turquoise!45!black,fill=Turquoise!45!black},
	 paperfacefill1/.style={fill=Turquoise!31, draw=none},
	 paperfacefill2/.style={fill=Turquoise!28, draw=none},
	 paperfacefill3/.style={fill=Turquoise!24, draw=none},
	 paperfacefill4/.style={fill=Turquoise!20, draw=none},
	 vertex/.style={inner sep=0pt,draw=none,fill=none}]

	\begin{scope}[xshift=-4cm]
	\coordinate (c000) at (0,0,0);
	\coordinate (c100) at (1,0,0);
	\coordinate (c010) at (0,1,0);
	\coordinate (c001) at (0,0,1);
	\coordinate (c110) at (1,1,0);
	\coordinate (c101) at (1,0,1);
	\coordinate (c011) at (0,1,1);
	\coordinate (c111) at (1,1,1);

	\fill[paperfacefill2] (c100) -- (c110) -- (c111) -- (c101) -- cycle;
	\fill[paperfacefill1] (c010) -- (c110) -- (c111) -- (c011) -- cycle;
	\fill[paperfacefill3] (c001) -- (c101) -- (c111) -- (c011) -- cycle;

	\draw[edge, back] (c000) -- (c100);
	\draw[edge, back] (c000) -- (c010);
	\draw[edge, back] (c000) -- (c001);
	\draw[edge] (c010) -- (c110);
	\draw[edge] (c001) -- (c101);
	\draw[edge] (c011) -- (c111);
	\draw[edge] (c100) -- (c110);
	\draw[edge] (c001) -- (c011);
	\draw[edge] (c101) -- (c111);
	\draw[edge] (c100) -- (c101);
	\draw[edge] (c010) -- (c011);
	\draw[edge] (c110) -- (c111);
	\node[distinguishedvertex] at (c111) {};
	\end{scope}

	\begin{scope}[xshift=+2cm]
	\fill[paperfacefill2] (1,1,0) -- (1,0,1) -- (1,1,1) -- cycle;
	\fill[paperfacefill1] (1,1,0) -- (0,1,1) -- (1,1,1) -- cycle;
	\fill[paperfacefill3] (0,1,1) -- (1,0,1) -- (1,1,1) -- cycle;

	\node[vertex] at (1,1,0) {};
	\node[vertex] at (0,1,1) {};
	\node[vertex] at (1,0,1) {};
	\node[vertex] at (1,1,1) {};
	\node[distinguishedvertex] at (1,1,1) {};

	\draw[edge] (0,1,1)  -- (1,1,0);
	\draw[edge] (0,1,1)  -- (1,0,1);
	\draw[edge] (1,1,0)  -- (1,0,1);
	\draw[edge] (1,1,1)  -- (1,1,0);
	\draw[edge] (1,1,1)  -- (1,0,1);
	\draw[edge] (1,1,1)  -- (0,1,1);

	\end{scope}

	\begin{scope}[xshift= 0cm]
	\fill[paperfacefill2] (1,0,0) -- (1,1,0) -- (1,0,1) -- cycle;
	\fill[paperfacefill1] (0,1,0) -- (1,1,0) -- (0,1,1) -- cycle;
	\fill[paperfacefill3] (0,0,1) -- (1,0,1) -- (0,1,1) -- cycle;
	\fill[distinguishedface] (1,1,0) -- (0,1,1) -- (1,0,1) -- cycle;
	\draw[edge, back] (0,0,0) -- (1,0,0);
	\draw[edge, back] (0,0,0) -- (0,1,0);
	\draw[edge, back] (0,0,0) -- (0,0,1);

	\node[vertex] at (0,0,0) {};
	\node[vertex] at (1,0,0) {};
	\node[vertex] at (0,1,0) {};
	\node[vertex] at (0,0,1) {};
	\node[vertex] at (1,1,0) {};
	\node[vertex] at (0,1,1) {};
	\node[vertex] at (1,0,1) {};

	\draw[edge] (1,0,0)  -- (1,1,0);
	\draw[edge] (0,1,0)  -- (1,1,0);
	\draw[edge] (0,1,0)  -- (0,1,1);
	\draw[edge] (0,0,1)  -- (0,1,1);
	\draw[edge] (0,1,1)  -- (1,1,0);
	\draw[edge] (0,0,1)  -- (0,1,1);
	\draw[edge] (0,0,1)  -- (1,0,1);
	\draw[edge] (0,1,1)  -- (1,0,1);
	\draw[edge] (1,0,0)  -- (1,0,1);
	\draw[edge] (1,1,0)  -- (1,0,1);

	\end{scope}

	\begin{scope}[xshift= -2cm]
	\fill[paperfacefill2]
		(1,0,0) -- (1,1,0) -- (1,1,1/2) -- (1,1/2,1) -- (1,0,1) -- cycle;
	\fill[paperfacefill1]
		(0,1,0) -- (1,1,0) -- (1,1,1/2) -- (1/2,1,1) -- (0,1,1) -- cycle;
	\fill[paperfacefill3]
		(0,0,1) -- (1,0,1) -- (1,1/2,1) -- (1/2,1,1) -- (0,1,1) -- cycle;
	\fill[distinguishedface]
		(1,1,1/2) -- (1/2,1,1) -- (1,1/2,1) -- cycle;
	\draw[edge, back] (0,0,0) -- (1,0,0);
	\draw[edge, back] (0,0,0) -- (0,1,0);
	\draw[edge, back] (0,0,0) -- (0,0,1);

	\node[vertex] at (0,0,0) {};
	\node[vertex] at (1,0,0) {};
	\node[vertex] at (0,1,0) {};
	\node[vertex] at (0,0,1) {};
	\node[vertex] at (1,1,0) {};
	\node[vertex] at (0,1,1) {};
	\node[vertex] at (1,0,1) {};
	\node[vertex] at (1,1,1/2) {};
	\node[vertex] at (1/2,1,1) {};
	\node[vertex] at (1,1/2,1) {};

	\draw[edge] (1,0,0)  -- (1,1,0);
	\draw[edge] (0,1,0)  -- (1,1,0);
	\draw[edge] (0,1,0)  -- (0,1,1);
	\draw[edge] (0,0,1)  -- (0,1,1);
	\draw[edge] (0,0,1)  -- (0,1,1);
	\draw[edge] (0,0,1)  -- (1,0,1);
	\draw[edge] (1,0,0)  -- (1,0,1);
	\draw[edge] (1,1,1/2)  -- (1/2,1,1) -- (1,1/2,1) -- cycle;
	\draw[edge] (1,1/2,1)  -- (1,0,1);
	\draw[edge] (1,1,1/2)  -- (1,1,0);
	\draw[edge] (1/2,1,1)  -- (0,1,1);
	
	\end{scope}
\end{tikzpicture}

%% file: tikz/badminkowski_example.tex
\begin{tikzpicture}[scale=0.82, line cap=round, line join=round]
  \tikzset{
    poly/.style={draw=Turquoise!25!black, line width=0.46pt, fill=Turquoise!27},
    guide/.style={draw=Turquoise!30, line width=0.35pt}
  }

  \newcommand{\axesandgrid}[1]{
    \begin{scope}[xshift=#1]
      \draw[guide] (0,0) rectangle (2,2);
      \draw[guide] (1,0) -- (1,2);
      \draw[guide] (0,1) -- (2,1);
    \end{scope}
  }

  \axesandgrid{0cm}
  \begin{scope}
    \filldraw[poly] (0,0) -- (2,0) -- (2,2) -- (0,2) -- cycle;
  \end{scope}

  \axesandgrid{3.2cm}
  \begin{scope}[xshift=3.2cm]
    \filldraw[poly] (1,0) -- (1,1) -- (0,1) -- cycle;
  \end{scope}

  \axesandgrid{6.4cm}
  \begin{scope}[xshift=6.4cm]
    \filldraw[poly] (0,0) -- (1,0) -- (1,1) -- (0,1) -- cycle;
  \end{scope}

  \axesandgrid{9.6cm}
  \begin{scope}[xshift=9.6cm]
    \filldraw[poly] (1,0) -- (2,0) -- (2,2) -- (0,2) -- (0,1) -- cycle;
    \draw[Turquoise!45!black, dashed, line width=0.45pt] (0,0) -- (1,0) -- (0,1) -- cycle;
  \end{scope}
\end{tikzpicture}

%% file: tikz/nonexample.tex
\begin{tikzpicture}[
	tdplot_main_coords,
	scale = 2.0,
	back/.style={color=Turquoise!25!black!30, solid, line width=0.35pt},
	edge/.style={color=Turquoise!25!black, line width=0.46pt},
	distinguishededge/.style={color=Turquoise!45!black, line width=1.05pt},
	distinguishedface/.style={fill=Turquoise!45,draw=none},
	paperfacefill0/.style={fill=Turquoise!35,draw=none},
	paperfacefill1/.style={fill=Turquoise!31,draw=none},
	paperfacefill2/.style={fill=Turquoise!28,draw=none},
	paperfacefill3/.style={fill=Turquoise!24,draw=none},
	paperfacefill4/.style={fill=Turquoise!20,draw=none},
	vertex/.style={inner sep=0pt,minimum size=0pt,draw=none,fill=none}]

	\begin{scope}[xshift=-3cm]
	\fill[paperfacefill2] (1,0,0) -- (1,1,0) -- (1,1,1/2) -- (1,0,1) -- cycle;
	\fill[paperfacefill1] (0,1,0) -- (1,1,0) -- (1,1,1/2) -- (0,1,1/2) -- cycle;
	\fill[paperfacefill3] (0,0,1) -- (1,0,1) -- (1,1,1/2) -- (0,1,1/2) -- cycle;
	\draw[edge, back] (0,0,0) -- (1,0,0);
	\draw[edge, back] (0,0,0) -- (0,1,0);
	\draw[edge, back] (0,0,0) -- (0,0,1);

	\node[vertex] at (0,0,0) {};
	\node[vertex] at (1,0,0) {};
	\node[vertex] at (0,1,0) {};
	\node[vertex] at (0,0,1) {};
	\node[vertex] at (1,1,0) {};
	\node[vertex] at (1,0,1) {};
	\node[vertex] at (0,1,1/2) {};
	\node[vertex] at (1,1,1/2) {};

	\draw[edge] (1,0,0)  -- (1,1,0);
	\draw[edge] (1,0,0)  -- (1,0,1);
	\draw[edge] (0,1,0)  -- (1,1,0);
	\draw[edge] (0,1,0)  -- (0,1,1/2);
	\draw[edge] (0,0,1)  -- (0,1,1/2);
	\draw[edge] (0,0,1)  -- (1,0,1);
	\draw[edge] (1,1,0)  -- (1,1,1/2);
	\draw[distinguishededge] (1,0,1)  -- (1,1,1/2);
	\draw[edge] (0,1,1/2)  -- (1,1,1/2);
	
	\end{scope}
	
	\begin{scope}[xshift=-1cm]
	\fill[paperfacefill2] (1,0,0) -- (1,1,0) -- (1,1,0.25) -- (1,0,0.75) -- cycle;
	\fill[paperfacefill1] (0,1,0) -- (1,1,0) -- (1,1,0.25) -- (0.50,1,0.50) -- (0,1,0.50) -- cycle;
	\fill[paperfacefill3] (0,0,1) -- (0.50,0,1) -- (0.50,1,0.50) -- (0,1,0.50) -- cycle;
	\fill[distinguishedface] (1,1,0.25) -- (1,0,0.75) -- (0.50,0,1) -- (0.50,1,0.50) -- cycle;

	\draw[edge,back] (0, 0, 0) -- (1, 0, 0);
	\draw[edge,back] (0, 0, 1) -- (0, 0, 0);
	\draw[edge,back] (0, 1, 0) -- (0, 0, 0);

	\draw[edge] (1, 1, 0) -- (1, 0, 0);
	\draw[edge] (1, 0, 0) -- (1, 0, 0.75);
	\draw[edge] (0, 1, 0.50) -- (0, 0, 1);
	\draw[edge] (0, 1, 0.50) -- (0, 1, 0);
	\draw[edge] (0, 1, 0.50) -- (0.50, 1, 0.50);
	\draw[edge] (0, 0, 1) -- (0.50, 0, 1);
	\draw[edge] (1, 1, 0) -- (0, 1, 0);
	\draw[edge] (1, 1, 0) -- (1, 1, 0.25);
	\draw[edge] (1, 1, 0.25) -- (1, 0, 0.75);
	\draw[edge] (1, 1, 0.25) -- (0.50, 1, 0.50);
	\draw[edge] (1, 0, 0.75) -- (0.50, 0, 1);
	\draw[edge] (0.50, 0, 1) -- (0.50, 1, 0.50);

	\node[vertex] at (1, 0, 0)     {};
	\node[vertex] at (0, 1, 0.50)     {};
	\node[vertex] at (0, 0, 1)     {};
	\node[vertex] at (1, 1, 0)     {};
	\node[vertex] at (0, 1, 0)     {};
	\node[vertex] at (0, 0, 0)     {};
	\node[vertex] at (1, 1, 0.25)     {};
	\node[vertex] at (1, 0, 0.75)     {};
	\node[vertex] at (0.50, 0, 1)     {};
	\node[vertex] at (0.50, 1, 0.50)     {};
		
	\end{scope}

	\begin{scope}[xshift=1cm]
	\fill[paperfacefill2] (1,0,0) -- (1,1,0) -- (1,0,0.5) -- cycle;
	\fill[paperfacefill1] (0,1,0) -- (1,1,0) -- (0,1,0.5) -- cycle;
	\fill[distinguishedface] (0,0,1) -- (1,0,0.5) -- (1,1,0) -- (0,1,0.5) -- cycle;

	\draw[edge,back] (1, 0, 0) -- (0, 0, 0);
	\draw[edge,back] (0, 0, 1) -- (0, 0, 0);
	\draw[edge,back] (0, 1, 0) -- (0, 0, 0);
	
	\node[vertex] at (1, 0, 0)     {};

	\draw[edge] (1, 0, 0) -- (1, 0, 0.5);
	\draw[edge] (1, 0, 0) -- (1, 1, 0);
	\draw[edge] (0, 1, 0.5) -- (0, 0, 1);
	\draw[edge] (0, 1, 0.5) -- (1, 1, 0);
	\draw[edge] (0, 1, 0.5) -- (0, 1, 0);
	\draw[edge] (0, 0, 1) -- (1, 0, 0.5);
	\draw[edge] (1, 1, 0) -- (1, 0, 0.5);
	\draw[edge] (1, 1, 0) -- (0, 1, 0);

	\node[vertex] at (0, 1, 0.5)     {};
	\node[vertex] at (0, 0, 1)     {};
	\node[vertex] at (1, 1, 0)     {};
	\node[vertex] at (1, 0, 0.5)     {};
	\node[vertex] at (0, 1, 0)     {};
	\node[vertex] at (0, 0, 0)     {};
	\end{scope}

	\begin{scope}[xshift=3cm]
	\fill[paperfacefill4] (1,0,0) -- (1,1,0) -- (0,0,1) -- cycle;
	\fill[paperfacefill1] (0,1,0) -- (1,1,0) -- (0,1,1/2) -- cycle;
	\fill[paperfacefill4] (0,0,1) -- (1,1,0) -- (0,1,1/2) -- cycle;
	\draw[edge, back] (0,0,0) -- (1,0,0);
	\draw[edge, back] (0,0,0) -- (0,1,0);
	\draw[edge, back] (0,0,0) -- (0,0,1);

	\node[vertex] at (0,0,0) {};
	\node[vertex] at (1,0,0) {};
	\node[vertex] at (0,1,0) {};
	\node[vertex] at (0,0,1) {};
	\node[vertex] at (1,1,0) {};
	\node[vertex] at (0,1,1/2) {};

	\draw[edge] (1,0,0)  -- (1,1,0);
	\draw[edge] (0,1,0)  -- (1,1,0);
	\draw[edge] (0,1,0)  -- (0,1,1/2);
	\draw[edge] (0,0,1)  -- (0,1,1/2);
	\draw[edge] (0,1,1/2)  -- (1,1,0);
	\draw[edge] (0,0,1)  -- (1,0,0);
	\draw[edge] (0,0,1)  -- (1,1,0);

	\end{scope}
\end{tikzpicture}

%% file: tikz/factor.tex
\begin{tikzpicture}[
	tdplot_main_coords,
	scale = 2.0,
	back/.style={color=Turquoise!30, solid, line width=0.35pt},
	edge/.style={color=Turquoise!25!black, line width=0.46pt},
	distinguishedface/.style={fill=Turquoise!45,draw=none},
	distinguishededge/.style={color=Turquoise!45!black,line width=1.05pt},
	facefar/.style={fill=Turquoise!20,draw=none},
	facemid/.style={fill=Turquoise!27,draw=none},
	facenear/.style={fill=Turquoise!35,draw=none},
	vertex/.style={inner sep=0pt,minimum size=0pt,draw=none,fill=none},
	vertex_f/.style={inner sep=0.8pt,circle,draw=Turquoise!45!black,fill=Turquoise!45!black}]

	\begin{scope}[xshift= -2cm]
	\fill[facefar] (0,0,0) -- (0,1,0) -- (0,1,1) -- (0,0,1) -- cycle;
	\fill[facemid] (0,1,0) -- (1,1,0) -- (0,1,1) -- cycle;
	\fill[distinguishedface] (1,0,0) -- (1,1,0) -- (0,1,1) -- (0,0,1) -- cycle;
	\draw[edge, back] (0,0,0) -- (1,0,0);
	\draw[edge, back] (0,0,0) -- (0,1,0);
	\draw[edge, back] (0,0,0) -- (0,0,1);

	\node[vertex] at (0,0,0) {};
	\node[vertex] at (1,0,0) {};
	\node[vertex] at (0,1,0) {};
	\node[vertex] at (0,0,1) {};
	\node[vertex] at (1,1,0) {};
	\node[vertex] at (0,1,1) {};

	\draw[edge] (1,0,0)  -- (1,1,0);
	\draw[edge] (0,1,0)  -- (1,1,0);
	\draw[edge] (0,1,0)  -- (0,1,1);
	\draw[edge] (0,0,1)  -- (0,1,1);
	\draw[edge] (0,1,1)  -- (1,1,0);
	\draw[edge] (0,0,1)  -- (1,0,0);
	\draw[edge] (0,0,1)  -- (0,1,1);

	\draw[distinguishededge, dashed] (1,1,1)  -- (1,0,1);
	\node[vertex_f] at (1,1,1) {};
	\node[vertex_f] at (1,0,1) {};
	
	\end{scope}

	\begin{scope}[xshift= -1cm]
	
	\node at (1/2,1/2,1/2) {$ = $};
	\end{scope}

	\begin{scope}[xshift= 0cm]

	\draw[edge] (0,1,0)  -- (0,0,0);
	
	\end{scope}
	
	\begin{scope}[xshift= 1cm]
	
	\node at (1/2,1/2,1/2) {$ \times  $};
	\end{scope}

	\begin{scope}[xshift= 1cm]
	\fill[facenear] (0,1,1)  -- (1,1,0) -- (0,1,0) -- cycle;
	\draw[edge] (0,1,1)  -- (1,1,0) -- (0,1,0) -- cycle;
	\draw[distinguishededge] (0,1,1) -- (1,1,0);
	\end{scope}

\end{tikzpicture}

%% file: tikz/Img_poly.tex
\begin{tikzpicture}%
	[x={(-0.196524cm, 0.304119cm)},
		y={(0.980499cm, 0.061010cm)},
		z={(-0.000056cm, 0.950678cm)},
		scale=1.5,
		back/.style={color=Turquoise!25!black!30, solid,
			line width=0.35pt},
		edge/.style={color=Turquoise!25!black, line width=0.46pt},
		facefar/.style={fill=Turquoise!20, draw=none},
		facemid/.style={fill=Turquoise!28, draw=none},
		facenear/.style={fill=Turquoise!35, draw=none},
		distinguishedface/.style={fill=Turquoise!45, draw=none},
		distinguishededge/.style={color=Turquoise!45!black,
			line width=1.05pt}]
%
%
\coordinate (-2, 0, 2) at (-2, 0, 2);
\coordinate (-1, -1, 1) at (-1, -1, 1);
\coordinate (-1, -1, 3) at (-1, -1, 3);
\coordinate (-1, 1, 1) at (-1, 1, 1);
\coordinate (-1, 1, 3) at (-1, 1, 3);
\coordinate (0, -2, 2) at (0, -2, 2);
\coordinate (0, 0, 4) at (0, 0, 4);
\coordinate (0, 2, 2) at (0, 2, 2);
\coordinate (1, -1, 1) at (1, -1, 1);
\coordinate (1, -1, 3) at (1, -1, 3);
\coordinate (1, 1, 1) at (1, 1, 1);
\coordinate (1, 1, 3) at (1, 1, 3);
\coordinate (2, 0, 2) at (2, 0, 2);
\fill[facemid] (2, 0, 2) -- (1, -1, 3) -- (0, 0, 4) -- (1, 1, 3) -- cycle;
\fill[facefar] (2, 0, 2) -- (1, -1, 1) -- (0, -2, 2) -- (1, -1, 3) -- cycle;
\fill[facefar] (2, 0, 2) -- (1, 1, 1) -- (0, 2, 2) -- (1, 1, 3) -- cycle;
\fill[facenear] (1, -1, 1) -- (-1, -1, 1) -- (0, -2, 2) -- cycle;
\fill[distinguishedface] (1, 1, 1) -- (-1, 1, 1) -- (-1, -1, 1) -- (1, -1, 1) -- cycle;
\fill[facemid] (1, 1, 1) -- (-1, 1, 1) -- (0, 2, 2) -- cycle;
\fill[facenear] (2, 0, 2) -- (1, -1, 1) -- (1, 1, 1) -- cycle;
\draw[back] (-2, 0, 2) -- (-1, -1, 1);
\draw[back] (-2, 0, 2) -- (-1, -1, 3);
\draw[back] (-2, 0, 2) -- (-1, 1, 1);
\draw[back] (-2, 0, 2) -- (-1, 1, 3);
\draw[back] (-1, -1, 3) -- (0, -2, 2);
\draw[back] (-1, -1, 3) -- (0, 0, 4);
\draw[back] (-1, 1, 3) -- (0, 0, 4);
\draw[back] (-1, 1, 3) -- (0, 2, 2);
\draw[edge] (-1, -1, 1) -- (-1, 1, 1);
\draw[edge] (-1, -1, 1) -- (0, -2, 2);
\draw[edge] (-1, -1, 1) -- (1, -1, 1);
\draw[edge] (-1, 1, 1) -- (0, 2, 2);
\draw[edge] (-1, 1, 1) -- (1, 1, 1);
\draw[edge] (0, -2, 2) -- (1, -1, 1);
\draw[edge] (0, -2, 2) -- (1, -1, 3);
\draw[edge] (0, 0, 4) -- (1, -1, 3);
\draw[edge] (0, 0, 4) -- (1, 1, 3);
\draw[edge] (0, 2, 2) -- (1, 1, 1);
\draw[edge] (0, 2, 2) -- (1, 1, 3);
\draw[edge] (1, -1, 1) -- (1, 1, 1);
\draw[edge] (1, -1, 1) -- (2, 0, 2);
\draw[edge] (1, -1, 3) -- (2, 0, 2);
\draw[edge] (1, 1, 1) -- (2, 0, 2);
\draw[edge] (1, 1, 3) -- (2, 0, 2);
\draw[distinguishededge]
	(1, 1, 1) -- (-1, 1, 1) -- (-1, -1, 1) -- (1, -1, 1) -- cycle;
\end{tikzpicture}

%% file: tikz/schlegel.tex
\begin{tikzpicture}[
scale = 1.5,
edge/.style={color=Turquoise!25!black, line width=0.46pt},
distinguishededge/.style={color=Turquoise!45!black, line width=1.05pt},
vertex/.style={inner sep=0.75pt,circle,draw=none,
	fill=Turquoise!45!black}]

\node[vertex] at (2,2)     {};
\node[vertex] at (-2,2)     {};
\node[vertex] at (2,-2)     {};
\node[vertex] at (-2,-2)     {};

\node[vertex] at (0,1.5)     {};
\node[vertex] at (1.5,0)     {};
\node[vertex] at (0,-1.5)     {};
\node[vertex] at (-1.5,0)     {};

\node[vertex] at (0.4,0.4)     {};
\node[vertex] at (-0.4,0.4)     {};
\node[vertex] at (0.4,-0.4)     {};
\node[vertex] at (-0.4,-0.4)     {};

\node[vertex] at (0,0)     {};

\fill[fill=Turquoise!10, draw=none]
	(2,2) -- (-2,2) -- (-2,-2) -- (2,-2) -- cycle;
\draw[distinguishededge]
	(2,2) -- (-2,2) -- (-2,-2) -- (2,-2) -- cycle;
\draw[edge] (2,2) -- (0,1.5) -- (-2,2) -- (-1.5,0) -- (-2,-2) -- (0,-1.5) -- (2,-2) -- (1.5,0) -- (2,2);
\draw[edge] (0,1.5) -- (0.4,0.4) -- (1.5,0) -- (0.4,-0.4) -- (0,-1.5) -- (-0.4, -0.4) -- (-1.5,0) -- (-0.4, 0.4) -- cycle;
\draw[edge] (0,0) -- (0.4,0.4);
\draw[edge] (0,0) -- (-0.4,0.4);
\draw[edge] (0,0) -- (0.4,-0.4);
\draw[edge] (0,0) -- (-0.4,-0.4);


%
\end{tikzpicture}

%% file: tikz/graph.tex
\begin{tikzpicture}[
scale = 1.5,
edge/.style={color=Turquoise!25!black, line width=0.46pt},
vertex/.style={inner sep=0.75pt,circle,draw=none,
	fill=Turquoise!45!black},
classone/.style={color=Turquoise!85!black, line width=0.9pt},
classtwo/.style={color=Turquoise!70!black, line width=0.9pt},
classthree/.style={color=Turquoise!55!black, line width=0.9pt},
classfour/.style={color=Turquoise!40!black, line width=0.9pt}]

\node[vertex] at (2,2)     {};
\node[vertex] at (-2,2)     {};
\node[vertex] at (2,-2)     {};
\node[vertex] at (-2,-2)     {};

\node[vertex] at (0,1.5)     {};
\node[vertex] at (1.5,0)     {};
\node[vertex] at (0,-1.5)     {};
\node[vertex] at (-1.5,0)     {};

\node[vertex] at (0.4,0.4)     {};
\node[vertex] at (-0.4,0.4)     {};
\node[vertex] at (0.4,-0.4)     {};
\node[vertex] at (-0.4,-0.4)     {};

\node[vertex] at (0,0)     {};

\fill[fill=Turquoise!10, draw=none]
	(2,2) -- (-2,2) -- (-2,-2) -- (2,-2) -- cycle;
\draw[edge] (2,2) --  (-2,2) -- (-2,-2) -- (2,-2) -- cycle;
\draw[edge] (2,2) -- (0,1.5) -- (-2,2) -- (-1.5,0) -- (-2,-2) -- (0,-1.5) -- (2,-2) -- (1.5,0) -- (2,2);
\draw[edge] (0,1.5) -- (0.4,0.4) -- (1.5,0) -- (0.4,-0.4) -- (0,-1.5) -- (-0.4, -0.4) -- (-1.5,0) -- (-0.4, 0.4) -- cycle;
\draw[edge] (0,0) -- (0.4,0.4);
\draw[edge] (0,0) -- (-0.4,0.4);
\draw[edge] (0,0) -- (0.4,-0.4);
\draw[edge] (0,0) -- (-0.4,-0.4);

\draw[classone] (1,-1.75) to [out = 60, in = -60] (0.95, -0.2);
\draw[classone] (0.95, -0.2) to [out = 120, in = -45] (0.2, 0.2);
\draw[classone] (0.2, 0.2) to [out = 135, in = -15] (-0.2, 0.95);
\draw[classone] (-0.2, 0.95) to [out = 165, in = 15] (-1.75, 1);

\draw[classtwo] (1.75,1) to [out = 150, in = 30] ( 0.2,0.95);
\draw[classtwo] ( 0.2,0.95) to [out = 210, in = 45] ( -0.2,0.2);
\draw[classtwo] ( -0.2,0.2) to [out = 225, in = 75] ( -0.95,-0.2);
\draw[classtwo] ( -0.95,-0.2) to [out = 255, in = 105] ( -1,-1.75);

\draw[classthree] (-1,1.75) to [out = 240, in = 120] (-0.95, 0.2);
\draw[classthree] (-0.95, 0.2) to [out = 300, in = 135] (-0.2, -0.2);
\draw[classthree] (-0.2, -0.2) to [out = 315, in = 165] (0.2, -0.95);
\draw[classthree] (0.2, -0.95) to [out = 345, in = 195] (1.75, -1);

\draw[classfour] (-1.75,-1) to [out = 330, in = 210] (-0.2,-0.95);
\draw[classfour] (-0.2,-0.95) to [out = 390, in = 225] (0.2,-0.2);
\draw[classfour] (0.2,-0.2) to [out = 405, in = 255] (0.95,0.2);
\draw[classfour] (0.95,0.2) to [out = 435, in = 285] (1,1.75);

\end{tikzpicture}